\documentclass[10pt]{article}

\usepackage{amsmath,amsfonts,amsthm,amssymb,mathrsfs}
\usepackage[T1]{fontenc}
\usepackage[utf8]{inputenc}
\usepackage{natbib} 
\usepackage[colorlinks,backref]{hyperref}
\usepackage{xcolor}
\usepackage{fullpage}
\usepackage{mathtools}
\usepackage{authblk}
\RequirePackage{tikz}
\usetikzlibrary{fit,positioning,arrows,automata,calc}
\tikzset{
   main/.style={circle, minimum size = 10mm, thick, 
        draw =black!80, node distance = 10mm},
   box/.style={rectangle, draw=black!100}
}
\usetikzlibrary{positioning}
\usepackage{enumitem}
\setlist{  
  listparindent=\parindent,
  parsep=0pt
}
\usepackage{thmtools}
\usepackage{thm-restate}
\usepackage[capitalize]{cleveref}
\crefname{assumption}{Assumption}{Assumptions}
\usepackage{titlesec}
\titleformat*{\section}{\large\bfseries}

\titleformat*{\subsection}{\normalfont\bfseries}

 \newcommand\numberthis{\addtocounter{equation}{1}\tag{\theequation}}

\DeclareMathOperator*{\argmin}{arg\,min}
\DeclareMathOperator*{\argmax}{arg\,max}

\newcommand{\eps}{\ensuremath{\varepsilon}}
   
\newcommand{\s}{^*}
\renewcommand{\t}{^{\top}}

\newcommand{\inv}{^{-1}}

\newcommand{\p}{\mathbb{P}}

\renewcommand{\hat}{\widehat}

\renewcommand{\tilde}{\widetilde}

\newtheorem{theorem}{Theorem}
\newtheorem{corollary}{Corollary}
\newtheorem{lemma}{Lemma}

\theoremstyle{definition}

\newtheorem{assumption}{Assumption}
\setcitestyle{round}
\hypersetup{%
  colorlinks=true,
  linkcolor=blue,
  citecolor = blue,
  urlcolor = blue
}

\usepackage{mathtools}
\usepackage{algorithmic}
\usepackage{algorithm}
\usepackage{bm}

\usepackage{multicol,caption}
\usepackage{natbib} 
\bibpunct{(}{)}{;}{a}{,}{,}
\title{A High-Dimensional Statistical Theory for Convex and Nonconvex Matrix Sensing}
\author{Joshua Agterberg\thanks{Department of Statistics, University of Illinois Urbana-Champaign} \and Ren\'e Vidal\thanks{Center for Innovation in Data Engineering and Science (IDEAS), Departments of Electrical and Systems Engineering, Radiology, Computer and Information Science, Statistics and Data Science, University of Pennsylvania}}
\date{\today}
\begin{document}

\maketitle

\begin{abstract}
The problem of matrix sensing, or trace regression, is a problem wherein one wishes to estimate a low-rank matrix from linear measurements perturbed with noise. A number of existing works have studied both convex and nonconvex approaches to this problem, establishing minimax error rates when the number of measurements is sufficiently large relative to the rank and dimension of the low-rank matrix, though a precise comparison of these procedures still remains unexplored.   
In this work we provide a high-dimensional   statistical analysis for symmetric low-rank matrix sensing observed under Gaussian measurements and noise.  
Our main result describes a novel phenomenon: in this statistical model and in an appropriate asymptotic regime, the behavior of any local minimum of the nonconvex factorized approach (with known rank)  is approximately  equivalent to that of the matrix hard-thresholding of a corresponding matrix denoising problem, and the behavior of the convex nuclear-norm regularized least squares approach is approximately equivalent to that of matrix soft-thresholding of the same matrix denoising problem.  Here ``approximately equivalent'' is understood in the sense of concentration of Lipchitz functions.  As a consequence, the nonconvex procedure uniformly dominates the convex approach in mean squared error. 
 Our arguments are based on a matrix operator generalization of the Convex Gaussian Min-Max Theorem (CGMT) together with studying the interplay between local minima of the convex and  nonconvex formulations and their ``debiased'' counterparts, and several of these results may be of independent interest. 
\end{abstract}

\tableofcontents

\section{Introduction}

In the \emph{matrix sensing} or \emph{trace regression} problem, the goal is to estimate a symmetric matrix $\bm{M}$ from $n$ observations  $(y_i,\bm{X}_i) \in \mathbb{R} \times \mathbb{R}^{d\times d}$ such that
\begin{align}
    y_i = \frac{1}{\sqrt{n}}\langle \bm{X}_i, \bm{M} \rangle + \eps_i, \numberthis \label{matsens}
\end{align}
where $\langle \cdot, \cdot \rangle$ denotes the Frobenius inner product, $\eps_i$ is observation noise, and $\bm{M}$ is assumed to be a symmetric positive semidefinite rank $r$ signal matrix of interest with $r$ typically taken to be much smaller than the dimension $d$ (the scaling by $\frac{1}{\sqrt{n}}$ is primarily for theoretical purposes). Equation \eqref{matsens}  serves as a simple-but-practical model for high-dimensional matrix regression, and this model has drawn attention from fields such as computer science, engineering, and statistics.   This problem can be viewed as a matrix generalization of the noisy sparse recovery problem, in which one instead observes linear measurements of a sparse signal vector  perturbed by random noise, with the rank of the signal matrix $\bm{M}$ playing the role of sparsity.

To address the low-rank structure of the signal matrix, drawing inspiration from the compressed sensing literature, initial work focused on a penalized least-squares approach to estimation, often referred to as the Matrix LASSO.
Explicitly,  the Matrix LASSO with tuning parameter $\lambda$, denoted $\bm{Z}^{(\lambda)}$, solves the problem
\begin{align}
    \bm{Z}^{(\lambda)} = \argmin_{\bm{Z} \succcurlyeq 0} g_{\sf cvx}^{(\lambda)}(\bm{Z}) :=  \frac{1}{2} \sum_{i=1}^{n} (y_i - \langle \bm{X}_i, \bm{Z}\rangle/\sqrt{n} )^2 + \lambda \| \bm{Z} \|_*,
\end{align}
where $\|\cdot\|_*$ denotes the nuclear norm on $d\times d$ matrices (here we take any minimizer of the objective function).  
The right-hand side is convex in the optimization variable $\bm{Z}$, hence the minimizer is computable in polynomial time. Much work has been devoted to the study of $\bm{Z}^{(\lambda)}$, including conditions on the sensing matrices $\bm{X}_i$ and the tuning parameter $\lambda$ such that approximate recovery is possible.  As a representative example, the work of \citet{candes_tight_2011} proved that for Gaussian observations $\{\bm{X}_i\}_{i=1}^{n}$, when $n \geq C_0 dr$ and $\eps_i \sim \mathcal{N}(0,\sigma^2)$, it holds that 
\begin{align}
\label{eq:MLASSO-recovery}
    \| \bm{Z}^{(\lambda)} - \bm{M}\|_F^2 \leq C \sigma^2 dr
\end{align}
with high probability as long as $\lambda$ is taken to be $C' \sigma \sqrt{d}$, where $C, C',$ and $C_0$  are some universal constants.  Moreover, \citet{candes_tight_2011} show that this procedure is minimax rate-optimal in the sense that no estimator can improve the order of the right hand side of Equation \eqref{eq:MLASSO-recovery}.

While the estimator $\bm{Z}^{(\lambda)}$ is computable in polynomial time, computing it in practice often requires running an iterative algorithm such as projected gradient descent. A major drawback to projected gradient descent and its variants is that each iteration requires taking a (full) singular value decomposition of a $d\times d$ matrix, which can be prohibitive when $d$ is large. To ameliorate the need for multiple full SVDs, many researchers have instead proposed using a pre-factorized version of the problem that operates directly on the factors of the matrix, often referred to as the  ``Burer-Monteiro'' factorization from the work of \citet{burer_nonlinear_2003}.  Explicitly, given a pre-specified rank $r$, the nonconvex factorized estimator $\bm{U}^{(\lambda)} \in \mathbb{R}^{d\times r}$ solves the problem
 \begin{align}
     \bm{U}^{(\lambda)} = \argmin_{\bm{U} \in \mathbb{R}^{d\times r}} f_{{\sf ncvx}}^{(\lambda)}(\bm{U}) := \frac{1}{4} \sum_{i=1}^{n} (y_i - \langle \bm{X}_i, \bm{UU}\t \rangle/\sqrt{n} )^2 + \frac{\lambda}{2} \| \bm{U} \|_F^2.
 \label{eq:Burer-Monteiro}
 \end{align}
  The right hand side of \eqref{eq:Burer-Monteiro} is nonconvex in the optimization variable $\bm{U}$, but iterative procedures such as gradient descent have the benefit of operating directly on the low-rank factors, thereby reducing the computational burden when $r$ is small relative to $d$.  In this paper we are interested both in $\lambda > 0$ but also in the special case where $\lambda = 0$, which is often used in practice.  

 Existing work has demonstrated that despite introducing nonconvexity, the particular form of nonconvexity is ``harmless'' in the sense that the optimization landscape is benign (in the noiseless setting) \citep{haeffele2014structured,bhojanapalli_global_2016,li_symmetry_2019,ma_optimization_2023}, or that local optima exhibit minimax-optimal statistical guarantees similar to existing results for the convex estimator $\bm{Z}^{(\lambda)}$ \citep{ge_no_2017,negahban_estimation_2011,tu_low-rank_2016,zheng_convergent_2015}. See \cref{sec:relatedwork} for more details.
 
 Despite these advancements, a more comprehensive statistical theory is still missing from the literature.  
 To draw an analogy, in the statistical literature on high-dimensional linear regression, 
 a precise asymptotic statistical theory for the LASSO estimator has emerged over a series of works \citep{bayati_lasso_2012,javanmard_confidence_2014,javanmard_debiasing_2018,bellec_debiasing_2023,celentano_lasso_2023,miolane_distribution_2021,thrampoulidis_precise_2018,thrampoulidis_regularized_2015}.  
 Therefore, partially motivated by these findings, in this paper we are interested in answering the following question:
 \begin{quote}
   \emph{What is the asymptotic behavior of $\bm{Z}^{(\lambda)}$ and $\bm{U}^{(0)}$ when $d^2 \gg n \gg d$ with $r$ held fixed?}
 \end{quote}
 It is worth noting that in our main results   we will allow $r$ to grow with $n$ so that $n \gg d r^2$.  We deliberately focus on the regime $d^2 \gg n \gg d$ since if $n \gg d^2$, then we are once again in the ``low-dimensional'' regime where least squares is consistent without any additional regularization.

The key takeaway from this work can be stated informally as follows:
\begin{quote} 
\emph{The convex estimator $\bm{Z}^{(\lambda)}$ behaves asymptotically equivalently to matrix soft thresholding, and the nonconvex estimator $\bm{U}^{(0)}$ behaves asymptotically equivalently to matrix hard thresholding, where both thresholding procedures are computed on a noisy version of $\bm{M}$. }
\end{quote}
In the subsequent sections we formalize what we mean by ``behaves asymptotically equivalently to.''  Importantly, this asymptotic equivalence also implies a form of statistical dominance: namely, the nonconvex estimator has a mean-squared error strictly smaller than the convex estimator.

\subsection{Notation}
For a veector $\bm{v}$, $\|\bm{v}\|$ denotes its Euclidean $\ell_2$ norm, and for a matrix $\bm{X}$, $\|\bm{X}\|, \|\bm{X}\|_*$ and $\|\bm{X}\|_F$ denote its spectral, nuclear, and Frobenius norm, respectively.   We write ${\sf vec}(\bm{X})$ to denote the (half) vectorization of a (symmetric) matrix.     We write $\lambda_i(\bm{X})$ to denote the $i$'th largest eigenvalue of a matrix, and we let $\lambda_{\min}(\bm{X})$ denote its smallest nonzero eigenvalue.     For a general function $f: \mathbb{R}^{d \times r} \to \mathbb{R}$ we let $\nabla f$ denote its gradient, viewed as an element of $\mathbb{R}^{d \times r}$ and $\nabla^2$ denote its Hessian, viewed as an operator on $\mathbb{R}^{d\times r} \times \mathbb{R}^{d\times r}$. For two matrices $\bm{U}$ and $\bm{U}'$ of  dimension $d \times r$, we denote $\mathcal{O}_{\bm{U},\bm{U}'}$ via
\begin{align}
    \mathcal{O}_{\bm{U},\bm{U}'} &= \argmin_{\mathcal{OO}\t = \bm{I}_r} \| \bm{U} \mathcal{O} - \bm{U}' \|_F.
\end{align}

 Throughout we let $C$ or $c$ denote a universal constant that may change from line to line.  For two sequences $a_n$ and $b_n$, we say $a_n \ll b_n$ if $a_n/b_n \to 0$ as $n \to\infty$, and we say $a_n \lesssim b_n$ if $a_n \leq C b_n$.  We also write $a_n = O(b_n)$ if for all $n$ sufficiently large $a_n \leq C b_n$.  Similarly, we write $a_n = o(b_n)$ if $a_n \ll b_n$, and we write $a_n = \tilde O(b_n)$ if there exists some constant $c > 0$ such that $a_n = O( \log^c(n) b_n )$.  

We write $\mathcal{N}(\mu,\sigma^2)$ to denote the normal distribution with mean $\mu$ and variance $\sigma^2$.  We say a matrix $\bm{X} \sim {\sf GOE}(d)$ if $\bm{X} \in \mathbb{R}^{d\times d}$ is symmetric with independent $\mathcal{N}(0,1)$ entries on the diagonal and independent $\mathcal{N}(0,1/2)$ entries on the off-diagonal. 
We let $\mathcal{X}: \mathbb{R}^{d\times d} \to \mathbb{R}^{n}$ be the operator with elements defined via $\mathcal{X}(\bm{M})_i = \langle \bm{X}_i, \bm{M} \rangle/\sqrt{n}$, and we let $\mathcal{X}\s$ denote its adjoint.  Throughout the rest of the paper, probabilities and expectations without adornment hold according to the randomness in $\{\bm{X}_i,\eps_i\}_{i=1}^{n}$.

We let $L^2(\mathbb{R}^{d \times d}; \mathbb{R}^{d\times d})$ denote the space of symmetric matrix-valued functions; for two matrix-valued functions $\bm{Q}_1(\cdot), \bm{Q}_2(\cdot): \mathbb{R}^{d\times d} \to \mathbb{R}^{d\times d}$, we define the inner product 
\begin{align}
    \langle \bm{Q}_1, \bm{Q}_2 \rangle_{L_2,F} := \int \langle \bm{Q}_1(\bm{G}), \bm{Q}_2(\bm{G}) \rangle d\mathbb{P}(\bm{G}),
\end{align}
where $\mathbb{P}(\bm{G})$ denotes the density with respect to the Gaussian Orthogonal Ensemble.  We let the caligraphic letter $\mathcal{E}$ denote events.

\subsection{Preliminaries and Technical Assumptions}
\label{sec:prelims}
Suppose we observe $\{y_i,\bm{X}_i\}_{i=1}^{n}$ of the form in \eqref{matsens}, where $(\eps_i,\bm{X}_i) \in \mathbb{R} \times \mathbb{R}^{d\times d}$ are such that 
\begin{align}
     \qquad \bm{X}_i \sim {\sf GOE}(d), \quad \text{and}\quad \eps_i \sim \mathcal{N}(0,\sigma^2). \numberthis \label{model}
\end{align}
Here we recall that $\bm{X}_i \sim {\sf GOE}(d)$  means that $\bm{X}_i$ is a $d \times d$ symmetric matrix with IID $\mathcal{N}(0,1)$ entries on the diagonal and IID $\mathcal{N}(0,1/2)$ entries on the off-diagonal, referred to as the \emph{Gaussian orthogonal ensemble}.

While these assumptions on $\bm{X}_i$ and $\eps_i$ may seem overly idealized, we may view such assumptions as making a ``tradeoff'' in model specificity and result specificity, with more precise results (typically) requiring stronger modeling assumptions.  Moreover, while we assume Gaussianity throughout this work, a series of works in high-dimensional regression  suggest that these assumptions can be generalized to other distributions on $\bm{X}_i$ (e.g., independent entries) via notions of \emph{universality} \citep{dudeja_spectral_2024,dudeja_universality_2023,han_entrywise_2024,han_universality_2023,montanari_universality_2023}. 
Indeed, the works \citet{dudeja_spectral_2024,dudeja_universality_2023} suggest that many existing results continue to hold only requiring certain asymptotic spectral properties of the observations.  
Therefore, while perhaps the statistical model studied herein may seem limited, it is likely that the insights gleaned are not.

Our results relate $\bm{U}^{(\lambda)}, \bm{Z}^{(\lambda)}$, and $\bm{U}^{(0)}$ to a certain corresponding matrix denoising problem.  Suppose one observes $\bm{M} + \sigma \bm{H}$, where  $\bm{H} \sim {{\sf GOE}}(d)$.
Define
\begin{align}
    \bm{Z}_{{\sf ST}}^{(\lambda)} := \argmin_{\bm{Z}} \frac{1}{2} \| \bm{Z} - \bm{M} - \sigma \bm{H} \|_F^2 + \lambda  \| \bm{Z} \|_*.
\end{align}
Equivalently, $\bm{Z}_{{\sf ST}}^{(\lambda)}$ is the \emph{matrix soft-thresholding} of $\bm{M} + \sigma\bm{H}$ at level $\lambda$, which can be computed as follows.  First, suppose that $\bm{M} + \sigma \bm{H}$ has SVD $\bm{\hat V} \bm{\hat \Lambda} \bm{\hat W}\t$.  Then $\bm{Z}_{{\sf ST}}^{(\lambda)} = \bm{\hat V}\big( \bm{\hat \Lambda} - \lambda \bm{I} \big)_+ \bm{\hat W}\t$, where for a diagonal matrix $\bm{D}$, $(\bm{D})_+$ sets the negative entries to zero.  In other words, $\bm{Z}_{{\sf ST}}^{(\lambda)}$ subtracts $\lambda$ from all the singular values of $\bm{M} + \sigma \bm{H}$ and sets the negative singular values to zero.  

 We also define the matrix denoising hard thresholding counterpart
\begin{align}
    \bm{Z}_{{\sf HT}} := \mathcal{P}_{{\sf rank}-r} \big( \bm{M} + \sigma \bm{H} \big),
\end{align}
where $\mathcal{P}_{{\sf rank}-r}$ simply computes the best rank $r$ approximation to $\bm{M} + \sigma \bm{H}$, which, by the Ekhart-Young Theorem, is equivalent to setting the $d-r$ smallest singular values of $\bm{M} + \sigma \bm{H}$ to zero.    We emphasize that  the quantities $\bm{Z}_{{\sf HT}}$ and $\bm{Z}_{{\sf ST}}^{(\lambda)}$ serve only as theoretical tools. Throughout this work we assume that $\bm{Z}_{{\sf HT}}$ is \emph{always} computed assuming that $r$ is known, so there is no ambiguity in its definition.  

In order to state our results we must impose some assumptions. 
Our first assumption concerns the signal strength of the matrix $\bm{M}$. 
\begin{assumption}[Signal Strength] \label{ass1}
The low-rank matrix $\bm{M}$ satisfies $\bm{M} = \sqrt{d} \bm{V\Lambda V}\t$, where $\bm{\Lambda}$ is a diagonal $r\times r$ matrix of positive eigenvalues of $\bm{M}$ and $\bm{V}$ is a $d \times r$ orthonormal matrix of corresponding eigenvectors.  
The matrix $\bm{\Lambda}$ has smallest eigenvalue $\lambda_r$ that satisfies
\begin{align}
    C_1 \sqrt{r} < \lambda_r/\sigma < C_2 \sqrt{r},
\end{align}
where $C_1$ and $C_2$ are some sufficiently large constants. Furthermore, it holds that $\frac{\|\bm{M}\|}{\lambda_r \sqrt{d}} \leq \kappa$ for some $\kappa = O(1)$.  Finally, it holds that $\sigma \geq \sigma_{\min} > c > 0$. 
\end{assumption}
If $\|\bm{M}\|$ is too small relative to $\sigma$, then the zero matrix achieves the optimal error rate, and hence \cref{ass1} ensures the problem is well-defined.   Furthermore, from classical results in random matrix theory,  when $r$ is fixed it is known that the leading eigenvectors of $\bm{M} + \sigma \bm{H}$ are asymptotically nontrivial estimators of $\bm{V}$ when $\lambda_r/\sigma > 1$ \citep{benaych-georges_eigenvalues_2011}.  Similarly, by the lower bounds of \citet{cai_rate-optimal_2018}, the condition $\lambda_r/\sigma \geq C \sqrt{r}$ is necessary for minimax consistency of the leading eigenspace of $\bm{M}$.  Therefore, our assumption  $\lambda_r/\sigma > C_1 \sqrt{r}$, is nearly optimal to ensure that the matrix denoising problem $\bm{M} + \sigma \bm{H}$ admits nontrivial estimators for its eigenspace. 
Finally, we assume that $\sigma$ is bounded below for technical convenience. In principle, our results can be modified to admit the regime where $\sigma \to 0$, but the statements of the results will change, as the probabilities may depend on $\sigma$. Similarly we assume that $\frac{\lambda_r}{\sigma } \leq C_2 \sqrt{r}$ for some large constant $C_2$, which is only a technical assumption to ensure that the signal to noise ratio (the ratio $\frac{\lambda_r}{\sigma \sqrt{r}}$) is of constant order.  This assumption is not strong, and is likely a technical artifact, as a diverging SNR only renders the problem simpler from a statistical perspective.  This regime mimics existing results for the LASSO in the sub-linear sparsity regime; see the discussion after \cref{thm:mainthm}.  

Our second assumption concerns the dimensionality of the problem and ensures we are given sufficiently many measurements.
\begin{assumption}[Sample Size and Dimension] \label{ass2}
It holds that $\gamma_n := \frac{n}{dr}$ satisfies $\gamma_n \geq C_3 r $, where $C_3$ is some sufficiently large constant (possibly depending on $C_1$,  $C_2$, and $\kappa$). 
\end{assumption}
From the minimax bounds of \citet{candes_tight_2011}, minimax consistency requires $\gamma_n \to \infty$, and hence such an assumption is only suboptimal by a factor of $r$.  However, the assumption $n \geq C d r^2$ is pervasive in the literature studying gradient descent and its variants on the factorized problem.  More specifically, such a condition is often required to ensure that the spectral initialization for gradient descent lies within the so-called \emph{basin of attraction}, and it is an open problem to determine whether there is an initialization for gradient descent that converges under the weaker condition $n \geq C d r$. Consequently, \cref{ass2} is not particularly strong.

Our third assumption concerns the choice of the regularization parameter $\lambda$ in the convex estimator $\bm{Z}^{(\lambda)}$:
intuitively, if $\lambda$ is too large, then there is too much regularization, and if $\lambda$ is too small, then there is too little.  
\begin{assumption}[Regularization Parameter] \label{ass3}
    There exist universal constants $C_4$ and $C_5$ (possibly depending on $\kappa$) such that $C_4\sqrt{d} \leq \lambda/\sigma \leq C_5 \sqrt{d}$.  
\end{assumption}
 The choice above matches the order of the regularization parameter taken in \citet{candes_tight_2011} to yield consistent estimation, and hence this choice of $\lambda$ is order-wise optimal.  Throughout we assume that $\lambda$ is fixed \emph{a priori}.

\subsection{Main Results}
We are now prepared to state our main result, \cref{thm:mainthm}, which relates the the properties of $\bm{U}^{(0)}$ and $\bm{Z}^{(\lambda)}$ to $\bm{Z}_{{\sf HT}}$ and $\bm{Z}_{{\sf ST}}^{(\lambda)}$ respectively.

\begin{theorem} \label{thm:mainthm}
Suppose that \cref{ass1,ass2,ass3} hold. 
There exists a constant $C > 0$ depending on our assumptions such that for all $n$ sufficiently large, 
for any $\eps> C\exp( - c dr )$, and for any $1-$Lipschitz function $\phi: \mathbb{R}^{d\times d} \to \mathbb{R}$,
\begin{align}
\begin{split}
    \p\bigg\{ &\bigg| \phi\bigg( \frac{1}{\sqrt{dr}}\bm{U}^{(0)} \bm{U}^{(0)\top} \bigg) - \mathbb{E}_{\bm{H}} \phi\big( \frac{1}{\sqrt{dr}}\bm{Z}_{{\sf HT}} \big) \bigg| > \eps + C \sigma \sqrt{\frac{r}{\gamma_n}}\bigg\} \\&\leq  
 O\bigg( \exp( - c d) +  \exp( - c dr) + \exp( - c n) + \frac{1}{\eps^2} \exp(-c dr \eps^4 ) + \exp( - \frac{(dr)^2}{n} \eps^4) \bigg); \end{split} \\
 \begin{split}
    \p\bigg\{ &\bigg| \phi\bigg( \frac{1}{\sqrt{dr}}\bm{Z}^{(\lambda)} \bigg) - \mathbb{E}_{\bm{H}} \phi\big( \frac{1}{\sqrt{dr}}\bm{Z}_{{\sf ST}}^{(\lambda)} \big) \bigg| > \eps + C \sigma \frac{\sqrt{r}}{\gamma_n}  \bigg\} \\
    &\leq  O\bigg( \exp( - c d) +  \exp( - c dr) + \exp( - c n) + \frac{1}{\eps^2} \exp(-c dr \eps^4 ) + \exp( - \frac{(dr)^2}{n} \eps^4) \bigg) \end{split}
\end{align}
The same result continues to hold if $\bm{Z}^{(\lambda)}$ is replaced with $\bm{U}^{(\lambda)} \bm{U}^{(\lambda)\top}$.  
\end{theorem}

We note that the scaling by $\frac{1}{\sqrt{dr}}$ is simply the scaling required in order for the error to be $\Theta(1)$. In addition, as we will demonstrate in the proof (see \cref{thm:cvxasymptotics}), the error term $C \sqrt{r}/\gamma_n $ for the convex estimator in the second line can be eliminated by replacing $\bm{Z}^{(\lambda)}_{{\sf ST}}$ with a slightly different quantity (see \cref{sec:cvxcontrol} for details).   Furthermore, the requirement that $\eps > C \exp( - c dr)$ is only technical, since if $\eps \ll \exp( -c dr)$, the probability bound becomes essentially vacuous (as it becomes arbitrarily close to one).  Finally, we note that the one of the probability bounds also requires that $n \ll (dr)^2$.  However, as throughout this paper we are primarily interested in $d^2 \gg n$, so this requirement is not stringent.

Before discussing the implications of \cref{thm:mainthm}, we first state the following corollary demonstrating its utility.  The proof can be found in \cref{sec:corproof}.
\begin{corollary} \label{cor:maincor}
Consider the asymptotic regime where $d\to\infty$ with $r$ held fixed and $n = n(d)$ satisfies $ d \ll n(d) \ll d^2$.   Let $\{\bm{M}(d)\}_{d=1}^{\infty}$ denote a sequence of rank $r$ signal matrices, and suppose the eigenvalues of $\bm{M}(d)/\sqrt{d}$ converge deterministically to some limits $\lambda_1,\dots,\lambda_r$ satisfying \cref{ass1}.  Suppose further that \cref{ass1,ass2,ass3} are satisfied for all $n$ and $d$, and suppose $\sigma$ is fixed. Let $\bm{Z}^{(\lambda)} = \bm{Z}^{(\lambda)}(d)$ denote the sequence of convex minimizers, and let $\bm{U}^{(0)} = \bm{U}^{(0)}(d)$ be defined similarly.  
Then it holds that 
\begin{align}
  \liminf_{d\to\infty} \frac{\| \bm{Z}^{(\lambda)} - \bm{M}(d) \|_F^2}{ \| \bm{U}^{(0)}\bm{U}^{(0)\top} - \bm{M}(d) \|_F^2} \geq 1, \numberthis \label{mse}
\end{align}
where the limit is in probability, the infimum is over the sequence of problem instances, and strict inequality can hold for particular sequences of matrices $\{\bm{M}(d)\}_{d=1}^{\infty}$.  
\end{corollary}
Note that the quantity $\| \bm{U}^{(0)} \bm{U}^{(0)\top} - \bm{M} \|_F^2$ is (a rescaled version of) the excess risk for the estimator $\bm{U}^{(0)} \bm{U}^{(0)\top}$, with a similar interpretation for $\| \bm{Z}^{(\lambda)} - \bm{M} \|_F^2$.  
Furthermore, \eqref{mse} implies that asymptotically $\|\bm{Z}^{(\lambda)} - \bm{M}\|_F \geq \| \bm{U}^{(0)} \bm{U}^{(0)\top} - \bm{M}\|_F$, with strict inequality possible. In particular, this implies that the convex estimator has a test error larger than the nonconvex estimator, and hence is inadmissible with respect to the squared loss.

 \begin{figure}[ht!]
 \centering
\includegraphics[width=0.32\textwidth,keepaspectratio]{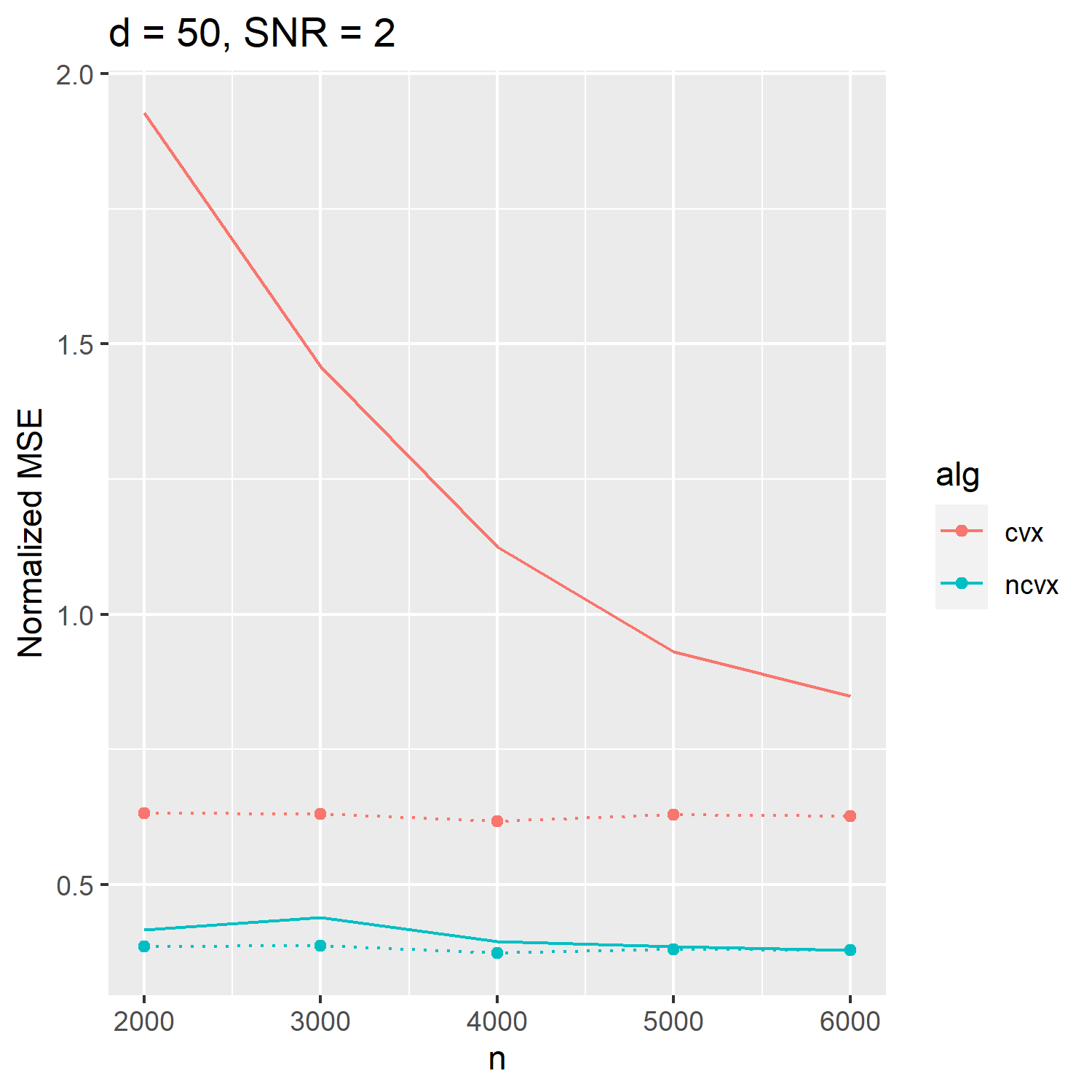}
\includegraphics[width=0.32\textwidth,keepaspectratio]{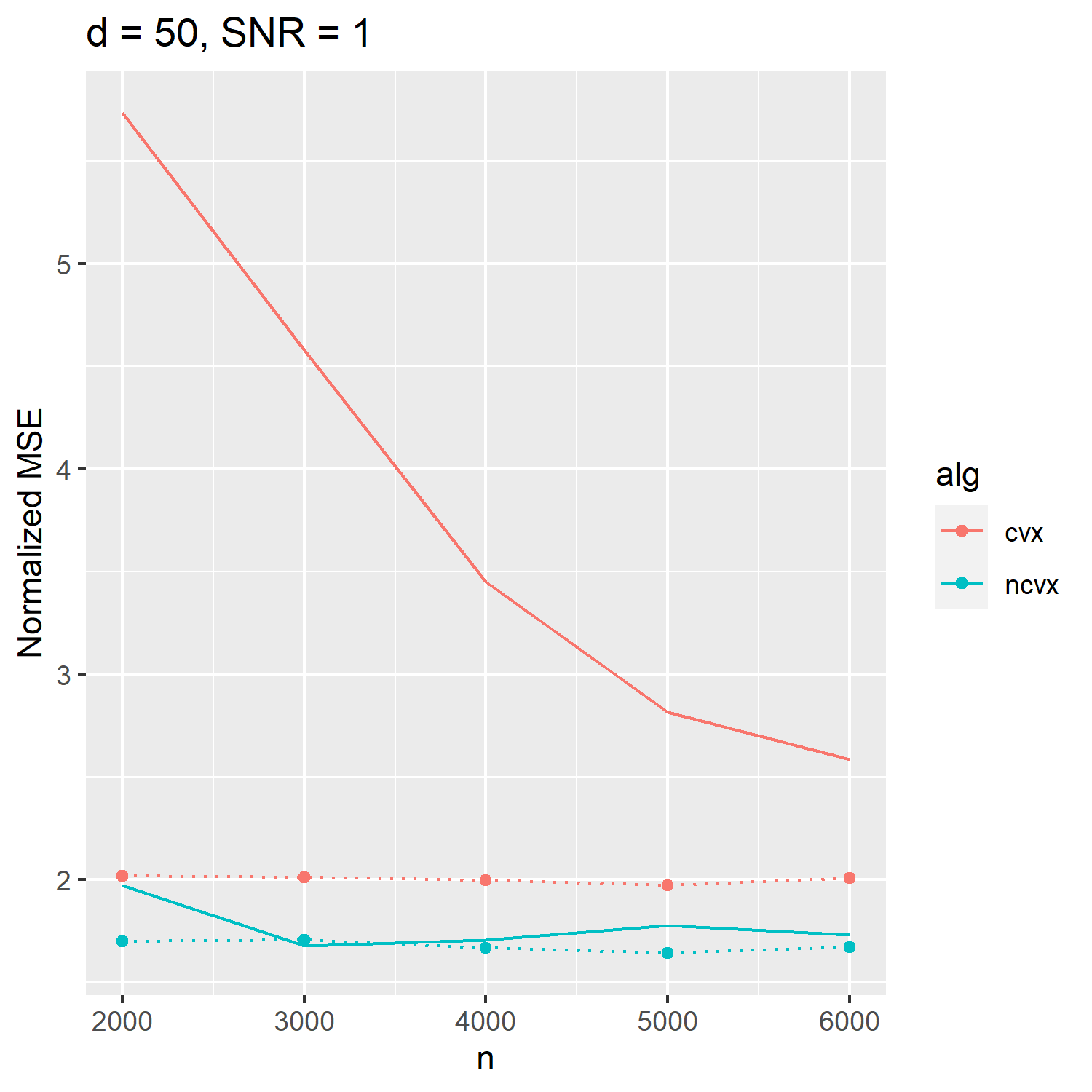}
\includegraphics[width=0.32\textwidth,keepaspectratio]{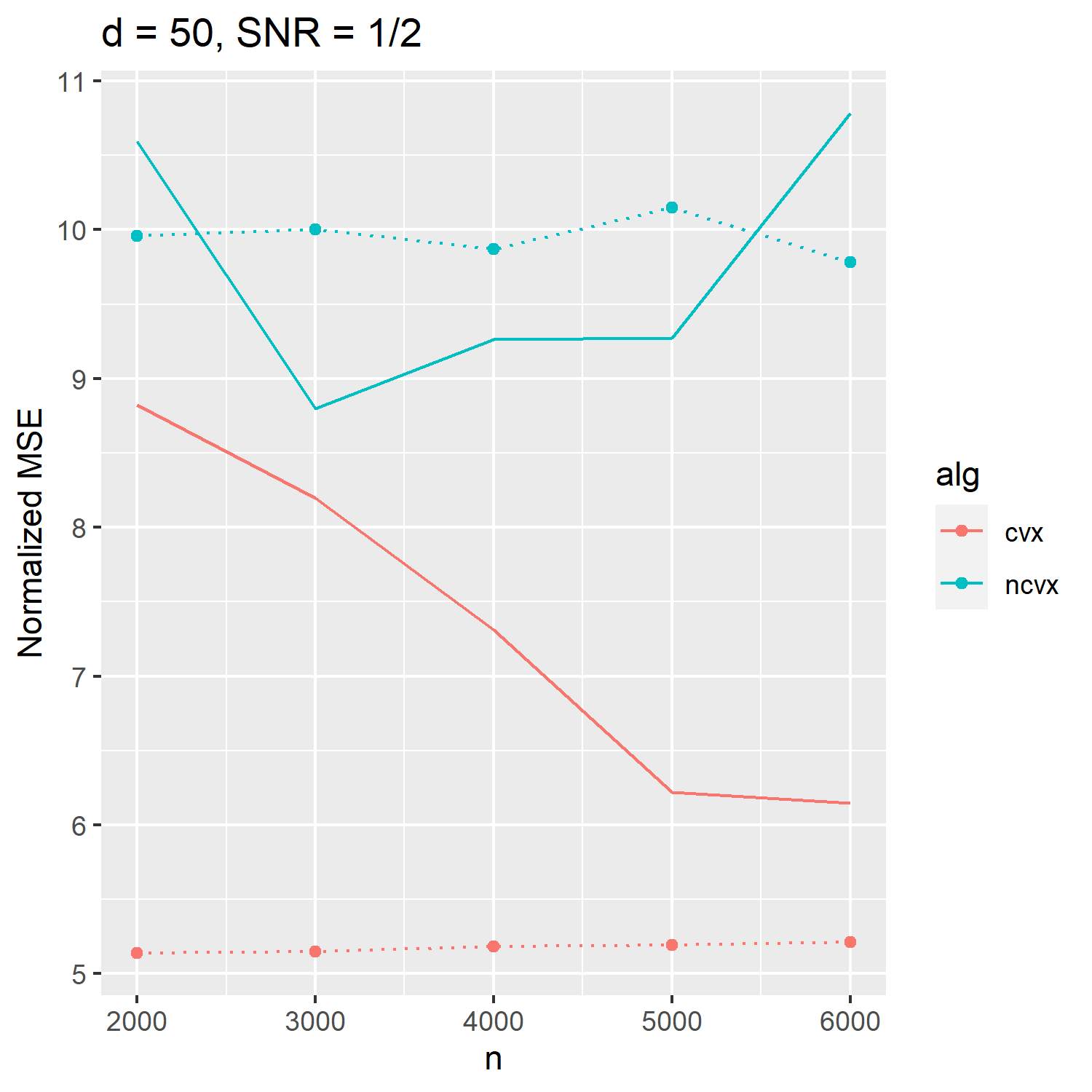}
\includegraphics[width=0.32\textwidth,keepaspectratio]{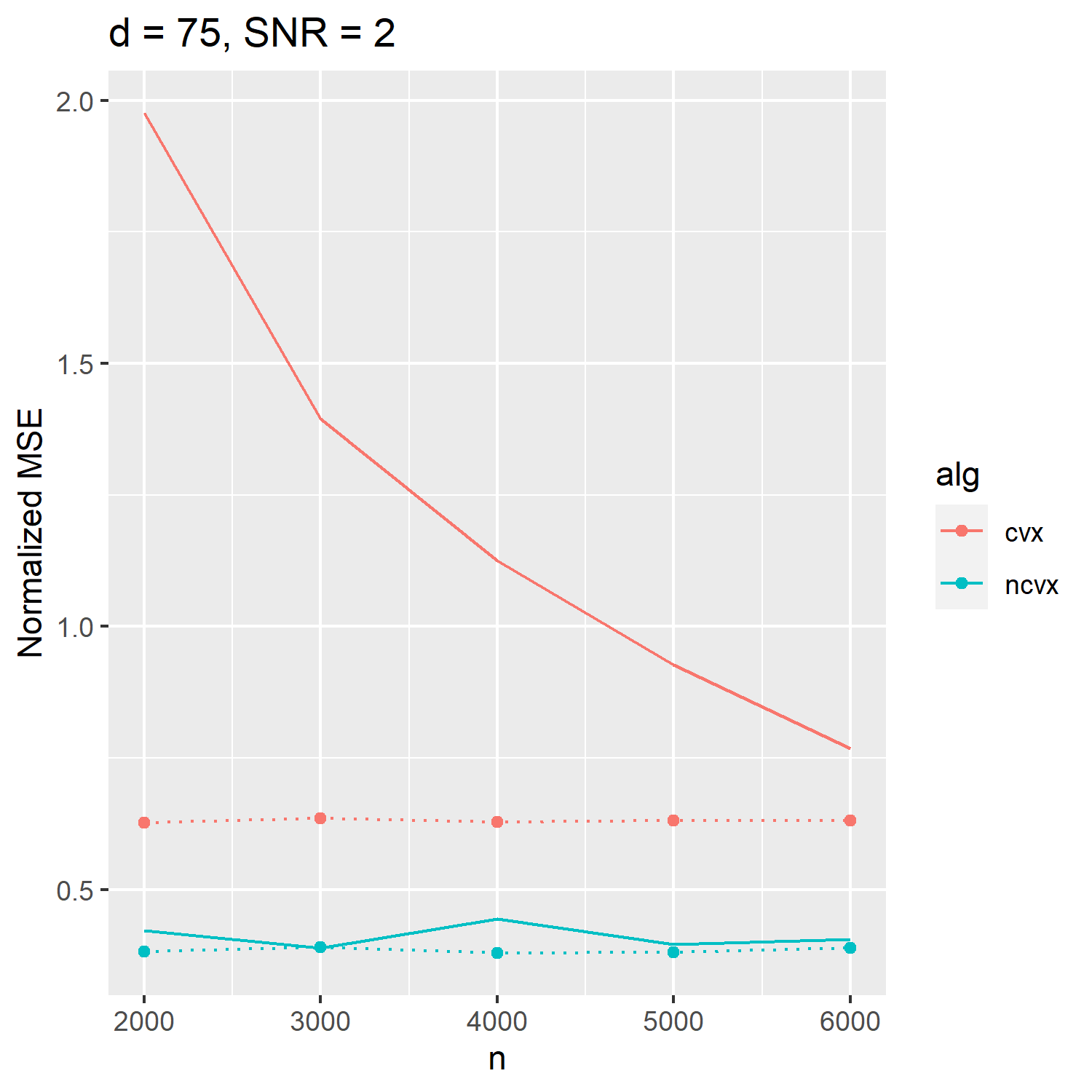}
\includegraphics[width=0.32\textwidth,keepaspectratio]{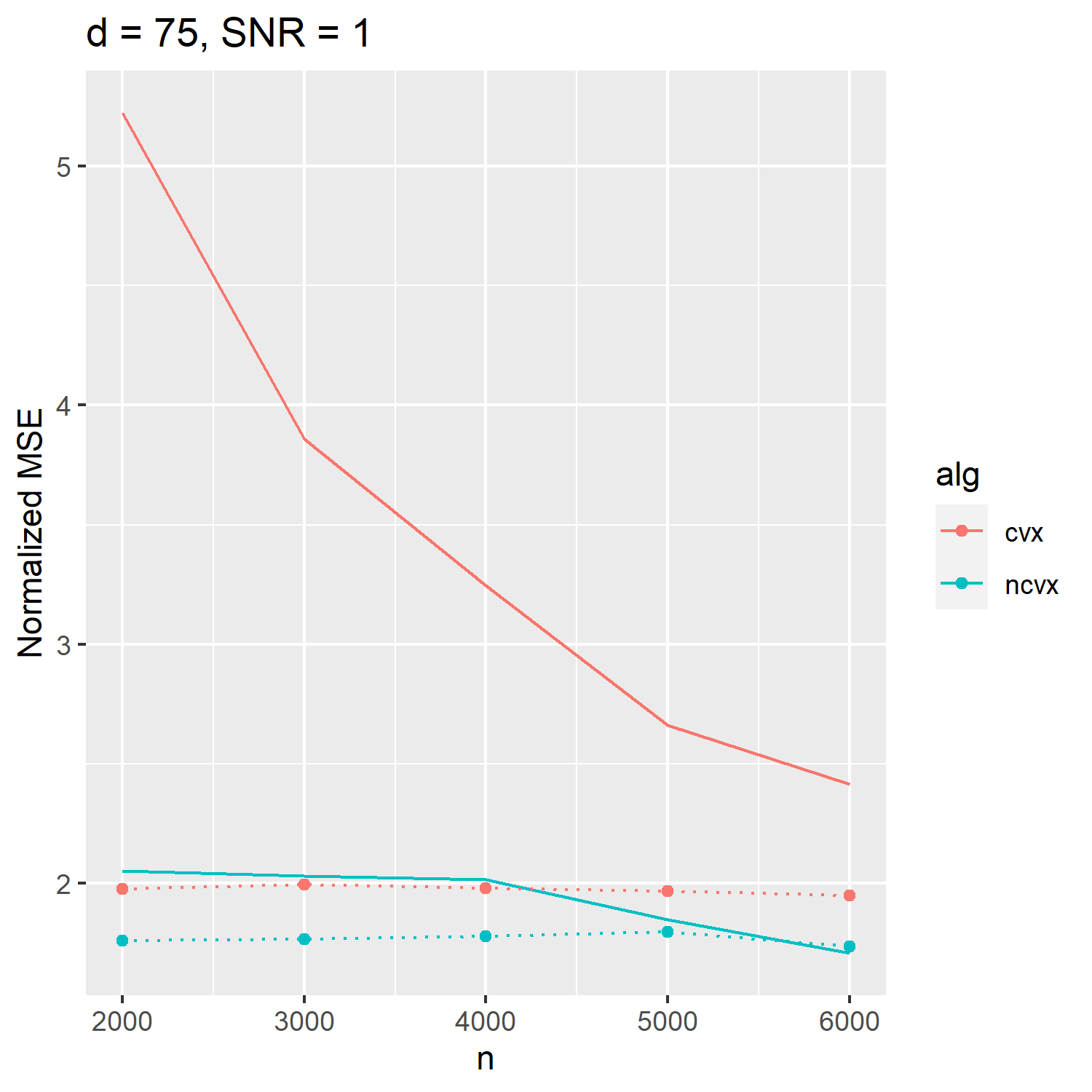}
\includegraphics[width=0.32\textwidth,keepaspectratio]{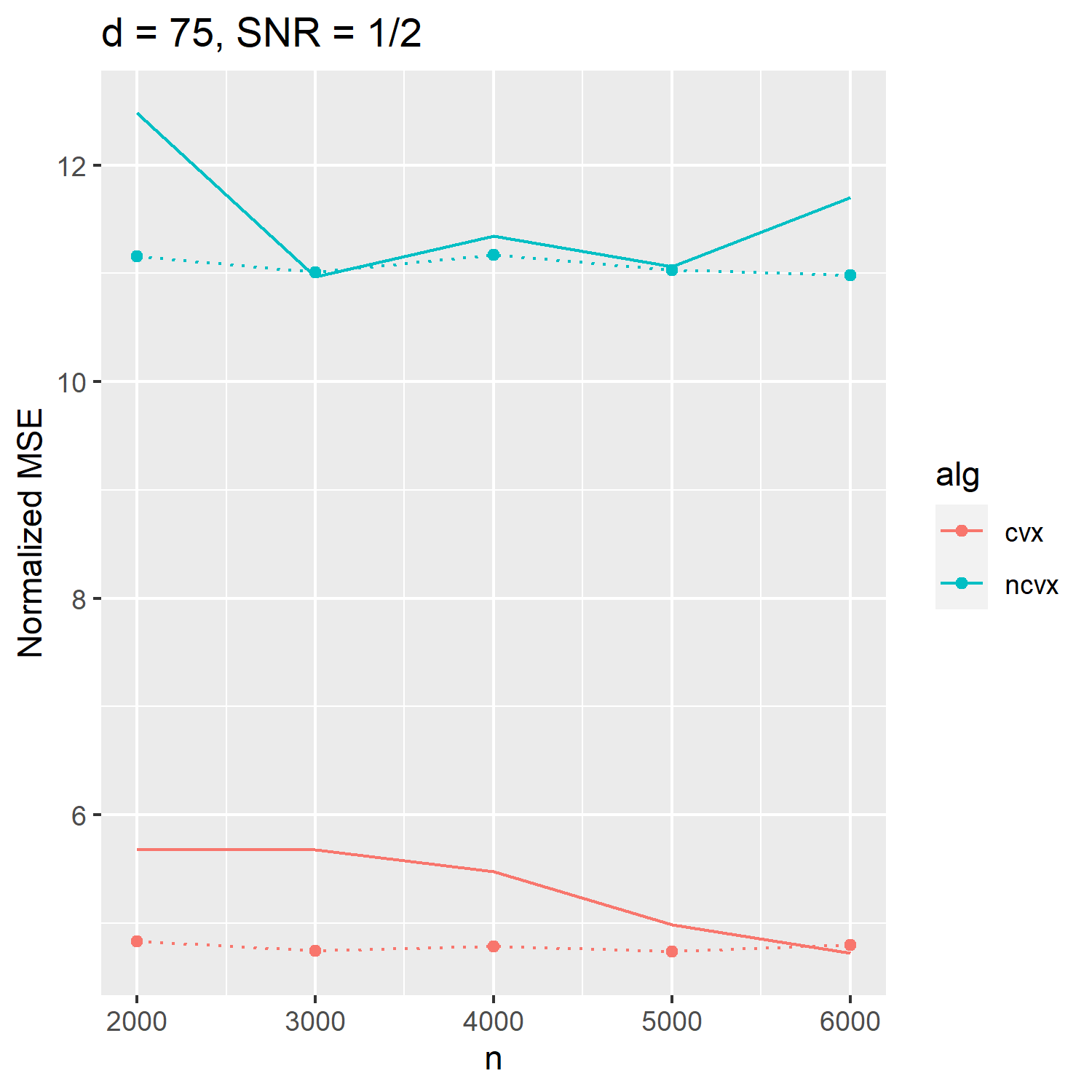}
       \caption{ For varying $n \in \{1000,2000,3000\}$, we generate $\bm{M}$ by first drawing a random matrix $\bm{U}$ of dimension $n\times r$ with entries $\mathcal{N}(0,1)$ and setting $\bm{\tilde M} := \bm{UU}\t$.  We then define $\bm{M}$ by setting the smallest nonzero singular value of $\bm{\tilde M}$ to be $\sqrt{d}$.  We control the SNR by varying $\sigma \in \{1/2,1,2\}$ corresponding to an SNR of $\{2,1,1/2\}$ respectively.  We take $\lambda = 0.5 \sigma \sqrt{d}$ in all simulations.  To estimate $\bm{U}^{(0)}$ we run gradient descent on $f^{(0)}_{{\sf ncvx}}$ starting from the ground truth with step size $2\times 10^{-4}$ for $d =50$ and $10^{-4}$ for $d = 75$ respectively, and we stop when the gradient is below $.01$.  Similarly, we estimate $\bm{Z}^{(\lambda)}$ by proximal gradient descent with step size $.01$, and we stopped when the gradient was below $.01$.   Our random seed was set to be $492025$.        The y-axis represents the error $\frac{1}{2d} \| \bm{\hat M} - \bm{M} \|_F^2$, with $\bm{\hat M}$ either equal to $\bm{Z}^{(\lambda)}$ or $\bm{U}^{(0)} \bm{U}^{(0)\top}$.  The solid line represents our empirical estimate and the points correspond to our theory, with the empirical estimate averaged over five iterations.}  \label{fig:d75snrvary}
     \end{figure}

\cref{fig:d75snrvary} displays the behavior illustrated by \cref{cor:maincor} for varying levels of the signal-to-noise ratio (left to right) and varying $d$ (top to bottom).   \cref{cor:maincor} shows that the nonconvex estimator dominates the convex estimator in least squares error, so we plot the error $\frac{1}{2d} \| \bm{\hat M} - \bm{M} \|_F^2$ under our model, with $\bm{\hat M}$ either $\bm{U}^{(0)} \bm{U}^{(0)\top}$ (blue) or $\bm{Z}^{(\lambda)}$ (red).  The dotted lines represent the theoretical prediction and the solid lines represent our empirical prediction, and the figure corroborates our theory for high SNR, the nonconvex estimator outperforms the convex estimator, and the approximation becomes more accurate as $n$ and $d$ increase.  Furthermore, when the SNR is less than one (something not covered by our theory), we find that the convex estimator outperforms the nonconvex estimator (the last panel).  This mirrors matrix denoising as discussed in \citet{gavish_optimal_2014}.   Interestingly, despite lack of theoretical backing, the error of the nonconvex estimator (the solid blue line) still seems to approximate the error of matrix hard thresholding (the dotted blue line) as in the case of high SNR. It would be interesting to derive an analogue of \cref{thm:mainthm} in such a ``weak signal'' regime.

We now highlight several features of \cref{thm:mainthm}:
\begin{itemize}
\item \textbf{Interpretation}. The main benefit of \cref{thm:mainthm} is to translate the behavior of the (complicated) estimators $\bm{U}^{(0)}$ and $\bm{Z}^{(\lambda)}$ into the behavior of the (simpler) estimators $\bm{Z}_{{\sf HT}}$ and $\bm{Z}_{{\sf ST}}^{(\lambda)}$. Specifically, \cref{thm:mainthm} shows that  Lipschitz functions of $\bm{U}^{(0)}\bm{U}^{(0)\top}$ (or $\bm{Z}^{(\lambda)}$) concentrate around expected values of Lipschitz functions of $\bm{Z}_{{\sf HT}}$ (or $\bm{Z}_{{\sf ST}}^{(\lambda)}$), 
with the error decaying roughly of order $o(1)$ when $\frac{n}{d r^2} \to \infty$. We elect to state \cref{thm:mainthm} in terms of a Lipschitz function $\phi:\mathbb{R}^{d\times d} \to \mathbb{R}$ to allow for generality.  Furthermore, convergence of Lipschitz functions is equivalent to convergence in distribution, so one (informal) interpretation of  \cref{thm:mainthm} could be that $\bm{U}^{(0)} \bm{U}^{(0)\top}$ (resp. $\bm{Z}^{(\lambda)}$) is approximately equal in distribution to $\bm{Z}_{{\sf HT}}$ (resp. $\bm{Z}_{{\sf ST}}^{(\lambda)}$). Also, throughout our analysis of $\bm{U}^{(0)}$ we assume that the rank $r$ is known.  Therefore, another (informal) interpretation of \cref{cor:maincor} is that knowledge of the rank $r$ improves estimation.  It is unclear whether the estimator $\bm{U}^{(0)}$ with $r$ over-estimated will continue to dominate $\bm{Z}^{(\lambda)}$, which is a regime of interest as $r$ is typically unknown in practice. However, such a result seems out of reach with our techniques, so we leave such analyses to future work.
  \item \textbf{High-Dimensional Asymptotics of Both Estimators}.   In the asymptotic regime $n \gg d r^2$, one can use \cref{thm:mainthm} to derive the leading-order constant for the asymptotic test error as in \cref{cor:maincor}, or it can be used for statistical inference. For example, suppose one is interested in providing confidence intervals for entries of $\bm{M}$.  Our results imply that the nonconvex estimator satisfies
  \begin{align*}
      \bigg(\bm{U}^{(0)} \bm{U}^{(0)\top}\bigg)_{ij} \approx \mathcal{P}_r\bigg( \bm{M} + \sigma \bm{H} \bigg)_{ij},
  \end{align*}
  where the ``$\approx$'' holds in the sense of distribution.  One can therefore derive asymptotically valid confidence intervals using techniques for matrix denoising as in \citet{chen_spectral_2021}.  Furthermore, our analysis shows that the convex estimator requires debiasing for statistical inference, whereas the nonconvex estimator is \emph{already approximately unbiased!}

\item  \textbf{Optimality of Assumptions}.  
When $d^2 \gg n \gg d$ and $r \asymp 1$, by taking $\eps \asymp d^{-1/4} \log(n)$, \cref{thm:mainthm} shows concentration at the rate $\tilde O( d^{-1/4}  + \sqrt{\frac{d}{n}})$ for the nonconvex estimator and concentration at the rate $O(d^{-1/4} + \frac{d}{n})$ for the convex estimator.  The quantity $d^{-1/4}$ probably has an incorrect fractional power on $d$ (it should be the parametric rate $d^{-1/2}$), although this is a known deficiency of CGMT-style analyses, which often result in a suboptimal dependence on dimension when deriving concentration bounds \citep{chandrasekher_sharp_2023,celentano_lasso_2023,miolane_distribution_2021}. The additional error terms depending on the ratio $\frac{d}{n}$ are likely necessary without significantly modifying the results, since consistency is only guaranteed in this regime when $n \gg d$. It would be of interest to improve the dependence on these rates. 

 Our result mimics asymptotics for sparse regression with $s$ nonzero coordinates and $p$ dimensions in the regime $\frac{s \log(p/s)}{n} \to 0$ \citep{javanmard_debiasing_2018,javanmard_hypothesis_2014,bellec_debiasing_2023}. 
 While there already exists a rich literature on sparse high-dimensional linear regression, similar analyses for matrix sensing are still in their infancy.   To the best of our knowledge, \cref{thm:mainthm} is the first result of its kind for both the convex and nonconvex estimators in this high-dimensional setup.

The choice of the order of the tuning parameter $\lambda$ is order-wise optimal (though implicitly requires knowledge of the noise level).  In practice one may wish to further optimize over $\lambda$, but our results may not cover this approach, as such an  analysis would require a \emph{uniform} control with respect to $\lambda$, and our results only apply for a (fixed) choice of $\lambda$.  It is of interest to determine whether there are settings for which optimizing over $\lambda$ can yield better error rates, but we leave such analyses to future work. 
\end{itemize}

\subsection{Proof Strategy and Novelty} \label{sec:proofstrategy}
We now highlight the key steps and novelty of our proof. 
Roughly speaking, given a Lipschitz function $\phi$, the major intuition behind our proof can be intuitively understood via the following ``approximate equivalences:''
\begin{align}
\phi\big(\bm{Z}^{(\lambda)}\big) &\approx \phi\big(\bm{U}^{(\lambda)} \bm{U}^{(\lambda)\top}\big) \approx  \phi\big(\bm{Z}_{{\sf ST}}^{(\lambda)}\big); \\
    \phi\big(\bm{U}^{(0)} \bm{U}^{(0)\top} \big)&\approx \phi\big(\bm{U}^{(\lambda)}_{{\sf deb}}\bm{U}^{(\lambda)\top}_{{\sf deb}} \big)\approx \phi\big(\bm{Z}^{(\lambda)}_{{\sf deb}}\big) \approx \phi\big(\bm{Z}_{{\sf HT}}\big),
    \end{align}
    where $\bm{U}^{(\lambda)}_{{\sf deb}}$ and $\bm{Z}_{{\sf deb}}^{(\lambda)}$ are ``debiased'' estimators  defined via 
 \begin{align}
     \bm{Z}_{{\sf deb}}^{(\lambda)} &:= \mathcal{P}_{{\mathrm{rank}-r}} \bigg( \bm{Z}^{(\lambda)} + \frac{1}{n} \sum_{i=1}^{n} \bigg( y_i - \langle \bm{X}_i, \bm{Z}^{(\lambda)} \rangle \bigg) \bm{X}_i \bigg); \\
     \bm{U}_{{\sf deb}}^{(\lambda)} &:= \bm{U}^{(\lambda)} \bigg( \bm{I} + \lambda \big( \bm{U}^{(\lambda)\top} \bm{U}^{(\lambda)} \big)\inv \bigg)^{1/2}.
 \end{align}
      We will elaborate on each of these approximate equivalences, as well as further discussion of $\bm{U}_{{\sf deb}}^{(\lambda)}$ and $\bm{Z}_{{\sf deb}}^{(\lambda)}$ in the subsequent paragraphs.   For convenience at  we have also included a diagram of the main steps of the proof in \cref{fig1}. %

\begin{figure}[ht] 
    \centering
{\footnotesize 
\begin{tikzpicture}
   \node[main] (t1) {\shortstack{$\bm{U}^{(\lambda)}\bm{U}^{(\lambda)\top}$ \\ \ \ $= \bm{Z}^{(\lambda)}$\\ (\cref{thm:relatecvxncvx})}}; 
   \node[main] (t2) [above right= 1.5cm of t1] {\shortstack{$\bm{U}_{{\sf deb}}^{(\lambda)} \bm{U}_{{\sf deb}}^{(\lambda)\top}$  \\ \ \ $= \bm{Z}^{(\lambda)}_{{\sf deb}}$ \\ (\cref{thm:udebzdeb})}}; 
   \node[main] (t3) [below right= 1.5cm of t1] {\shortstack{$\phi\big( \bm{Z}^{(\lambda)} \big)$ \\ $\approx \mathbb{E} \phi\big( \bm{Z}^{(\lambda)}_{{\sf ST}}\big)$\\ (\cref{thm:cvxasymptotics})\\(\cref{lem:ZSTtau})}}; 
   
\node[main] (t4) [right = 2cm of t1]{\shortstack{$\phi\big(\bm{Z}^{(\lambda)}_{{\sf deb}}\big)$ \\ $ \approx \phi\big(\bm{Z}_{{\sf HT}}\big)$\\ (\cref{thm:zdeb})\\(\cref{lem:ZHTtau})}}; 
\node[main] (t5) [right  = .5cm of t4]{\shortstack{ $\bm{U}^{(0)}\bm{U}^{(0)\top}$ \\ $\approx \bm{U}_{{\sf deb}}^{(\lambda)}\bm{U}_{{\sf deb}}^{(\lambda)\top}$\\  (\cref{thm:debunreg})}}; 
\node[main] (t6) [right = .5 cm of t5]{\shortstack{$\phi(\bm{U}^{(0)}\bm{U}^{(0)\top}\big)$\\ $\approx \mathbb{E} \phi\big( \bm{Z}_{{\sf HT}} \big)$\\(\cref{thm:mainthm})}};
   \node[main] (t7)[left = .5 cm of t3]{\shortstack{Matrix\\ CGMT\\ (\cref{thm:matrixcgmt})}};

\draw [->] (t1) to (t2);
\draw [->] (t1) to (t3);
\draw [->] (t5) to (t6);
\draw [->] (t7) to (t3);
\draw [->] (t2) to (t4);
\draw [->] (t3) to (t4);
\draw [->] (t4) to (t5);

\end{tikzpicture} 
}
\caption{Diagram of Proof Dependencies}\label{fig1}
\end{figure}
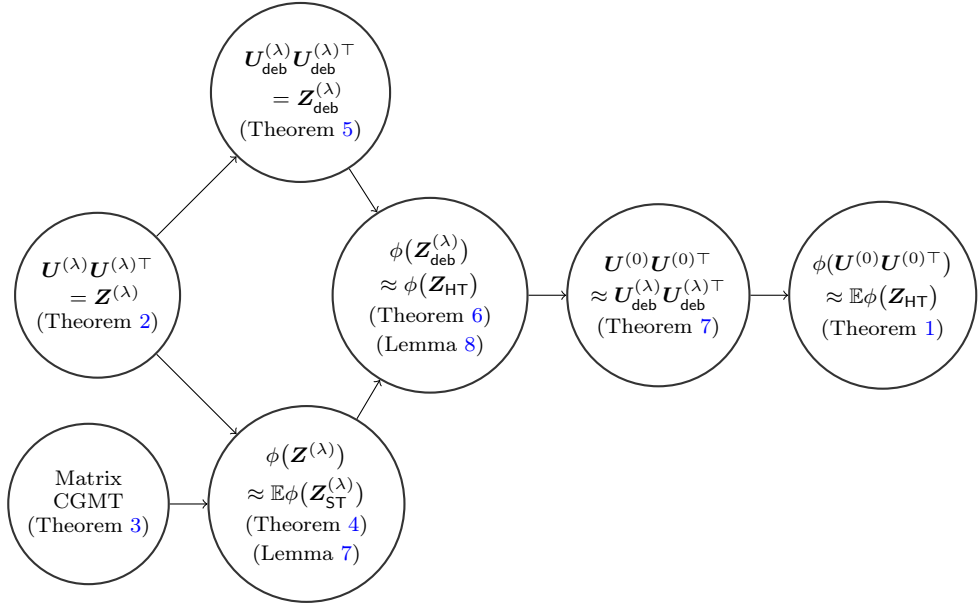

\paragraph{\cref{thm:relatecvxncvx}: Equivalence of Convex and (Regularized) Nonconvex Estimators, $\bm{Z}^{(\lambda)} = \bm{U}^{(\lambda)} \bm{U}^{(\lambda)\top}$, via Study of Local Minima.}  
The first equivalence we establish concerns the minima of $f^{(\lambda)}_{{\sf ncvx}}$ and $g^{(\lambda)}_{{\sf cvx}}$ with the same regularization parameter.  Indeed; recall from the variational form of the nuclear norm that for all positive semidefinite matrices $\bm{Z} \in \mathbb{R}^{d\times d}$ of rank at most $r$ it holds that 
\begin{align}
   \|\bm{Z} \|_* =  \min_{\bm{U} \in \mathbb{R}^{d \times r}: \bm{UU}\t = \bm{Z}} \| \bm{U} \|_F^2.
\end{align}
Consequently, since $f^{(\lambda)}_{{\sf ncvx}}$ includes a penalty term of $\|\bm{U}\|_F^2$ and $g^{(\lambda)}_{{\sf cvx}}$ includes a penalty term of $\|\bm{Z}\|_*$ it is natural that perhaps the local minima of these objectives ought to ``track'' each other.   
Since $\bm{U}^{(\lambda)}$ is computed with a \emph{pre-specified} rank $r$, to ensure that local minima of $f^{(\lambda)}_{{\sf ncvx}}$ and $g^{(\lambda)}_{{\sf cvx}}$ correspond, one needs to guarantee that the minimizer of $g^{(\lambda)}_{{\sf cvx}}$ is rank at most $r$. 

It turns out that there are simple conditions on the noise $\{\eps_i\}_{i=1}^{n}$ and the operator $\mathcal{X}$ such that the minimizer of $g_{{\sf cvx}}^{(\lambda)}(\bm{Z})$ is rank at most $r$. First, a simple condition is that the noise $\{\eps_i\}_{i=1}^{n}$ is sufficiently well-behaved, which can be handled directly via techniques from random matrix theory.  In addition, the other condition is that $\mathcal{X}\s\mathcal{X}$ behaves approximately as the identity operator on rank at most $2r$ matrices, which is known as the \emph{restricted isometry property} in the literature.  In particular, one formulation of the restricted isometry property is that $\mathcal{X}$ satisfies 
\begin{align}
    (1 - \delta_{2r}) \|\bm{A} \|_F \leq \| \mathcal{X}\s\mathcal{X}(\bm{A}) \| \leq (1 + \delta_{2r} ) \|\bm{A} \|_F,
\end{align}
for all matrices $\bm{A}$ of rank at most $2r$, where $\|\cdot\|$ is the operator norm. Here, $\delta_{2r}\geq 0$ is the \emph{restricted isometry constant} (roughly speaking, a quantitative measurement of the condition number of $\mathcal{X}\s\mathcal{X}$ on rank $2r$ matrices).  We encapsulate these ``good'' properties in the ``good'' event $\mathcal{E}_{{\sf Good}}$ defined formally in \eqref{egood}.

In order to formally prove \cref{thm:relatecvxncvx}, we first establish that first-order critical points  of $f^{(\lambda)}_{{\sf ncvx}}$ and $g^{(\lambda)}_{{\sf cvx}}$ are in correspondence provided they lie within radius $R = O(\sigma \sqrt{dr})$ of $\bm{M}$, which allows us to eliminate potentially bad critical points. Next, we show that any local minimum of $f^{(\lambda)}_{{\sf ncvx}}$  satisfying the second-order necessary conditions for a local minimum must lie within such a region of $\bm{M}$ on the event $\mathcal{E}_{{\sf Good}}$, which in turn yields the equivalence of the minima.   As a byproduct of this proof, on this same event, $\bm{Z}^{(\lambda)}$ is rank at most $r$ and is within radius $R$ of $\bm{M}$, which turns out to be a useful fact for the subsequent analysis.

Our arguments here bear some similarity to a number of similar lines of work in matrix completion  establishing connections between convex and nonconvex procedures \citep{yang_optimal_2023,chen_bridging_2021,chen_convex_2021,wang_robust_2025}, most notably \citet{chen_noisy_2020}.  However, unlike \citet{chen_noisy_2020}, we prove the equivalence by directly studying the second-order necessary conditions for a local minimum, whereas \citet{chen_noisy_2020} provide an inductive argument by considering the iterates of gradient descent initialized at the (unknown) ground truth $\bm{M}$.  By starting directly from the second-order necessary conditions we bypass the inductive arguments, and our results are therefore algorithm-agnostic. 

\paragraph{\cref{thm:cvxasymptotics}: Concentration of $\phi\big(\bm{Z}^{(\lambda)}\big)$ about $\mathbb{E}\phi\big(\bm{Z}_{{\sf ST}}^{(\lambda)}\big)$ via Matrix CGMT.} 
  Having shown the (exact) equivalence of convex and (regularized) nonconvex estimators, our next major technical argument is to demonstrate the approximate equivalence of minima of $g^{(\lambda)}_{{\sf cvx}}$ with $\bm{Z}_{{\sf ST}}^{(\lambda)}$, the soft-thresholding estimator.  This result is significantly different in spirit from the previous argument concerning local minima, as the form of approximate equivalence herein is taken to mean that Lipschitz functions of the former concentrate about expected values of Lipschitz functions of the latter  as in \cref{thm:mainthm}.

In \cref{thm:cvxasymptotics} we demonstrate that $\bm{Z}^{(\lambda)}$ approximates another soft-thresholding estimator which we then show approximates $\bm{Z}^{(\lambda)}_{{\sf ST}}$ asymptotically via \cref{lem:ZSTtau}.
  It is worth emphasizing that \cref{thm:cvxasymptotics} is able to yield precise asymptotics in the regime $\frac{d r^2}{n}\asymp C$ as long as \cref{ass1,ass2,ass3} hold.    To prove this approximate equivalence, we leverage Gaussian comparison inequalities via a novel generalization of the Convex Gaussian Min-Max Theorem of \citet{thrampoulidis_gaussian_2015} tailored to matrix ensembles. We dub this result, which may be of independent interest (see \cref{thm:matrixcgmt}), the ``Matrix CGMT.''

  Informally, the (vector) CGMT allows one to translate probabilistic statements about the optimal value of one min-max problem over Gaussian processes into the optimal value of another, often much simpler, min-max problem with different Gaussian processes.  
 This technique has been used to great success in a number of works on  high-dimensional regression, though we leave a more in-depth literature review to \cref{sec:relatedwork}.  Our work is most related to the analysis in \citet{celentano_lasso_2023}. 
  
  Many analyses based on the CGMT follow a similar recipe:
  \begin{enumerate}
      \item First, show that the empirical minimization  problem (e.g. $\min g^{(\lambda)}_{{\sf cvx}}$) can be written as an appropriate convex-concave min-max problem, and  apply the CGMT to translate probabilistic statements about this min-max problem into the ``simpler'' min-max problem (which we refer to as the ``CGMT objective'').
      \item Next, study the behavior of the CGMT objective about its optimal value to derive probabilistic statements about minima of the empirical minimization problem.
  \end{enumerate}    
  Step one above is a straightforward generalization of the arguments of \citet{celentano_lasso_2023} to the matrix setting via the Matrix CGMT.  After carrying out this argument, we arrive at the ``MCGMT objective'' $L_{\lambda}^{(n)}(\bm{Z})$ defined in \eqref{MCGMT_objective}.

The second step (studying the probabilistic behavior of the CGMT objective) is more complicated.    Unlike \citet{celentano_lasso_2023}, the function $L_{\lambda}^{(n)}(\cdot)$ is not  necessarily strongly convex on $d \times d$ matrices in a neighborhood of $\bm{Z}_{{\sf ST}}^{(\lambda)}$, and analyzing its probabilistic behavior becomes more difficult due to this lack of strong convexity.  However, by our previous analysis, we are able to restrict $L_{\lambda}^{(n)}$ to rank at most $r$ matrices within a certain radius of $\bm{M}$, and one of our key insights is to demonstrate that this restriction results in a form of strong convexity.  This restricted strong convexity mimics similar types of restricted strong convexity in other low-rank matrix optimization (e.g.,  \citet{chi_nonconvex_2019}). However, to the best of our knowledge, we are the first to combine these arguments with the CGMT, and our analysis requires a number of additional technical considerations beyond those required for the LASSO in \citet{celentano_lasso_2023}.

\paragraph{\cref{thm:udebzdeb,thm:zdeb}: Concentration of $\phi(\bm{U}^{(\lambda)}_{{\sf deb}} \bm{U}^{(\lambda)\top}_{{\sf deb}})$ and $\phi(\bm{Z}^{(\lambda)}_{{\sf deb}})$ about $\mathbb{E} \phi(\bm{Z}_{{\sf ST}})$.}  Recall that our analysis uses two intermediate quantities $\bm{U}^{(\lambda)}_{{\sf deb}}$ and $\bm{Z}_{{\sf deb}}^{(\lambda)}$    
    that can be computed from $\bm{U}^{(\lambda)}$ and $\bm{Z}^{(\lambda)}$ respectively.
While we view these two quantities as theoretical tools, in principle $\bm{Z}_{{\sf deb}}^{(\lambda)}$ and $\bm{U}_{{\sf deb}}^{(\lambda)}$ could be used for statistical inference as they are computable from data.   We show in \cref{thm:udebzdeb} that $\bm{Z}_{{\sf deb}}^{(\lambda)} = \bm{U}^{(\lambda)}_{{\sf deb}}\bm{U}^{(\lambda)\top}_{{\sf deb}}$ with high probability. 

To understand the reasoning behind introducing these estimators, consider the matrix denoising problem momentarily. In this case, the quantity $\bm{Z}_{{\sf ST}}^{(\lambda)}$ has its leading $r$ singular values shrunken by $\lambda$, which leads to a bias.  In order to debias this estimator, one would ideally re-impute all of the singular values of $\bm{Z}_{{\sf ST}}^{(\lambda)}$.  However, it turns out that only imputing the leading few singular values is sufficient, as keeping the remaining singular values to zero serves to act as additional denoising.  Therefore, the ``debiasing'' operation one arrives at is simply to add back $\lambda$ to the leading $r$ singular values, which results in $\bm{Z}_{{\sf HT}}$. 

Observe that operation $\bm{U}^{(\lambda)}\bm{U}^{(\lambda)\top} \mapsto \bm{U}_{{\sf deb}}^{(\lambda)}\bm{U}_{{\sf deb}}^{(\lambda)\top}$ can be equivalently be expressed as adding $\lambda$ to the leading $r$ singular values of $\bm{Z}^{(\lambda)}$. Consequently, both the operations $\bm{U}^{(\lambda)} \bm{U}^{(\lambda)\top} \mapsto \bm{U}^{(\lambda)}_{{\sf deb}} \bm{U}^{(\lambda)\top}_{{\sf deb}}$ and $\bm{Z}_{{\sf ST}}^{(\lambda)} \mapsto \bm{Z}_{{\sf HT}}$ can be understood as adding $\lambda$ to the nonzero singular values.  Since  we have already shown that $\phi\big(\bm{Z}^{(\lambda)}\big)$ concentrates about $\mathbb{E}\phi\big(\bm{Z}^{(\lambda)}_{{\sf ST}}\big)$, this discussion shows that it is intuitive to argue that $\phi\big(\bm{Z}_{{\sf deb}}^{(\lambda)}\big)$ concentrates about $\mathbb{E}\phi\big( \bm{Z}_{{\sf HT}}\big)$.   This result is accomplished in \cref{thm:zdeb}, which shows that $\phi(\bm{Z}_{{\sf deb}}^{(\lambda)})$ concentrates around around another hard thresholding estimator, which we then show approximates $\bm{Z}_{{\sf HT}}$ asymptotically via \cref{lem:ZHTtau}.  The proof of \cref{thm:zdeb} follows by applying \cref{thm:cvxasymptotics}, where we use the fact that the debiasing operator (i.e., adding $\lambda$ to the nonzero singular values) is a Lipschitz function on an appropriate domain.  

\paragraph{\cref{thm:debunreg}: Concentration of $\bm{U}_{{\sf deb}}^{(\lambda)}$ about $\bm{U}^{(0)}$.}    
 Our final  major technical contribution is to demonstrate that  $\bm{U}_{{\sf deb}}^{(\lambda)}$ and $\bm{U}^{(0)}$ are close in Frobenius norm (up to some nonidentifiable orthogonal transformation).  Informally, our main idea is to show that $\bm{U}_{{\sf deb}}^{(\lambda)}$ is an approximate local minimum for $f^{(0)}_{{\sf ncvx}}$, hence ``closeness'' follows via a Taylor expansion argument.  However, as we demonstrate in \cref{thm:debunreg}, of all of the steps of our proofs, this step is the only step that essentially requires that $n \gg  d r^2$. In order to formally prove this result, we first show that $\nabla f^{(0)}_{{\sf ncvx}}\big( \bm{U}_{{\sf deb}}^{(\lambda)}\big)$, the gradient of the unregularized loss function $f^{(0)}_{{\sf ncvx}}$ evaluated at $\bm{U}_{{\sf deb}}^{(\lambda)}$, 
 is close to zero with high probability.  In addition, since $\bm{U}^{(0)}$ is a local minimum of $f^{(0)}_{{\sf ncvx}}$ we show that on the event $\mathcal{E}_{{\sf Good}}$ the second derivative of $f^{(0)}_{{\sf ncvx}}$ is sufficiently well-behaved in a neighborhood of $\bm{U}^{(0)}$, and $\bm{U}^{(\lambda)}_{{\sf deb}}$ falls in this neighborhood on this event.

\subsection{Related Work} \label{sec:relatedwork}
In this section we review related work.  First, we consider low-rank matrix optimization both in the noiseless and noisy settings.  Next our main technical tools are closely related to similar tools used in the literature on precise asymptotics for high-dimensional linear regression, and we review these related works as well.  \\
\\ \ \noindent
\textbf{Low-Rank Matrix Optimization}. A number of works have studied factorized nonconvex approaches to low-rank matrix estimation, with several works focusing on the optimization landscape and geometry \citep{ge_no_2017,bhojanapalli_global_2016}.  
There has also been a few works considering the geometry with $r$ overparameterized \citep{zhang_sharp_2021,ma_geometric_2023,ma_optimization_2023}.   Subsequent research has also been devoted to studying the iterates of gradient descent and its variants, including settings where the parameter $r$ is overspecified and the initialization is random 
\citep{stoger_small_2021,soltanolkotabi_implicit_2025,ma_geometric_2023} or  the matrix $\bm{M}$ is ill-conditioned \citep{tong_low-rank_2021,zhang_fast_2023,xu_power_2023}. For a more complete list of references, see the survey \citet{chi_nonconvex_2019}. 

For matrix sensing the primary focus of this line of work has been the noiseless setting, and one theme that has emerged is that  when the rank is correctly specified, all local minima are in fact global minima, and all other critical points are ``strict saddles,'' meaning there is a search direction along which the Hessian is negative. In particular, in the noiseless setting, all local minima can be shown to yield the ground truth signal matrix.  Therefore, 
our results can be viewed as an extension from the noiseless setting to the noisy setting: in the noisy setting all local minima behave asymptotically equivalently to the truncated SVD of the corresponding matrix denoising problem, which matches the ground truth in the noiseless case.    A major novelty of our work is in demonstrating that this form of ``asymptotic equivalence'' holds in a precise distributional sense (in the sense of Lipschitz functions).
 
Our work is also closely related to several works on low-rank matrix estimation under various statistical models.  For example, \citet{chen_convex_2021} study noisy blind deconvolution, \citet{candes_phase_2015,chen_gradient_2019,cai_optimal_2016} study phase retrieval, \citet{chen_bridging_2021} study robust PCA, \citet{chen_noisy_2020} study noisy matrix completion, and \citet{ma_implicit_2020,charisopoulos_low-rank_2021} consider several different low-rank problems.  From a technical point of view, our analysis differs significantly from these works, who rely on a ``leave-one-out'' argument to study the iterates of nonconvex procedures. Typically the focus of these works has been to study the statistical error \emph{rates}, whereas our focus is on the precise asymptotic statistical behavior.  A work markedly similar to ours is  \citet{chen_inference_2019}, in which the authors consider developing confidence regions for the convex estimator in matrix completion via a debiasing procedure.  Besides the fact that they consider matrix completion, whereas we consider matrix sensing, unlike their work our results apply to both convex and nonconvex estimators directly, whereas the results of  \citet{chen_inference_2019} only apply to the convex estimator, though they also consider an analogue of $\bm{U}^{(\lambda)}$.  The focus of our work is on explicitly studying the statistical behavior of $\bm{U}^{(0)}$ and $\bm{Z}^{(\lambda)}$, whereas the focus of their work is on statistical inference.

Perhaps the closest related works to our setting are the works \citet{chandrasekher_alternating_2024,chandrasekher_sharp_2023} who derive precise predictions for iterative procedures in high-dimensional regimes similar to the one we consider herein. In \citet{chandrasekher_alternating_2024} the authors study alternating minimization for rank one matrix sensing, though the statistical model is significantly different from the one we consider herein.  In addition,  \citet{chandrasekher_sharp_2023} uses the CGMT to derive iterate-by-iterate predictions for nonconvex procedures for variants of high-dimensional regression, provided that each iteration minimizes some convex problem.  
However, both \citet{chandrasekher_alternating_2024} and \citet{chandrasekher_sharp_2023} rely heavily on the assumption of \emph{sample-splitting}, meaning that a fresh set of data is used at each iteration of their algorithm.  In contrast, our work has no such assumption, as our results apply directly to local minima regardless of  algorithm.  
\\ \ \\
\noindent
\textbf{Precise Asymptotics for High-Dimensional Statistical Models}.  Besides the work \citet{chandrasekher_sharp_2023}, the CGMT has been used in a number of works studying the asymptotics of estimators of various procedures in high dimensions. 
For example, \citet{thrampoulidis_regularized_2015,thrampoulidis_precise_2018} studied  regularized linear regression and M-estimation respectively, \citet{javanmard_precise_2020,javanmard_precise_2022} considered adversarial training, \citet{liang_precise_2022,donhauser_fast_2022,wang_tight_2022,stojanovic_tight_2024,han_distribution_2023} considered minimum norm interpolators for regression and classification, \citet{montanari_generalization_2025} considered margin-based classfication, and \citet{taheri_sharp_2020,taheri_fundamental_2021,han_noisy_2023,loureiro_learning_2021,salehi_impact_2019} considered other variants of linear models.  Our model is significantly different from these works, as our model is based on matrix-valued observations, necessitating a generalization of the CGMT used in these prior works to matrix ensembles (\cref{thm:matrixcgmt}).

As mentioned previously, our analyses are partially motivated by those in \citet{celentano_lasso_2023}, who provide precise statistical guarantees for the LASSO and its debiased counterpart under general Gaussian measurements in the \emph{inconsistency regime} (i.e., the asymptotic regime where minimax error does not vanish asymptotically).  While some of our techniques are similar, our focus is significantly different, as we primarily focus on a regime in which minimax error \emph{does} vanish asymptotically, though we only require this assumption in the final step of our proof.  
However, we also study the minimizer of the nonconvex objective function $f_{{\sf ncvx}}^{(\lambda)}$, to which there is no such immediate analogue in sparse linear regression. In essence, this difference is due to the fact that singular values are \emph{ordered} automatically by magnitude, whereas linear regression coefficients do not have any such ordering \emph{a priori}.

Finally, the debiased estimators we consider in this work are motivated by \citet{chen_inference_2019}, though similar forms of $\bm{Z}_{{\sf deb}}^{(\lambda)}$ have appeared elsewhere \citep{xia_confidence_2019,cai_geometric_2016}.  The estimator $\bm{Z}_{{\sf deb}}^{(\lambda)}$ without the rank $r$ projection step is also similar to the debiased LASSO first introduced in \citep{zhang_confidence_2014} and studied further in the works of \citet{javanmard_confidence_2014,javanmard_debiasing_2018,miolane_distribution_2021,celentano_lasso_2023}, among others.

\subsection{Discussion} \label{sec:discussion}
In this paper we have established the precise asymptotic behavior of both the convex and nonconvex formulations to the matrix sensing problem.  We have seen that even by picking the parameter $\lambda$ to be order-wise optimal  \emph{a priori}, the nonconvex estimator performs better in terms of its generalization ability.  The heart of this result relies on an intimate connection between the nonconvex minima and the debiased convex minima, which we show to be asymptotically ``equivalent'' in the regime $\gamma_n/r \to \infty$. 

While we believe our analysis to shed new light on the matrix sensing problem, a number of open problems remain, which we detail below. 
\begin{itemize}
    \item \textbf{Asymmetric matrices:} In our analysis we have assumed throughout that the underlying matrices (both the observations and the matrix to be estimated) are symmetric.  This assumption greatly simplifies our analyses, but in practice matrices need not be symmetric.   The paper of \citet{ma_beyond_2021} showed that in the noiseless setting gradient descent on asymmetric matrices without additional regularization converges to a local minimum, but extending this result to the noisy setting would require significant additional considerations for both the left and right factors.
\item \textbf{Choice of $r$:} We assumed throughout that $r$ was known \emph{a priori}.  Such an assumption is unrealistic in practice, as $r$ is often overspecified. There has been some work studying such a regime \citep{stoger_small_2021,soltanolkotabi_implicit_2025,ma_geometric_2023}, but these works do not obtain any statistical characterizations similar to the one found herein.  
\item \textbf{The inconsistency regime:}  Our work assumed that $\gamma_n/r \to \infty$ (or at least $\gamma_n \geq C_3 r$), which is essentially the minimal assumption in order to have consistency up to the factor of $r$.  Extension of our results to the inconsistency regime will likely require new arguments, as our form of ``approximate equivalence'' of $\bm{U}_{{\sf deb}}^{(\lambda)}$ and $\bm{U}^{(0)}$ relied heavily on this regime.  
\item \textbf{Heavy-tailed noise:}  Our analysis requires that $\bm{\eps}$ is a Gaussian random vector, though the extension to subgaussian noise is likely minor.  However, extending our results to encompass the heavy-tailed setting may require analyzing a new estimator that incorporates a robust loss function.  
\item \textbf{Weak signals:} In \cref{ass1} we require that $\lambda_r/\sigma \geq C_1 \sqrt{r}$, which implies that $\bm{M}$ has enough signal.  In our numerical experiments  we have seen that the nonconvex estimator no longer dominates the convex estimator in weak signal regimes, a regime not covered by our theory.  It would be of interest to pin down all regimes where the nonconvex estimator dominates the convex estimator (or vice versa).  
\item \textbf{Matrix sensing as a testbed for neural networks:}  Finally, one of our major results shows that the nonconvex formulation with correctly specified rank dominates the nonconvex formulation.  Insofar as matrix sensing can be viewed as a toy model for two-layer neural networks, it may be interesting to study how to extend our results to deep neural networks with many hidden layers.      
\end{itemize}

\subsection{Paper Organization}
The rest of this paper is organized as follows.  First, in \cref{sec:equivalence} we demonstrate the equivalence of the convex and regularized nonconvex estimators, and in \cref{sec:cvxcontrol} we study the precise high-dimensional behavior of the convex estimator. \cref{sec:debequivalence} is dedicated to studying the debiased estimators, and  \cref{sec:debnearequivalence} is dedicated to the study of the concentration of the unregularized estimator about the debiased estimator.  We finish with a proof of \cref{thm:mainthm} in \cref{sec:mainthmproof}.  More detailed proofs are deferred to the appendices.

\section{Equivalence of Convex and Regularized Nonconvex Estimators} \label{sec:equivalence}
In this section we demonstrate the equivalence of the convex and regularized nonconvex estimators; i.e, \cref{thm:relatecvxncvx} in \cref{fig1}.  The result relies on the existence of a high-probability event $\mathcal{E}_{{\sf Good}}$ such that $\bm{U}^{(\lambda)} \bm{U}^{(\lambda)\top} = \bm{Z}^{(\lambda)}$.  To wit, 
define the ``good'' event $\mathcal{E}_{{\sf Good}}$ via
\begin{subequations} 
\label{egood}
\begin{align}
    \mathcal{E}_{{\sf Good}} &:= \bigg\{ \bigg\| \frac{1}{n} \sum_i \eps_i \bm{X}_i \bigg\| \leq 8\sigma \sqrt{\frac{d}{n}}  \bigg\}  \label{concentration} \\
    &\bigcap \bigg\{ (1 - \delta_{2r}) \|\bm{A}\|_F^2 \leq  \|  \mathcal{X}( \bm{A}) \|^2 \leq (1 + \delta_{2r}) \|\bm{A} \|_F^2 \text{ for all matrices $\bm{A}$ of rank at most $2r$}   \bigg\}  \label{rip} \\
    &\bigcap \bigg\{ \bigg\| \frac{1}{n} \sum_{i} \eps_i \langle \bm{X}_i , \bm{A} \rangle \bigg\| \leq C \sigma \sqrt{\frac{dr}{n}}\| \bm{A} \|_F  \text{ for all matrices $\bm{A}$ of rank at most $2r$} \bigg\}. \label{rip2}
 \end{align}
 \end{subequations}
 Here $\delta_{2r} > 0$ is a fixed constant.  Intuitively, the first event \eqref{concentration} ensures that the random matrix $\frac{1}{n} \sum_{i} \eps_i \bm{X}_i$ (which behaves as a random linear combination of ${\sf GOE}(d)$ random matrices) does not deviate too far in spectral norm, which guarantees that the noise is sufficiently well-behaved.  The next event \eqref{rip} is the well-known \emph{restricted isometry property} for the operator $\mathcal{X}$.  Many analyses for matrix sensing rely only on this assumption or something similar \citep{zhang_sharp_2021,zheng_convergent_2015,zhang_fast_2023,ge_no_2017,ma_beyond_2021,soltanolkotabi_implicit_2025,stoger_small_2021}. The final event, \eqref{rip2}, argues that $\bm{X}_i$ concentrates at a rate that holds \emph{uniformly} over all matrices $\bm{A}$ of rank at most $2r$.  In this sense, this event blends both probabilistic concentration and something resembling the restricted isometry property (as it must hold uniformly over matrices of rank at most $2r$).

The following lemma shows that this event holds with high probability for an appropriate choice of $\delta_{2r}$.  

\begin{lemma} \label{lem:Egood}
The event $\mathcal{E}_{{\sf Good}}$ holds with probability at least $1 - O(\exp(- c d))$ whenever $n \geq C dr /\delta_{2r}^2$. In particular, $\mathcal{E}_{{\sf Good}}$ holds under \cref{ass2} (i.e., we may take $\delta_{2r} = \frac{c}{\sqrt{r}}$).
\end{lemma}
\begin{proof}
    See \cref{sec:proofofegood}.
\end{proof}

Next,  the following theorem demonstrates that on the event $\mathcal{E}_{{\sf Good}}$, local minima of $f^{(\lambda)}_{{\sf ncvx}}$ correspond to local (in fact, global) minima of $g^{(\lambda)}_{{\sf cvx}}$.   
\begin{theorem}[Relating Convex and Nonconvex Regularized Estimators] \label{thm:relatecvxncvx}
Let $\bm{Z}^{(\lambda)}$ denote (any) global minimum of $g^{(\lambda)}_{{\sf cvx}}$, and let $\bm{U}^{(\lambda)}$ denote any critical point of $f^{(\lambda)}_{{\sf ncvx}}$ satisfying the first and second-order necessary conditions for a local minimum.  Then under \cref{ass2,ass3}, 
on the event $\mathcal{E}_{{\sf Good}}$ defined in \eqref{egood}, 
\begin{align}
    \bm{U}^{(\lambda)} \bm{U}^{(\lambda)\top} &= \bm{Z}^{(\lambda)}.
\end{align}
\end{theorem}
This result is entirely deterministic when $\mathcal{E}_{{\sf Good}}$ holds. 

\subsection{Proof of \cref{thm:relatecvxncvx}}
In order to prove \cref{thm:relatecvxncvx} we first prove that the conclusion holds under a sufficient condition via the following lemma.
\begin{lemma}[Relating the Convex and Nonconvex Regularized Estimators Under a Technical Condition] \label{lem:relatecvxncvx}
    Let $\bm{Z}^{(\lambda)}$ denote the global minimum of $g^{(\lambda)}_{{\sf cvx}}$, and let $\bm{U}^{(\lambda)}$ denote a stationary point of $f^{(\lambda)}_{{\sf ncvx}}$.  Then as long as $\lambda \geq C_4 \sigma \sqrt{d}$, and $ \| \bm{U}^{(\lambda)}\bm{U}^{(\lambda)\top} - \bm{M} \|_F < \lambda/(2\delta_{2r})$,     on the event $\mathcal{E}_{{\sf Good}}$ it holds that
    \begin{align}
        \bm{U}^{(\lambda)} \bm{ U}^{(\lambda)\top} &= \bm{Z}^{(\lambda)}.
    \end{align}
\end{lemma}
\begin{proof}[Proof of \cref{lem:relatecvxncvx}]
    First, observe that any  optimizer $\bm{Z}^{(\lambda)}$ of $g^{(\lambda)}_{{\sf cvx}}$ must satisfy the first-order necessary condition
\begin{align}
  \frac{1}{\sqrt{n}} \sum_{i=1}^{n} \bigg( \langle \bm{X}_i, \bm{Z}^{(\lambda)} - \bm{M} \rangle /\sqrt{n} - \eps_i \bigg) \bm{X}_i = -\lambda \bm{VV}\t + \bm{W},
\end{align}
where $\bm{Z}^{(\lambda)} = \bm{V} \Sigma \bm{V}\t$, and $\bm{W}$ lies in the tangent space of $\bm{Z}^{(\lambda)}$ and satisfies $\|\bm{W} \| \leq \lambda $.  In addition, from the first-order necessary condition for a stationary point, $\bm{U}^{(\lambda)}$ satisfies
\begin{align}
   \frac{1}{\sqrt{n}} \sum_{i=1}^{n} \bigg( \langle \bm{X}_i, \bm{U}^{(\lambda)}\bm{U}^{(\lambda)\top} - \bm{M} \rangle/\sqrt{n} - \eps_i \bigg) \bm{X}_i \bm{U}^{(\lambda)} + \lambda \bm{U}^{(\lambda)} &= 0. \numberthis \label{uncvx}
\end{align}
Denote
\begin{align}
    \bm{N}_{\eps} := \frac{1}{\sqrt{n}}\sum_i \eps_i \bm{X}_i.
\end{align}
Equation \cref{uncvx} then implies 
\begin{align}
\mathcal{X}\s \mathcal{X}( \bm{U}^{(\lambda)} \bm{U}^{(\lambda)\top} - \bm{M} ) - \bm{N}_{\eps} &= 
    -\lambda \bm{V}_{\bm{U}} \bm{V}_{\bm{U}}\t + \bm{W}_{U}, \numberthis \label{cnvx}
\end{align}
where $\bm{V}_{{\bm{U}}}$ is the orthonormal matrix associated to the column space of $\bm{U}^{{(\lambda)}}$, and $\bm{W}_U$ is orthogonal to $\bm{V}_{{\bm{U}}}$.  Therefore, it suffices to show that $\|\bm{W}_U \| < \lambda$.  
We have that 
\begin{align}
\mathcal{X}\s \mathcal{X} \big( \bm{U}^{(\lambda)} \bm{U}^{(\lambda)\top} - \bm{M} \big) - \bm{N}_{\eps} &= 
 -\bm{M} + \bm{U}^{(\lambda)}\bm{U}^{(\lambda)\top} + \bigg( \mathcal{X}\s \mathcal{X} - \mathcal{I} \bigg) \big( \bm{U}^{(\lambda)}\bm{U}^{(\lambda)\top} - \bm{M}  \big) - \bm{N}_{\eps}.
\end{align}
Plugging this identity into \eqref{cnvx} and rearranging yields
\begin{align}
    \bm{M} + \bigg( \mathcal{I} - \mathcal{X}\s \mathcal{X}  \bigg) \big( \bm{U}^{(\lambda)}\bm{U}^{(\lambda)\top} - \bm{M}  \big) + \bm{N}_{\eps} &=  \bm{U}^{(\lambda)}\bm{U}^{(\lambda)\top} + \lambda \bm{V}_{U} \bm{V}_U\t - \bm{W}_U \\
    &= \bm{V}_{U} \bigg( \Sigma + \lambda \bm{I} \bigg) \bm{V}_U - \bm{W}_U.  
\end{align}
Since $\bm{W}_U$ is orthogonal to $\bm{V}_{\bm{U}}$, the eigenvalues of $\bm{W}_U$ are equal to the bottom $d - r$ eigenvalues of $\bm{V}_U(\Sigma + \lambda \bm{I}) \bm{V}_U\t$. 
Therefore, by Weyl's inequality, for $r < i \leq d$,
\begin{align}
    \sigma_i \bigg( \bm{V}_{U} \bigg( \Sigma + \lambda \bm{I} \bigg) \bm{V}_U\t - \bm{W}_U \bigg) &\leq \sigma_i(\bm{M}) + \bigg\| \bigg( \mathcal{I} - \mathcal{X}\s \mathcal{X}  \bigg) \big( \bm{U}^{(\lambda)}\bm{U}^{(\lambda)\top} - \bm{M}  \big) + \bm{N}_{\eps} \bigg\| \\
    &\leq \bigg\|  \bigg( \mathcal{I} - \mathcal{X}\s \mathcal{X}  \bigg) \big( \bm{U}^{(\lambda)}\bm{U}^{(\lambda)\top} - \bm{M}  \big) \bigg\| + \bigg\| \bm{N}_{\eps} \bigg\| \\
    &\leq  \delta_{2r} \| \bm{U}^{(\lambda)}\bm{U}^{(\lambda)\top} - \bm{M} \|_F + \| \bm{N}_{\eps} \| \\
    &< \frac{\lambda}{2} + \frac{\lambda}{2} \\
    &< \lambda,
\end{align}
since by \cref{ass3} $\lambda$ is taken to be  larger than $C_4 \sigma \sqrt{d}$, since then on $\mathcal{E}_{{\sf Good}}$ it holds that $\|\bm{N}_{\eps}\| \leq C\sigma \sqrt{d} < \frac{C_4 \sigma \sqrt{d}}{2}$, and  $ \| \bm{U}^{(\lambda)}\bm{U}^{(\lambda)\top} - \bm{M} \|_F < \lambda/(2\delta_{2r})$ by  assumption.  Thus, we have shown that $\| \bm{W}_U\| < \lambda$, which completes the proof.
\end{proof}

Next, note that \cref{lem:relatecvxncvx} implies \cref{thm:relatecvxncvx} if we can show that $ \| \bm{U}^{(\lambda)}\bm{U}^{(\lambda)\top} - \bm{M} \|_F < \lambda/(2\delta_{2r})$ for a local minimum $\bm{U}^{(\lambda)}$.  This result is accomplished in the following lemma.   
\begin{lemma} \label{lem:uncvxgood}
Let $\bm{U}^{(\lambda)}$ be a minimizer of $f^{(\lambda)}_{{\sf ncvx}}$ satisfying the second-order necessary condition.  Then on the event $\mathcal{E}_{{\sf Good}}$ it holds that
\begin{align}
    \frac{1}{\sqrt{dr}}\| \bm{U}^{(\lambda)} \bm{U}^{(\lambda)\top} - \bm{M} \|_F &\leq C \sigma. 
\end{align}
In particular, under \cref{ass2} and \cref{ass3} the conditions of \cref{lem:relatecvxncvx} hold.
\end{lemma}
\begin{proof}
   See \cref{sec:uncvxgoodproof}.
\end{proof}
To prove this result we study the Hessian of $f^{(\lambda)}_{{\sf ncvx}}$ evaluated at any local minimum and show that on the event $\mathcal{E}_{{\sf Good}}$ the necessary conditions for a local minimum give rise to the conclusion of \cref{lem:uncvxgood}.  

The proof of \cref{thm:relatecvxncvx} is now straightforward.
\begin{proof}[Proof of \cref{thm:relatecvxncvx}]
By \cref{lem:uncvxgood}, on the event $\mathcal{E}_{{\sf Good}}$, the conditions of \cref{lem:relatecvxncvx} hold.  The result is proven.
\end{proof}

\section{Control of Convex and Regularized Nonconvex Estimators} \label{sec:cvxcontrol}

In this section we study the convex estimator $\bm{Z}^{(\lambda)}$ as well as its nonconvex counterpart $\bm{U}^{(\lambda)}$, which up to the high-probability event $\mathcal{E}_{{\sf Good}}$, are equivalent in the sense of \cref{thm:relatecvxncvx}. 
The main purpose of this section is to state \cref{thm:cvxasymptotics}, which demonstrates that $\phi\big(\bm{Z}^{(\lambda)})$ concentrates around a  soft-thresholding estimator that asymptotically approximates $\bm{Z}_{{\sf ST}}^{(\lambda)}$ by \cref{lem:ZSTtau} stated in the following subsection.  This new soft-thresholding estimator arises from our application of the Matrix CGMT that we introduce in the following subsection, the primary technical ingredient in our proof.  We then give the definition of this new quantity and state \cref{lem:ZSTtau} in \cref{sec:softthresholding}.  Finally, we state \cref{thm:cvxasymptotics} and give a high-level proof in \cref{sec:cvxasymptotics}.

\subsection{Matrix CGMT} \label{sec:matrixcgmt}

The proof of \cref{thm:cvxasymptotics} uses the following generalization of the Convex Gaussian  Min-Max Theorem tailored to the matrix observations that may be of independent interest.   The proof is self-contained.   

\begin{theorem}[Matrix CGMT] \label{thm:matrixcgmt}
Let $\mathcal{A}: \mathbb{R}^{d \times d} \to \mathbb{R}^{n}$ be an operator such that $\mathcal{A}( \bm{M})_i =\langle \bm{A}_i, \bm{M} \rangle $, where each $\bm{A}_i$ is a GOE matrix. Let $\bm{H} \in \mathbb{R}^{d\times d}$ be a GOE matrix, and let $\bm{g}$ be a standard $n$-dimensional Gaussian random vector.  Let $\mathcal{S}_{\bm{W}} \in \mathbb{R}^{d\times d}$ and $\mathcal{S}_{\bm{\nu}} \in \mathbb{R}^{n}$ be compact sets, and let $\psi: \mathcal{S}_{\bm{W}} \times \mathcal{S}_{\bm{\nu}} \to \mathbb{R}$ be a continuous function.  Define
\begin{align}
    \Phi(\mathcal{A}) :&= \min_{\bm{W}\in \mathcal{S}_{\bm{W}}} \max_{\bm{\nu} \in \mathcal{S}_{\bm{\nu}}} \langle \mathcal{A}[ \bm{W}], \bm{\nu} \rangle  + \psi(\bm{W}, \bm{\nu}); \\
    \phi(\bm{g},\bm{H}) :&= \min_{\bm{W}\in \mathcal{S}_{\bm{W}}} \max_{\bm{\nu} \in \mathcal{S}_{\bm{\nu}}} \| \bm{W} \|_F \langle \bm{g}, \bm{\nu} \rangle  + \| \bm{\nu}\| \langle \bm{H}, \bm{W} \rangle + \psi(\bm{W}, \bm{\nu}).
\end{align}
Then it holds that
\begin{align}
    \p\bigg(  \Phi(\mathcal{A})  \leq  t \bigg) \leq 2 \p \bigg(  \phi(\bm{g,H})  \leq t\bigg). 
\end{align}
If in addition both $\mathcal{S}_{\bm{W}}$ and $\mathcal{S}_{\bm{\nu}}$ are convex and $\psi(\bm{W},\bm{\nu})$ is convex-concave, then
\begin{align}
    \p\bigg(  \Phi(\mathcal{A})  \geq  t \bigg) \leq 2 \p \bigg(  \phi(\bm{g,H})  \geq t\bigg). 
\end{align}
\end{theorem}
\begin{proof}
    See \cref{sec:cgmtproof}. 
\end{proof}
Observe that unlike the \emph{vector} CGMT \citep{thrampoulidis_gaussian_2015}, the operator $\mathcal{A}$ consists of ${\sf GOE}(d)$ matrices, whereas the vector CGMT uses a single random matrix, and the minimization is over the bilinear form generated by this random matrix.  In \cref{thm:matrixcgmt} the new quantity $\phi(\bm{g},\bm{H})$ converts the randomness in $\mathcal{A}$ to two independent sources of randomness: the random vector $\bm{g}$ and the random matrix $\bm{H}$.  In contrast, the vector CGMT converts the randomness in the single random matrix into two independent random Gaussian vectors.

\subsection{A Family of Soft-Thresholding Estimators and Associated Fixed-Point Equations} \label{sec:softthresholding}
In this subsection we give the definition of the particular soft-thresholding estimator about which $\bm{Z}^{(\lambda)}$ concentrates in  \cref{thm:cvxasymptotics}. We will also state \cref{lem:ZSTtau}, which shows that this soft-thresholding estimator is asymptotically equal to $\bm{Z}_{{\sf ST}}^{(\lambda)}$.  En route to stating this result, we state several intermediate results that we rely on throughout our proofs.

For two scalars  $\tau,\zeta$, we define the soft-thresholding estimator
\begin{align}
    \bm{Z}_{{\sf ST}}^{(\lambda/\zeta)}(\tau) &= \argmin_{\bm{Z} \succcurlyeq 0} \frac{\zeta}{2} \| \bm{Z} - (\bm{M} + \tau \bm{H}) \|_F^2 + \lambda \| \bm{Z} \|_*.
\end{align}
In \cref{thm:cvxasymptotics} we show that $\bm{Z}^{(\lambda)}$ approximates $\bm{Z}^{(\lambda/\zeta\s)}(\tau\s)$ for appropriate choices of $(\tau\s,\zeta\s)$ that are asymptotically equal to $(\sigma,1)$. In fact, observe that by taking $\zeta = 1$ and $\tau = \sigma$ above we arrive at $\bm{Z}_{{\sf ST}}^{(\lambda)}$ as in \cref{thm:mainthm}.

The pair ($\tau\s,\zeta\s$) arise directly from our analysis via the Matrix CGMT.  Here we give a heuristic derivation, with additional details in \cref{sec:heuristicderivation}.  Recall that $\bm{Z}^{(\lambda)}$ satisfies
\begin{align}
 \bm{Z}^{(\lambda)} =    \argmin_{\bm{Z}\succcurlyeq0} \frac{1}{2} \| \mathcal{X}(\bm{Z - M}) - \bm{\eps} \|_F^2 + \lambda \| \bm{Z} \|_*.
\end{align}
By reparameterizing this cost function with $\bm{W} := \bm{Z} - \bm{M}$ and expanding out the Frobenius norm in its variational form, we arrive at the new cost function
\begin{align}
    C_{\lambda}^{(n)}(\bm{W}) &=  \max_{\bm{\nu} \in \mathbb{R}^{n}} \langle \mathcal{X}[\bm{W}], \bm{\nu} \rangle - \langle \bm{\eps} , \bm{\nu} \rangle - \frac{1}{2} \| \bm{\nu} \|^2 + \lambda \big\{ \| \bm{W+M} \|_{*} - \|\bm{M}\|_{*}\big\}.
\end{align}
Consequently, since we are minimizing over $\bm{W}$, the minimization above takes the form required for the Matrix CGMT.  Therefore, 
we can show that the minimum value of $C_{\lambda}^{(n)}$ over any set $\mathcal{W}$ is related to that of the function $L_{\lambda}^{(n)}(\mathcal{W})$, defined via
\begin{subequations}
\label{l_lambda_n}    
\begin{align}
    L_{\lambda}^{(n)}(\mathcal{W}) &= \max_{\beta > 0} \min_{\tau \geq \sigma} \bigg( \frac{\sigma^2}{\tau} + \tau \bigg) \frac{\|\bm{g}\|}{\sqrt{n}}\frac{\gamma_n \beta}{2} - \gamma_n \frac{\beta^2}{2} \\
    &\quad + \min_{ \substack{\bm{W} \in\mathcal{W}\\\bm{W} + \bm{M} \succcurlyeq 0}} \bigg\{  \beta \frac{\|\bm{W}\|_F^2}{2 dr \tau} + \beta \frac{\langle \bm{H}, \bm{W} \rangle}{dr} + \frac{\lambda}{dr} \big\{ \| \bm{W} + \bm{M} \|_* - \|\bm{M} \|_* \big\} \bigg\}.
\end{align}
\end{subequations}
Note that the argument of $L_{\lambda}^{(n)}$ is a set (the set for which $\bm{W}$ is minimized over).  We will also abuse notation and denote $L_{\lambda}(\bm{W}) = L_{\lambda}(\{\bm{W}\})$.  

It turns out that $L_{\lambda}^{(n)}(\mathbb{R}^{d\times d})$ concentrates about the quantity $L_{\lambda}^{*}$ (roughly speaking, its expected value) defined via 
\begin{align}
    L_{\lambda}^* := \max_{\beta > 0} \min_{\tau \geq \sigma} \psi_{\lambda}^* (\tau,\beta),
\end{align}
where  $\psi^*_{\lambda}(\cdot,\cdot)$ is the (scalar) valued function defined via
\begin{align}
        \psi^{*}_{\lambda}(\tau,\beta) &:= \bigg( \frac{\sigma^2}{\tau} + \tau \bigg) \frac{\gamma_n \beta}{2} - \gamma_n \frac{\beta^2}{2} +  \mathbb{E} \min_{\bm{Z}} \beta\frac{\|\bm{Z} - \bm{M} \|_F^2}{2dr \tau}  + \beta \frac{\langle \bm{H}, \bm{Z} - \bm{M} \rangle}{dr}  + \frac{\lambda}{dr} \big( \| \bm{Z}\|_* - \|\bm{M} \|_* \big). 
\end{align}
Note that $\psi^*(\tau,\beta)$ resembles \eqref{l_lambda_n}.   
From this derivation, we arrive at the definition of $(\tau\s,\zeta\s)$. We define $(\tau\s,\zeta\s)$ such that 
\begin{align*}
    L_{\lambda}\s &= \psi_{\lambda}\s(\tau\s, \zeta\s/\tau\s).
\end{align*}
In words, the quantities $\tau\s$ and $ \beta\s = \zeta\s/\tau\s$ are simply the critical points such that $L_{\lambda}^{*} = \psi^*_{\lambda}(\tau\s,\beta\s)$.

It is not immediate that $(\tau\s,\zeta\s)$ are well-defined.  To establish existence and uniqueness, first we give the following lemma that gives an alternative characterization in terms of  of fixed-point equations which can be obtained by studying the first-order necessary conditions for the critical points of $\psi_{\lambda}^*$.

\begin{lemma}\label{lem:saddlepoint}
Define the system of equations
\begin{subequations}
\label{fxdpt}
 \begin{align}
    \tau^2 &= \sigma^2 + \frac{dr}{n} {\sf R}_{\lambda}\big( \tau^2, \zeta); \label{fxtpt1} \\
    \zeta &=1 -  \frac{{\sf df}_{\lambda}\big( \tau^2,\zeta)}{n},  \numberthis \label{fxdpt2}
\end{align}
\end{subequations}
where  
\begin{align}
    {\sf R}_{\lambda}(\tau^2,\zeta) :&=  \frac{1}{dr} \mathbb{E} \bigg[ \big\| \bm{ Z}_{\sf ST}^{(\lambda/\zeta)}(\tau) - \bm{M} \|_F^2 \bigg ];  \\ 
    {\sf df}_{\lambda}(\tau^2,\zeta) &= \frac{1}{\tau} \mathbb{E} \bigg\langle \bm{H} ,  \bm{ Z}_{\sf ST}^{(\lambda/\zeta)}(\tau) \bigg\rangle 
\end{align}
where the expectation is with respect to $\bm{H} \sim {\sf GOE}(d)$.   Let $\tau\s$ and $\beta\s = \zeta\s \tau\s$ be such that $(\tau\s,\zeta\s)$ satisfy the system of equations \eqref{fxdpt}.  Then it holds that $L^{*}_{\lambda} = \psi^{*}_{\lambda}(\tau\s,\beta\s)$. 
\end{lemma}
\begin{proof}
    Simply take the derivative and set equal to zero.  
\end{proof}

Readers familiar with the literature on proportional asymptotics for linear regression may recognize that the fixed-point equations \eqref{fxdpt} bear significant similarity to other fixed-point equations that arise in the study of high-dimensional (regularized) M-estimators \citep{bellec_existence_2024,donoho_high_2016,el_karoui_robust_2013,koriyama_phase_2025,berthier_state_2020}.  In particular, the equations \eqref{fxdpt} have an interpretation as the matrix analogue of the fixed-point equations for the LASSO studied in \citet{celentano_lasso_2023} and \citet{miolane_distribution_2021}.  For more details, see the explanation in \citet{celentano_lasso_2023}.  

Next, the following lemma relies on the system of equations \eqref{fxdpt} to establish the existence and uniqueness of $(\tau\s,\zeta\s)$.
\begin{lemma}[Existence and Uniqueness of $(\tau\s,\zeta\s)$] \label{prop:fxdpt} For any fixed $\lambda >0$, if $n,d \geq C$, there exist unique $\zeta\s > 0 $ and $\tau\s \geq \sigma$ that satisfy the system of equations \eqref{fxdpt}. 
\end{lemma}

\begin{proof}
    See \cref{sec:fixedpointproof}.
\end{proof}

Thus far we have shown that $\bm{Z}^{(\lambda/\zeta\s)}(\tau\s)$, which is more directly related to $\bm{Z}^{(\lambda)}$, is well-defined.  However, in \cref{thm:mainthm} we show that $\bm{Z}^{(\lambda)}$ concentrates around $\bm{Z}^{(\lambda)}_{{\sf ST}} \equiv \bm{Z}^{(\lambda)}_{{\sf ST}}(\sigma)$.  Therefore, the following result shows that under \cref{ass1,ass2,ass3} that $\tau\s$ and $\zeta\s$ approximate $\sigma$ and $1$ respectively.

\begin{lemma}[Properties of the Fixed-Point Solutions]\label{lem:fxdpointprops}
Suppose that \cref{ass1,ass2,ass3} are satisfied.  Then it holds that $\sigma \leq \tau\s \leq C \sigma$ and 
\begin{align}
    |(\tau\s)^2 - \sigma^2 | =O\bigg( \frac{d r}{n} \sigma^2 \bigg); \qquad |\zeta\s - 1 | = O\bigg( \frac{d r^{3/2}}{n} \bigg),
\end{align}
where the implicit constants depend on the constants in our assumptions.  
\end{lemma}
\begin{proof}
    See \cref{sec:fixedpointprops}.
\end{proof}

 In  \cref{sec:fixedpointprops} we also establish several key properties of the soft-thresholding estimator $\bm{Z}_{{\sf ST}}^{(\lambda/\zeta\s)}(\tau\s)$ that we use throughout our proofs.  Finally, the following lemma establishes that the replacement of $\bm{Z}_{{\sf ST}}^{(\lambda/\zeta\s)}(\tau\s)$ with $\bm{Z}_{{\sf ST}}^{(\lambda)}$ is asymptotically negligible.
 
\begin{lemma}\label{lem:ZSTtau} Suppose that \cref{ass1,ass2,ass3} are satisfied. Then for any $1$-Lipschitz function $\phi: \mathbb{R}^{d\times d} \to \mathbb{R}$ it holds that
    \begin{align}
         \bigg| \mathbb{E} \phi \big( \bm{Z}_{{\sf ST}}^{(\lambda/\zeta\s)}(\tau\s)/\sqrt{dr} \big) - \mathbb{E} \phi \big( \bm{Z}_{{\sf ST}}^{(\lambda)}/\sqrt{dr} \big) \bigg| = O\bigg( \sigma \frac{d r^{3/2}}{n}\bigg).
    \end{align}
\end{lemma}

\begin{proof}
    See \cref{sec:ZHTtau}.  
\end{proof}

\subsection{High-Dimensional Properties of Convex and Regularized Nonconvex Estimators} \label{sec:cvxasymptotics}
We are now prepared to establish the high-dimensional properties of the convex estimator and its nonconvex counterpart for the choice $\tau\s$ and $\zeta\s$ from the previous section.  The full proof of this result is in \cref{sec:step1}, though we give a high-level sketch of the proof momentarily.  
\begin{theorem} 
 \label{thm:cvxasymptotics}
Suppose that \cref{ass1,ass2,ass3} are satisfied. Then for any 1-Lipschitz function $\phi: \mathbb{R}^{d \times d} \to \mathbb{R}$ and for any $\eps$ satisfying $\eps > C \exp( - c d r)$, it holds that
\begin{align}
   \p\bigg\{ &\bigg| \phi\bigg( \frac{1}{\sqrt{dr}} \bm{U}^{(\lambda) } \bm{U^}{(\lambda)\top} \bigg) - \mathbb{E}_{\bm{H}} \phi\bigg( \frac{1}{\sqrt{dr}} \bm{Z}^{(\lambda/\zeta\s)}_{{\sf ST}}(\tau\s) \bigg) \bigg| > \eps \bigg\}\\
   &\leq  O\bigg( \exp( - c d) +  \exp( - c dr) + \exp( - c n) + \frac{1}{\eps^2} \exp(-c dr \eps^4 ) - \exp( - \frac{(dr)^2}{n} \eps^4) \bigg),
\end{align}
where the constants depend on problem parameters.  The same result holds if $\bm{U}^{(\lambda)} \bm{U}^{(\lambda)\top}$ is replaced with $\bm{Z}^{(\lambda)}$.
\end{theorem}

We remark that \cref{thm:cvxasymptotics} does not contain an additional error term depending on $\gamma_n$.  In essence, this is due to the fact that we replace $\bm{Z}_{{\sf ST}}^{(\lambda/\zeta\s)}(\tau\s)$ with $\bm{Z}_{{\sf ST}}^{(\lambda)} = \bm{Z}_{{\sf ST}}^{(\lambda)}(\sigma)$ in the statement of \cref{thm:mainthm}.   As shown in \cref{lem:fxdpointprops}, it holds that $\tau\s = \sigma + O(\gamma_n\inv)$ when $\gamma_n$ is large. Therefore, the additional error term  in \cref{thm:mainthm} arises  due to this replacement.  In the asymptotic regime $\frac{n}{d r^2} \asymp 1$, this replacement is non-negligible, but \cref{thm:cvxasymptotics} ameliorates this.  

\cref{thm:cvxasymptotics} is comparable to the results on the LASSO developed in \citet{miolane_distribution_2021,celentano_lasso_2023,javanmard_debiasing_2018}, among others.  The statement is most directly related to the main result of \citet{celentano_lasso_2023}, who establish an equivalence between the LASSO estimator for sparse linear regression with the corresponding soft-thresholding of a vector denoising problem.  The primary difference between our result and theirs comes from the fact that we are considering matrices as opposed to vectors, whence the function $L_{\lambda}^{(n)}$ is no longer strongly convex. One of our major contributions and key technical ingredients is to show that $L_{\lambda}^{(n)}$ maintains strong convexity when restricted to rank at most $r$ matrices.

\begin{proof}[Sketch of the Proof of \cref{thm:cvxasymptotics}]
    The high-level strategy will be to argue that any minimizer of $C_{\lambda}^{(n)}(\bm{W})$ over the shifted positive semidefinite cone $\{\bm{W}: \bm{W} + \bm{M} \succcurlyeq 0 \}$ must be within $\eps \sqrt{dr}$ of $\bm{ Z}_{{\sf ST}}^{(\lambda/\zeta\s)}(\tau\s)$ with high probability.  
Let $\bm{\hat W} := \bm{ Z}_{{\sf ST}}^{(\lambda/\zeta\s)}(\tau\s)- \bm{M}$, and let $\phi$ be a fixed 1-Lipschitz function. Consider the set $\mathcal{D}$ defined via
\begin{align}
    \mathcal{D} := \bigg\{ \bm{W} :\bigg| \phi\bigg( \frac{\bm{W} + \bm{M}}{\sqrt{dr}}\bigg) - \mathbb{E} \phi\bigg( \frac{\bm{\hat W} + \bm{M}}{\sqrt{dr}} \bigg) \bigg|\leq \frac{\eps}{2} \bigg\} \cap \bigg\{ \bm{W} + \bm{M} \succcurlyeq 0 \bigg\},
\end{align}
and let $\mathcal{D}_{\eps}$ denote the $\eps$-enlargement of the set $\mathcal{D}$; i.e. $\mathcal{D}_{\eps} := \{ \bm{V} = \bm{W} + \bm{W}': \| \bm{W}' \|_F \leq \eps, \bm{W} \in \mathcal{D}\}$.  Since by the proof of \cref{thm:relatecvxncvx} $\bm{Z}^{(\lambda)}$ is rank at most $r$ and within radius $R = C \sigma \sqrt{dr}$ of $\bm{M}$ with high probability, we will show that any rank at most $r$ minimum of $C^{(n)}_{\lambda}$ on $\mathcal{D}_{\eps/2}^c$ (outside of the $\eps/2$ ball) but within radius $R$ of $\bm{M}$ must be suboptimal with high probability, which will imply that the minimizer must lie within $\mathcal{D}_{\eps/2}$ with high probability.

Carrying out this argument, in our full proof of \cref{thm:cvxasymptotics} in \cref{sec:step1}, we show via the Matrix CGMT that it holds that
\begin{align}
    \p\bigg\{ \argmin ~ &C^{(n)}_{\lambda}(\bm{W}) \notin \mathcal{D}_{\eps/2} \bigg\} \\
    &\leq 2 \p\bigg\{ \bm{\hat W} \notin \mathcal{D} \bigg\}+ 2 \p\bigg\{ L_{\lambda}^{(n)}( \mathbb{R}^{d\times d}) \geq L_{\lambda}^* + C \eps^2 \bigg\} \\
    &+ 2 \p\bigg\{ L_{\lambda}^{(n)}\bigg( \mathbb{B}_{\eps/2}^c(\bm{\hat W}/\sqrt{dr}) \cap \{\mathrm{rank}(\bm{W + M}) \leq r \} \cap \{\| \bm{W}\|_F \leq R\} \bigg)  \leq L_{\lambda}^* + C \eps^2 \bigg\},
\end{align}
where $L_{\lambda}^*$ is the quantity defined in the previous subsection. Therefore, the proof boils down to analyzing the three terms on the right hand side above. It turns out that the first two terms are straightforward to bound: the first term is exponentially small as $\bm{Z}^{(\lambda/\zeta\s)}_{{\sf ST}}(\tau\s)$ is a proximal operator and hence Lipschitz.  The second term is also well-behaved as $L_{\lambda}^{(n)}$ is expected to concentrate about $L_{\lambda}^*$  at its global minimum.

The final remaining quantity is significantly more difficult to analyze.  The term inside the probability can be understood as the minimum of value of $L_{\lambda}^{(n)}$ over the set of matrices $\bm{W}$ such that:
\begin{enumerate}
    \item  They are at least $\eps/2$ distance away from $\bm{\hat W}$;
    \item $\bm{W} + \bm{M}$ is rank at most $r$;
    \item $\bm{W}$ has norm at most $R$.
\end{enumerate}
In order to analyze the function $L_{\lambda}^{(n)}$ over matrices such that $\mathrm{rank}(\bm{W} + \bm{M}) \leq r $, we define the auxiliary function
\begin{align}
    h(\bm{U}) &= \frac{\gamma_n}{2} \bigg( \sqrt{\frac{\|\bm{UU}\t - \bm{M}\|_F^2}{n} + \sigma^2} \frac{\|\bm{g}\|}{\sqrt{n}} - \frac{\langle \bm{H}, \bm{UU}\t - \bm{M}\rangle}{n} \bigg)_+^2 + \frac{\lambda}{dr} \big\{ \| \bm{U}\|_F^2 - \|\bm{M}\|_* \big\}.
\end{align}
It is straightforward to see that if $\bm{W} = \bm{UU}\t - \bm{M}$, then $h(\bm{U}) = L_{\lambda}^{(n)}(\bm{W})$.  Therefore, analyzing $L_{\lambda}^{(n)}$ over rank at most $r$ matrices is equivalent to analyzing the function $h(\bm{U})$ over $\mathbb{R}^{d \times r}$.   Intuitively, since  $L_{\lambda}^{(n)}(\mathbb{R}^{d \times d})$ is close to $L_{\lambda}\s$, the optimization geometry of $L_{\lambda}^{(n)}$ should be similar to that of matrix soft-thresholding.  Therefore, since $h(\bm{U})$ is the ``nonconvex version'' of $L_{\lambda}^{(n)}$, the geometry of $h(\bm{U})$ should behave similarly to that of nonconvex matrix factorization.  

Carrying out this intuition, in \cref{lem:hstronglyconvex} we show that the function $h(\bm{U})$ is strongly convex on an appropriate high-probability event $\mathcal{E}_h$ within a region $\mathcal{R} \subset \mathbb{R}^{d \times r}$ defined via
\begin{align}
    \mathcal{R} := \bigg\{ \bm{U}: \| \bm{U} - \bm{U}\s \mathcal{O}_{\bm{U}\s,\bm{U}} \| \leq \frac{c}{\kappa} \| \bm{U}\s \| \bigg\},
\end{align}
where $c$ is some universal constant, $\bm{U}\s$ is such that $\bm{M}  = \bm{U}\s \bm{U}^{*\top}$, and $\mathcal{O}_{\bm{U}\s,\bm{U}} $ is the Frobenius-optimal orthogonal matrix aligning $\bm{U}$ and $\bm{U}\s$.  This argument is significantly different from previous analyses for the LASSO, in which the analogue of $L_{\lambda}^{(n)}$ is already strongly convex in a region near its optimal value \citep{miolane_distribution_2021}.  In contrast, in our setting the function $L_{\lambda}^{(n)}$ need not be strongly convex near $\bm{\hat W}$, but $h(\bm{U})$ can be shown to be (near $\bm{Z}_{{\sf ST}}^{(\lambda/\zeta\s)}(\tau\s)$).  This discrepancy is due to the fact that $\langle \bm{H},\bm{W} \rangle$ does not concentrate at rate $O(\sqrt{dr} \|\bm{W}\|_F)$ uniformly over all matrices $\bm{W}$ (which is what is required on the event $\mathcal{E}_h$ for strong convexity), but $\langle \bm{H}, \bm{UU}\t - \bm{M} \rangle$ does concentrate at this rate uniformly over all matrices $\bm{U}$.

We then demonstrate that when $\|\bm{UU}\t - \bm{M} \|_F \leq R$, then this automatically implies that $\bm{U} \in \mathcal{R}$. Consequently, by strong convexity, we show that there is a high-probability event $\mathcal{E}_{{\sf ST}}$ (depending only on $\bm{Z}^{(\lambda/\zeta\s)}(\tau\s)$) such that if
\begin{align}
    \frac{\|\bm{UU}\t - \bm{Z}^{(\lambda/\zeta\s)}(\tau\s)\|_F}{\sqrt{dr}} > \frac{\eps}{2},
\end{align}
then $\bm{U}$ must satisfy 
\begin{align}
    h(\bm{U}) > \min_{\bm{U}: \|\bm{U}- \bm{U}\s \mathcal{O}_{\bm{U}\s,\bm{U}} \|_F \leq \frac{R}{8 \|\bm{U}\s\|}} h(\bm{U}) + C\eps^2,
\end{align}
where $C$ is some universal constant.   The above display turns out to be equivalent to the expression
\begin{align}
    L_{\lambda}^{(n)}\bigg( \mathbb{B}_{\eps/2}^c(\bm{\hat W}/\sqrt{dr}) \cap \mathrm{rank}(\bm{W} + \bm{M}) \leq r \cap \{ \|\bm{W}\|_F \leq R \} \bigg) > L_{\lambda}^* + C \eps^2,
\end{align}
which occurs on the event $\mathcal{E}_h \cap \mathcal{E}_{{\sf ST}}$.  Therefore, the proof is completed by tabulating the probabilities of these events.  
\end{proof}

\section{Study of Debiased Estimators} \label{sec:debequivalence}
In this section we study the debiased estimators $\bm{Z}_{{\sf deb}}^{(\lambda)}$ and $\bm{U}_{{\sf deb}}^{(\lambda)}$ defined via  
\begin{align}
    \bm{Z}^{(\lambda)}_{{\sf deb}} &:= \mathcal{P}_{{\sf rank}-r} \bigg( \bm{Z}^{(\lambda)} + \frac{1}{n} \sum_{i=1}^{n} \bigg( y_i - \langle \bm{X}_i, \bm{Z}^{(\lambda)} \bigg) \bm{X}_i \bigg); \\
    \bm{U}^{(\lambda)}_{{\sf deb}} &:= \bm{U}^{(\lambda)} \bigg( \bm{I} + \lambda \big( \bm{U}^{(\lambda)\top} \bm{U}^{(\lambda)} \big)\inv \bigg)^{1/2}.
\end{align}
In \cref{sec:equivdebiased} we prove that both definitions above are equivalent on the event $\mathcal{E}_{{\sf Good}}$, and then in \cref{sec:debcontrol} we study their high-dimensional properties.  

\subsection{Equivalence of Debiased Estimators}
\label{sec:equivdebiased}
Similar to the case of $\bm{U}^{(\lambda)}$ and $\bm{Z}^{(\lambda)}$, on the event $\mathcal{E}_{{\sf Good}}$, both $\bm{U}_{{\sf deb}}^{(\lambda)}$ and $\bm{Z}^{(\lambda)}_{{\sf deb}}$ are equivalent, which we demonstrate in the following theorem.   
\begin{theorem} \label{thm:udebzdeb}    On the event $\mathcal{E}_{{\sf Good}}$, under \cref{ass2,ass3} it holds that 
   \begin{align}
         \bm{Z}_{{\sf deb}}^{(\lambda)} &= \bm{V}_{U} \big( \bm{ \Sigma}+ \lambda \bm{I} \big) \bm{V}_U\t, \label{eq:zdebidentity1}
     \end{align}
     where $\bm{V}_U \bm{\Sigma}^{1/2} \bm{W}_U\t$ is the singular value decomposition of $\bm{U}^{(\lambda)}$.   Furthermore, on this same event it holds that 
 \begin{align}
     \bm{U}_{{\sf deb}}^{(\lambda)} \bm{U}_{{\sf deb}}^{(\lambda)\top} &= \bm{Z}_{{\sf deb}}^{(\lambda)}. \label{eq:zdebidentity2}
     \end{align}
\end{theorem}
Similar to \cref{thm:relatecvxncvx}, this result is deterministic on the event $\mathcal{E}_{{\sf Good}}$.  
\begin{proof}
    Observe that from \cref{thm:relatecvxncvx}, on the event $\mathcal{E}_{{\sf Good}}$ it holds that $\bm{Z}^{(\lambda)} = \bm{U}^{(\lambda)} \bm{U}^{(\lambda)\top}$ under \cref{ass2,ass3}. Furthermore, from the proof of \cref{thm:relatecvxncvx}  we have that
    \begin{align}
        \mathcal{X}\s \mathcal{X} \big( \bm{U}^{(\lambda)} \bm{U}^{(\lambda)\top} - \bm{M} \big) - \bm{N}_{\eps} = - \lambda \bm{V}_U \bm{V}_U\t + \bm{W},
    \end{align}
    where $\bm{W}$ is orthogonal to $\bm{V}_U$ and has spectral norm bounded by $\lambda$.  Therefore, we have that
\begin{align}
    \bm{Z}^{(\lambda)} + \frac{1}{n} \sum_{i=1}^{n} \bigg( y_i - \langle \bm{X}_i , \bm{Z}^{(\lambda)} \rangle \bigg) \bm{X}_i &= \bm{U}^{(\lambda)} \bm{U}^{(\lambda)\top} + \mathcal{X}\s \mathcal{X} \big( \bm{M} - \bm{U}^{(\lambda)} \bm{U}^{(\lambda)\top} \big) + \bm{N}_{\eps} \\
    &= \bm{U}^{(\lambda)} \bm{U}^{(\lambda)\top} + \lambda \bm{V}_U \bm{V}_U\t - \bm{W} \\
    &= \bm{V}_U \big( \bm{\Sigma} + \lambda \bm{I}  \big) \bm{V}_U\t - \bm{W},
\end{align}
where the final equality is due to the fact that $\bm{U}^{(\lambda)}$ has singular vectors $\bm{V}_U$.  After taking the rank $r$ projection and recognizing that the singular values of the matrix $\bm{V}_U \big( \bm{\Sigma} + \lambda \bm{I} \big) \bm{V}_U\t$ are strictly larger than $\lambda$, we obtain that
\begin{align}
    \bm{Z}^{(\lambda)}_{{\sf deb}} &= \bm{V}_U \big( \bm{\Sigma} + \lambda \bm{I} \big) \bm{V}_U\t,
\end{align}
which proves \eqref{eq:zdebidentity1}.
Furthermore, if $\bm{U}^{(\lambda)} = \bm{V}_U \bm{\Sigma}^{1/2} \bm{W}_U\t$, then
    \begin{align}
        \bm{U}^{(\lambda)}_{{\sf deb}} \bm{U}^{(\lambda)\top}_{{\sf deb}} &= \bm{U}^{(\lambda)} \bigg( \bm{I} + \lambda \big( \bm{U}^{(\lambda)\top} \bm{U}^{(\lambda)} \big)\inv \bigg)  \bm{U}^{(\lambda)\top} \\
        &= \bm{V}_U \bm{\Sigma}^{1/2} \bm{W}_U\t \bigg(   \bm{I} + \lambda \bm{W}_U \bm{\Sigma}^{-1} \bm{W}_U\t \bigg) \bm{W}_U \bm{\Sigma}^{1/2} \bm{V}_U\t \\
        &= \bm{V}_U \big( \bm{\Sigma} + \lambda \bm{I}\big) \bm{V}_U\t,
     \end{align}
     which proves \eqref{eq:zdebidentity2}.
\end{proof}

\subsection{Concentration of Debiased Estimators about Hard-Thresholding Estimator} \label{sec:debcontrol}
Next, as in the case of the previous section, we will show that both $\bm{U}^{(\lambda)}_{{\sf deb}}$ and $\bm{Z}^{(\lambda)}_{{\sf deb}}$ concentrate around a particular hard-thresholding estimator (that may not necessarily be equal to $\bm{Z}_{{\sf HT}}$ if $\gamma_n$ remains bounded).  Define 
\begin{align}
    \bm{Z}^{(\lambda)}_{{\sf HT}} := \mathcal{P}_{{\sf rank}-r} \bigg( \bm{M} + \tau\s \bm{H} \bigg),
\end{align}
where $\tau\s$ denotes the parameter from the fixed-point equations in \eqref{fxdpt}.  Here we note that $\bm{Z}_{{\sf HT}}^{(\lambda)}$ only depends on $\lambda$ implicitly through $\tau\s$.  
The following result shows that $\bm{Z}_{{\sf deb}}^{(\lambda)}$ closely approximates $\bm{Z}_{{\sf HT}}^{(\lambda)}$.

\begin{theorem}[High-Dimensional Asymptotics for $\bm{Z}^{(\lambda)}_{{\sf HT}}$]\label{thm:zdeb} Suppose that \cref{ass1,ass2,ass3} are satisfied.
Let $\bm{Z}^{(\lambda)}_{{\sf deb}}$ denote the debiased convex estimator.  Let $\phi$ be any $1$-Lipschitz function on $\mathbb{R}^{d \times d}$.  Then
\begin{align}
    \p\bigg\{ \bigg|   &\phi\bigg( \frac{1}{\sqrt{dr}}\bm{ Z}^{(\lambda)}_{{\sf deb}} \bigg) - \mathbb{E}_{\bm{H}} \phi\bigg( \frac{1}{\sqrt{dr}}\bm{Z}^{(\lambda)}_{\sf HT} \bigg) \bigg| > \eps \bigg\} \\
    &\leq  O\bigg( \exp( - c d) +  \exp( - c dr) + \exp( - c n) + \frac{1}{\eps^2} \exp(-c dr \eps^4 ) - \exp( - \frac{(dr)^2}{n} \eps^4) \bigg).
\end{align}
\end{theorem}
\begin{proof}
We will argue that $\phi\big( \frac{1}{\sqrt{dr}} \bm{Z}^{(\lambda)}_{{\sf deb}} \big)$ is a Lipschitz function of $\bm{Z}^{(\lambda)}$ on the event $\mathcal{E}_{{\sf Good}}$. Indeed, on this event, by the proof of \cref{thm:relatecvxncvx}, it holds that $\|\bm{Z}^{(\lambda)} - \bm{M} \|_F\leq C \sigma \sqrt{dr}$. From the condition $\lambda_{\min}(\bm{M}) \geq C_0 \sigma \sqrt{dr}$,  Weyl's inequality implies that
\begin{align}
\lambda_r(\bm{Z}^{(\lambda)}) \geq \lambda_r(\bm{M}) - C \sigma \sqrt{dr} \geq (C_0 - C) \sigma \sqrt{dr} \geq \lambda \sqrt{r},
\end{align}
provided that $C_0$ is sufficiently large.  From \cref{thm:udebzdeb}, on the event $\mathcal{E}_{{\sf Good}}$ it holds that $\bm{Z}^{(\lambda)}_{{\sf deb}} = \bm{V}_U \big( \Sigma + \lambda \bm{I} \big) \bm{V}_U\t$, where $\bm{V}_U$ are the leading $r$ eigenvectors of $\bm{Z}^{(\lambda)}$.  Let $\eta$ denote the function which takes a matrix and adds $\lambda/\sqrt{dr}$ to the leading $r$ eigenvalues of a matrix.  Then since $\bm{Z}^{(\lambda)} = \bm{V}_U \bm{\Sigma} \bm{V}_U\t$,
\begin{align}
    \eta( \bm{Z}^{(\lambda)}/\sqrt{dr}) = \frac{1}{\sqrt{dr}}\bm{V}_U \big(  \bm{\Sigma} + \lambda \bm{I} \big) \bm{V}_U\t = \bm{Z}^{(\lambda)}_{{\sf deb}}/\sqrt{dr},
\end{align}
so that $\eta(\bm{Z}^{(\lambda)}/\sqrt{dr}) = \bm{Z}_{{\sf deb}}^{(\lambda)}/\sqrt{dr}$.  We claim that $\eta $ is Lipschitz on the set
\begin{align}
  \mathcal{S} :=  \{\bm{A} \in\mathbb{R}^{d\times d} : \lambda_{r+1}(\bm{A}) =0; \lambda_{r}(\bm{A}) \geq \frac{\lambda}{\sqrt{d}}\},
\end{align}
which contains $\bm{Z}^{(\lambda)}/\sqrt{dr}$ on the event $\mathcal{E}_{{\sf Good}}$.  
To see that $\eta$ is Lipschitz, suppose that $\bm{A}$ and $\bm{B}$ are two matrices in $\mathcal{S}$, and let them have eigendecompositions $\bm{U}_A \bm{\Sigma}_A \bm{U}_A\t$ and $\bm{U}_B \bm{\Sigma}_B \bm{U}_B\t$ respectively.  First, if $\|\bm{A} - \bm{B} \|_F < \lambda/\sqrt{d}$, then
\begin{align}
    \| \eta (\bm{A} ) - \eta (\bm{B}) \|_F &= \| \bm{U}_A (\bm{\Sigma} + \frac{\lambda}{\sqrt{dr}}\bm{I}) \bm{U}_A\t - \bm{U}_B ( \bm{\Sigma}_B + \frac{\lambda}{\sqrt{dr}} \bm{I}) \bm{U}_B\t \|_F \\
    &\leq \| \bm{A} - \bm{B} \|_F +\frac{\lambda}{\sqrt{dr}} \| \bm{U}_A \bm{U}_A\t - \bm{U}_B \bm{U}_B\t \|_F \\
    &\leq \| \bm{A} - \bm{B}\|_F +\frac{\lambda}{\sqrt{dr}}\frac{\|\bm{A} - \bm{B}\|_F}{\frac{\lambda}{\sqrt{d}}} \\
    &\leq 2 \| \bm{A} - \bm{B}\|_F,
\end{align}
where we have used the Davis-Kahan Theorem.  On the other hand, if $\|\bm{A} - \bm{B}\|_F \geq \lambda/\sqrt{d}$, then
\begin{align}
     \| \eta (\bm{A} ) - \eta (\bm{B}) \|_F &\leq \| \bm{A} - \bm{B}\|_F + \frac{\lambda}{\sqrt{dr}} \| \bm{U}_A \bm{U}_A\t - \bm{U}_B \bm{U}_B\t \|_F \\
     &\leq \| \bm{A} - \bm{B}\|_F + \frac{1}{\sqrt{r}} \| \bm{A} - \bm{B}\|_F \| \bm{U}_A \bm{U}_A\t - \bm{U}_B \bm{U}_B\t \|_F \\
     &\leq 3 \| \bm{A} - \bm{B} \|_F,
\end{align}
since the Frobenius norm of the difference of projection matrices is always at most $\sqrt{2r}$.  Consequently, $\eta$ is at most a $3$-Lipschitz function on this set.  By Kirszbaum's theorem, there exists a Lipschitz extension $\tilde \eta : \mathbb{R}^{d\times d} \to \mathbb{R}^{d\times d}$ such that $\tilde \eta $ has the same Lipschitz constant of $\eta $ and $\tilde \eta |_{\mathcal{S}} = \eta$.  In particular, on the event $\mathcal{E}_{{\sf Good}}$, $\eta $ and $\tilde \eta$ coincide.  Let $\phi$ be any Lipschitz function, and let $\tilde \phi = \phi \circ \tilde \eta$, which is also Lipschitz.  By \cref{thm:cvxasymptotics},
\begin{align}
    \p\bigg\{  \phi \bigg( &\frac{1}{\sqrt{dr}} \bm{Z}^{(\lambda)}_{{\sf deb}} \bigg) - \mathbb{E} \phi \bigg( \frac{1}{\sqrt{dr}} \bm{Z}^{(\lambda/\zeta\s)}_{{\sf ST}}(\tau\s) \bigg) > \eps \bigg\} \\
    &\leq  \p\bigg\{  \phi \bigg( \frac{1}{\sqrt{dr}} \bm{Z}^{(\lambda)}_{{\sf deb}} \bigg) - \mathbb{E} \phi \bigg( \frac{1}{\sqrt{dr}} \bm{Z}^{(\lambda/\zeta\s)}_{{\sf ST}}(\tau\s) \bigg) > \eps \bigg\} \cap \mathcal{E}_{{\sf Good}} + \p \mathcal{E}_{{\sf Good}}^c \\
    &\leq \p\bigg\{ \tilde \phi \bigg( \frac{1}{\sqrt{dr}} \bm{Z}^{(\lambda)} \bigg) - \mathbb{E} \tilde \phi \bigg( \frac{1}{\sqrt{dr}} \bm{Z}^{(\lambda/\zeta\s)}_{{\sf ST}}(\tau\s) \bigg) > \eps \bigg\} + \p \mathcal{E}_{{\sf Good}}^c \\
    &\leq   O\bigg( \exp( - c d) +  \exp( - c dr) + \exp( - c n) + \frac{1}{\eps^2} \exp(-c dr \eps^4 ) - \exp( - \frac{(dr)^2}{n} \eps^4) \bigg),
\end{align}
which completes the proof.
\end{proof}

As in the previous steps observe that the rate of concentration does not have an additional error term depending on $\gamma_n$, which is due to the fact that $\bm{Z}_{{\sf HT}}^{(\lambda)} \neq \bm{Z}_{{\sf HT}}$ when $\gamma_n$ is bounded away from infinity, and, moreover, this replacement is non-negligible in such a regime.  However, the following lemma shows that $\bm{Z}_{{\sf HT}}^{(\lambda)}$ approximates $\bm{Z}_{{\sf HT}}$ when $\gamma_n \to \infty$, the latter of which does not depend on $\lambda$.
\begin{lemma}\label{lem:ZHTtau} Suppose that \cref{ass1,ass2,ass3} are satisfied. Then for any $1$-Lipschitz function $\phi: \mathbb{R}^{d\times d} \to \mathbb{R}$ it holds that
    \begin{align}
         \bigg| \mathbb{E} \phi \big( \bm{Z}_{{\sf HT}}^{(\lambda)}/\sqrt{dr} \big) - \mathbb{E} \phi \big( \bm{Z}_{{\sf HT}}^{}/\sqrt{dr} \big) \bigg| = O\bigg( \sigma\frac{dr}{n} \bigg).
    \end{align}
\end{lemma}
\begin{proof}
    See \cref{sec:ZHTtau}.  
\end{proof}

\section{Concentration of Unregularized Estimator about Debiased Estimator} \label{sec:debnearequivalence}
Our final argument is based on demonstrating $\bm{U}_{{\sf deb}}^{(\lambda)}$ approximates $\bm{U}^{(0)}$.  The following result establishes the near-equivalence of $\bm{U}_{{\sf deb}}^{(\lambda)}$ and $\bm{U}^{(0)}$.  We prove this result in the following subsection.
\begin{theorem}[Relating the Debiased and Unregularized Nonconvex Estimators] \label{thm:debunreg} Suppose that \cref{ass1,ass2,ass3} are satisfied.  Then the debiased estimator $\bm{U}^{(\lambda)}_{{\sf deb}}$ satisfies the bound
    \begin{align}
      \frac{1}{\sqrt{dr}}   \|\bm{U}^{(\lambda)}_{{\sf deb}} \bm{U}^{(\lambda)\top}_{{\sf deb}} - \bm{U}^{(0)} \bm{U}^{(0)\top}\|_F = O\bigg( \sigma r\sqrt{\frac{d}{n}} \bigg)
    \end{align}
    with probability at least $1 - O\bigg( \exp( - c dr) + \exp(- c n) + \exp( - c d) \bigg)$.  
\end{theorem}

This step is the only step essentially requiring that $\gamma_n/r \to \infty$, since observe that the rate of approximation is of order $(\gamma_n/r)^{-1/2}$ from \cref{thm:debunreg}.  In essence, this rate is due to the 
the rate of concentration of $\mathcal{X}\s \mathcal{X} - \mathcal{I}$ on rank $r$ matrices.  Therefore, in order to replace $\bm{U}_{{\sf deb}}^{(\lambda)} \bm{U}_{{\sf deb}}^{(\lambda)\top}$ with  $\bm{U}^{(0)}\bm{U}^{(0)\top}$, we incur an error of order $O((\gamma_n/r)^{-1/2})$, which is partially where this extra term in \cref{thm:mainthm} arises.

\subsection{Proof of \cref{thm:debunreg}}
In order to prove \cref{thm:debunreg} we will require a number of intermediate results.  The first result bounds the norms of the gradients of $f^{(0)}_{{\sf ncvx}}$ when evaluated at $\bm{U}_{{\sf deb}}^{(\lambda)}$. 
\begin{lemma}\label{lem:concentrationbound1} Under \cref{ass1,ass2,ass3}, with probability at least $1 - O\bigg(  \exp( -c  d r) -  \exp( - c n ) \bigg)$, it holds that
    \begin{align}
        \| \nabla f^{(0)}_{{\sf ncvx}} \big( \bm{U}^{(\lambda)}_{{\sf deb}} \big) \|_F \lesssim \frac{\lambda r \sqrt{d}}{\sqrt{n}} \| \bm{U}^{(\lambda)}_{{\sf deb}} \|_F.
    \end{align}
\end{lemma}

\begin{proof}
    See \cref{sec:concentrationbound1proof}.
\end{proof}

The next result characterizes the landscape of $f^{(0)}_{{\sf ncvx}}$ about the quantity $\bm{U}^{(0)}$. Without loss of generality we may take $ \mathcal{O}_{\bm{U}^{(0)},\bm{U}_{{\sf deb}}^{(\lambda)}} = \bm{I}$, since if $\bm{U}^{(0)}$ is a local minimum of $f^{(0)}_{{\sf ncvx}}$ then so is $\bm{U}^{(0)} \mathcal{O}$ for any orthogonal matrix $\mathcal{O}$.
\begin{lemma}\label{lem:convexity} Suppose without loss of generality that 
$ \mathcal{O}_{\bm{U}^{(0)},\bm{U}_{{\sf deb}}^{(\lambda)}} = \bm{I}$. Under \cref{ass1,ass2,ass3}, with probability at least $1 - O\big( \exp( - c dr ) + \exp( - c d) \big)$ it holds that
\begin{align}
    20 \lambda_{\max}(\bm{M}) \| \bm{U}^{(0)} - \bm{U}^{(\lambda)}_{{\sf deb}} \|_F^2  &\geq \nabla^2 f^{(0)}_{{\sf ncvx}}\big( t \bm{U}^{(0)} + (1 - t) \bm{U}^{(\lambda)}_{{\sf deb}} \big)[\bm{U}^{(0)} - \bm{U}_{{\sf deb}}^{(\lambda)} ,\bm{U}^{(0)} - \bm{U}_{{\sf deb}}^{(\lambda)}] \\
    &\geq \frac{\lambda_{\min}(\bm{M})}{20} \| \bm{U}^{(0)} - \bm{U}^{(\lambda)}_{{\sf deb}} \|_F^2.
\end{align}
Furthermore, it also holds that
\begin{align}
    \max\{ \| \bm{U}^{(0)}\|^2, \| \bm{U}_{{\sf deb}}^{(\lambda)} \|^2 \} \leq 8 \| \bm{M} \|
\end{align}
\end{lemma}

\begin{proof}
See \cref{sec:landscapeproof}.
\end{proof}
The proof of this result bears some similarity to existing works in the literature studying the optimization geometry of nonconvex problems \citep{ge_no_2017,luo_nonconvex_2022}. In fact, en route to the proof of this result we actually prove a slightly stronger statement; see \cref{lem:convexitydeterministic}.  Our proof is based on arguing that $\bm{U}(t) = t \bm{U}^{(0)} + (1 - t) \bm{U}^{(\lambda)}_{{\sf deb}}$ lies in a geodesic ball of $\bm{U}\s$ with radius at most $C \sigma \sqrt{dr}$ with high probability; strong convexity and smoothness follows.

\begin{proof}[Proof of \cref{thm:debunreg}]
First, since $\nabla f^{(0)}_{{\sf ncvx}}(\bm{U}^{(0)})$ vanishes, for any $\alpha > 0$ it holds that
\begin{align}
    \bm{U}^{(\lambda)}_{{\sf deb}} - \bm{U}^{(0)} &=    \bm{U}^{(\lambda)}_{\sf deb}- \alpha \nabla f^{(0)}_{{\sf ncvx}}(\bm{U}^{(\lambda)}_{{\sf deb}} )- \bm{U}^{(0)}+ \alpha \nabla f^{(0)}_{{\sf ncvx}}(\bm{U}^{(0)}) + \alpha \nabla f^{(0)}_{{\sf ncvx}}(\bm{U}^{(\lambda)}_{{\sf deb}}) \\
    &= \bigg( \bm{I}_{dr} - \alpha \int_{0}^{1} \nabla^2 f^{(0)}_{{\sf ncvx}}\big( t \bm{U}^{(0)} + (1 - t) \bm{U}^{(\lambda)}_{{\sf deb}} \big) dt \bigg) \bigg( \bm{U}_{{\sf deb}}^{(\lambda)} - \bm{U}^{(0)} \bigg) + \alpha \nabla f^{(0)}_{{\sf ncvx}}(\bm{U}^{(\lambda)}_{{\sf deb}}),
\end{align}
where the second line follows from the Fundamental Theorem of Calculus, where above we slightly abuse notation by letting $\nabla^2 f^{(0)}_{{\sf ncvx}}$ be of appropriate compatible dimension so that the equation above is meaningful (note that no generality is lost by doing so).
Taking norms, we have that
\begin{align}
    \|  \bm{U}^{(\lambda)}_{{\sf deb}} - \bm{U}^{(0)} \|_F &\leq \bigg\| \bigg( \bm{I}_{dr} - \alpha \int_{0}^{1} \nabla^2 f^{(0)}_{{\sf ncvx}}\big( t \bm{U}^{(0)} + (1 - t) \bm{U}^{(\lambda)}_{{\sf deb}} \big) dt \bigg) \bigg( \bm{U}_{{\sf deb}}^{(\lambda)} - \bm{U}^{(0)} \bigg)\bigg\|_F + \alpha \| \nabla f^{(0)}_{{\sf ncvx}} \big( \bm{U}_{{\sf deb}}^{(\lambda)} \big) \|_F.
\end{align}
Observe that
\begin{align}
 \bigg\| &\bigg( \bm{I}_{dr} - \alpha \int_{0}^{1} \nabla^2 f^{(0)}_{{\sf ncvx}}\big( t \bm{U}^{(0)} + (1 - t) \bm{U}^{(\lambda)}_{{\sf deb}} \big) dt \bigg) \bigg( \bm{U}_{{\sf deb}}^{(\lambda)} - \bm{U}^{(0)} \bigg)\bigg\|_F^2 \\
 &= {\sf vec} \bigg( \bm{U}_{{\sf deb}}^{(\lambda)} - \bm{U}^{(0)} \bigg)\t \bigg( \bm{I}_{dr} - \alpha \int_{0}^{1} \nabla^2 f^{(0)}_{{\sf ncvx}}\big( t \bm{U}^{(0)} + (1 - t) \bm{U}^{(\lambda)}_{{\sf deb}} \big) dt \bigg)^2 {\sf vec}\bigg( \bm{U}_{{\sf deb}}^{(\lambda)} - \bm{U}^{(0)} \bigg).
\end{align}
Denote 
\begin{align}
    \bm{\Gamma} := \int_{0}^{1} \nabla^2 f^{(0)}_{{\sf ncvx}}\big( t \bm{U}^{(0)} + (1 - t) \bm{U}^{(\lambda)}_{{\sf deb}} \big) dt .
\end{align}
Let $\alpha = \frac{c_0}{20\lambda_{\max}(\bm{M})}$.  Then by \cref{lem:convexity} it holds that on the event therein,
\begin{align}
    \alpha {\sf vec}\big( \bm{U}^{(0)} - \bm{U}_{{\sf deb}}^{(\lambda)}\big)\t \bm{\Gamma} \big( \bm{U}^{(0)} - \bm{U}_{{\sf deb}}^{(\lambda)}\big) &\leq 20 \alpha  \lambda_{\max} \| \bm{U}^{(0)} - \bm{U}_{{\sf deb}}^{(\lambda)} \|_F^2 \\
    &\leq c_0 \| \bm{U}^{(0)} - \bm{U}_{{\sf deb}}^{(\lambda)} \|_F^2,
\end{align}
and also
\begin{align}
       \alpha {\sf vec}\big( \bm{U}^{(0)} - \bm{U}_{{\sf deb}}^{(\lambda)}\big)\t \bm{\Gamma} \big( \bm{U}^{(0)} - \bm{U}_{{\sf deb}}^{(\lambda)}\big) &\geq \frac{c_0}{400 \kappa} \| \bm{U}^{(0)} - \bm{U}_{{\sf deb}}^{(\lambda)} \|_F^2.
\end{align}
Consequently,
\begin{align}
    {\sf vec}\big( \bm{U}^{(0)} - \bm{U}_{{\sf deb}}^{(\lambda)}\big)\t \bigg( \bm{I} - \alpha \bm{\Gamma} \bigg) {\sf vec}\big( \bm{U}^{(0)} - \bm{U}_{{\sf deb}}^{(\lambda)}\big) &\leq (1 - \frac{c_0}{400 \kappa} ) \| \bm{U}^{(0)} - \bm{U}_{{\sf deb}}^{(\lambda)} \|_F^2.
\end{align}
As a result,
\begin{align}
    {\sf vec} \bigg( \bm{U}_{{\sf deb}}^{(\lambda)} - \bm{U}^{(0)} \bigg)\t& \bigg( \bm{I}_{dr} - \alpha \bm{\Gamma} \bigg)^2 {\sf vec}\bigg( \bm{U}_{{\sf deb}}^{(\lambda)} - \bm{U}^{(0)} \bigg) \\
    &= {\sf vec} \bigg( \bm{U}_{{\sf deb}}^{(\lambda)} - \bm{U}^{(0)} \bigg)\t \bigg( \bm{I}_{dr} - \alpha \bm{\Gamma} \bigg) {\sf vec}\bigg( \bm{U}_{{\sf deb}}^{(\lambda)} - \bm{U}^{(0)} \bigg) \\
    &\quad - \alpha {\sf vec} \bigg( \bm{U}_{{\sf deb}}^{(\lambda)} - \bm{U}^{(0)} \bigg)\t \bigg( \bm{I}_{dr} - \alpha \bm{\Gamma} \bigg)\bm{\Gamma} {\sf vec}\bigg( \bm{U}_{{\sf deb}}^{(\lambda)} - \bm{U}^{(0)} \bigg) \\
    &= {\sf vec} \bigg( \bm{U}_{{\sf deb}}^{(\lambda)} - \bm{U}^{(0)} \bigg)\t \bigg( \bm{I}_{dr} - \alpha \bm{\Gamma} \bigg) {\sf vec}\bigg( \bm{U}_{{\sf deb}}^{(\lambda)} - \bm{U}^{(0)} \bigg) \\
    &\quad - \alpha {\sf vec} \bigg( \bm{U}_{{\sf deb}}^{(\lambda)} - \bm{U}^{(0)} \bigg)\t \bm{\Gamma} {\sf vec}\bigg( \bm{U}_{{\sf deb}}^{(\lambda)} - \bm{U}^{(0)} \bigg) \\
    &\quad + \alpha^2 {\sf vec} \bigg( \bm{U}_{{\sf deb}}^{(\lambda)} - \bm{U}^{(0)} \bigg)\t \bm{\Gamma}^2  {\sf vec}\bigg( \bm{U}_{{\sf deb}}^{(\lambda)} - \bm{U}^{(0)} \bigg) \\
    &\leq ( 1 - \frac{c_0}{400\kappa} ) \| \bm{U}^{(0)} - \bm{U}_{{\sf deb}}^{(\lambda)} \|_F^2 - \frac{c_0}{400\kappa}\|\bm{U}^{(0)} - \bm{U}_{{\sf deb}}^{(\lambda)} \|_F^2 + c_0^2 \|\bm{U}^{(0)} - \bm{U}_{{\sf deb}}^{(\lambda)} \|_F^2 \\
    &=  \| \bm{U}_{{\sf deb}}^{(\lambda)} - \bm{U}^{(0)} \|_F^2 \bigg( 1 - \frac{c_0}{800\kappa} + c_0^2 \bigg) \\
    &\leq (1 - \frac{c_0}{1600\kappa}) \| \bm{U}_{{\sf deb}}^{(\lambda)} - \bm{U}^{(0)} \|_F^2.
\end{align}
as long as $c_0 \leq \frac{1}{800 \kappa}$.  
Therefore, we have that
\begin{align}
    \| \bm{U}_{{\sf deb}} - \bm{U}^{(0)} \|_F \leq (1 - \frac{c_0}{1600\kappa})   \| \bm{U}_{{\sf deb}} - \bm{U}^{(0)} \|_F + \alpha \| \nabla f^{(0)}_{{\sf ncvx}} \bm{U}^{(\lambda)}_{{\sf deb}} \|_F.
\end{align}
After rearrangement we obtain that 
\begin{align}
    \| \bm{U}^{(\lambda)}_{{\sf deb}} - \bm{U}^{(0)} \|_F &\leq \frac{1600\kappa }{c_0} \alpha  \| \nabla f^{(0)}_{{\sf ncvx}}(\bm{U}^{(\lambda)}_{{\sf deb}}) \|_F = \frac{80}{\kappa \lambda_{\max}(\bm{M})} \| \nabla f^{(0)}_{{\sf ncvx}}(\bm{U}^{(\lambda)}_{{\sf deb}}) \|_F, \label{bebebe}
\end{align}
 By \cref{lem:convexity}, on the event therein,
\begin{align}
    \max\bigg\{\| \bm{U}^{(\lambda)}_{{\sf deb}} \|, \| \bm{U}^{(0)} \| \bigg\} &\leq C \sqrt{\lambda_{\max}(\bm{M})}. \label{blahblahboy2}
\end{align}
By \cref{lem:concentrationbound1} we have that with the probability therein
\begin{align}
      \| \nabla f^{(0)}_{{\sf ncvx}} (\bm{U}^{(\lambda)}_{{\sf deb}} ) \| _F &\lesssim  \frac{\lambda r \sqrt{d}}{\sqrt{n}} \| \bm{U}^{(\lambda)}_{{\sf deb}}  \|_F. \label{blahblahboy}
\end{align}
Therefore, plugging \eqref{blahblahboy} and \eqref{blahblahboy2} into \eqref{bebebe}, it holds that
\begin{align*}
     \| \bm{U}_{{\sf deb}} - \bm{U}^{(0)} \|_F &\lesssim \frac{1}{\kappa \lambda_{\max}(\bm{M})} \frac{\lambda r \sqrt{d}}{\sqrt{n}} \| \bm{U}_{{\sf deb}}^{(\lambda)} \|_F \lesssim \frac{\sigma r^{3/2} d}{ \sqrt{\lambda_{\max}(\bm{M})} \sqrt{n}}.
\end{align*}
Therefore, with the probability from \cref{lem:concentrationbound1,lem:convexity} it holds that 
\begin{align}
    \frac{1}{\sqrt{dr}} \| \bm{U}^{(0)} \bm{U}^{(0)\top} - \bm{U}^{(\lambda)}_{{\sf deb}} \bm{U}^{(\lambda)\top}_{{\sf deb}} \|_F &\leq \frac{1}{\sqrt{dr}} \bigg( \big\| \bm{U}^{(0)}  - \bm{U}^{(\lambda)}_{{\sf deb}} \big\|_F \big\| \bm{U}^{(0)}\big\| + \big\| \bm{U}^{(0)} - \bm{U}^{(\lambda)}_{{\sf deb}} \big\|_F \big\| \bm{U}^{(\lambda)}_{{\sf deb}} \big\| \bigg) \\
    &\lesssim \frac{1}{\sqrt{dr}} \frac{r^{3/2} d \sigma  \sqrt{\lambda_{\max}(\bm{M})}}{\sqrt{n} \sqrt{\lambda_{\max}(\bm{M})}} \\
    &\asymp \sigma \frac{r \sqrt{d}}{\sqrt{n}},
\end{align}
which proves the result after tabulating these probabilities.
\end{proof}

\section{Proof of \cref{thm:mainthm} } \label{sec:mainthmproof}
We now have all the pieces to prove \cref{thm:mainthm}.

\begin{proof}[Proof of \cref{thm:mainthm}]
Let $\phi(\cdot)$ be any 1-Lipschitz function. For the convex estimator, we note that the bound holds directly from \cref{thm:cvxasymptotics}, \cref{thm:relatecvxncvx}, and \cref{lem:ZSTtau} from the fact that
\begin{align*}
\bigg|    \mathbb{E} \phi\big( \bm{Z}_{{\sf ST}}^{(\lambda/\zeta\s)}(\tau\s) /\sqrt{dr}\big) - \mathbb{E} \phi \big( \bm{Z}_{{\sf ST}}^{(\lambda)}/\sqrt{dr} \big) \bigg| &= O \bigg( \sigma \frac{dr^{3/2}}{n} \bigg).
\end{align*}
Therefore, we focus on the proof for the nonconvex estimator. 

Define the events 
\begin{align}
   \mathcal{E}_{\mathrm{\cref{thm:zdeb}}} &:= \bigg\{ \bigg| \phi\bigg( \frac{1}{\sqrt{dr}} \bm{Z}^{(\lambda)}_{{\sf deb}} \bigg) - \mathbb{E} \phi \bigg( \frac{1}{\sqrt{dr}} \bm{Z}^{(\lambda)}_{{\sf HT}} \bigg) \bigg| \leq \eps \bigg\} \\
  \mathcal{E}_{\mathrm{\cref{thm:debunreg}}} &:= \bigg\{   \frac{1}{ \sqrt{dr}} \|  \bm{U}^{(\lambda)}_{{\sf deb}} \bm{U}^{(\lambda)\top }_{{\sf deb}} - \bm{U}^{(0)} \bm{U}^{(0)\top} \|_F = O\bigg( \frac{  \sigma r\sqrt{d}}{\sqrt{n}} \bigg) \bigg\}.
\end{align}
By \cref{lem:ZHTtau} it holds that
\begin{align*}
\bigg|    \mathbb{E} \phi\big( \bm{Z}_{{\sf HT}}^{(\lambda)}/\sqrt{dr} \big) - \mathbb{E} \phi\big( \bm{Z}_{{\sf HT}}/\sqrt{dr} \big) \bigg| \leq C \sigma \frac{dr}{n} \leq C \sigma \sqrt{\frac{dr}{n}}.
\end{align*}
By \cref{thm:udebzdeb} on the event $\mathcal{E}_{{\sf Good}}$ it holds that $\bm{U}^{(\lambda)}_{{\sf deb}} \bm{U}^{(\lambda)\top}_{{\sf deb}} = \bm{Z}_{{\sf deb}}^{(\lambda)}$.  As a result, on the event $\mathcal{E}_{\mathrm{\cref{thm:zdeb}}} \cap \mathcal{E}_{\mathrm{\cref{thm:debunreg}}}  \cap \mathcal{E}_{{\sf Good}}$ it holds that
\begin{align}
\bigg| \phi\bigg( \frac{1}{\sqrt{dr}} \bm{U}^{(0)} \bm{U}^{(0)\top} \bigg) - \mathbb{E} \phi \bigg( \frac{1}{\sqrt{dr}} \bm{Z}_{{\sf HT}} \bigg) \bigg| &\leq 
    \bigg| \phi \bigg( \frac{1}{\sqrt{dr}}\bm{U}^{(0)} \bm{U}^{(0)\top} \bigg) - \mathbb{E} \phi\bigg( \frac{1}{\sqrt{dr}} \bm{Z}^{(\lambda)}_{{\sf HT}} \bigg) \bigg| + C \sigma \sqrt{\frac{dr}{n}} \\
    &\leq \bigg| \phi \bigg( \frac{1}{\sqrt{dr}}\bm{U}^{(0)} \bm{U}^{(0)\top} \bigg) - \phi \bigg( \frac{1}{\sqrt{dr}} \bm{U}^{(\lambda)}_{{\sf deb}} \bm{U}^{(\lambda)\top}_{{\sf deb}} \bigg) \bigg|  + C \sigma \sqrt{\frac{dr}{n}} \\
    &\quad + \bigg| \phi \bigg( \frac{1}{\sqrt{dr}} \bm{Z}^{(\lambda)}_{{\sf deb}}  \bigg) -  \mathbb{E} \phi \bigg( \frac{1}{\sqrt{dr}} \bm{Z}^{(\lambda)}_{{\sf HT}} \bigg) \bigg| + C \sigma \sqrt{\frac{dr}{n}}  \\
    &\leq C \sigma r \sqrt{\frac{d}{n}} + \eps.
\end{align}
Therefore, by \cref{thm:debunreg,thm:zdeb,lem:Egood},
\begin{align*}
    \p\bigg\{\bigg| \phi\bigg( &\frac{1}{\sqrt{dr}} \bm{U}^{(0)} \bm{U}^{(0)\top} \bigg) - \mathbb{E}_{\bm{H}} \phi\big( \frac{1}{\sqrt{dr}} \bm{Z}_{{\sf HT}} \big) \bigg| > \eps + C \sigma r \sqrt{\frac{d}{n}} \bigg\}\\
    &\leq \p \mathcal{E}_{\mathrm{\cref{thm:zdeb}}}^c  + \p  \mathcal{E}_{\mathrm{\cref{thm:debunreg}}}^c + \p\mathcal{E}_{{\sf Good}}^c \\
    &\leq O\bigg( \exp( - c d) +  \exp( - c dr) + \exp( - c n) + \frac{1}{\eps^2} \exp(-c dr \eps^4 ) + \exp( - \frac{(dr)^2}{n} \eps^4) \bigg),
\end{align*}
which completes the proof. 
\end{proof}

\section*{Acknowledgements}
Ren\'e Vidal acknowledges the support of the Office of Naval Research under grant MURI 503405-78051, the National Science Foundation under grant 2031985, and the Simons Foundation under grant 814201. Joshua Agterberg thanks Tianjiao Ding, Vijay Giri, Benjamin Haeffele, Lachlan MacDonald, Hancheng Min, Nghia Nguyen, Uday Kiran Tadipatri, Salma Tarmoun,  and Ziqing Xu for productive discussions.

\appendix

\section{Proofs of Preliminary Results}
In this section we prove our key preliminary results that we rely on throughout our analysis.

\subsection{Proof of Lemma \ref{lem:Egood}}
\label{sec:proofofegood}

\begin{proof}
We will bound each quantity separately and take a union bound.  First, we note that 
 $\frac{1}{n}\sum_{i} \eps_i \bm{X}_i$ is equal in distribution to a mean-zero random matrix with off-diagonal variance $\frac{\sigma^2}{n}$ and diagonal variance $\frac{2\sigma^2}{n}$.  By Corollary 3.9 of \citet{bandeira_sharp_2016}, we have that
 \begin{align}
     \bigg\|  \frac{1}{n} \sum_i \eps_i \bm{X}_i \bigg\| \leq 8 \sigma \sqrt{\frac{d}{n}} 
 \end{align}
 with probability at least $1 - \exp\big( - c d \big)$.  

 For the next event, this result follows from Lemma 1 of \citet{candes_tight_2011}, where the modification to ${\sf GOE}(d)$ is straightforward.

 The final term is handled via $\eps$-net. First, let $\bm{A}$ be any rank at most $2r$ matrix.  Observe that $\langle \bm{X}_i, \bm{A} \rangle$ is distributed as a mean-zero Gaussian random variable with variance $\|\bm{A}\|_F^2$.  In addition, the vector $\bm{\eps}$ is a Gaussian vector with covariance $\sigma^2 \bm{I}_n$.  Therefore, letting $A \sim B$ note equality in distribution, it holds that
 \begin{align}
    \frac{1}{n} \sum_{i} \eps_i \langle \bm{X}_i, \bm{A} \rangle \sim   \frac{\sigma \|\bm{A}\|_F}{n}\langle \bm{z}_1, \bm{z}_2 \rangle 
 \end{align}
 where $\bm{z}_1$ and $\bm{z}_2$ are standard $\mathcal{N}(0,\bm{I}_n)$ random variables.  By rotational invariance, note that since $\langle \bm{z}_1, \bm{z}_2 \rangle = \| \bm{z}_1 \|\frac{\langle \bm{z}_1, \bm{z}_2 \rangle}{\| \bm{z}_1 \|}$, it holds that $\langle \bm{z}_1, \bm{z}_2 \rangle$ equals in distribution to the random variable $ \| \bm{z}_1 \| G$, where $G$ is a standard Gaussian random variable. By Lipschitz concentration $\big| \|\bm{z}_1 \| - \sqrt{n} \big| \leq t_1 \sqrt{n} $ with probability at least $1 -2 \exp( - c  n t_1^2)$.  Furthermore, $|G| \leq t_2$ with probability at least $1 - \exp(- t_2^2)$.  Therefore, with probability at least $1 - 2 \exp( - c n t_1^2) - 2 \exp( - t_2^2)$, it holds that
 \begin{align}
      \frac{\sigma \| \bm{A} \|_F}{n} \big|\langle \bm{z}_1 , \bm{z}_2 \rangle \big| \leq   \frac{\sigma \| \bm{A} \|_F }{n} \big( \sqrt{n} + t_1 \sqrt{n} \big) t_2 = \frac{\sigma \| \bm{A} \|_F}{\sqrt{n}} \big( 1 + t_1\big) t_2.
 \end{align}
 Taking $t_1= c$ and $t_2 = C_{t_2} \sqrt{d r}$ shows that with probability at least $1 - 2 \exp( - c n) - 2 \exp( - C_{t_2}^2 dr )$,
 \begin{align}
     \bigg|  \frac{1}{n} \sum_{i} \eps_i \langle \bm{X}_i, \bm{A} \rangle  \bigg| \leq C_{t_2} \sigma \| \bm{A} \|_F  \sqrt{\frac{dr}{n}}.
 \end{align}
 Now let $\mathcal{N}_{\eps}$ denote an $\eps$-net for the set of rank at most $2r$ matrices with Frobenius norm at most $1$. By Lemma 3.1 of \citet{candes_tight_2011}, $| \mathcal{N}_{\eps} | \leq \bigg( \frac{9}{\eps} \bigg)^{3 d r}$.   For any fixed $\bm{A} \in \mathcal{N}_{\eps}$ it holds that
 \begin{align}
     \bigg| \frac{1}{n} \sum_{i} \eps_i \langle \bm{X}_i, \bm{A} \rangle \bigg| \leq C_{t_2} \sigma \sqrt{\frac{dr}{n}}.  
 \end{align}
Consequently, the union bound implies that
\begin{align}
    \sup_{\bm{A} \in \mathcal{N}_{\eps}}   \bigg| \frac{1}{n} \sum_{i} \eps_i \langle \bm{X}_i, \bm{A} \rangle \bigg| \leq C_{t_2} \sigma \sqrt{\frac{dr}{n}} \numberthis \label{1025}
\end{align}
with probability at least
\begin{align}
    1 - \exp \bigg( 3  d r \log( 9/\eps) - c n\bigg) - \exp\bigg( 4 dr \log(9/\eps) - C_{t_2}^2 dr \bigg). 
\end{align}
Now denote
 \begin{align}
     M := \sup_{\bm{A}: \| \bm{A} \|_F \leq 1 }  \bigg| \frac{1}{n} \sum_{i} \eps_i \langle \bm{X}_i, \bm{A} \rangle \bigg|,
 \end{align}
 and let $\bm{A}_0$ denote the maximizer above.  Then there is some $\bm{A}_{\eps} \in \mathcal{N}_{\eps}$ such that $\|\bm{A}_0 - \bm{A}_{\eps} \|_F \leq \eps$, and, furthermore, we can write $\bm{A}_0 - \bm{A}_{\eps} = \bm{A}_1 + \bm{A}_2$, where both $\bm{A}_1$ and $\bm{A}_2$ are rank at most $2r$ and $\|\bm{A}_1 \|_F, \|\bm{A}_2\|_F \leq \eps$.  Then on the event \eqref{1025}, it holds that
 \begin{align}
       M &= \bigg| \frac{1}{n} \sum_{i} \eps_i \langle \bm{X}_i, \bm{A}_0 \rangle \bigg| \\
       &\leq \bigg| \frac{1}{n} \sum_{i} \eps_i \langle \bm{X}_i, \bm{A}_0 - \bm{A}_\eps \rangle \bigg| + \bigg| \frac{1}{n} \sum_{i} \eps_i \langle \bm{X}_i,  \bm{A}_\eps \rangle \bigg| \\
       &\leq \bigg| \frac{1}{n} \sum_{i} \eps_i \langle \bm{X}_i, \bm{A}_1 + \bm{A}_2  \rangle \bigg| +  C_{t_2} \sigma \sqrt{\frac{dr}{n}} \\
       &\leq \bigg| \frac{1}{n} \sum_{i} \eps_i \langle \bm{X}_i, \bm{A}_1 \rangle \bigg| + \bigg| \frac{1}{n} \sum_{i} \eps_i \langle \bm{X}_i,\bm{A}_2  \rangle \bigg| +  C_{t_2} \sigma \sqrt{\frac{dr}{n}} \\
       &\leq 2 \eps M + C_{t_2} \sigma \sqrt{\frac{dr}{n}}.
 \end{align}
 Rearranging shows that on this same event
 \begin{align}
     M \leq \frac{1}{2\eps }C_{t_2} \sigma \sqrt{\frac{dr}{n}}.
 \end{align}
 Consequently, by taking $\eps = 1/4$, as long as $C_{t_2}$ is larger than some universal constant, this probability is at least $1 - C \exp( - C dr )$ since $n \geq C_0 dr$.   
\end{proof}

\subsection{Proof of \cref{prop:fxdpt}} \label{sec:fixedpointproof}
\begin{proof}[Proof of \cref{prop:fxdpt}]
To prove existence and uniqueness of $\tau\s,\zeta\s$, we follow the argument of Lemma A.2 of \citet{celentano_lasso_2023}.  Define functions $\mathcal{T}, \mathcal{Z} : L^2( \mathbb{R}^{d\times d} ; \mathbb{R}^{d \times d} ) \to \mathbb{R}$ via
\begin{align}
    \mathcal{T}(\bm{Z}) := \sigma^2 + \frac{\| \bm{Z} - \bm{M} \|_{L^2,F}^2}{n}, \\
    \mathcal{Z}(\bm{Z}) :=  \bigg( 1 - \frac{\langle \bm{H}, \bm{Z} \rangle_{L^2,F} }{n \mathcal{T}(\bm{Z})} \bigg)_+.
\end{align}
Define the function $\mathcal{E} : L^2(  \mathbb{R}^{d\times d} \times \mathbb{R}^{d \times d} ) \to \mathbb{R}$ via 
\begin{align}
\mathcal{E}(\bm{Z}) := \frac{1}{2} \bigg( \sqrt{\sigma^2 + \frac{\|\bm{Z - M}\|_{L^2,F}^2}{n}} - \frac{\langle \bm{H}, \bm{Z - M} \rangle_{L^2,F}}{n} \bigg)^2_+ + \frac{\lambda}{n} \mathbb{E} \big\{ \| \bm{Z}  \|_* - \| \bm{M} \|_* \big\},
\end{align}
where we note that $\bm{Z}$ is viewed as a function  $\mathbb{R}^{d\times d} \to \mathbb{R}^{d\times d}_{\succcurlyeq 0}$ in $L^{2,F}(\mathbb{R}^{d\times d})$.  
\\ \ \\
\noindent
\textbf{Step 1: Showing $\mathcal{E}(\bm{Z})$ has minimizers.} 
 We will show that $\mathcal{E}(\bm{Z}) \to \infty$ as $\| \bm{Z} \|_{L^2,F} \to \infty$, whence $\mathcal{E}(\bm{Z})$ has a minimizer.  
To wit, we note that
\begin{align}
\frac{\lambda}{n} \mathbb{E} \| \bm{Z} \|_* - \| \bm{M} \|_* &\geq \frac{\lambda}{n} \mathbb{E} \| \bm{Z} - \bm{M} \|_F - 2\| \bm{M} \|_*,
\end{align}
which holds from the standard inequality $\| \bm{Z} \|_* \geq \| \bm{Z} \|_F$.  In addition, for any $M > 0$, it holds that
\begin{align}
     | \langle \bm{Z} - \bm{M}, \bm{H} \rangle | &= \big| \mathbb{E} \langle \bm{Z}(\bm{H}) - \bm{M}, \bm{H} \mathbb{I}_{\| \bm{H} \|_F > M } \rangle_{L^2,F} + \mathbb{E} \langle \bm{Z}(\bm{H}) - \bm{M}, \bm{H} \rangle_{L^2_F} \mathbb{I}_{\| \bm{H} \|_F \leq M} \big| \\
     &\leq \| \bm{Z}- \bm{M} \|_{L^2,F} (\mathbb{E} \| \bm{H} \|_F^2 \mathbb{I}_{\| \bm{H} \|_F > M} )^{1/2} + M \mathbb{E} \| \bm{Z}(\bm{H}) - \bm{M} \|_F.
\end{align}
Let $M$ be sufficiently large such that $(\mathbb{E} \| \bm{H} \|_F^2 \mathbb{I}_{\| \bm{H} \|_F > M} )^{1/2} < \sqrt{n}/2.$  Then
\begin{align}
    \mathcal{E}(\bm{Z}) &\geq \frac{1}{2} \bigg( \frac{\| \bm{Z} - \bm{M} \|_{L^2,F}}{2\sqrt{n}} - \frac{M}{n} \mathbb{E} \| \bm{Z}(\bm{H}) - \bm{M}\|_F \bigg)^2_+ +\frac{\lambda}{n} \mathbb{E} \| \bm{Z} - \bm{M} \|_F - 2\frac{\lambda}{n}\| \bm{M} \|_*.
\end{align}
We now analyze three cases:
\begin{itemize}
    \item \textbf{Case 1: $\|\bm{Z} - \bm{M} \|_{L^2,F}/(2\sqrt{n}) \leq M \mathbb{E} \| \bm{Z}(\bm{H}) - \bm{M} \|_F/n$}.  In this case, the term inside the parentheses is zero, and it holds that
    \begin{align}
        \frac{\lambda}{n} \mathbb{E} \| \bm{Z} - \bm{M} \|_F \geq \frac{\lambda}{2 M \sqrt{n}} \| \bm{Z} - \bm{M} \|_{L^2,F}.
    \end{align}
    \item \textbf{Case 2: $\| \bm{Z} - \bm{M} \|_{L^2,F}/(4/\sqrt{n}) \leq M\mathbb{E} \| \bm{Z}(\bm{H}) - \bm{M} \|_F/n <  \|\bm{Z} - \bm{M} \|_{L^2,F}/(2\sqrt{n})$}. In this case, we have that
    \begin{align}
        \frac{1}{2} \bigg( \frac{\|\bm{Z} - \bm{M} \|_{L^2,F}}{2\sqrt{n}} - \frac{M}{n} \mathbb{E} \| \bm{Z} - \bm{M} \|_F \bigg)^2_+ &\geq \frac{1}{2} \bigg( \frac{\|\bm{Z} - \bm{M}\|_{L^2,F}}{4 \sqrt{n}} \bigg)^2 \geq \frac{\|\bm{Z} - \bm{M}\|_{L^2,F^2}^2}{32 n}
    \end{align}
    \item \textbf{Case 3: $\| \bm{Z} - \bm{M} \|_{L^2,F}/(4/\sqrt{n}) >  M\mathbb{E} \| \bm{Z}(\bm{H}) - \bm{M} \|_F/n$ }.  In this case we again have the bound 
      \begin{align}
        \frac{1}{2} \bigg( \frac{\|\bm{Z} - \bm{M} \|_{L^2,F}}{2\sqrt{n}} - \frac{M}{n} \mathbb{E} \| \bm{Z} - \bm{M} \|_F \bigg)^2_+ &\geq \frac{1}{2} \bigg( \frac{\|\bm{Z} - \bm{M}\|_{L^2,F}}{4 \sqrt{n}} \bigg)^2 \geq \frac{\|\bm{Z} - \bm{M}\|_{L^2,F^2}^2}{32 n}.
    \end{align}
\end{itemize}
Combining these cases, we see that we have that
\begin{align}
    \mathcal{E}(\bm{Z}) \geq \min\bigg\{ \frac{\|\bm{Z} - \bm{M}\|_{L^2,F}^2}{32n}, \frac{\lambda}{2m \sqrt{n}} \| \bm{Z} - \bm{M} \|_{L^2,F} - 2 \frac{\lambda}{n} \| \bm{M}\|_* \bigg\}.
\end{align}
In any case, when $\|\bm{Z}\|_{L^2,F} \to \infty$, since $\bm{M}$ is fixed, it holds that $\mathcal{E}(\bm{Z}) \to \infty$, whence $\mathcal{E}(\cdot)$ has minimizers.  
\\ \ \\
\noindent
\textbf{Step 2: Establishing a correspondence between soft-thresholding and minimizers}.  Suppose $\bm{Z}\s$ is any minimizer of $\mathcal{E}(\bm{Z})$.   Define the following function
 \begin{align}
     \mathcal{\tilde E}(\bm{Z},\zeta,\tau) := \frac{\zeta}{2 n} \| \bm{Z} - \bm{M} - \tau \bm{H} \|_{L^2,F}^2 + \frac{\lambda}{n} \mathbb{E} \big\{ \| \bm{Z} \|_* - \| \bm{M} \|_* \big\}.
 \end{align}
Given $\zeta$ and $\tau$, it is clear that $\bm{Z}\s$ minimizes $\mathcal{\tilde E}(\bm{Z},\zeta,\tau)$ if and only if $\bm{Z}$ satisfies
\begin{align}
    \bm{Z}(\bm{H}) &= {\sf S.T.} \big( \bm{M} + \tau \bm{H} ; \frac{\lambda}{\zeta} \big)
\end{align}
almost surely. Denote
\begin{align}
    \mathcal{F}(\bm{Z}) := \frac{1}{2} \bigg( \sqrt{\sigma^2  + \frac{\|\bm{Z - M}\|_F^2}{n}} - \frac{\langle \bm{H}, \bm{Z - M} \rangle_{L_2}}{n} \bigg)^2_+.
\end{align}
For fixed $\bm{Z}_1, \bm{Z}_2$, we have that
\begin{align}
    \frac{d}{d\eps} \mathcal{F}\big(\bm{Z}_1 + \eps \bm{Z}_2 \big)  &=  \bigg( \sqrt{\sigma^2 + \frac{\| \bm{Z}_1 -\eps \bm{Z}_2 - \bm{M} \|_{L^2}^2}{n}} - \frac{\langle \bm{H}, \bm{Z}_1 + \eps \bm{Z}_2 - \bm{M} \rangle }{n}\bigg)_{+} \\
    &\times \bigg( \frac{1}{2} \big( \sigma^2 + \frac{\|\bm{Z}_1 + \eps \bm{Z}_2 - \bm{M} \|_F^2}{n} \big)^{-1/2} \bigg( \frac{2\langle \bm{Z}_1 - \bm{M}, \bm{Z}_2 \rangle + \eps^2 \| \bm{Z}_2 \|^2 }{n} \bigg) -  \frac{\langle \bm{H}, \bm{Z}_2 \rangle }{n} \bigg).
\end{align}
Evaluating at $\eps= 0$ yields
\begin{align}
    \frac{d}{d\eps} \mathcal{F}(\bm{Z}_1 + \eps \bm{Z}_2 ) \bigg|_{\eps = 0} &=  \bigg[ \sqrt{\sigma^2 + \frac{\|\bm{Z}_1 - \bm{M} \|_{L^2}^2}{n}} - \frac{\langle \bm{H}, \bm{Z}_1 - \bm{M} \rangle }{n} \bigg]_+ \bigg[\sigma^2 + \frac{\|\bm{Z}_1 - \bm{M} \|_{L^2}^2}{n} \bigg]^{-1/2} \frac{\langle \bm{Z}_1 - \bm{M}, \bm{Z}_2 \rangle}{n} \\
    &\quad -  \bigg[\sqrt{\sigma^2 + \frac{\|\bm{Z}_1 - \bm{M} \|_{L^2}^2}{n}} - \frac{\langle \bm{H}, \bm{Z}_1 - \bm{M} \rangle }{n} \bigg]_+ \frac{\langle \bm{H}, \bm{Z}_2 \rangle}{n} \\
    &=  \mathcal{Z}(\bm{Z}_1) \frac{\langle \bm{Z}_1 - \bm{M}, \bm{Z}_2 \rangle}{n} -  \mathcal{T}(\bm{Z}_1) \mathcal{Z}(\bm{Z}_1) \frac{\langle \bm{H}, \bm{Z}_2 \rangle}{n} \\
    &=  \mathcal{Z}(\bm{Z}_1)  \frac{\langle \bm{Z}_1 - \bm{M} - \mathcal{T}(\bm{Z}_1) \bm {H} , \bm{Z}_2\rangle}{n}.
\end{align}
Therefore, it holds that
\begin{align}
    \mathcal{F}(\bm{Z}_1 + \eps \bm{Z}_2) &= \mathcal{F}(\bm{Z}_1) + \eps  \mathcal{Z}(\bm{Z}_1) \frac{\langle \bm{Z}_1 - \bm{M} - \mathcal{T}(\bm{Z}_1) \bm {H} , \bm{Z}_2\rangle}{n} + O(\eps^2).
\end{align}
In addition, observe that
\begin{align}
    \mathcal{\tilde E}(\bm{Z}_1 + \eps \bm{Z}_2, \zeta,\tau) &= \frac{\zeta}{2n} \| \bm{Z}_1 + \eps \bm{Z}_2 - \bm{M} - \tau \bm{H} \|_{L^2}^2 + \frac{\lambda}{n} \mathbb{E} \big\{ \|\bm{Z}_1 + \eps \bm{Z}_2 \|_* - \bm{M}\|_* \big\} \\
    &= \frac{\zeta}{2n} \bigg( \| \bm{Z}_1 - \bm{M} - \tau \bm{H} \|^2 + 2 \eps \langle \bm{Z}_1 - \bm{M} - \tau \bm{H}, \bm{Z}_2 \rangle + \eps^2 \| \bm{Z}_2\|^2 \bigg) + \frac{\lambda}{n} \mathbb{E} \big\{ \| \bm{Z}_1 + \eps \bm{Z}_2 \|_* - \| \bm{M} \|_* \big\} \\
    &= \frac{\zeta}{2n} \| \bm{Z}_1 - \bm{M} - \tau \bm{H} \|^2 +  \eps  \zeta  \frac{\langle \bm{Z}_1 - \bm{M} -\tau  \bm {H} , \bm{Z}_2\rangle}{n} + O(\eps^2) + \frac{\lambda}{n} \mathbb{E} \big\{ \| \bm{Z}_1 + \eps \bm{Z}_2 \|_* - \| \bm{M} \|_* \big\}.
\end{align}
In particular, for $\zeta = \mathcal{Z}(\bm{Z}_1)$ and $\tau = \mathcal{T}(\bm{Z}_1)$ it holds that
\begin{align}
    \mathcal{\tilde E}(\bm{Z}_1 + \eps \bm{Z}_2, \mathcal{Z}(\bm{Z}_1), \mathcal{T}(\bm{Z}_1)) - \mathcal{E} ( \bm{Z}_1 + \eps \bm{Z}_2 ) &=  \mathcal{\tilde E}(\bm{Z}_1 , \mathcal{Z}(\bm{Z}_1), \mathcal{T}(\bm{Z}_1)) - \mathcal{E} ( \bm{Z}_1 ) + O(\eps^2),
\end{align}
Consequently,  $\bm{Z}_1$ is a minimizer of $\mathcal{E}(\bm{Z}_1)$ if and only if $\bm{Z}_1$ is a minimizer of $\bm{Z}_1 \mapsto \mathcal{\tilde E}(\bm{Z}_1, \mathcal{Z}(\bm{Z}_1), \mathcal{T}(\bm{Z}_1))$, which are simply given via
\begin{align}
    \bm{Z}\s &= {\sf S.T.} \bigg( \bm{M} + \mathcal{T}(\bm{Z}\s) \bm{H} ; \frac{\lambda}{\mathcal{Z}(\bm{Z}\s)} \bigg).
\end{align}
Since minimizers of $\mathcal{E}(\bm{Z})$ exist, it must hold that $\bm{Z}\s$ exists, and hence solutions to the fixed-point equations exist.
\\ \ \\
\noindent \textbf{Step 3: Establishing Uniqueness}.   Now suppose that $\mathcal{Z}(\bm{Z}\s) = 0$ for some minimizer $\bm{Z}\s$.  Then $\bm{Z}\s = 0$ almost surely, and hence 
\begin{align}
    \mathcal{Z}(\bm{Z}\s) =  \big( 1- \frac{\langle \bm{H}, 0 \rangle_{L_2}}{n\mathcal{T}(\bm{Z}\s)} \big)_+ = 1
\end{align}
a contradiction.  Therefore $\mathcal{Z}(\bm{Z}\s) > 0$ for all minimizers of $\mathcal{E}$.  Finally, the function $\mathcal{E}$ is strictly convex on $\mathcal{Z}(\bm{Z}) > 0$.  Because all minimizers satisfy $\mathcal{Z}(\bm{Z}) > 0$, restricting $\mathcal{E}$ to the set $\mathcal{Z}(\bm{Z}) > 0$ renders $\mathcal{E}$ strictly convex and does not eliminate any  minimizers, whence the minimizer must be unique by strict convexity.
\end{proof}

\subsection{Proof of \cref{lem:fxdpointprops} and Additional Properties of Soft-Thresholding Estimator} \label{sec:fixedpointprops}
In this section we provide some additional characterizations of $\tau\s$ and $\zeta\s$ and the soft thresholding estimator $\bm{Z}_{{\sf st}}^{(\lambda/\zeta\s)}(\tau\s)$ that are useful in subsequent section.

We will also need the following lemma, which concerns the estimator $\bm{ Z}_{{\sf ST}}^{(\lambda/\zeta\s)}(\tau\s)$. 
\begin{lemma} \label{lem:Zstgoodproperties}
Suppose \cref{ass1,ass2,ass3} are satisfied, and let $s > 0$.  Then it holds that
    \begin{align}
\p\bigg\{      \bigg|  \| \bm{ Z}_{{\sf ST}}^{(\lambda/\zeta\s)}(\tau\s) - \bm{M} \|_F^2 - \mathbb{E} \|\bm{ Z}_{{\sf ST}}^{(\lambda/\zeta\s)}(\tau\s) - \bm{M} \|_F^2 \bigg|  > s d r \bigg\} \leq C \exp(- c dr (s \vee s^2) ).
    \end{align}
    Moreover, $\p\bigg\{ \mathrm{rank}(\bm{ Z}_{{\sf ST}}^{(\lambda/\zeta\s)}(\tau\s)) > r \bigg\} \leq C \exp(- c d).$  In addition, $\p\bigg\{\| \bm{Z}_{{\sf ST}}^{(\lambda/\zeta\s)}(\tau\s) \| > 2 \| \bm{M} \|\bigg\} \leq C\exp(- c d)$.
\end{lemma}
\begin{proof}
    See \cref{sec:Zstgoodproperties_proof}.
\end{proof}

\subsubsection{Proof of \cref{lem:fxdpointprops}}

\begin{proof}[Proof of \cref{lem:fxdpointprops}]
The proof of \cref{prop:fxdpt} shows that
\begin{align}
    \zeta\s &= \bigg( 1 - \frac{\langle \bm{H}, \bm{Z}\s \rangle_{L_2} }{n \mathcal{T}(\bm{Z}\s)} \bigg)_+ > 0
\end{align}
for any minimizer $\bm{Z}\s$ of $\mathcal{E}$.  This also implies that 
\begin{align}
\frac{{\sf df}_{\lambda}(\tau^{*2},\zeta\s)}{n} &=
\frac{1}{n\tau\s} \mathbb{E} \bigg \langle \bm{H}, \bm{Z}^{(\lambda/\zeta\s)}(\tau\s) \bigg \rangle  =  \frac{\langle \bm{H}, \bm{Z}\s\rangle_{L_2}}{n \mathcal{T}(\bm{Z}\s)} < 1,
\end{align}
which shows that $\zeta\s \leq 1$.  

We will now demonstrate that $\tau^2$ is bounded.  First, observe that $\tau\s$ satisfies 
\begin{align}
    (\tau\s)^2 = \sigma^2 + \frac{\mathbb{E}\| \bm{Z}^{(\lambda/\zeta\s)}(\tau\s) - \bm{M} \|_F^2}{n}.  
\end{align}
Since $1 \geq \zeta\s > 0$, $\lambda/\zeta\s \geq \lambda$.  Therefore, for any fixed $\tau\s$, with probability one it holds that
\begin{align}
    \mathrm{rank}(\bm{Z}^{(\lambda/\zeta\s)}_{{\sf ST}}(\tau\s)) \leq \mathrm{rank}(\bm{Z}^{(\lambda)}_{{\sf ST}}(\tau\s).
\end{align}
Moreover, by \citet{bandeira_sharp_2016}, it holds that
\begin{align}
    \| \bm{H} \| \leq 3 \big( \sqrt{d} +  t \big)
\end{align}
with probability at least $1 - \exp( - c t^2 )$.  Consequently, by Weyl's inequality, on this event
\begin{align}
    |\lambda_{r+1}( \bm{M} + \tau\s \bm{H} ) | \leq \tau\s \| \bm{H} \| \leq 3 \tau\s (\sqrt{d} + t ).
\end{align}
Therefore, by \cref{ass3}
\begin{align}
    \mathbb{E} \| \bm{Z}^{(\lambda/\zeta\s)}(\tau\s) - \bm{M} \|_F^2 &\leq 2 \| \bm{M} \|_F^2 + 2\mathbb{E} \mathrm{rank}(\bm{Z}^{(\lambda)}(\tau\s))\| \bm{Z}^{(\lambda)}(\tau\s)\|^2 \\
    &\leq 4 r \| \bm{M} \|^2 + 4(\tau\s)^2 \mathbb{E} \mathrm{rank}(\bm{Z}^{(\lambda)}(\tau\s) \| \bm{H} \|^2 \\
    &\leq 4 r \| \bm{M} \|^2 + 4 (\tau\s)^2 \mathbb{E} \bigg\{ \bigg[ \sum_{i=1}^{d} \mathbb{I}\{ \lambda_i( \bm{M} + \tau\s \bm{H}) > \lambda \} \bigg]  \| \bm{H} \|^2 \bigg\}\\
    &\leq 4 r \| \bm{M} \|^2 + 4 (\tau\s)^2 \mathbb{E} \bigg\{ \bigg[ \sum_{i=1}^{d} \mathbb{I}\{ \lambda_i( \bm{M} + \tau\s \bm{H}) > \lambda \} \bigg]  \| \bm{H} \|^2  \mathbb{I}\{ \| \bm{H} \| \leq C \sqrt{d} \bigg\} \\
    &\quad ++ 4 (\tau\s)^2 \mathbb{E} \bigg\{ \bigg[ \sum_{i=1}^{d} \mathbb{I}\{ \lambda_i( \bm{M} + \tau\s \bm{H}) > \lambda \} \bigg]  \| \bm{H} \|^2  \mathbb{I}\{ \| \bm{H} \| > C \sqrt{d} \bigg\} \\
    &\leq 4 r \| \bm{M} \|^2 + 4 (\tau\s)^2 \mathbb{E} \bigg\{ \bigg[ \sum_{i=1}^{d} \mathbb{I}\{ \lambda_i( \bm{M} + \tau\s \bm{H}) > \lambda \} \bigg]  \| \bm{H} \|^2  \mathbb{I}\{ \| \bm{H} \| \leq C \sqrt{d} \bigg\} \\
    &\quad + 8 (\tau\s)^2  d \int_{C^2 d}^{\infty}  e^{- c s^2} ds \\
    &\leq 4 r \| \bm{M} \|^2 + 36 (\tau\s)^2 d r 
    + 4 (\tau\s)^2  d \int_{C \sqrt{d}}^{\infty} s^2 e^{- c s^2} ds \\
    &\leq 4 r \| \bm{M} \|^2 + 36 (\tau\s)^2 d r 
    + 4 (\tau\s)^2  C' d^2 \exp( - c d^2) \\
    &\leq  4 r \| \bm{M} \|^2 + 36 (\tau\s)^2 d r \\
    &\leq 4 r \kappa^2 \lambda_{\min}(\bm{M})^2 + 36 (\tau\s)^2 d r,
\end{align}
provided $d$ is sufficiently large, where we have used the identity $\mathbb{E} X = \int_{0}^{\infty} \p\big( X > s \big) ds$
for positive random variable $X$.  
Therefore, we have that $\tau\s$ satisfies
\begin{align}
    (\tau\s)^2 &\leq \sigma^2 + \frac{4 r \kappa^2 \lambda_{\min}(\bm{M})^2}{n} + \frac{4 \cdot 36 (\tau\s)^2 dr}{n},
\end{align}
which, after rearranging   yields
\begin{align}
    \bigg( 1 - \frac{4 \cdot 36 dr}{n} \bigg) (\tau\s)^2 \leq \sigma^2 + \frac{4 \kappa^2 d r^2}{n} \frac{\lambda_{\min}(\bm{M})^2}{d r}  \leq \sigma^2 + \frac{4 \kappa^2 d r^2}{n} \frac{\lambda_{\min}^2}{r}.
\end{align}
Recalling that $n \geq C_2 d r^2$ by \cref{ass2} shows that
\begin{align}
    (\tau\s)^2 \leq 2\bigg( \sigma^2 +\frac{4 \kappa^2 d r^2}{n} \frac{\lambda_{\min}^2}{r} \bigg) \leq  6 \kappa^2 \sigma^2,
\end{align}
where we have used that $\frac{\lambda_{\min}^2}{\sigma^2 r} \leq C_2$ by \cref{ass1} and $n \geq C_3 d r^2$ for some large constant $C_3$ by \cref{ass2}.  
This same argument shows that
\begin{align}
    \bigg| (\tau\s)^2 - \sigma^2 \bigg| &\leq \frac{\mathbb{E} \| \bm{Z}^{(\lambda/\zeta\s)}(\tau\s) - \bm{M} \|_F^2}{n} \\
    &\leq \frac{4 r \kappa^2 \lambda_{\min}(\bm{M})^2}{n} + \frac{4 \cdot 36 (\tau\s)^2 dr}{n}  \\
    &\leq \frac{4 dr \kappa^2 \lambda_{\min}^2 \sigma^2}{\sigma^2 n r} + \frac{24 \cdot 36 \kappa^2  \sigma^2 d r}{n} \\&= O\bigg( \sigma^2\frac{dr}{n} \bigg),
\end{align}
where we have used the bound $\frac{\lambda_{\min}^2}{\sigma^2 r} \leq C_2$ by \cref{ass1}.  
It remains to bound $|\zeta\s - 1 |$.  We have
\begin{align}
    \bigg| \zeta\s - 1 \bigg| &= \frac{\mathbb{E} \langle \bm{H}, \bm{Z}\s \rangle}{n \tau\s}  
\end{align}
Suppose that $\bm{Z}\s = \bm{Z}^{(\lambda/\zeta\s)}_{{\sf ST}}(\tau\s) = \bm{\hat U} \bm{\hat \Lambda} \bm{\hat U}\t$.  Then
\begin{align}
    \langle \bm{H}, \bm{Z}\s \rangle &= \langle \bm{\hat U}\t \bm{H} \bm{\hat U}, \bm{\hat \Lambda} \rangle \leq \| \bm{\hat U}\t \bm{H} \bm{\hat U} \|_F \| \bm{\hat \Lambda} \|_F \\
    &\leq \mathrm{rank}(\bm{Z}^{(\lambda)}_{{\sf ST}}(\tau\s)) \| \bm{H} \| \| \bm{M} + \tau\s \bm{H} \| \\
    &\leq \mathrm{rank}(\bm{Z}^{(\lambda)}_{{\sf ST}}(\tau\s)) \bigg( \tau\s \|\bm{H} \|^2 + \| \bm{H} \| \| \bm{M} \| \bigg).
\end{align}
Plugging this in above yields
\begin{align}
    \big| \zeta\s - 1 \big| &\leq \frac{1}{n\tau\s} \mathbb{E} \bigg\{ \mathrm{rank}(\bm{Z}^{(\lambda)}_{{\sf ST}}(\tau\s)) \bigg( \tau\s \|\bm{H} \|^2 + \| \bm{H} \| \| \bm{M} \| \bigg) \bigg\} \\
    &\leq \frac{1}{n} \mathbb{E} \bigg\{ \| \bm{H} \|^2 \mathrm{rank}(\bm{Z}^{(\lambda)}_{{\sf ST}}(\tau\s) \bigg\} + \frac{\lambda_r \sqrt{d}}{n\tau\s} \mathbb{E} \bigg\{ \| \bm{H} \| \mathrm{rank}(\bm{Z}^{(\lambda)}_{{\sf ST}}(\tau\s) \bigg\}.
\end{align}
Our previous argument has already shown that there is a universal constant $C > 0$ such that as long as $d \geq C$, 
\begin{align}
    \mathbb{E} \bigg\{ \| \bm{H} \|^2 \mathrm{rank}(\bm{Z}^{(\lambda)}_{{\sf ST}}(\tau\s) \bigg\} \leq C d r.
\end{align}
A similar argument shows that
\begin{align}
    \mathbb{E} \bigg\{ \| \bm{H} \| \mathrm{rank}(\bm{Z}^{(\lambda)}_{{\sf ST}}(\tau\s) \bigg\} \leq C \sqrt{d} r.
\end{align}
As a result, by \cref{ass1} we have that
\begin{align}
    \big| \zeta\s - 1 \big| \leq C \frac{d r}{n} \bigg( 1 + \frac{\lambda_r}{\tau\s} \bigg) \leq  C \frac{d r}{n} \frac{\lambda_r}{\sigma} \leq C C_2 \frac{d r^{3/2}}{n} = O\bigg( \frac{d r^{3/2}}{n} \bigg).
\end{align}
\end{proof}

\subsubsection{Proof of \cref{lem:Zstgoodproperties}}
\label{sec:Zstgoodproperties_proof}
\begin{proof}[Proof of \cref{lem:Zstgoodproperties}]
We mimic the proof in \citet{celentano_lasso_2023}.  
First we observe that for $\zeta\s,\tau\s$ fixed, the function $\bm{H} \mapsto \bm{Z}_{{\sf ST}}^{(\lambda/\zeta\s)}(\tau\s)$ is a proximal operator, and hence is $\tau\s$-Lipschitz.  Therefore, the random variable $\|\bm{Z}_{{\sf ST}}^{(\lambda/\zeta\s)}(\tau\s) - \bm{M} \|_F/\sqrt{dr}$ is $\frac{(\tau\s)^2}{dr}-$subgaussian.  In addition, by \cref{lem:saddlepoint}, it holds that 
\begin{align}
    \frac{\mathbb{E} \| \bm{Z}_{{\sf ST}}^{(\lambda/\zeta\s)}(\tau\s) - \bm{M} \|_F^2}{n} = (\tau\s)^2 - \sigma^2 = O\bigg( \frac{dr}{n} \sigma^2 \bigg),
\end{align}
where the last bound comes from \cref{lem:fxdpointprops}.  Therefore one has
\begin{align}
      \frac{\mathbb{E} \| \bm{Z}_{{\sf ST}}^{(\lambda/\zeta\s)}(\tau\s) - \bm{M} \|_F^2}{dr} &= O(1).
\end{align}
Consequently, we see that
\begin{align}
    \bm{H} \mapsto \frac{1}{dr} \mathbb{E} \| \bm{Z}_{{\sf ST}}^{(\lambda/\zeta\s)}(\tau\s) - \bm{M} \|_F^2
\end{align}
is $(\frac{C_\sigma}{dr},\frac{C_\sigma}{dr})$ sub-Gamma by Proposition G.5 of \citet{miolane_distribution_2021}.  Therefore, by tail bounds on sub-Gamma random variables,
     \begin{align}
\p\bigg\{      \bigg|  \| \bm{ Z}_{{\sf ST}}^{(\lambda/\zeta\s)}(\tau\s) - \bm{M} \|_F^2 - \mathbb{E} \|\bm{ Z}_{{\sf ST}}^{(\lambda/\zeta\s)}(\tau\s) - \bm{M} \|_F^2 \bigg|  > s d r \bigg\} \leq C \exp(- c dr (s \vee s^2) ),
    \end{align}
    which is what was required to show.     

    For the other assertion, we first observe that
    \begin{align}
        \p\bigg\{ \mathrm{rank}(\bm{Z}^{(\lambda/\zeta\s)}(\tau\s) \leq r \bigg\} = \p\bigg\{ |\hat \lambda_{r+1}\big( \bm{M} + \tau\s \bm{H} \big) |\leq \lambda/\zeta\s \bigg\}.
    \end{align}
By \cref{lem:fxdpointprops} it holds that 
\begin{align*}
    \frac{\lambda}{\zeta\s} = \frac{\lambda}{1 + O( \frac{d r^{3/2}}{n})}  \geq \lambda/2. 
\end{align*}
Therefore, 
\begin{align}
    \p\bigg\{ |\hat \lambda_{r+1}\big( \bm{M} + \tau\s \bm{H} \big) |\leq \lambda/\zeta\s \bigg\} \geq \p\bigg\{ |\hat \lambda_{r+1}\big( \bm{M} + \tau\s \bm{H} \big) |\leq  \lambda/2 \bigg\}.
\end{align} By Weyl's inequality,
    \begin{align}
        | \hat \lambda_{r+1}\big( \bm{M} + \tau\s \bm{H} \big)| \leq \tau\s \| \bm{H} \|.
    \end{align}
    Since $\bm{H}$ is a ${\sf GOE}(d)$ random matrix, it holds that $\| \bm{H} \| \leq 32 \sqrt{d}$ with probability at least $1 - \exp( - c d)$ for $d \geq C$. In addition, by \cref{lem:fxdpointprops} it holds that $(\tau\s)^2 \leq C \sigma^2$, where $C$ depends only on $\kappa$.  Consequently, as long as $\lambda \geq C_3 \sigma \sqrt{d}$, for $C_3$ larger than some constant (possibly depending on $\kappa$), it holds that $\hat \lambda_{r+1}( \bm{M} + \tau\s \bm{H}) \leq \tau\s \|\bm{H} \| \leq C \sigma \sqrt{d}\leq  \lambda/2$.  The final assertion holds by the same logic and the triangle inequality.  
\end{proof}

\subsubsection{Informal Derivation of Fixed-Point Equations} \label{sec:heuristicderivation}
We now give more details on the heuristic derivation of $(\tau\s,\zeta\s)$ in \cref{sec:softthresholding}.  
Recall that our goal is to study $\bm{Z}^{(\lambda)}$ defined via
\begin{align}
 \bm{Z}^{(\lambda)} =    \argmin_{\bm{Z}} \frac{1}{2} \| \mathcal{X}(\bm{Z - M}) - \bm{\eps} \|_F^2 + \lambda \| \bm{Z} \|_*.
\end{align}
Define the reparameterized cost function:
\begin{align}
C^{(n)}_{\lambda}(\bm{W}) := \frac{1}{2}     \| \mathcal{X}(\bm{W}) - \bm{\eps} \|_F^2 + \lambda\big\{ \| \bm{W + M}  \|_* - \| \bm{M} \|_* \big\}. \numberthis \label{Cndef}
\end{align}
It is immediately evident that if $\bm{W}$ minimizes $C_{\lambda}^{(n)}$ over the set of matrices $\bm{W}$ such that $\bm{W} + \bm{M}$ is positive semidefinite, then $\bm{Z}^{(\lambda)} = \bm{W} + \bm{M}$.

Observe that we can equivalently write
\begin{align}
    C_{\lambda}^{(n)}(\bm{W}) &=  \max_{\bm{\nu} \in \mathbb{R}^{n}} \langle \mathcal{X}[\bm{W}], \bm{\nu} \rangle - \langle \bm{\eps} , \bm{\nu} \rangle - \frac{1}{2} \| \bm{\nu} \|^2 + \lambda \big\{ \| \bm{W+M} \|_{*} - \|\bm{M}\|_{*}\big\}.
\end{align}
Consequently, the min-max above takes the form required for the Matrix CGMT.  Therefore, 
we can show that the minimum value of $C_{\lambda}^{(n)}$ is related to that of the function $L_{\lambda}^{(n)}$, defined via
\begin{align}
    L_{\lambda}^{(n)}(\mathcal{W}) :=\frac{1}{dr} \bigg( \min_ {\substack{\bm{W}  \in \mathcal{W}\\ \bm{W} + \bm{M} \succcurlyeq 0}} \max_{\bm{\nu} \in \mathbb{R}^n} \| \bm{W} \|_F \langle \bm{g}, \bm{\nu} \rangle /\sqrt{n} +&\| \bm{\nu}\| \langle \bm{H}, \bm{W}
    \rangle /\sqrt{n}- \langle \bm{\eps}, \bm{\nu} \rangle - \frac{1}{2} \| \bm{\nu} \|^2 \\
    &\qquad\qquad + \lambda (\| \bm{W} + \bm{M} \|_{*} - \|\bm{M}\|_* )\bigg).
\end{align}
Again, the argument of $L_{\lambda}^{(n)}$ is a set (the set for which $\bm{W}$ is minimized over). 
Therefore, it suffices to analyze the minimum value of $L_{\lambda}^{(n)}$ over $\mathbb{R}^{d\times d}$.  By expanding $L_{\lambda}^{(n)}$ out further and manipulating, we can show that 
\begin{subequations}
\label{l_lambda_n1}    
\begin{align}
    L_{\lambda}^{(n)}(\mathbb{R}^{d\times d}) &= \max_{\beta > 0} \min_{\tau \geq \sigma} \bigg( \frac{\sigma^2}{\tau} + \tau \bigg) \frac{\|\bm{g}\|}{\sqrt{n}}\frac{\gamma_n \beta}{2} - \gamma_n \frac{\beta^2}{2} \\
    &\quad + \min_{ \substack{\bm{W} \in\mathbb{R}^{d\times d}\\\bm{W} + \bm{M} \succcurlyeq 0}} \bigg\{  \beta \frac{\|\bm{W}\|_F^2}{2 dr \tau} + \beta \frac{\langle \bm{H}, \bm{W} \rangle}{dr} + \frac{\lambda}{dr} \big\{ \| \bm{W} + \bm{M} \|_* - \|\bm{M} \|_* \big\} \bigg\}.
\end{align}
\end{subequations}
After changing variables back to $\bm{Z} = \bm{W} + \bm{M}$, the inner minimization can be obtained by soft-thresholding $\bm{M} + \tau \bm{H}$ at level $\lambda/\zeta$ where $\zeta = \tau\beta$.

\subsection{Proof of \cref{lem:ZSTtau,lem:ZHTtau}} \label{sec:ZHTtau}
\begin{proof}[Proof of \cref{lem:ZSTtau,lem:ZHTtau}]
We prove the result first for hard thresholding.
 First, define the event
    \begin{align}
        \mathcal{E} 
        &:= \bigg\{ \| \bm{H} \| \leq C \sqrt{d} \bigg\}.
    \end{align}
 We observe that on $\mathcal{E}$, since $\lambda \geq C_3 \sigma \sqrt{d}$ and $\tau\s \leq C \sigma$ by \cref{lem:fxdpointprops}, it holds that $\bm{Z}_{{\sf HT}}^{(\lambda)}$ is rank $r$. Consequently, on this event,
\begin{align}
    \phi\bigg( \bm{Z}_{{\sf HT}}^{(\lambda)} \bigg) &= \phi\bigg( \mathcal{P}_{{\sf rank}-r}(\bm{M} + \tau\s \bm{H} ) \bigg) \\
    &= \phi\bigg( \mathcal{P}_{{\sf rank}-r}\bigg[\bm{M} + \sigma \bigg( 1 + O\big( \frac{dr}{n} \big)\bigg) \bm{H} \bigg] \bigg).
\end{align}
Let $\bm{U}_{\sigma} \bm{\Lambda}_{\sigma} \bm{U}_{\sigma}\t$ denote the leading rank $r$ approximation of $\bm{M} + \sigma \bm{H}$, and let $\bm{U}_{\tau} \bm{\Lambda}_{\tau} \bm{U}_{\tau}\t$ be the leading rank $r$ approximation of $\bm{M} + \tau\s \bm{H}$.  Recall that we let $\mathcal{O}_{\bm{U}_{\sigma},\bm{U}_{\tau}}$ denote the Frobenius-optimal orthogonal matrix aligning $\bm{U}_{\sigma}$ to $\bm{U}_{\tau}$.  We have
\begin{align*}
    \| \bm{Z}_{{\sf HT}}^{(\lambda)} - \bm{Z}_{{\sf HT}} \|_F &= \| \bm{U}_{\sigma} \bm{\Lambda}_{\sigma} \bm{U}_{\sigma}\t - \bm{U}_{\tau} \bm{\Lambda}_{\tau} \bm{U}_{\tau}\t \|_F \\
    &\leq \| \bm{U}_{\sigma} \bm{\Lambda}_{\sigma} \big(  \bm{U}_{\sigma} - \bm{U}_{\tau} \mathcal{O}_{\bm{U}_{\tau}, \bm{U}_{\sigma}} \big)\t  \|_F +  \| \bm{U}_{\sigma} \big[ \bm{\Lambda}_{\sigma} \mathcal{O}_{\bm{U}_{\tau},\bm{U}_{\sigma}}\t -  \mathcal{O}_{\bm{U}_{\tau},\bm{U}_{\sigma}}\t  \bm{\Lambda}_{\tau} \big] \bm{U}_{\tau}\t \|_F \\
    &\quad + \| \big[ \bm{U}_{\sigma} \mathcal{O}_{\bm{U}_{\tau},\bm{U}_{\sigma}}\t - \bm{U}_{\tau} \big] \bm{\Lambda}_{\tau} \bm{U}_{\tau}\t \|_F \\
    &\leq \underbrace{\| \bm{U}_{\sigma} \bm{\Lambda}_{\sigma} \big(  \bm{U}_{\sigma} - \bm{U}_{\tau} \mathcal{O}_{\bm{U}_{\tau}, \bm{U}_{\sigma}} \big)\t  \|_F + \| \big[ \bm{U}_{\sigma} \mathcal{O}_{\bm{U}_{\tau},\bm{U}_{\sigma}}\t - \bm{U}_{\tau} \big] \bm{\Lambda}_{\tau} \bm{U}_{\tau}\t \|_F}_{=: T_1} \\
    &\quad + \underbrace{\| \bm{\Lambda}_{\sigma}\big[ \mathcal{O}_{\bm{U}_{\tau},\bm{U}_{\sigma}}\t - \bm{U}_{\sigma}\t \bm{U}_{\tau} \big] \|_F  + \| \big[ \bm{U}_{\sigma}\t \bm{U}_{\tau} - \mathcal{O}_{\bm{U}_{\sigma},\bm{U}_{\tau}} \big] \bm{\Lambda}_{\tau} \|_F}_{=:T_2} + \underbrace{\| \bm{\Lambda}_{\sigma}\bm{U}_{\sigma}\t \bm{U}_{\tau} - \bm{U}_{\sigma}\t \bm{U}_{\tau} \bm{\Lambda}_{\tau} \|_F}_{=:T_3}.
\end{align*}
We will analyze each term in turn.  First, we note that on the event $\mathcal{E}$ it holds that 
\begin{align*}
    \max\big\{ | \lambda_{r}(\bm{M} + \sigma \bm{H}) - \lambda_{r+1}(\bm{M} + \tau\s \bm{H})|, | \lambda_{r+1}(\bm{M} + \sigma \bm{H}) - \lambda_{r}(\bm{M} + \tau\s \bm{H})| \big\} \leq C \sigma \sqrt{d} \leq \lambda_{\min}(\bm{M}).
\end{align*}
Consequently, the Davis-Kahan Theorem implies
\begin{align*}
    \| \bm{U}_{\sigma} \mathcal{O}_{\bm{U}_{\sigma} \bm{U}_{\tau}} - \bm{U}_{\tau} \|_F \leq \sqrt{r}  \| \bm{U}_{\sigma} \mathcal{O}_{\bm{U}_{\sigma} \bm{U}_{\tau}} - \bm{U}_{\tau} \| \leq C \frac{|\tau\s - \sigma| \sqrt{dr}}{\lambda_r(\bm{M})}.
\end{align*}
We also note that
\begin{align*}
    \| \mathcal{O}_{\bm{U}_{\sigma} \bm{U}_{\tau}} - \bm{U}_{\sigma}\t \bm{U}_{\tau} \| \leq 1 - \| \cos\Theta(\bm{U}_{\sigma},\bm{U}_{\tau}) \|  \leq \| \sin\Theta(\bm{U}_{\sigma},\bm{U}_{\tau}) \|^2 \leq C^2 \frac{|\tau\s - \sigma |^2 dr}{\lambda_r^2(\bm{M})},
\end{align*}
which follows from the fact that $\mathcal{O}_{\bm{U}_{\sigma},\bm{U}_{\tau}}$ is computed from the SVD of $\bm{U}_{\sigma}\t \bm{U}_{\tau}$, and the singular values of $\bm{U}_{\sigma}\t \bm{U}_{\tau}$ are equal to the cosines of the principal angles between the subspaces.  We now anlyze each term.
\begin{itemize}
    \item \textbf{The term $T_1$.}  We have that 
    \begin{align*}
         \| \big[ \bm{U}_{\sigma} \mathcal{O}_{\bm{U}_{\tau},\bm{U}_{\sigma}}\t - \bm{U}_{\tau} \big] \bm{\Lambda}_{\tau} \bm{U}_{\tau}\t \|_F \leq \| \bm{\Lambda}_{\tau} \| \frac{C |\tau\s - \sigma| \sqrt{dr}}{\lambda_r(\bm{M})} \lesssim | \tau\s - \sigma | \sqrt{dr}.
    \end{align*}
    The same bound holds for the other term.  
    \item \textbf{The term $T_2$.} We have 
    \begin{align}
   \| \bm{\Lambda}_{\sigma}\big[ \mathcal{O}_{\bm{U}_{\tau},\bm{U}_{\sigma}}\t - \bm{U}_{\sigma}\t \bm{U}_{\tau} \big] \|_F &\leq \| \bm{\Lambda}_{\sigma} \| C^2 \frac{|\tau\s - \sigma |^2 dr}{\lambda_r^2(\bm{M})} \leq C'  |\tau\s - \sigma |^2 \sqrt{dr} \frac{\sqrt{dr}}{\lambda_{r}(\bm{M})},
    \end{align}
    which follows by \cref{ass1}.  A similar bound holds for the other quantity.
    \item \textbf{The term $T_3$.} We observe that  
    \begin{align*}
        \| \bm{\Lambda}_{\sigma} \bm{U}_{\sigma}\t \bm{U}_{\tau} - \bm{U}_{\sigma}\t \bm{U}_{\tau} \bm{\Lambda}_{\tau} \|_F &= \| \bm{U}_{\sigma}\t \big[ \bm{M} + \tau\s \bm{H} - ( \bm{M} + \sigma \bm{H}) \big] \bm{U}_{\tau} \|_F \\
        &\lesssim \sqrt{rd} | \tau\s - \sigma|.
    \end{align*}
\end{itemize}
Combining these results, we have that on the event $\mathcal{E}$,
\begin{align*}
    \| \bm{Z}_{{\sf HT}}^{(\lambda)} - \bm{Z}_{{\sf HT}} \|_F \leq C | \tau\s - \sigma | \sqrt{dr} + C \sqrt{dr} | \tau\s - \sigma |^2.
\end{align*}
By \cref{lem:fxdpointprops} it holds that 
\begin{align*}
    |\tau\s - \sigma| \leq \frac{|(\tau\s)^2 - \sigma^2|}{\tau\s + \sigma} &= O\bigg( \frac{dr}{n} \sigma \bigg).
\end{align*}
Therefore, we have that  on the event $\mathcal{E}$,
\begin{align*}
    \| \bm{Z}_{{\sf HT}}^{(\lambda)} - \bm{Z}_{{\sf HT}} \|_F &= O \bigg( \frac{dr}{n} \sigma \bigg).
\end{align*}
Consequently,
 \begin{align}
 \bigg|        \mathbb{E}\bigg\{ \phi\bigg( \frac{\bm{Z}_{{\sf HT}}^{(\lambda)}}{\sqrt{dr}} \bigg) - \phi \bigg( \frac{\bm{Z}_{{\sf HT}}}{\sqrt{dr}} \bigg) \bigg\} \bigg| &\leq \frac{1}{\sqrt{dr}}\mathbb{E} \| \bm{Z}_{{\sf HT}}^{(\lambda)} - \bm{Z}_{{\sf HT}} \|_F \\
 &\leq O\bigg( \frac{dr}{n} \sigma \bigg) + \frac{1}{\sqrt{dr}} \bigg( \mathbb{E}  \| \bm{M} + \tau\s \bm{H} \|_F \mathbb{I}_{\mathcal{E}^c}  + \mathbb{E}  \| \bm{M} + \sigma\s \bm{H} \|_F \mathbb{I}_{\mathcal{E}^c} \bigg)\\ 
 &\leq O \bigg( \frac{dr}{n} \sigma \bigg) + \frac{\| \bm{M} \|}{\sqrt{dr}} \p\{ \mathcal{E}^c \} +  \frac{1}{\sqrt{r}} \mathbb{E} \| \bm{H} \| \mathbb{I}_{\mathcal{E}^c} \\
 &\leq O\bigg( \frac{dr}{n} \sigma \bigg) + C \exp( - c d) + \frac{1}{\sqrt{r}} \int_{C\sqrt{d}}^{\infty} \exp( - c s^2) ds \\
 &= O\bigg( \frac{dr}{n} \sigma \bigg).
 \end{align}
 Here we have used the fact that $\p\{ \|\bm{H}\| \geq C \sqrt{d} + s \} \leq 2 \exp( - c s^2)$ by standard tail bounds for norms of random matrices \citep{bandeira_sharp_2016}.

We now prove the result for soft thresholding.  First, we note that the exact same argument continues applies if we
make the replacement $\bm{M} \mapsto \bm{M} - \lambda \bm{I}$.  Consequently, we have that 
\begin{align}
\bigg|     \mathbb{E} \phi\bigg( \frac{\bm{Z}_{{\sf ST}}^{(\lambda/\zeta\s)}(\tau\s)}{\sqrt{dr}} \bigg) - \mathbb{E} \phi\bigg( \frac{\bm{Z}_{{\sf ST}}^{(\lambda)}(\sigma)}{\sqrt{dr}} \bigg) \bigg| &\leq \bigg|     \mathbb{E} \phi\bigg( \frac{\bm{Z}_{{\sf ST}}^{(\lambda)}(\tau\s)}{\sqrt{dr}} \bigg) - \mathbb{E} \phi\bigg( \frac{\bm{Z}_{{\sf ST}}^{(\lambda)}(\sigma)}{\sqrt{dr}} \bigg) \bigg| \\
&\quad + \bigg|     \mathbb{E} \phi\bigg( \frac{\bm{Z}_{{\sf ST}}^{(\lambda/\zeta\s)}(\tau\s)}{\sqrt{dr}} \bigg) - \mathbb{E} \phi\bigg( \frac{\bm{Z}_{{\sf ST}}^{(\lambda)}(\tau\s)}{\sqrt{dr}} \bigg) \bigg| \\
&\leq O\bigg( \frac{dr}{n} \sigma \bigg) + \bigg|     \mathbb{E} \phi\bigg( \frac{\bm{Z}_{{\sf ST}}^{(\lambda/\zeta\s)}(\tau\s)}{\sqrt{dr}} \bigg) - \mathbb{E} \phi\bigg( \frac{\bm{Z}_{{\sf ST}}^{(\lambda)}(\tau\s)}{\sqrt{dr}} \bigg) \bigg|.
\end{align}
Therefore, it suffices to focus on the second term.  

On the event $\mathcal{E}$ the quantity $\bm{Z}_{{\sf ST}}^{(\lambda/\zeta\s)}(\tau\s)$ is simply hard thresholding of $\bm{M} -\lambda + \tau\s \bm{H}$ at level $\lambda/\zeta\s$, and similarly for  $\bm{Z}_{{\sf ST}}^{(\lambda)}(\tau\s)$.  Therefore, both matrices share the same eigenvectors, and hence on the event $\mathcal{E}$,
\begin{align*}
    \|\bm{Z}_{{\sf ST}}^{(\lambda/\zeta\s)}(\tau\s) - \bm{Z}_{{\sf ST}}^{(\lambda)}(\tau\s) \|_F \leq \sqrt{2r} \lambda \bigg( \frac{1}{\zeta\s} - 1 \bigg) = O\bigg( \sigma \sqrt{rd} \times \frac{d r^{3/2}}{n} \bigg),
\end{align*}
where we have used the fact that $|\zeta\s  - 1 | = O\bigg( \frac{d r^{3/2}}{n} \bigg)$ by \cref{lem:fxdpointprops}.  Therefore, the same argument as for $\bm{Z}_{{\sf HT}}$ shows that 
\begin{align}
\bigg|     \mathbb{E} \phi\bigg( \frac{\bm{Z}_{{\sf ST}}^{(\lambda/\zeta\s)}(\tau\s)}{\sqrt{dr}} \bigg) - \mathbb{E} \phi\bigg( \frac{\bm{Z}_{{\sf ST}}^{(\lambda)}(\sigma)}{\sqrt{dr}} \bigg) \bigg| &= O\bigg( \frac{d r^{3/2}}{n} \bigg).
\end{align}
\end{proof}

\subsection{Proof of the Matrix CGMT (\cref{thm:matrixcgmt})}
\label{sec:cgmtproof}
In this section we prove the Matrix CGMT.  Our proof is based on \citet{miolane_distribution_2021}.  

\begin{proof}[Proof of \cref{thm:matrixcgmt}]
Define the Gaussian processes
\begin{align}
    X(\bm{W},\bm{\nu}) &:= \| \bm{W} \|_F \langle \bm{g}, \bm{\nu} \rangle + \| \bm{\nu} \|\langle \bm{H}, \bm{W} \rangle; \\
    Y(\bm{W},\bm{\nu} ) &:= \langle \mathcal{A}[\bm{W}], \bm{\nu} \rangle + \| \bm{\nu} \| \| \bm{W} \| z,
\end{align}
where $z \sim \mathcal{N}(0,1)$ is independent from $\mathcal{A}(\cdot)$, $\bm{g}$ is a standard Gaussian vector, and $\bm{H}$ is a ${\sf GOE}(d)$ random matrix.  Observe that since $\bm{H}$ is ${\sf GOE}(d)$, $\langle \bm{H}, \bm{W} \rangle \sim \mathcal{N}(0, \| \bm{W} \|_F^2)$ for symmetric $\bm{W}$, whence we need only define the processes for symmetric matrices $\bm{W}$.

We have
\begin{align}
    \mathbb{E} \bigg[ Y(\bm{W},\bm{\nu}) Y( \bm{W}' ,\bm{\nu'} )\bigg] &- \mathbb{E}  \bigg[ X(\bm{W},\bm{\nu}) X( \bm{W}' ,\bm{\nu'} )\bigg] \\
    &= \| \bm{W} \|_F \| \bm{\nu} \| \| \bm{W}' \|_F \| \bm{\nu}' \| + \langle \bm{\nu}, \bm{\nu}' \rangle \langle \bm{W}, \bm{W}' \rangle - \| \bm{\nu} \| \|\bm{\nu}' \| \langle \bm{W}, \bm{W}' \rangle -  \| \bm{W} \|_F \|\bm{W}' \|_F \langle \bm{\nu}, \bm{\nu}' \rangle \\
    &= \bigg( \| \bm{W} \|_F \| \bm{W}'\|_F - \langle \bm{W}, \bm{W}' \rangle \bigg) \bigg( \|\bm{\nu} \| \| \bm{\nu}' \| - \langle \bm{\nu}, \bm{\nu}' \rangle \bigg) \geq 0,
\end{align}
with equality when $\bm{W} = \bm{W}'$ or $\bm{\nu} = \bm{\nu}'$.  Therefore, by Theorem G.2 of \citet{miolane_distribution_2021}, 
\begin{align}
    \p\bigg\{ \min_{\bm{W} \in \mathcal{S}_{\bm{W}}} \max_{\bm{\nu} \in \mathcal{S}_{\nu}} Y(\bm{W},\bm{\nu}) + \psi(\bm{W},\bm{\nu}) \leq t \bigg\} \leq  \p\bigg\{\min_{\bm{W} \in \mathcal{S}_{\bm{W}}} \max_{\bm{\nu} \in \mathcal{S}_{\nu}} X(\bm{W},\bm{\nu}) + \psi(\bm{W}, \bm{\nu}) \leq t \bigg\}.
\end{align}
Furthermore, 
\begin{align}
    \p\bigg\{\min_{\bm{W} \in \mathcal{S}_{\bm{W}}} \max_{\bm{\nu} \in\mathcal{S}_{\bm{\nu}}} Y(\bm{W},\bm{\nu}) + \psi(\bm{W},\bm{\nu}) \leq t \bigg\} &\geq \frac{1}{2}  \p\bigg\{ \min_{\bm{W} \in \mathcal{S}_{\bm{W}}} \max_{\bm{\nu} \in \mathcal{S}_{\bm{\nu}}} Y(\bm{W},\bm{\nu}) + \psi(\bm{W},\bm{\nu}) \leq t | z \leq 0 \bigg\} \\
    &\geq \frac{1}{2}  \p\bigg\{ \min_{\bm{W} \in \mathcal{S}_{\bm{W}}} \max_{\bm{\nu} \in \mathcal{S}_{\bm{\nu}}}\langle \mathcal{A}(\bm{W}), \bm{\nu} \rangle + \psi(\bm{W},\bm{\nu}) \leq t | z \leq 0 \bigg\}\\
    &=  \frac{1}{2} \p\bigg\{ \Phi(\mathcal{A}) \leq t \bigg\}.
\end{align}
Rewriting the equation above yields
\begin{align}
    \p\bigg\{ \Phi(\mathcal{A}) \leq t \bigg\} \leq 2 \p\bigg\{\min_{\bm{W} \in \mathcal{S}_{\bm{W}}} \max_{\bm{\nu} \in \mathcal{S}_{\bm{\nu}}} X(\bm{W},\bm{\nu}) + \psi(\bm{W}, \bm{\nu}) \leq t \bigg\} = 2 \p \bigg\{ \phi(\bm{g},\bm{H}) \leq t \bigg\}.
\end{align}
This proves the first result.  

Now suppose that $\mathcal{S}_{\bm{W}}$ and $\mathcal{S}_{\bm{\nu}}$ are convex and that $\psi$ is convex-concave. Note that the argument above relies not on the minimization over $\mathcal{S}_{\bm{W}}$ and maximization over $\mathcal{S}_{\bm{\nu}}$ -- we are free to minimize instead over $\mathcal{S}_{\bm{\nu}}$ and maximize over $\mathcal{S}_{\bm{W}}$.  Then this argument above shows that (using $-\psi$ instead of $\psi$)
\begin{align}
    \p\bigg\{ \min_{\bm{\nu} \in \mathcal{S}_{\bm{\nu}}}\max_{\bm{W} \in \mathcal{S}_{\bm{W}}}  \langle \mathcal{A}(\bm{W}), \bm{\nu} \rangle - \psi(\bm{W}, \bm{\nu}) \leq - t \bigg\} \leq 2 \p\bigg\{ \min_{\bm{\nu} \in \mathcal{S}_{\bm{\nu}}}\max_{\bm{W} \in \mathcal{S}_{\bm{W}}}  \|\bm{W} \|_F \langle \bm{g}, \bm{\nu} \rangle + \| \bm{\nu} \| \langle \bm{H}, \bm{W} \rangle   - \psi(\bm{W}, \bm{\nu}) \leq - t \bigg\}.
\end{align}
Multiplying everything by $-1$ and using the fact that $\mathcal{A} (\cdot)$, $\bm{H}$, and $\bm{g}$ are equal in distribution to $-\mathcal{A}(\cdot)$, $- \bm{H}$ and $-\bm{g}$ respectively yields
\begin{align}
   \p\bigg\{ \max_{\bm{\nu} \in \mathcal{S}_{\bm{\nu}}}\min_{\bm{W} \in \mathcal{S}_{\bm{W}}}  \langle \mathcal{X}(\bm{W}), \bm{\nu} \rangle +  \psi(\bm{W}, \bm{\nu}) \geq  t \bigg\} \leq 2 \p\bigg\{ \max_{\bm{\nu} \in \mathcal{S}_{\bm{\nu}}}\min_{\bm{W} \in \mathcal{S}_{\bm{W}}}  \|\bm{W} \|_F \langle \bm{g}, \bm{\nu} \rangle + \| \bm{\nu} \| \langle \bm{H}, \bm{W} \rangle   + \psi(\bm{W}, \bm{\nu}) \geq t \bigg\}.
\end{align}
We observe that since $\mathcal{S}_{\bm{W}}$ and $\mathcal{S}_{\bm{\nu}}$ are both convex and $\psi$ is convex-concave, by Sion's minimax theorem we are able to switch the minimum and maximum on the left hand side.  For the right hand side we need not switch, but instead only observe that
\begin{align}
    \max_{\bm{\nu} \in \mathcal{S}_{\bm{\nu}}}\min_{\bm{W} \in \mathcal{S}_{\bm{W}}}  \|\bm{W} \|_F \langle \bm{g}, \bm{\nu} \rangle + \| \bm{\nu} \| \langle \bm{H}, \bm{W} \rangle   + \psi(\bm{W}, \bm{\nu}) \leq \min_{\bm{W} \in \mathcal{S}_{\bm{W}}} \max_{\bm{\nu} \in \mathcal{S}_{\bm{\nu}}} \|\bm{W} \|_F \langle \bm{g}, \bm{\nu} \rangle + \| \bm{\nu} \| \langle \bm{H}, \bm{W} \rangle   + \psi(\bm{W}, \bm{\nu}).
\end{align}
This completes the proof.
\end{proof}

\section{Proof of Lemmas in \cref{sec:equivalence} (Theorem \ref{thm:relatecvxncvx})}
\label{sec:cvxncvx}

 Throughout this section we fix a constant $\delta_{2r} = \frac{c_0}{\sqrt{r}}$, which we are able to do on the event $\mathcal{E}_{{\sf Good}}$ by \cref{ass2} and \cref{lem:Egood}.   Throughout this section we denote $\bm{U}\s$ as any matrix such that $\bm{U}\s \bm{U}^{*\top} = \bm{M}$.  

\subsection{Proof of \cref{lem:uncvxgood}} \label{sec:uncvxgoodproof}
In this subsection we prove \cref{lem:uncvxgood}. The proof is based on studying the second derivative of $f^{(\lambda)}_{{\sf ncvx}}$ under the necessary conditions for local minima. First, the following lemma provides a calculation of the Hessian of $f^{(\lambda)}_{{\sf ncvx}}$.  

\begin{lemma}\label{lem:secondderivcalc}
    Let $\Delta = \bm{U} - \bm{U}_2\mathcal{O}_{\bm{U},\bm{U}_2}$.  Then it holds that
    \begin{align}
     \nabla^2 f^{(\lambda)}_{{\sf ncvx}}(\bm{U}) [ \bm{\Delta},\bm{\Delta}] &= .5 \langle \mathcal{X}(\bm{\Delta \Delta}\t), \mathcal{X}(\bm{\Delta \Delta}\t) \rangle \\
    &\qquad + 4  \langle \nabla f^{(\lambda)}_{{\sf ncvx}}(\bm{U}),\Delta \rangle + \frac{4}{\sqrt{n}} \sum_{i=1}^{n} \eps_i\langle \bm{X}_i, \bm{U\Delta}\t \rangle - 4 \lambda \langle \bm{U},\bm{\Delta} \rangle \\
   &\qquad - 1.5 \langle \mathcal{X}(\bm{UU}\t - \bm{M}) ,\mathcal{X}(\bm{UU}\t - \bm{M}) \rangle \\
    &\qquad + \frac{1}{\sqrt{n}} \sum_{i=1}^{n} \big[ \eps_i \langle \bm{X}_i, \bm{\Delta},\bm{\Delta \rangle}] +  \lambda \| \bm{\Delta}\|_F^2.
\end{align}
\end{lemma}
\begin{proof}
    See \cref{sec:secondderivcalc}.
\end{proof}

With this calculation at hand, we may prove \cref{lem:uncvxgood}.

\begin{proof}[Proof of \cref{lem:uncvxgood}]
First we will show that $\nabla^2 f^{(\lambda)}_{{\sf ncvx}}(\bm{U})[ \bm{\Delta}, \bm{\Delta}]$ is upper bounded when $\bm{U}$ is a critical point of $f^{(\lambda)}_{{\sf ncvx}}$.  The proof is motivated by \citet{ge_no_2017}.

First, on the event $\mathcal{E}_{{\sf Good}}$, we have that
\begin{align}
\bigg|    \frac{4}{\sqrt{n}} \sum_{i=1}^{n} \eps_i \langle \bm{X}_i, \bm{U\Delta}\t \rangle \bigg| &\leq C \sigma \sqrt{dr} \| \bm{U\Delta}\t \|_F; \\
\bigg| \frac{1}{\sqrt{n}} \sum_{i=1}^{n} \eps_i \langle \bm{X}_i, \bm{\Delta \Delta}\t \rangle \bigg| &\leq C \sigma \sqrt{dr} \| \bm{\Delta\Delta}\t \|_F
\end{align}
By the restricted isometry property (e.g. Lemma 3 of \citet{chi_nonconvex_2019}), we have that
\begin{align}
    \bigg| \big\langle \mathcal{X}(\bm{UU}\t - \bm{M}), \mathcal{X}(\bm{UU}\t - \bm{M}) \big \rangle - \langle \bm{UU}\t - \bm{M},  \bm{UU}\t - \bm{M}\rangle \bigg| \leq \delta_{2r} \| \bm{UU}\t - \bm{M} \|_F^2
\end{align}
and
\begin{align}
    \bigg| \langle \mathcal{X}( \bm{\Delta \Delta}\t  ) ,  \mathcal{X}( \bm{\Delta \Delta}\t) \rangle - \langle \bm{\Delta \Delta}\t, \bm{\Delta \Delta}\t \rangle \bigg| \leq \delta_{2r}  \| \bm{\Delta \Delta}\t \|_F^2
\end{align}
Consequently, if $\bm{U}$ satisfies $\nabla f^{(\lambda)}_{{\sf ncvx}}(\bm{U}) = 0$, 
    \begin{align}
     \nabla^2 f^{(\lambda)}_{{\sf ncvx}}(\bm{U})[ \bm{\Delta},\bm{\Delta}] & \leq .5\big( 1 + \delta_{2r}) \| \bm{\Delta \Delta}\t \|_F^2 \\
     &\qquad + C \sigma \sqrt{dr} \big( \| \bm{U\Delta}\t \|_F + \| \bm{\Delta\Delta}\t \|_F \big) \\
     &\quad - 1.5( 1 - \delta_{2r}) \| \bm{UU}\t - \bm{M} \|_F^2 \\
     &\qquad + \lambda \| \bm{\Delta}\|_F^2 - 4 \lambda \langle \bm{U,\Delta} \rangle.
     \end{align}
     Next, we note that by Lemma 6 of \citet{ge_no_2017}, if $\bm{\Delta} = \bm{U}\s - \bm{U}\mathcal{O}_{\bm{U},\bm{U}\s}$  and $\bm{U}\s \bm{U}^{*\top} = \bm{M}$, then 
     \begin{align}
         \| \bm{\Delta\Delta}\t \|_F^2 \leq 2 \| \bm{UU}\t - \bm{M} \|_F^2; \qquad \| \bm{\Delta}\|_F^2 \leq \frac{1}{2(\sqrt{2} - 1) \lambda_r(\bm{M})} \| \bm{UU}\t - \bm{M} \|_F^2.
     \end{align}
     Plugging in these bounds shows that
    \begin{align}
     \nabla^2 f^{(\lambda)}_{{\sf ncvx}}(\bm{U})[ \bm{\Delta},\bm{\Delta}] & \leq \big( 1 + \delta_{2r}) \| \bm{UU}\t - \bm{M} \|_F^2 \\
     &\qquad + C \sigma \sqrt{dr} \big( \| \bm{U\Delta}\t \|_F + \| \bm{UU}\t - \bm{M} \|_F \big) \\
     &\quad - 1.5( 1 + \delta_{2r}) \| \bm{UU}\t - \bm{M} \|_F^2 \\
     &\qquad + \lambda \frac{1}{2(\sqrt{2} - 1) \lambda_r(\bm{M})} \| \bm{UU}\t - \bm{M} \|_F^2 - 4 \lambda \langle \bm{U,\Delta} \rangle \\
     &= \| \bm{UU}\t - \bm{M} \|_F^2 \bigg( - .5 + .5 \delta_{2r} + \frac{\lambda}{2 (\sqrt{2} - 1) \lambda_r(\bm{M})} \bigg)  \\
     &\qquad + C \sigma \sqrt{dr} \bigg( \| \bm{U}\bm{\Delta}\t \|_F + \| \bm{UU}\t - \bm{M} \|_F\bigg) \\
     &\qquad - 4 \lambda \langle \bm{U,\Delta} \rangle.
     \end{align}
     We further have that
     \begin{align}
         \| \bm{U\Delta}\t \|_F &\leq \| ( \bm{U} - \bm{U}\s (\mathcal{O}_{\bm{U},\bm{U}\s})\t) \bm{\Delta}\t \|_F + \| \bm{U}\s \| \| \bm{\Delta} \|_F \\
         &\leq \| \bm{\Delta\Delta}\t \|_F +  \| \bm{U}\s \| \| \bm{\Delta} \|_F \\
         &\leq \sqrt{2} \| \bm{UU}\t - \bm{M} \|_F + \sqrt{\frac{\lambda_1(\bm{M})}{\lambda_r(\bm{M})}} \frac{1}{\sqrt{2(\sqrt{2}-1)}} \| \bm{UU}\t - \bm{M} \|_F \\
         &\leq (\sqrt{\kappa} + \sqrt{2}) \| \bm{UU}\t - \bm{M} \|_F,
     \end{align}
     where the same bound holds for $\langle \bm{U},\bm{\Delta} \rangle$.  Consequently, plugging this in, we obtain
    \begin{align}
     \nabla^2 f^{(\lambda)}_{{\sf ncvx}}(\bm{U})[ \bm{\Delta},\bm{\Delta}] & \leq \big( 1 + \delta_{2r}) \| \bm{UU}\t - \bm{M} \|_F^2 \\
     &\qquad + C \sigma \sqrt{dr} \big( \| \bm{U\Delta}\t \|_F + \| \bm{UU}\t - \bm{M} \|_F \big) \\
     &\quad - 1.5( 1 + \delta_{2r}) \| \bm{UU}\t - \bm{M} \|_F^2 \\
     &\qquad + \lambda \frac{1}{2(\sqrt{2} - 1) \lambda_r(\bm{M})} \| \bm{UU}\t - \bm{M} \|_F^2 - 4 \lambda \langle \bm{U,\Delta} \rangle \\
     &= \| \bm{UU}\t - \bm{M} \|_F^2 \bigg( - .5 + .5 \delta_{2r} + \frac{\lambda}{2 (\sqrt{2} - 1) \lambda_r(\bm{M})} \bigg)  \\
     &\qquad + C \sigma \sqrt{dr} \bigg( (\sqrt{\kappa} + \sqrt{2}) \| \bm{UU}\t - \bm{M} \|_F + \| \bm{UU}\t - \bm{M} \|_F\bigg) \\
     &\qquad +  4 \lambda(\sqrt{\kappa} + \sqrt{2}) \| \bm{UU}\t - \bm{M} \|_F \\
     &= \| \bm{UU}\t - \bm{M} \|_F^2 \bigg( - .5 + .5 \delta_{2r} + \frac{\lambda}{2 (\sqrt{2} - 1) \lambda_r(\bm{M})} \bigg) \\
     &\qquad + \| \bm{UU}\t - \bm{M} \|_F \bigg( C \sigma  \sqrt{\kappa dr} +  4 \lambda \sqrt{\kappa} \bigg).
     \end{align}
 Under the assumption that $\delta_{2r} \leq \frac{1}{10}$, $\lambda_r(\bm{M}) \geq C_1 \sigma \sqrt{dr}$, and $C_3 \sigma \sqrt{d} \leq \lambda \leq C_4 \sigma \sqrt{d}$, it holds that
 \begin{align}
     - .5 + .5 \delta_{2r} + \frac{\lambda}{2 (\sqrt{2} - 1) \lambda_r(\bm{M})}  \leq - .1.
 \end{align}
    If $\bm{U}$ is a local minimum, then it holds that $0 \leq \nabla^2 f^{(\lambda)}_{{\sf ncvx}}(\bm{U})[ \bm{\Delta},\bm{\Delta}]$.  
         Rearranging this yields
     \begin{align}
       .1  \| \bm{UU}\t - \bm{M} \|_F^2 &\leq \| \bm{UU}\t - \bm{M} \|_F\bigg( C \sigma  \sqrt{\kappa dr} +  4 \lambda \sqrt{\kappa} \bigg).
     \end{align}
     Note that if $\bm{UU}\t = \bm{M}$ then the result is trivially proven, so without loss of generality we may assume $\|\bm{UU}\t - \bm{M} \|_F$ is positive.  Rearranging shows that we must have
     \begin{align}
         \| \bm{UU}\t - \bm{M} \|_F \leq 10 C \sigma \sqrt{\kappa dr} + 40 \lambda \sqrt{\kappa}, 
     \end{align}
    which completes the proof, since $\kappa = O(1)$ by \cref{ass1} and $\lambda \leq C_5 \sigma \sqrt{d}$ by \cref{ass3}.
\end{proof}

\subsection{Proof of \cref{lem:secondderivcalc}} \label{sec:secondderivcalc}
\begin{proof}[Proof of \cref{lem:secondderivcalc}]
First, we have that 
\begin{align}
    f^{(\lambda)}_{{\sf ncvx}}(\bm{U}) &= \frac{1}{4} \sum_{i=1}^{n} \bigg[ y_i - \langle \bm{UU}\t, \bm{X}_i \rangle/\sqrt{n} \bigg]^2 + \frac{\lambda}{2} \| \bm{U} \|_F^2,
\end{align}
whence
\begin{align}
    \nabla f^{(\lambda)}_{{\sf ncvx}}(\bm{U}) &= \sum_{i=1}^{n} \bigg[ \langle \bm{UU}\t , \bm{X}_i \rangle/\sqrt{n} - y_i \bigg] \bm{X}_i\bm{U} + \lambda \bm{U}.
\end{align}
As a result,
\begin{align}
    \nabla^2 &f^{(\lambda)}_{{\sf ncvx}}(\bm{U})[ \bm{\Delta},\bm{\Delta}] \\
    &= \frac{d}{dt} \langle \nabla f^{(\lambda)}_{{\sf ncvx}}(\bm{U} + t \bm{\Delta}) , \bm{\Delta} \rangle\bigg|_{t=0} \\
    &= \frac{d}{dt} \bigg\langle \sum_{i=1}^{n} \bigg[ \langle \bm{(U + t \bm{\Delta})(U + t \Delta)}\t , \bm{X}_i \rangle/\sqrt{n} - y_i \bigg] \bm{X}_i (\bm{U} + t \bm{\Delta}) + \lambda (\bm{U} + t \bm{\Delta}) , \bm{\Delta} \bigg\rangle\bigg|_{t=0} \\
    &=   \frac{1}{n}\sum_{i=1}^{n} \bigg \langle \bm{\Delta U}\t + \bm{U\Delta}\t, \bm{X}_i \bigg \rangle \bigg\langle \bm{X}_i \bm{U}, \bm{\Delta} \bigg \rangle + \frac{1}{\sqrt{n}} \sum_{i=1}^{n} \bigg[  \langle \bm{UU}\t , \bm{X}_i \rangle/\sqrt{n} - y_i \bigg] \langle \bm{X}_i \bm{\Delta},\bm{\Delta} \rangle +  \lambda \| \bm{\Delta} \|_F^2
    \\
    &=  \frac{1}{n} 2 \sum_{i=1}^{n} \bigg \langle  \bm{U\Delta}\t, \bm{X}_i \bigg \rangle \bigg\langle \bm{X}_i \bm{U}, \bm{\Delta} \bigg \rangle + \frac{1}{\sqrt{n}} \sum_{i=1}^{n} \bigg[  \langle \bm{UU}\t , \bm{X}_i \rangle/\sqrt{n} - y_i \bigg] \langle \bm{X}_i \bm{\Delta},\bm{\Delta} \rangle + \lambda \| \bm{\Delta} \|_F^2  \\
    &= \frac{1}{n} 2 \sum_{i=1}^{n} \bigg \langle  \bm{U\Delta}\t, \bm{X}_i \bigg \rangle^2  + \frac{1}{n} \sum_{i=1}^{n} \bigg[  \langle \bm{UU}\t - \bm{M} , \bm{X}_i \rangle \bigg] \langle \bm{X}_i \bm{\Delta},\bm{\Delta} \rangle - \frac{1}{\sqrt{n}}\sum_{i=1}^{n} \bigg[ \eps_i \langle \bm{X}_i \bm{\Delta}, \bm{\Delta} \rangle \bigg] +\lambda \| \bm{\Delta} \|_F^2 \\
    &= .5 \big \langle \mathcal{X}( \bm{U\Delta}\t + \bm{\Delta U}\t),\mathcal{X}( \bm{U\Delta}\t + \bm{\Delta U}\t) \big \rangle \\
    &\qquad +  \langle \mathcal{X}( \bm{UU}\t -  \bm{M}   ) ,  \mathcal{X}( \bm{\Delta \Delta}\t) \rangle  + \frac{1}{\sqrt{n}}\sum_{i=1}^{n} \bigg[ \eps_i \langle \bm{X}_i \bm{\Delta}, \bm{\Delta} \rangle \bigg] +\lambda \| \bm{\Delta} \|_F^2.
\end{align}
Next, we note that
\begin{align}
    \bm{UU}\t - \bm{M} + \bm{\Delta \Delta}\t = \bm{U\Delta}\t + \bm{\Delta U}\t.
\end{align}
Plugging this in yields
\begin{align}
   \nabla^2 f^{(\lambda)}_{{\sf ncvx}}(\bm{U})[ \bm{\Delta},\bm{\Delta}] 
    &= .5 \langle \mathcal{X}( \bm{\Delta\Delta}\t ),\mathcal{X}( \bm{\Delta\Delta}\t ) \rangle \\
    &\qquad + 2\langle \mathcal{X}(\bm{UU}\t - \bm{M}), \mathcal{X}(\bm{U\Delta}\t + \bm{\Delta U}\t) \rangle  \\
    &\qquad - 1.5 \langle \mathcal{X}(\bm{UU}\t - \bm{M}) ,\mathcal{X}(\bm{UU}\t - \bm{M}) \rangle \\
    &\qquad + \frac{1}{\sqrt{n}} \sum_{i=1}^{n} \big[ \eps_i \langle \bm{X}_i, \bm{\Delta},\bm{\Delta \rangle}] +  \lambda \| \bm{\Delta}\|_F^2.
\end{align}
In addition, we have that
\begin{align}
  \langle \nabla f^{(\lambda)}_{{\sf ncvx}}(\bm{U}), \Delta \rangle &= .5 \langle \mathcal{X}( \bm{UU}\t - \bm{M}), \mathcal{X}(\bm{U} \Delta\t + \bm{\Delta U}\t ) \rangle - \frac{1}{\sqrt{n}} \sum_{i=1}^{n} \eps_i \langle \bm{X}_i, \bm{U \Delta}\t\rangle + \lambda \langle \bm{U}, \bm{\Delta} \rangle,
\end{align}
whence
\begin{align}
     \nabla^2 f^{(\lambda)}_{{\sf ncvx}}(\bm{U}) [ \bm{\Delta},\bm{\Delta}] &= .5 \langle \mathcal{X}(\bm{\Delta \Delta}\t), \mathcal{X}(\bm{\Delta \Delta}\t) \rangle \\
    &\qquad + 4  \langle \nabla f^{(\lambda)}_{{\sf ncvx}}(\bm{U}),\Delta \rangle + \frac{4}{\sqrt{n}} \sum_{i=1}^{n} \eps_i\langle \bm{X}_i, \bm{U\Delta}\t \rangle - 4 \lambda \langle \bm{U},\bm{\Delta} \rangle \\
   &\qquad - 1.5 \langle \mathcal{X}(\bm{UU}\t - \bm{M}) ,\mathcal{X}(\bm{UU}\t - \bm{M}) \rangle \\
    &\qquad + \frac{1}{\sqrt{n}} \sum_{i=1}^{n} \big[ \eps_i \langle \bm{X}_i, \bm{\Delta},\bm{\Delta \rangle}] +  \lambda \| \bm{\Delta}\|_F^2.
\end{align}
This completes the proof.
\end{proof}

\section{Proof of Theorem \ref{thm:cvxasymptotics}} \label{sec:step1}
In this section we study the convex estimator $\bm{Z}^{(\lambda)}$ defined via
\begin{align}
    \argmin_{\bm{Z}} \frac{1}{2} \| \mathcal{X}(\bm{Z - M}) - \bm{\eps} \|_F^2 + \lambda \| \bm{Z} \|_*.
\end{align}
Recall we denote $\mathcal{X}(\cdot)$ as the operator such that $\mathcal{X}(\bm{M}) = \langle \bm{X}_i, \bm{M} \rangle/ \sqrt{n}$.  %
Define 
\begin{align}
C^{(n)}_{\lambda}(\bm{W}) := \frac{1}{2}     \| \mathcal{X}(\bm{W}) - \bm{\eps} \|_F^2 + \lambda\big\{ \| \bm{Z}  \|_* - \| \bm{M} \|_* \big\}. \numberthis \label{Cndef}
\end{align}
Define also
\begin{align}
 \psi^{(n)}_{\lambda}( \bm{W} , \bm{\nu}) :=  \frac{1}{dr} \| \bm{W} \|_F \langle \bm{g}, \bm{\nu} \rangle/\sqrt{n} - \frac{1}{dr}\| \bm{\nu}\| \langle \bm{H}, \bm{W}
    \rangle /\sqrt{n}- \frac{1}{dr}\langle \bm{\eps}, \bm{\nu} \rangle - \frac{1}{2} \frac{1}{dr}\| \bm{\nu} \|^2 +\frac{1}{dr} \lambda \big\{ \| \bm{W} + \bm{M}\|_* - \|\bm{M}\|_* \big\}. \numberthis \label{psindef}
\end{align}
We define, for a set $\mathcal{D} \subset \mathbb{R}^{d\times d}$, 
\begin{align}
    L^{(n)}_{\lambda}(\mathcal{D}) := \min_{\bm{W} \in \mathcal{D}: \bm{W} + \bm{M} \succcurlyeq 0 } \max_{\bm{\nu}} \psi^{(n)}_{\lambda}( \bm{W} , \bm{\nu}).
\end{align}
In essence, $L_{\lambda}^{(n)}(\mathcal{D})$ denotes the minimizer of $L_{\lambda}(\cdot)$ over the set $\mathcal{D}$.  By slight abuse of notation, we let $L_{\lambda}^{(n)}(\bm{W}) := L_{\lambda}^{(n)}(\{\bm{W}\})$ (equivalently, the minimizer over the singleton set $\{\bm{W}\}$).  
We observe that 
\begin{align}
dr & \times L_{\lambda}^{(n)} (\mathcal{W}) \\
&= \min_ {\substack{\bm{W} \in \mathcal{W}\\ \bm{W} + \bm{M} \succcurlyeq 0}} \max_{\bm{\nu}} \| \bm{W} \|_F \langle \bm{g}, \bm{\nu} \rangle/\sqrt{n} - \| \bm{\nu}\| \langle \bm{H}, \bm{W}
    \rangle /\sqrt{n}- \langle \bm{\eps}, \bm{\nu} \rangle - \frac{1}{2} \| \bm{\nu} \|^2 + \lambda \big\{  \| \bm{W} + \bm{M}\|_* -  \| \bm{M} \|_* \big\}\\
    &=     \min_ {\substack{\bm{W} \in \mathcal{W}\\ \bm{W} + \bm{M} \succcurlyeq 0}} \max_{\beta > 0} \max_{\|\bm{\nu}\| = \beta}\| \bm{W} \|_F \langle \bm{g}, \bm{\nu} \rangle /\sqrt{n} - \beta  \langle \bm{H}, \bm{W}
    \rangle /\sqrt{n} - \langle \bm{\eps}, \bm{\nu} \rangle - \frac{1}{2} \beta^2 + \lambda \big\{ \| \bm{W} + \bm{M}\|_* -  \| \bm{M} \|_* \big\}\\
    &= \min_ {\substack{\bm{W} \in \mathcal{W}\\ \bm{W} + \bm{M} \succcurlyeq 0}} \max_{\beta > 0} \max_{\|\bm{\nu}\| = \beta} \bigg(  \big\langle\| \bm{W} \|_F \bm{g}/\sqrt{n}- \bm{\eps},  \bm{\nu}/\beta  \big\rangle - \langle \bm{H}, \bm{W}
    \rangle /\sqrt{n}\bigg) \beta  - \frac{1}{2} \beta^2 + \lambda \big\{ \| \bm{W} + \bm{M} \|_*-  \| \bm{M} \|_* \big\}\\
    &= \min_ {\substack{\bm{W} \in \mathcal{W}\\ \bm{W} + \bm{M} \succcurlyeq 0}} \max_{\beta > 0} \bigg( \big\| \| \bm{W} \|_F \bm{g}-  \sqrt{n}\bm{\eps} \big\| -  \langle \bm{H}, \bm{W} \rangle \bigg) \beta
     - \frac{n}{2} \beta^2 + \lambda \big\{ \| \bm{W} + \bm{M} \|_* -\| \bm{M} \|_* \big\} \\
     &= \min_ {\substack{\bm{W} \in \mathcal{W}\\ \bm{W} + \bm{M} \succcurlyeq 0}} \max_{\beta > 0} \bigg( \sqrt{ \| \bm{W}\|_F^2 + n \sigma^2} \| \bm{g} \| -  \langle \bm{H}, \bm{W} \rangle \bigg) \beta
     - \frac{n}{2} \beta^2 + \lambda\big\{ \| \bm{W} + \bm{M} \|_* - \| \bm{M} \|_* \big\}, \label{L_lambda_derivation}
\end{align}
where we slightly abuse notation and let $\bm{g}$ be a new Gaussian random variable with identity covariance.  Dividing through by $dr$ we obtain
\begin{align}
    L_{\lambda}^{(n)}(\mathcal{W}) &= \min_{\substack{\bm{W} \in \mathcal{W}\\ \bm{W} + \bm{M} \succcurlyeq 0}} \max_{\beta} \bigg( \frac{1}{dr} \sqrt{ \| \bm{W} \|_F^2 + n \sigma^2} \| \bm{g} \| - \frac{\langle{\bm{H}, \bm{W}}\rangle}{dr} \bigg) \beta - \frac{\gamma_n \beta^2}{2} + \frac{\lambda}{dr}\big\{ \| \bm{W} + \bm{M} \|_* - \| \bm{M} \|_* \big\} \\
    &= \min_{\substack{\bm{W} \in \mathcal{W}\\ \bm{W} + \bm{M} \succcurlyeq 0}} \max_{\beta} \bigg( \gamma_n \sqrt{ \frac{\| \bm{W} \|_F^2}{n} +  \sigma^2} \frac{\|\bm{g}\|}{\sqrt{n}} - \frac{\langle{\bm{H}, \bm{W}}\rangle}{dr} \bigg) \beta - \frac{\gamma_n \beta^2}{2} + \frac{\lambda}{dr}\big\{ \| \bm{W} + \bm{M} \|_* - \| \bm{M} \|_* \big\} \\
    &= \min_{\substack{\bm{W} \in \mathcal{W}\\ \bm{W} + \bm{M} \succcurlyeq 0}} \frac{\gamma_n}{2}  \bigg(  \sqrt{\frac{\| \bm{W} \|_F^2}{n} + \sigma^2} \frac{\|\bm{g}\|}{\sqrt{n}} - \frac{\langle \bm{H} , \bm{W} \rangle}{n} \bigg)_+^2 + \frac{\lambda}{dr} \big\{ \| \bm{W} + \bm{M} \|_* - \| \bm{M} \|_* \big\}. \label{MCGMT_objective}
\end{align}
The following lemma uses the Matrix CGMT (\cref{thm:matrixcgmt}) to demonstrate that the minimizer of $C_{\lambda}^{(n)}(\bm{W})$ over any set $\mathcal{D}$ is close to $L_{\lambda}^{(n)}(\mathcal{D})$ with comparable probability.
\begin{lemma} \label{lem:applycgmt}
Let $\mathcal{D}$ be any closed subset of $\mathbb{R}^{d\times d}$.   Then for all $t \in \mathbb{R}$  it holds that
    \begin{align}
      \p\bigg\{ \min_{\bm{W} \in \mathcal{D}: \bm{W+M}  \succcurlyeq 0 } \frac{1}{dr} \mathcal{C}^{(n)}_{\lambda}(\bm{W}) 
         \leq t \bigg\} \leq 2 \p\bigg\{ L^{( n)}_{\lambda}(\mathcal{D})  \leq   t \bigg\}.
    \end{align}
    If $\mathcal{D}$ is also convex, then
    \begin{align}
          \p\bigg\{ \min_{\bm{W} \in \mathcal{D}: \bm{W} + \bm{M} \succcurlyeq 0 } \frac{1}{dr} \mathcal{C}^{(n)}_{\lambda}(\bm{W}) 
         > t \bigg\} \leq 2 \p\bigg\{ L^{( n)}_{\lambda}(\mathcal{D})  >  t \bigg\}; 
    \end{align}
\end{lemma}
\begin{proof}
    See \cref{sec:lem:applycgmt}. 
\end{proof}
We will also need the following lemma which relates, for appropriate choice of $\mathcal{D}$, the quantity $L^{(n)}_{\lambda}(\mathcal{D})$ with $L_{\lambda}^{*}$, where we define
\begin{align}
    L_{\lambda}^{*} &= \max_{\beta} \min_{\tau} \psi_{\lambda}^{*}(\tau,\beta) = \psi_{\lambda}^{*}(\tau\s,\beta\s); \\
    \psi^{*}_{\lambda}(\tau,\beta) &:= \bigg( \frac{\sigma^2}{\tau} + \tau \bigg) \frac{\gamma_n \beta}{2} - \gamma_n \frac{\beta^2}{2} + \mathbb{E} \min_{\bm{Z}} \frac{\|\bm{Z} - \bm{M} \|_F^2}{2dr \tau} \beta + \frac{\langle \bm{H}, \bm{Z} - \bm{M} \rangle}{dr} \beta + \frac{\lambda}{dr} \big( \| \bm{Z}\|_* - \|\bm{M} \|_* \big). \label{psistardef}
\end{align}
Here  $\beta\s = \zeta\s \tau\s$, where $(\tau\s,\zeta\s)$ solve the fixed-point equations. 
\begin{lemma} \label{lem:concentrationlemma} Suppose \cref{ass1,ass2,ass3} are satisfied.  Then for any $\eps$ satisfying $\eps^2 \geq C \exp( - c dr)$  with probability at least
$1 - O\bigg( \exp( - c d) +  \exp( - c dr) + \exp( - c n) + \frac{1}{\eps^2} \exp(-c dr \eps^4 ) - \exp( - \frac{(dr)^2}{n} \eps^4) \bigg)$ it holds that
    \begin{align}
        L_{\lambda}^{(n)}\bigg( \| \bm{W} \|_F \leq R \bigg) \leq L_{\lambda}\s + C \eps^2 ;  \label{eq:thinktoprove1} \\
        L_{\lambda}^{(n)}\big( \mathbb{B}_{\eps/2}(\bm{\hat W}/\sqrt{dr})^c \cap \mathrm{rank}(\bm{W} + \bm{M}) \leq r \cap \|\bm{W}\|_F \leq C \sigma \sqrt{dr} \bigg) > L_{\lambda}\s + 2 C \eps^2 \label{eq:thingtoprove}
    \end{align}
\end{lemma}

\begin{proof}
    See \cref{sec:concentrationlemmaproof}.
\end{proof}

\begin{proof}[Proof of \cref{thm:cvxasymptotics}]
Observe that because
\begin{align}
    C^{(n)}(\bm{W}) &= \frac{1}{2} \| \mathcal{X}(\bm{W}) - \bm{\eps} \|_F^2 + \lambda \| \bm{W} + \bm{M} \|_* - \lambda \| \bm{M} \|_*,
\end{align}
we need to argue that any minimizer of $C^{(n)}(\bm{W})$ (over the shifted cone $\bm{M} + \mathbb{R}^{d\times d}_{\succcurlyeq 0})$ must must be within $\eps$ of $\bm{\hat W}$ with high probability.  Let $\mathcal{D}$ be the set \begin{align}
    \mathcal{D} := \big\{ \bm{W} : \bigg| \phi \big( \frac{\bm{M} + \bm{W}}{\sqrt{dr}}\big)- \mathbb{E} \big[ \phi \big(  \frac{\bm{M} + \bm{W}}{\sqrt{dr}}  \big) \big] \bigg| \leq \eps/2 \big\} \cap \{ \bm{W} + \bm{M} \succcurlyeq 0\}.
\end{align}
Then we need to show that the minimizer of $C^{(n)}$ lies within $\mathcal{D}$ with high probability.  Let $\mathcal{D}_{\eps}$ denote the $\eps$-enlargement of the set $\mathcal{D}$; i.e., $\mathcal{D}_{\eps} := \{ \bm{W} :\inf_{\bm{W}'\in \mathcal{D}} \|\bm{W} - \bm{W}'\|_F < \eps \}$.  We will show that any rank at most $r$ minimum of $C^{(n)}$ on $\mathcal{D}_{\eps/2}^c$ (i.e. outside an $\eps/2$-region of $\mathcal{D}$) must be suboptimal with high probability, which implies that the minimizer lies within $\mathcal{D}_{\eps/2}$.  To accomplish this, first note that 
\begin{align}
    \p\bigg\{ \argmin C^{(n)}(\bm{W}) \notin \mathcal{D}_{\eps/2} \bigg\} &\leq \p\bigg\{ \argmin C^{(n)}( \bm{W}) \notin \mathcal{D}_{\eps/2} \cap \mathcal{E}_{{\sf Good}} \bigg\} + \p  \mathcal{E}_{{\sf Good}}^c.
\end{align}
Observe that the first event above satisfies
\begin{align}
    \bigg\{ &\argmin C^{(n)}( \bm{W}) \notin \mathcal{D}_{\eps/2} \cap \mathcal{E}_{{\sf Good}} \bigg\}\\
    &\subset \bigg\{ \argmin C^{(n)}( \bm{W}) \notin \mathcal{D}_{\eps/2} \cap \mathrm{rank}(\bm{W} + \bm{M}) \leq r \cap \| \bm{W} \|_F \leq C \sigma \sqrt{dr} \bigg\} \\
    &\subset \bigg\{ \exists \bm{W} : \bm{W} + \bm{M} \succcurlyeq 0, \mathrm{rank}(\bm{W} + \bm{M}) \leq r, \bm{W} \notin \mathcal{D}_{\eps/2},\|\bm{W}\|_F \leq C \sigma \sqrt{dr}, \\
    &\quad \quad \quad C^{(n)} (\bm{W}) \leq \min_{\bm{W}:\|\bm{W}\|_F \leq C \sigma \sqrt{dr}} C^{(n)}(\bm{W}) \bigg\} \\
    &\subset \bigg\{ \exists \bm{W} : \bm{W} + \bm{M} \succcurlyeq 0, \mathrm{rank}(\bm{W} + \bm{M}) \leq r, \bm{W} \in \mathcal{D}^{c}_{\eps/2}, \|\bm{W}\|_F \leq C \sigma \sqrt{dr}, \\
    &\qquad \qquad C^{(n)}(\bm{W}) \leq \min_{\bm{W}:\|\bm{W}\|_F \leq C \sigma \sqrt{dr}} C^{(n)}(\bm{W}) + C \eps^2 \bigg\} \\
    &= \bigg\{ \min_{\substack{\bm{W} \in \mathcal{D}^c_{\eps/2}\\\mathrm{rank}(\bm{W} + \bm{M}) \leq r\\\|\bm{W}\|_F \leq C \sigma \sqrt{dr}}} C^{(n)}(\bm{W}) \leq \min_{\substack{\bm{W} \in \mathbb{R}^{d\times d}\\\bm{W} + \bm{M} \succcurlyeq 0\\\|\bm{W}\|_F \leq C \sigma \sqrt{dr}}} C^{(n)} (\bm{W}) + C \eps^2 \bigg\}.
\end{align}
As a result,
\begin{align}
    \p\bigg\{ &\argmin C^{(n)}(\bm{W}) \notin \mathcal{D}_{\eps/2} \bigg\} \\
    &\leq \p \bigg\{\min_{\substack{\bm{W} \in \mathcal{D}^c_{\eps/2}\\\mathrm{rank}(\bm{W} + \bm{M}) \leq r \\ \|\bm{W}\|_F \leq C \sigma \sqrt{dr}}} C^{(n)}(\bm{W}) \leq \min_{\substack{\bm{W} \in \mathbb{R}^{d\times d}\\ \bm{W} + \bm{M} \succcurlyeq 0\\\|\bm{W}\|_F \leq C \sigma \sqrt{dr}}} C^{(n)} (\bm{W}) + C \eps^2 \bigg\} \\
    &\leq  \p\bigg\{ \big\{ \min_{\substack{\bm{W} \in \mathcal{D}^c_{\eps/2}\\\mathrm{rank}(\bm{W} + \bm{M}) \leq r\\\|\bm{W}\|_F \leq C \sigma \sqrt{dr}}} C^{(n)}(\bm{W}) \leq \min_{\substack{\bm{W} \in \mathbb{R}^{d\times d }\\\bm{W} + \bm{M} \succcurlyeq 0\\\|\bm{W}\|_F \leq C \sigma \sqrt{dr}}} C^{(n)}(\bm{W}) + C \eps ^2 \} \\
    &\qquad\qquad \qquad\qquad\bigcap \{  \min_{\substack{\bm{W} \in \mathbb{R}^{d\times d}\\ \bm{W} + \bm{M} \succcurlyeq 0\\\|\bm{W}\|_F \leq C \sigma \sqrt{dr}}} C^{(n)}(\bm{W}) \leq L^{*}_{\lambda} + C \eps^2 \big\} \bigg\}   \\
    &\quad + \p\bigg\{   \min_{\bm{W} \in \mathbb{R}^{d\times d}: \bm{W} + \bm{M} \succcurlyeq 0,\|\bm{W}\|_F \leq C \sigma \sqrt{dr}} C^{(n)}(\bm{W}) \geq L^{*}_{\lambda} + C \eps^2 \bigg\} \\
    &\leq  \p\bigg\{  \min_{\bm{W} \in \mathcal{D}^c_{\eps/2} \cap \mathrm{rank}(\bm{W} + \bm{M}) \leq r} C^{(n)}(\bm{W}) \leq   L^{*}_{\lambda} + C \eps^2  \bigg\}   \\
    &\quad + \p\bigg\{   \min_{\bm{W} \in \mathbb{R}^{d\times d}: \bm{W} + \bm{M} \succcurlyeq 0} C^{(n)}(\bm{W}) \geq  L^{*}_{\lambda} + C \eps^2 \bigg\} \\
    &\leq 2 \p\bigg\{ L_{\lambda}^{(n)}(\mathcal{D}_{\eps/2}^c \cap \mathrm{rank}(\bm{W} + \bm{M}) \leq r \cap \|\bm{W}\|_F \leq C \sigma \sqrt{dr}) \leq  L^{*}_{\lambda} + C \eps^2  \bigg\} \\
    &\qquad + 2\p\bigg\{ L_{\lambda}^{(n)}\big(\mathbb{B}(0,C \sigma \sqrt{dr}) \big) \geq  L^{*}_{\lambda} + C \eps^2 \bigg\},
\end{align}
where the final inequality is by \cref{lem:applycgmt}.  Next, note that either $\bm{\hat W} \notin \mathcal{D} $ or $\mathcal{D}_{\eps/2}^c \subset \mathbb{B}_{\eps/2 }\big( \bm{\hat W}\big)^c \cap \{ \bm{W} + \bm{M} \succcurlyeq 0 \}.$ As a result,
\begin{align}
    \p\bigg\{ L_{\lambda}^{(n)}&\big( \mathcal{D}_{\eps/2}^c  \cap \mathrm{rank}(\bm{W} + \bm{M}) \leq r) \cap \|\bm{W}\|_F \leq C \sigma \sqrt{dr}\big) \leq L_{\lambda}^{*} + C \eps^2 \bigg\} \\
    &\leq \p \big\{ \bm{\hat W} \notin \mathcal{D} \big\}  +  \p\bigg\{ L_{\lambda}^{(n)}( \mathbb{B}_{\eps/2 }\big( \bm{\hat W}\big)^c  \cap \mathrm{rank}(\bm{W} + \bm{M}) \leq r) \cap \|\bm{W}\|_F \leq C \sigma \sqrt{dr}\big) \leq L^{*}_{\lambda} + C \eps^2 \bigg\}.
\end{align}
Therefore, 
\begin{align}
    \p\bigg\{& \min_{\bm{W} \in \mathcal{D}_{\eps/2}^c} C^{(n)}(\bm{W}) \leq \min_{\bm{W} \in \mathbb{R}^{d\times d}: \bm{W} - \bm{M} \succcurlyeq 0} C^{(n)}(\bm{W}) + C \eps ^2 \bigg\}  \\
    &\leq 2 \p \big\{ \bm{\hat W} \notin \mathcal{D} \big\} +  2\p\bigg\{ L_{\lambda}^{(n)}( \mathbb{B}_{\eps/2 }\big( \bm{\hat W}\big)^c  \cap \mathrm{rank}(\bm{W} + \bm{M}) \leq r)\big)  \leq L^{*}_{\lambda} + C \eps^2 \bigg\} \\
    &\quad + 2 \p\bigg\{ L_{\lambda}^{(n)} (\mathbb{R}^{d\times d}) \geq L_{\lambda}^{*} + C \eps^2 \bigg\} \\
    &\leq 2 \p \big\{ \bm{\hat W} \notin \mathcal{D} \big\} + O\bigg( \exp( - c d) +  \exp( - c dr) + \exp( - c n) + \frac{1}{\eps^2} \exp(-c dr \eps^4 ) - \exp( - \frac{(dr)^2}{n} \eps^4) \bigg),
\end{align}
where we have applied \cref{lem:concentrationlemma}.   
Note that by construction $\bm{W} + \bm{M} \succcurlyeq 0$ almost surely.  In addition, we have that the soft-thresholding operator is a proximal operator and hence nonexpansive, implying it is 1-Lipschitz.  As a result, the function 
\begin{align}
    \frac{1}{\sqrt{dr}}  \bm{M} - \frac{\tau\s}{\sqrt{dr}} \bm{H} \mapsto \bm{\hat W}
\end{align}
is a $\tau\s/\sqrt{dr}$ Lipschitz function of $\bm{H}$.  Therefore,
\begin{align}
    \p\big\{ \bm{\hat W} \notin \mathcal{D} \big\} &\leq 2 \exp \big\{ - c \frac{d r \eps^2}{\tau\s} \big\}.
\end{align}
By \cref{lem:fxdpointprops},  $\sigma \leq \tau\s \leq C \sigma $, and hence by adjusting constants if necessary, we complete the proof.
\end{proof}

\subsection{Proof of \cref{lem:applycgmt}} \label{sec:lem:applycgmt}

\begin{proof}[Proof of \cref{lem:applycgmt}]
The proof amounts to rewriting the appropriate quantities.  
 We have that
\begin{align}
\min_{\bm{W} \in \mathbb{R}^{d\times d}: \bm{W+M}  \succcurlyeq 0 \cap \mathcal{D}} \mathcal{C}_{\lambda}^{(n)} (\bm{W}) 
&=
  \min_{\bm{W} \in \mathbb{R}^{d\times d}: \bm{W+M}  \succcurlyeq 0 \cap \mathcal{D}} \frac{1}{2}\| \mathcal{X}[\bm{W}] - \bm{\eps}  \|_F^2  + \lambda \| \bm{W} + \bm{M} \|_{*} \\
    &=\min_{\bm{W} \in \mathbb{R}^{d\times d}: \bm{W+M}  \succcurlyeq 0 \cap \mathcal{D}}\max_{\bm{\nu} \in \mathbb{R}^{n}} \langle \mathcal{X}[\bm{W}], \bm{\nu} \rangle - \langle \bm{\eps} , \bm{\nu} \rangle - \frac{1}{2} \| \bm{\nu} \|^2 + \lambda \| \bm{W} + \bm{M}\|_{*}.
    \end{align}
Furthermore, observe that $\mathcal{C}_{\lambda}^{(n)}(\bm{W})$ is convex and satisfies $\mathcal{C}_{\lambda}^{(n)}(\bm{W}) \to \infty$ as $\|\bm{W} \|_F\to \infty$, and hence admits minimizers.  Therefore, there exists a constant $M$ sufficiently large such that the maximizer $\nu$ is bounded with $\|\bm{\nu} \| \leq M$.  As a result, since the min-max is attained on a closed set, we may apply the MCGMT to deduce that
\begin{align}
\p\bigg\{  &   \min_{\bm{W} \in \mathbb{R}^{d\times d}: \bm{W+M}  \succcurlyeq 0 \cap \mathcal{D}} \mathcal{C}_{\lambda}^{(n)} (\bm{W}) > t \bigg\} \\
&\leq 2 \p \bigg\{  \min_{\bm{W} \in \mathbb{R}^{d\times d}: \bm{W+M}  \succcurlyeq 0 \cap \mathcal{D}}\max_{\bm{\nu} \in \mathbb{R}^n} \| \bm{W} \|_F \langle \bm{g}, \bm{\nu} \rangle /\sqrt{n} +\| \bm{\nu}\| \langle \bm{H}, \bm{W}
    \rangle /\sqrt{n}- \langle \bm{\eps}, \bm{\nu} \rangle - \frac{1}{2} \| \bm{\nu} \|^2 + \lambda \| \bm{W} + \bm{M} \|_{*} > t \bigg\},
\end{align}
where the final equality is due to the fact that $-\langle \bm{\eps},\bm{\nu}\rangle - \frac{1}{2} \|\bm{\nu}\|^2$ is concave (and hence the function  $(\bm{W},\bm{\nu}) \mapsto -\langle \bm{\eps},\bm{\nu}\rangle - \frac{1}{2} \|\bm{\nu}\|^2 + \lambda \|\bm{W}\|_{*}$ is convex-concave).  Here $\bm{g}$ is a standard Gaussian random vector and $\bm{H}$ is a GOE matrix.  The first equality follows from the fact that $\frac{1}{2}\|\bm{v} \|_2^2 = \max_{\bm{\nu}} \langle \bm{v},\bm{\nu} \rangle - \frac{1}{2} \|\bm{\nu}\|^2$.  Since the statement is probabilistic in nature, we are free to make the replacement $\bm{H} \mapsto - \bm{H}$ by symmetry.  Diving the resulting expressions inside the probability by $dr$ and recognizing that the right hand side is simply $\psi^{(n)}_{\lambda}(\bm{W},\bm{\nu})$, we complete the proof (after rescaling $t$ by $dr$).  The reverse inequality follows by the same argument.  
\end{proof}

\subsection{Proof of \cref{lem:concentrationlemma}}\label{sec:concentrationlemmaproof}
To prove this result we first require the following lemma.  
\begin{lemma}\label{lem:concentrationlemmapart1}
Suppose \cref{ass1,ass2,ass3} are satisfied.  Let $R = C \sigma \sqrt{dr}$, with $C$ as in \cref{lem:uncvxgood}. 
Then with probability at least 
$1 -   \frac{1}{\eps} \exp( - c n \eps^2) - \exp(-dr \eps^2) $
it holds that
\begin{align}
    L_{\lambda}^{(n)}\bigg( \| \bm{W} \|_F \leq R; \bm{W} + \bm{M} \succcurlyeq 0 \bigg) \geq L_{\lambda}\s - C \sqrt{\gamma_n} \eps.
\end{align}
\end{lemma}
\begin{proof}
    See \cref{sec:concentrationlemmapart1proof}. 
\end{proof}

In order to give the proof we will require some additional notation.  Define the function $h: \mathbb{R}^{d\times r} \to \mathbb{R}$ as follows.  First, set 
\begin{align}
f &: \mathbb{R}^{d\times r} \to \mathbb{R} \\
   \bm{U} &\mapsto  \sqrt{\frac{\| \bm{UU}\t - \bm{M} \|_F^2}{n} + \sigma^2} \frac{\|\bm{g}\|}{\sqrt{n}} -  \frac{\langle \bm{H}, \bm{UU}\t - \bm{M}\rangle}{n}.
\end{align}
Define
\begin{align}
    h(\bm{U}) :&= \frac{\gamma_n}{2} ( f(\bm{U}) )^2_+ + \frac{\lambda}{dr} \big\{ \| \bm{U} \|_F^2 - \| \bm{M} \|_* \big\}.\numberthis \label{hdef}
\end{align}
Recall we denote $\mathcal{O}_{\bm{U},\bm{U}'}$ via
\begin{align}
    \mathcal{O}_{\bm{U},\bm{U}'} &= \argmin_{\mathcal{OO}\t = \bm{I}_r} \| \bm{U} \mathcal{O} - \bm{U}' \|_F.
\end{align}

The following lemma shows that $h(\bm{U})$ is strongly convex on a high probability event.  
\begin{lemma} \label{lem:hstronglyconvex} Suppose \cref{ass1,ass2,ass3} are satisfied.
    Let $h(\bm{U})$ be defined in $\eqref{hdef}$.  Let $\bm{M} = \bm{U}\s \bm{U}^{*\top}$ with $\bm{U}\s \in \mathbb{R}^{d \times r}$.  Define the event
    \begin{align}
       \mathcal{E}_h &:= \bigg\{ \big| \langle \bm{H}, \bm{Q} \rangle \big| \leq C \sqrt{dr} \| \bm{Q} \|_F \text{ for all matrices $\bm{Q}$ of rank at most $2r$} \bigg\} \cap \big\{ \frac{1}{2} \leq  \frac{\|\bm{g}\|}{\sqrt{n}} \leq 2 \big\} \cap \bigg\{ \| \bm{H} \| \leq 3 \sqrt{d} \bigg\}.
    \end{align}
    Then for all $\bm{U}$ and $\bm{V} = \bm{U} - \bm{U}' \mathcal{O}_{\bm{U}'\bm{U}}$ satisfying
    \begin{align}
        \max\bigg\{ \| \bm{U} - \bm{U}\s \mathcal{O}_{\bm{U}\s,\bm{U}} \|, \| \bm{U}' - \bm{U}\s \mathcal{O}_{\bm{U}\s,\bm{U}'} \| \bigg\} \leq \frac{c}{\kappa(\bm{M})} \| \bm{U}\s \|, \numberthis \label{region}
    \end{align}
    for some sufficiently small constant $c$, 
    it holds that
    \begin{align}
        \nabla^2 h(\bm{U})[ \bm{V}, \bm{V}] \geq  \bigg(  \frac{ \lambda_{\min}\big( \bm{M} \big)}{32 dr}  + 2 \frac{\lambda}{dr}\bigg) \|\bm{V}\|_F^2.
    \end{align}
\end{lemma}
\begin{proof}
    See \cref{sec:hstronglyconvexproof}.  
\end{proof}
In addition, the following additional lemma shows that $\mathcal{E}_h$ above holds with high probability.
\begin{lemma} \label{lem:Eh_highprob}
    The event $\mathcal{E}_h$ holds with probability at least $1 - 2 \exp( - c dr) - 2 \exp( - c n)$.
\end{lemma}
\begin{proof}
    See \cref{sec:Eh_highprob_proof}. 
\end{proof}
\noindent 
Finally, the following lemma shows that when $\bm{U} \in \mathcal{R}$ satisfies $h(\bm{U}) \leq \min_{\bm{U} \in \mathcal{R}} h(\bm{U}) + \eps$, it holds that $\bm{\hat U}$ must be close to $\bm{U}$.  
\begin{lemma} \label{lem:miolaneb1analogue}
Let $h: \mathbb{R}^{d \times r} \to \mathbb{R}$ be a function such that $h(\bm{U} \mathcal{O}) = h(\bm{U})$ for all orthogonal matrices $\mathcal{O}$. 
Let $\mathbb{B}(\bm{\hat U},K)$ denote the set of matrices $\bm{U}$ such that $\|\bm{U} \mathcal{O}_{\bm{U},\bm{\hat U}} - \bm{\hat U} \|_F \leq K$. 
Suppose that for all $\bm{U} \in \mathbb{B}(\bm{\hat U},r)$ it holds that
    \begin{align}
        \nabla^2 h(\bm{U})[ \bm{V},\bm{V}] \geq \gamma \| \bm{V} \|_F^2,
    \end{align}
    where $\bm{V} = \bm{U}_1 - \bm{U} \mathcal{O}_{\bm{U},\bm{U}_1}$.  Suppose further that
    \begin{align}
        h(\bm{\hat U}) \leq \min_{\mathbb{B}(\bm{\hat U},K)} h(\bm{U}) + \eps.
    \end{align}
    Then if $\bm{U}_0$ is any minimizer of $h$ over this set, it  holds that
    \begin{align}
        \| \bm{\hat U} - \bm{U}_0 \mathcal{O}_{\bm{U}_0, \bm{\hat U}} \|_F^2 \leq \frac{2}{\gamma} \eps.  
    \end{align}
    Furthermore, it must be that for all $\bm{\tilde U} \in \mathbb{B}(\bm{\hat U},K)$, $h(\bm{\tilde U}) \leq \min_{\mathbb{B}(\bm{\hat U},K)} h(\bm{U}) + \eps$ and $\| \bm{\tilde U} \mathcal{O}_{\bm{\tilde U},\bm{\hat U}} - \bm{\hat U} \|_F^2 \leq \frac{8}{\gamma} \eps $. 
\end{lemma}
\begin{proof}
    See \cref{sec:miolaneb1analogueproof}.
\end{proof}

We also need the following result.
\begin{lemma} \label{lem:h_yhat_bounded} Suppose \cref{ass1,ass2,ass3} are satisfied.
With probability at least $1 - O\bigg( \exp\big( - C dr \big( C \exp(- c dr) + \eps \big)^2 \big) + \exp( - c dr) \bigg)$ it holds that $L_{\lambda}^{(n)} \big( \bm{Z}_{{\sf ST}}^{(\lambda/\zeta\s)}(\tau\s) - \bm{M} \big) \leq L_{\lambda}\s + \eps + C \exp( - c dr)$. 
\end{lemma}
\begin{proof}
    See \cref{sec:h_yhat_bounded_proof}.
\end{proof}

\begin{proof}[Proof of \cref{lem:concentrationlemma}]
Recall that we aim to show that:
\begin{align}
    L_{\lambda}^{(n)} \bigg( \| \bm{W} \|_F \leq R \cap PSD \bigg) \leq L_{\lambda}^* + C \eps^2; \\
    L_{\lambda}^{(n)} \bigg( \mathbb{B}_{\eps/2} \bigg( \bm{\hat W}/\sqrt{dr} \bigg)^c \cap \mathrm{rank}(\bm{W+ M} ) \leq r \cap\|\bm{W} \|_F \leq R \bigg) > L_{\lambda}^* + C \eps^2
    \end{align}
    with high probability. 
    We split up the proof into steps.
\\
 \ \\ \noindent \textbf{Step 1: Good events.} First, let $\mathcal{E}_{\text{\cref{lem:Zstgoodproperties}}}$ denote the event from \cref{lem:Zstgoodproperties}. On this event, it holds that $\mathrm{rank}(\bm{ Z}_{{\sf ST}}^{(\lambda/\zeta\s)}(\tau\s)) \leq r$, whence the quantity $\bm{\hat U}$ is well-defined on this event.  When $\bm{\hat U}$ is well-defined,  define the event
\begin{align}
\mathcal{E}_1 : =     \bigg\{ h(\bm{\hat U}) \leq L_{\lambda}\s + C \sqrt{\gamma_n} \tilde \eps \bigg\},
\end{align}
where $\tilde \eps$ will be chosen momentarily.  We also define the event 
\begin{align}
\mathcal{E}_2 :=  \bigg\{    L_{\lambda}^{(n)}\big( \| \bm{W} \|_F \leq R \cap \bm{W} + \bm{M} \succcurlyeq 0 \cap \mathrm{Rank}(\bm{W} + \bm{M}) \leq r \big) \geq L_{\lambda}\s - C \tilde \eps \sqrt{\gamma_n} \bigg\},
\end{align}
where $C$ is defined as in \cref{lem:concentrationlemmapart1}. In fact, by this same lemma, we have on the event therein and $\mathcal{E}_{\text{\cref{lem:Zstgoodproperties}}}$ it holds that
\begin{align}
     L_{\lambda}^{(n)}\big( \| \bm{W} \|_F \leq R \cap \bm{W} + \bm{M} \succcurlyeq 0 \cap \mathrm{Rank}(\bm{W} + \bm{M}) \leq r \big) &\geq   L_{\lambda}^{(n)}\big( \| \bm{W} \|_F \leq R \cap \bm{W} + \bm{M} \succcurlyeq 0  \big) \\
     &\geq L_{\lambda}\s - C \tilde \eps\sqrt{\gamma_n}.
\end{align}
\\
 \ \noindent
 \textbf{Step 2: The bound \eqref{eq:thinktoprove1}.} We have that
on these events
    \begin{align}
    L_{\lambda}\s + C \sqrt{\gamma_n} \tilde \eps \geq h(\bm{\hat U}) &\geq L_{\lambda}^{(n)}\big( \|\bm{W}\|_F \leq R \cap PSD \cap \mathrm{rank}(\bm{W} + \bm{M}) \leq r \big) \\
    &\geq L_{\lambda}^{(n)}\big( \|\bm{W}\|_F \leq R \cap PSD \big), 
    \end{align}
    which proves \eqref{eq:thinktoprove1} whenever $ C\tilde \eps \sqrt{\gamma_n} = C'  \eps^2$, which we will show is valid at the end of Step 3 below.  
 \\ \ \\
 \noindent
 \textbf{Step 3: The bound \eqref{eq:thingtoprove}.}   The bound \eqref{eq:thingtoprove} will require a little more analysis.   
    Define the regions
\begin{align}
    \mathcal{R}_1 := \{ \bm{U} : \| \bm{U} - \bm{\hat U} \mathcal{O}_{\bm{\hat U},\bm{U}}\|_F \leq \frac{1}{8 \| \bm{U}\s \|} R \}; \\
    \mathcal{R}_2 := \{ \bm{U} : \| \bm{UU}\t - \bm{\hat U} \bm{\hat U}\t \|_F \leq R \}.
\end{align}
We claim that have the containment $\mathcal{R}_1 \subset \mathcal{R}_2$ on the event $\mathcal{E}_{\text{\cref{lem:Zstgoodproperties}}}$.  
In order to see this, observe that for any $\bm{U} \in \mathcal{R}_1$, 
\begin{align}
    \| \bm{UU}\t - \bm{\hat U} \bm{\hat U} \|_F \leq \big( \| \bm{U} \| + \| \bm{\hat U} \| \big)  \frac{1}{8 \| \bm{U}\s \|} R.
\end{align}
On the event $\mathcal{E}_{\text{\cref{lem:Zstgoodproperties}}}$ it holds that $\|\bm{Z}_{{\sf ST}}^{(\lambda)} \| \leq 2 \| \bm{M}\|$.  Consequently, on this same event, we have that $\|\bm{\hat U}\| \leq 2 \|\bm{U}\s \|$, since $\|\bm{\hat U}\|^2 = \| \bm{Z}_{{\sf ST}}^{(\lambda/\zeta\s)}(\tau\s)\|$.  
Furthermore, we have that by virtue of $\bm{U} \in \mathcal{R}_1$,
\begin{align}
    \| \bm{U} \| &\leq \| \bm{U} \mathcal{O}_{\bm{U},\bm{\hat U}} - \bm{\hat U} \| + \| \bm{\hat U} \| \leq  \frac{1}{8 \|\bm{U}\s \|} R + 2 \|\bm{U}\s \| \\
    &\leq  \frac{\lambda_{\min}(\bm{M})}{8 \sqrt{\lambda_{\max}(\bm{M})}} + 2 \|\bm{U}\s \|\leq 4 \| \bm{U}\s \|.
\end{align}
Consequently,
\begin{align}
     \| \bm{UU}\t - \bm{\hat U} \bm{\hat U} \|_F \leq \big( \| \bm{U} \| + \| \bm{\hat U} \| \big)  \frac{1}{8 \| \bm{U}\s \|} R \leq R.
\end{align}
Therefore, $\mathcal{R}_1 \subset \mathcal{R}_2$. Note that we have used the implicit assumption that $\lambda_r \geq C_1 \sigma \sqrt{r}$ for some sufficiently large constant $C_1$.  

We will now show that $h$ is strongly convex on the region $\mathcal{R}_2$ for a particular parameter $\alpha$.  By \cref{lem:hstronglyconvex} on the event $\mathcal{E}_h$, $h$ is strongly convex with the parameter 
\begin{align}
    \alpha =  \frac{\lambda_{\min}(\bm{M})}{32 dr} + 2 \frac{\lambda}{dr},
\end{align}
provided $\bm{U}$ lies in the region $\mathcal{R}$ defined via
\begin{align}
    \mathcal{R} := \bigg\{ \bm{U}: \| \bm{U} - \bm{U}\s \mathcal{O}_{\bm{U}\s, \bm{U}} \| \leq \frac{c}{\kappa(\bm{M})} \|\bm{U}\s \|.\bigg\} 
\end{align}
We will show that $\mathcal{R}_2 \subset \mathcal{R}$. 
It holds that $\|\bm{Z}_{{\sf ST}}^{(\lambda/\zeta\s)}(\tau\s) - \bm{M} \|_F \leq C \sigma \sqrt{dr}$ on the event $\mathcal{E}_{\text{\cref{lem:Zstgoodproperties}}}$ by appropriately choosing $s = O(1)$ in the statement therein.  (See, for example, the proof of \cref{lem:h_yhat_bounded}.)  We then have that 
\begin{align}
    \| \bm{UU}\t - \bm{U}\s \bm{U}^{*\top} \|_F &\leq R + \|  \bm{U}\s \bm{U}^{*\top} -\bm{\hat U \hat U}\t \|_F = R + \| \bm{Z}_{{\sf ST}}^{(\lambda)} - \bm{M} \|_F \\
    &\leq C' \sigma \sqrt{dr}.
\end{align} 
As a consequence, we have that
\begin{align}
    \| \bm{U} - \bm{U}\s \mathcal{O}_{\bm{U}\s,\bm{U}} \| \leq  \frac{c_0 C'}{\sqrt{\lambda_{\min}}} \sigma \sqrt{dr} 
    &\leq \frac{c}{\kappa(\bm{M})} \| \bm{U}\s \|,
\end{align}
which holds as long as  $\lambda_{\min}(\bm{M})/\sigma \geq C_1 \kappa \sqrt{dr}$, for some appropriately large constant $C_1$, which holds according to \cref{ass1}.

Consequently, on the event $\mathcal{E}_1 \cap \mathcal{E}_2 \cap \mathcal{E}_{\text{\cref{lem:Zstgoodproperties}}} \cap \mathcal{E}_h$, 
\begin{align}
  h(\bm{\hat U}) &=  L_{\lambda}^{(n)}( \bm{\hat W}) \leq  L_{\lambda}^{(n)}\big( \| \bm{W} \|_F \leq R \cap \bm{W} + \bm{M} \succcurlyeq 0 \cap \mathrm{rank}(\bm{W} + \bm{M}) \leq r\big) + 2 C \tilde\eps \sqrt{\gamma_n} \\
  &= \min_{\mathcal{R}_2} h(\bm{U}) + 2C \tilde\eps \sqrt{\gamma_n} \\
  &\leq \min_{\mathcal{R}_1} h(\bm{U}) + 2C \tilde\eps \sqrt{\gamma_n} \\
  &= \min_{\bm{U}: \| \bm{U} - \bm{U}\s \mathcal{O}_{\bm{U}\s,\bm{U}} \|_F \leq \frac{1}{8 \| \bm{U}\s \|} R} h(\bm{U}) + 2 C \tilde\eps \sqrt{\gamma_n}, \label{contradictionbb}
\end{align}
were the final bound holds from the containment $\mathcal{R}_1 \subset \mathcal{R}_2$.   By \cref{lem:miolaneb1analogue}, since $h$ is strongly convex on the region $\mathcal{R}_2$ with parameter $\alpha$, for any $\bm{U} \in \mathbb{B}(\bm{\hat U}, \frac{R}{8\|\bm{U}\s\|} )$, it holds that $\| \bm{U}  \mathcal{O}_{\bm{U},\bm{\hat U}}- \bm{\hat U}\|_F^2 \leq \frac{8 C \sqrt{\gamma_n} \tilde\eps}{\alpha}$. 
In addition, for any $\bm{U} \in \mathbb{B}\big( \bm{\hat U}, \frac{R}{8\|\bm{U}\s\|} \big)$, we have that
\begin{align}
    \| \bm{UU}\t - \bm{\hat U} \bm{\hat U}\t \|_F \leq \big( \|\bm{U}\| + \| \bm{\hat U} \| \big) \| \bm{U} - \bm{\hat U} \mathcal{O}_{\bm{\hat U},\bm{U}} \|_F 
\leq  4 \| \bm{U}\s \| \| \bm{U} - \bm{\hat U} \mathcal{O}_{\bm{\hat U},\bm{U}} \|_F.
\end{align}
Consequently, if
\begin{align}
    \frac{\| \bm{UU}\t - \bm{\hat U \hat U}\t \|_F}{\sqrt{dr}} > \frac{\eps}{2},
\end{align}
then it holds that 
\begin{align}
    \| \bm{U} - \bm{\hat U} \mathcal{O}_{\bm{\hat U},\bm{U}} \|_F \geq 
       \frac{1}{4\| \bm{U}\s \|} \| \bm{UU}\t - \bm{\hat U\hat U}\t \|_F >  \frac{\eps \sqrt{dr}}{ 4\| \bm{U}\s \|} .%
\end{align}
The contrapositive of \cref{lem:miolaneb1analogue} implies that for $\bm{U}$ as above, it must hold that
\begin{align}
    h(\bm{U}) > \min_{\bm{U} : \| \bm{U} - \bm{U}\s \mathcal{O}_{\bm{U}\s,\bm{U}} \|_F \leq \frac{R}{8 \| \bm{U}\s \|}}h(\bm{U}) + \bigg(\frac{\eps \sqrt{dr}}{ 4\| \bm{U}\s \|}\bigg)^2 \frac{\alpha}{8} \label{hguy}
\end{align}
Define $\tilde \eps$ via
\begin{align}
     \tilde \eps = \frac{\eps^2 dr \alpha }{128 C \sqrt{\gamma}_n \| \bm{U}\s\|^2}.
\end{align}
Then the display equation \eqref{hguy} simplifies to
\begin{align}
    h(\bm{U}) > \min_{\bm{U} : \|\bm{U - U\s} \mathcal{O}_{\bm{U}\s,\bm{U}}\|_F \leq \frac{R}{8\|\bm{U}\s\|} } h(\bm{U}) + C \sqrt{\gamma_n} \tilde \eps,
\end{align}
which is a contradiction to \eqref{contradictionbb}.  Recalling the definition of $\alpha$ from earlier, we have that 
\begin{align}
    C \sqrt{\gamma_n} \tilde \eps &= \frac{\eps^2dr}{128 \| \bm{U}\s \|^2} \bigg( \frac{\lambda_{\min}(\bm{M})}{32 dr} + 2 \frac{\lambda}{dr} \bigg) \\
    &= \frac{\eps^2}{128} \bigg( \frac{1}{32 \kappa} + 2 \frac{\lambda}{\|\bm{M}\|} \bigg) = C \eps^2.  
    \end{align}
which justifies the previous choice of $\tilde \eps$ from the previous step.  
In particular, on these events it holds that
\begin{align}
    L_{\lambda}^{(n)} \bigg( \mathbb{B}_{\eps/2}^c( \bm{\hat W}/\sqrt{dr}) \cap \mathrm{rank}(W + M) \leq r \cap \| \bm{W} \|_F \leq R \bigg) > L_{\lambda}^* + 2 \eps^2 
    \label{eq:thingtoproveproven}
    \end{align}
which is precisely  \eqref{eq:thingtoprove}.  This algebra also verifies the choice of $\tilde \eps$ from the previous step of the proof.
 \\ \ \\ \noindent
 \textbf{Step 4: Tabulating probabilities}. 
We bound each event in turn. 
\begin{itemize}
\item \textbf{The events $\mathcal{E}_{\cref{lem:Zstgoodproperties}}$ and $\mathcal{E}_h$}: The first event holds with probability at least $1 - C\exp(- c d)$.  Furthermore, by \cref{lem:Eh_highprob}, 
$\mathcal{E}_h$ holds with probability at least $1 - 2 \exp( -cdr) - 2 \exp(-c n)$.  
    \item \textbf{The event $\mathcal{E}_1$}: From \cref{lem:h_yhat_bounded} we have that 
\begin{align}
      L_{\lambda}^{(n)}\big( \bm{Z}_{{\sf ST}}^{(\lambda/\zeta\s)}(\tau\s) - \bm{M} \big) \leq L_{\lambda}\s + s+ C \exp(- c dr)
\end{align}
with probability at least $$1 - O\bigg( \exp\big( - C dr \big( C \exp(- c dr) + s \big)^2 \big) + \exp( - c dr) \bigg)$$ Let $\eps^2 = s + C \exp(- c dr)$.  Then we have that
\begin{align}
     L_{\lambda}^{(n)}\big( \bm{Z}_{{\sf ST}}^{(\lambda/\zeta\s)}(\tau\s) - \bm{M} \big) \leq L_{\lambda}\s + \eps^2
\end{align}
with probability at least $1 - O( \exp( - c dr \eps^2 ) + \exp(- c dr))$ as long as $\eps^2 \geq C \exp(- c dr)$, as in the statement of the proof.
    \item \textbf{The event $\mathcal{E}_2$}: From \cref{lem:concentrationlemmapart1}, it holds that 
    \begin{align}
            L_{\lambda}^{(n)} \bigg( \|\bm{W}\|\leq R; \bm{W + M} \succcurlyeq 0 \bigg) &\geq L_{\lambda}\s - C t \sqrt{\gamma_n} 
    \end{align}
    with probability at least $1 - \frac{1}{t} \exp( - c n t^2) - \exp( - d r t^2)$.

    \begin{align}
        L_{\lambda}^{(n)} \bigg( \|\bm{W}\|\leq R; \bm{W + M} \succcurlyeq 0 \bigg) \geq L_{\lambda}\s - C \eps^2
    \end{align}
    with probability at least $1 -  O\big( \frac{1}{\eps^2} \exp(-c dr \eps^4 ) - \exp( - \frac{(dr)^2}{n} \eps^4)$.  Since 
    \begin{align}
        L_{\lambda}^{(n)} \bigg( \|\bm{W}\|_F \leq R; \bm{W} + \bm{M} \succcurlyeq 0; \mathrm{rank}(\bm{W} + \bm{M}) \leq r \bigg) \geq  L_{\lambda}^{(n)} \bigg( \|\bm{W}\|\leq R; \bm{W + M} \succcurlyeq 0 \bigg)
    \end{align}
    always holds, we see that $\mathcal{E}_2$ holds with this same probability.
\end{itemize}
Combining these bounds we have that the overall probability is at least
\begin{align}
    1 - O\bigg( \exp( - c d) +  \exp( - c dr) + \exp( - c n) + \frac{1}{\eps^2} \exp(-c dr \eps^4 ) - \exp( - \frac{(dr)^2}{n} \eps^4) \bigg),
\end{align}
which completes the proof. 
\end{proof}

\subsubsection{Proof of \cref{lem:concentrationlemmapart1}}
\label{sec:concentrationlemmapart1proof}
\begin{proof}[Proof of \cref{lem:concentrationlemmapart1}]
The proof of this result resembles the proof of Lemma B.3 in \citet{celentano_lasso_2023}. Throughout we assume implicitly that $\bm{W} \succcurlyeq -\bm{M}$.  
Throughout the proof, denote $R := C \sigma \sqrt{dr}$. 
We observe that we can rewrite $L_{\lambda}(\mathcal{W})$ (see the derivation in \eqref{L_lambda_derivation}) via
\begin{align}
    L_{\lambda}(\bm{W}) &:= \max_{\beta > 0}  \bigg(\gamma_n \sqrt{\frac{\| \bm{W} \|_F^2 }{n} + \sigma^2} \frac{\|\bm{g}\|}{\sqrt{n}} - \frac{\langle \bm{H},\bm{W}\rangle}{dr} \bigg) \beta - \gamma_n \frac{\beta^2}{2} + \frac{\lambda}{dr} \big\{ \|\bm{W} + \bm{M} \|_* - \| \bm{M}\|_* \big\} \\
    &=: \max_{\beta \geq 0} \ell_{\lambda}(\bm{W},\beta; \bm{g}) \\
    &\geq \ell_{\lambda}(\bm{W}, \beta\s; \bm{g}).
\end{align}
Define the event
\begin{align}
    \mathcal{E}_{{\bm{g}}} := \bigg\{ \bigg| \frac{\|\bm{g}\|}{\sqrt{n}} - 1 \bigg| \leq t /\sqrt{\gamma_n} \bigg\}.
\end{align}
Since $\bm{g} \mapsto \frac{\|\bm{g}\|}{\sqrt{n}}$ is a 1-Lipschitz function, we have that $\p\big(\mathcal{E}_{{\bm{g}}}\big) \geq 1 - \exp( -c dr t^2)$ for all $t > 0$.  
Denote the modified objective
\begin{align}
    \tilde\ell_{\lambda}(\bm{W},\beta) := \bigg(\gamma_n \sqrt{\frac{\| \bm{W} \|_F^2 }{n} + \sigma^2} - \frac{\langle \bm{H},\bm{W}\rangle}{dr} \bigg) \beta - \gamma_n \frac{\beta^2}{2} + \frac{\lambda}{dr} \big\{ \|\bm{W} + \bm{M} \|_* - \| \bm{M}\|_* \big\}.
\end{align}
If  $\|\bm{W}\|_F \leq R$, it holds that on  the event $\mathcal{E}_{\bm{g}}$,
\begin{align}
    \big| \ell_{\lambda}(\bm{W},\beta\s)- \tilde \ell_{\lambda}(\bm{W},\beta) \big| \leq \beta_{\max} \gamma_n^{1/2} t   \sqrt{\frac{R^2}{n} + \sigma^2} \leq 2C \gamma_n^{1/2} t \sigma.
\end{align}
Consequently, on this event
\begin{align}
    \min_{\|\bm{W}\|_F \leq R }L_{\lambda}(\bm{W}) &= \min_{\|\bm{W}\|_F \leq R} \max_{\beta \geq 0} \ell(\bm{W},\beta) \geq \min_{\| \bm{W}\|_F \leq R} \ell(\bm{W}, \beta\s) \\
    &\geq \min_{\|\bm{W}\|_F \leq R }\tilde \ell(\bm{W},\beta\s) - 2\beta_{\max} \gamma_n^{1/2} \eps \sigma. 
\end{align}
Next, for every $\tau \in [\sigma ,\sqrt{\sigma^2 + \frac{R^2}{n}}]$, the function
\begin{align}
    \bm{H} \mapsto \min_{\|\bm{W}\|_F \leq R} \bigg\{\frac{\gamma_n \beta\s\|\bm{W}\|_F^2}{2\tau n} - \frac{\gamma_n \beta\s}{n} \langle \bm{H},\bm{W} \rangle + \frac{\lambda}{dr} \{ \|\bm{W} + \bm{M} \|_* - \| \bm{M} \|_* \} \bigg\}
\end{align}
is $\frac{\gamma_n \beta\s}{n} R$-Lipschitz, and hence so is $\bm{H} \mapsto F(\tau,\bm{H})$.  Consequently, 
\begin{align}
    \p\bigg\{ \bigg| F(\tau,\bm{H}) - \mathbb{E} F(\tau,\bm{H}) \bigg| > \eps \frac{R \gamma_n}{\sqrt{n}} \bigg\} &\leq 2 \exp\bigg( - c \frac{n}{( \beta\s)^2} \eps^2 \bigg).
\end{align}
Furthermore, since $\tau \mapsto F(\tau,\bm{H})$ is $\big(\beta\s \gamma_n \big[ \frac{R^2}{n \sigma} + \frac{1}{2} + \frac{\sqrt{\sigma^2 + \frac{R^2}{n}}}{2}\big] \big) \leq C \beta\s \gamma_n (\sigma + 1)$-Lipschitz, we can conclude via an $\eps$-net argument that there exists constants depending on $\beta\s,\sigma$ such that
\begin{align}
    \p\bigg\{ \sup_{\tau \in [\sigma, \sqrt{\sigma^2 + \frac{R^2}{n}}]} \bigg| F(\tau,\bm{H}) - \mathbb{E}F(\tau,\bm{H}) \bigg| \leq  \eps \frac{R \gamma_n}{\sqrt{n}} \bigg\} &\geq 1 - \frac{C \gamma_n^{3/2}}{\eps } \exp\bigg( - c n\eps^2 \bigg) \\
    &\geq 1 - \frac{C}{\eps } \exp\bigg( - c  n \eps^2 \bigg),
\end{align}
since $\gamma_n \geq 1$.  Assuming this event holds we have that
\begin{align}
    \min_{\|\bm{W}\|_F \leq R} \tilde \ell(\bm{W},\beta\s) &\geq \min_{\tau\in [ \sigma,\sqrt{\sigma^2 + R^2}]} \mathbb{E} F(\tau,\bm{H}) - \eps \frac{R\gamma_n}{\sqrt{n}}.
\end{align}
Recalling that $R = C \sigma \sqrt{dr}$, it holds that $\eps \frac{R \gamma_n}{\sqrt{n}} =C \eps \sigma \sqrt{\gamma_n}$.  Furthermore, observe that 
\begin{align}
    \mathbb{E} F(\tau,\bm{H}) &\geq  \min_{\tau \in [\sigma ,\sqrt{\sigma^2 + R^2}]}  \bigg\{ \bigg( \frac{\sigma^2}{2\tau} + \frac{\tau}{2} \bigg) \beta\s \gamma_n - \frac{\gamma_n(\beta\s)^2}{2}\\
    &\quad +  \mathbb{E} \min_{\bm{W}} \bigg\{ \frac{\gamma_n\beta\s}{2\tau n} \| \bm{W}\|_F^2 - \frac{\gamma_n\beta\s}{n} \langle \bm{H,W} \rangle + \frac{\lambda}{dr} \big\{ \| \bm{W} + \bm{M}\|_* - \bm{M}\|_*\big\} \bigg\} \bigg\} \\
    &=\min_{\tau \in [\sigma ,\sqrt{\sigma^2 + R^2}]} \psi_{\lambda}\s(\tau,\beta\s) \\
    &\geq  \psi_{\lambda}\s(\tau\s,\beta\s) = L_\lambda^*
\end{align}
where $\psi_{\lambda}\s(\tau,\beta\s)$ is defined in \eqref{psistardef}.  Combining these arguments, on these events we have that
\begin{align}
    L_{\lambda}^{(n)} \geq L_{\lambda}\s - C \eps \sqrt{\gamma_n} 
\end{align}
which holds cumulatively with probability at least
\begin{align}
    1 -     \frac{1}{\eps} \exp( - c n \eps^2) - \exp(-dr \eps^2 ).
\end{align}
\end{proof}

\subsubsection{Proof of \cref{lem:hstronglyconvex}} \label{sec:hstronglyconvexproof}
For convenience we recall that $h(\bm{U})$ is defined via
\begin{align}
    h(\bm{U}) := \frac{\gamma_n}{2} \bigg( \frac{\|\bm{g}\|}{\sqrt{n}} \sqrt{\frac{\|\bm{UU}\t - \bm{M} \|_F^2}{n} + \sigma^2} - \frac{\langle \bm{H},\bm{UU}\t - \bm{M} \rangle}{n} \bigg)_+^2 + \frac{\lambda}{dr} \big\{ \| \bm{U} \|_F^2 - \|\bm{M}\|_* \big\}.
\end{align}
We will define
\begin{align}
    \tau(\bm{U}) :&= \sqrt{\frac{\|\bm{UU}\t - \bm{M}\|_F^2}{n} + \sigma^2},
\end{align}
which implies that
\begin{align}
    f(\bm{U}) &= \sqrt{\frac{\|\bm{UU}\t - \bm{M}\|_F^2}{n} + \sigma^2} \frac{\|\bm{g}\|}{\sqrt{n}} - \frac{\langle \bm{H}, \bm{UU}\t - \bm{M} \rangle}{n} \equiv \tau(\bm{U}) \frac{\|\bm{g}\|}{\sqrt{n}} - \frac{\langle \bm{H}, \bm{UU}\t - \bm{M} \rangle}{n}; \\
    h(\bm{U}) &= \frac{\gamma_n}{2} ( f(\bm{U}))_+^2 + \frac{\lambda}{dr} \big\{ \|\bm{U}\|_F^2 - \|\bm{M}\|_* \big\}.
\end{align}
To prove \cref{lem:hstronglyconvex} we present the following lemma computing the second derivative of $h$.
\begin{lemma} \label{lem:hessiancalculation}
It holds that
\begin{align}
    \nabla^2 h(\bm{U})[ \bm{V}, \bm{V}] &=  \frac{4 \gamma_n}{n^2 \tau^2(\bm{U})} \frac{\|\bm{g}\|}{\sqrt{n}} \bigg( \frac{\|\bm{g}\|}{\sqrt{n}} - \frac{f(\bm{U})_+}{\tau(\bm{U})} \bigg) \bigg \langle \big( \bm{UU}\t - \bm{M}\big) \bm{U}, \bm{V} \bigg \rangle^2 \\
    &\quad + \frac{4\gamma_n}{n^2} \bigg \langle \bm{HU}, \bm{V} \bigg\rangle^2 \\
    &\quad + \frac{4 \gamma_n}{n^2 \tau(\bm{U})} \frac{\|\bm{g}\|}{\sqrt{n}} \bigg( \frac{f(\bm{U})_+}{\tau^2(\bm{U})} - 2 \bigg) \bigg \langle \bm{HU}, \bm{V} \bigg \rangle \bigg \langle \big( \bm{UU}\t - \bm{M} \big) \bm{U}, \bm{V} \bigg \rangle \\
    &\quad +\gamma_n f(\bm{U})_+ 2\tau(\bm{U})\inv \frac{\|\bm{g}\|}{\sqrt{n}}  \bigg\langle \bigg(  \frac{\big(  \bm{VU}\t +  \bm{UV}\t \big) }{n}  \bigg)\bm{U}, \bm{V} \bigg \rangle \\
    &\quad +\gamma_n f(\bm{U})_+ 2\tau(\bm{U})\inv \frac{\|\bm{g}\|}{\sqrt{n}}  \bigg\langle \bigg(  \frac{\bm{UU}\t - \bm{M}- \bm{H}}{n}  \bigg) \bm{V} , \bm{V} \bigg \rangle \\
    &\quad + 2 \frac{\lambda}{dr} \| \bm{V} \|_F^2.
\end{align}
\end{lemma}
\begin{proof}
    See \cref{sec:hessiancalculationproof}.
\end{proof}

With this calculation in place we are free to prove \cref{lem:hstronglyconvex}.

\begin{proof}[Proof of \cref{lem:hstronglyconvex}]
Recall we define the event
\begin{align}
    \mathcal{E}_h &:= \bigg\{ \big| \langle \bm{H}, \bm{Q} \rangle \big| \leq C \sqrt{dr} \| \bm{Q} \|_F \text{ for all matrices $\bm{Q}$ of rank at most $2r$} \bigg\} \cap \big\{ \frac{1}{2} \leq  \frac{\|\bm{g}\|}{\sqrt{n}} \leq 2 \big\} \cap \bigg\{ \| \bm{H} \| \leq 3 \sqrt{d} \bigg\}.
\end{align}
We decompose the hessian by \cref{lem:hessiancalculation} via
\begin{align}
    \nabla^2 h(\bm{U})[\bm{V}, \bm{V}] := T_1 + T_2 + T_3 + T_4 + T_5 + T_6
\end{align}
where
\begin{align}
 \\   T_1 &:= \frac{4 \gamma_n}{n^2 \tau^2(\bm{U})} \frac{\|\bm{g}\|}{\sqrt{n}} \bigg( \frac{\|\bm{g}\|}{\sqrt{n}} - \frac{f(\bm{U})_+}{\tau(\bm{U})} \bigg) \bigg \langle \big( \bm{UU}\t - \bm{M}\big) \bm{U}, \bm{V} \bigg \rangle^2\\  
 T_2 &:= \frac{4\gamma_n}{n^2} \bigg \langle \bm{HU}, \bm{V} \bigg\rangle^2 \\  
 T_3 &:= \frac{4 \gamma_n}{n^2 \tau(\bm{U})} \frac{\|\bm{g}\|}{\sqrt{n}} \bigg( \frac{f(\bm{U})_+}{\tau^2(\bm{U})} - 2 \bigg) \bigg \langle \bm{HU}, \bm{V} \bigg \rangle \bigg \langle \big( \bm{UU}\t - \bm{M} \big) \bm{U}, \bm{V} \bigg \rangle \\
 T_4 &:= \gamma_n f(\bm{U})_+ 2\tau(\bm{U})\inv \frac{\|\bm{g}\|}{\sqrt{n}}  \bigg\langle \bigg(  \frac{\big(  \bm{VU}\t +  \bm{UV}\t \big) }{n}  \bigg)\bm{U}, \bm{V} \bigg \rangle \\   
 T_5 &:= \gamma_n f(\bm{U})_+ 2\tau(\bm{U})\inv \frac{\|\bm{g}\|}{\sqrt{n}}  \bigg\langle \bigg(  \frac{\bm{UU}\t - \bm{M}- \bm{H}}{n}  \bigg) \bm{V} , \bm{V} \bigg \rangle \\   
 T_6 &:=  2 \frac{\lambda}{dr} \| \bm{V} \|_F^2.
\end{align}
We will show that some of the terms above are small-order terms on the event $\mathcal{E}_h$, and then some of the terms provide a lower bound.
\begin{itemize}
    \item \textbf{The term $T_1$}:  Observe that if $f(\bm{U}) > 0$,
\begin{align}
    \frac{\|\bm{g}\|}{\sqrt{n}} - \frac{f(\bm{U})_+}{\tau(\bm{U})} &= \frac{\langle \bm{H}, \bm{UU}\t - \bm{M} \rangle}{n \tau(\bm{U})}.
\end{align}
In addition, we have that
\begin{align}
    \| \bm{UU}\t - \bm{M} \|_F/\sqrt{n} \leq \tau(\bm{U}).
\end{align}
Consequently, we have that
\begin{align}
    T_1 &\leq \frac{4 \gamma_n}{n} \frac{\|\bm{g}\|}{\sqrt{n}}\bigg| \frac{\langle \bm{H},\bm{UU}\t - \bm{M} \rangle \rangle}{n \tau(\bm{U})} \bigg| \frac{\|\bm{UU}\t - \bm{M}\|_F^2 \|\bm{U}\|^2 \|\bm{V} \|_F^2}{n \tau^2(\bm{U})} \\
    &\leq \frac{4}{dr} \frac{\|\bm{g}\|}{\sqrt{n}}\bigg| \frac{\langle \bm{H},\bm{UU}\t - \bm{M} \rangle \rangle}{n \tau(\bm{U})} \bigg| \| \bm{U}\|^2 \|\bm{V} \|_F^2.
\end{align}
Note that on the event $\mathcal{E}_h$,
\begin{align}
    \bigg| \bigg\langle \bm{H}, \bm{UU}\t - \bm{M} \bigg \rangle \bigg| \leq C \sqrt{dr} \| \bm{UU}\t - \bm{M} \|_F; \qquad  \frac{\|\bm{g}\|}{\sqrt{n}} \leq 2.   \end{align}
Therefore,
\begin{align}
    T_1 &\leq \frac{8C}{\sqrt{ndr}} \frac{\|\bm{UU}\t - \bm{M}\|_F}{\sqrt{n} \tau(\bm{U})} \|\bm{U} \|^2 \|\bm{V} \|_F^2 \\
    &\leq \frac{8C}{\sqrt{ndr}}\|\bm{U} \|^2 \|\bm{V} \|_F^2.
\end{align}
Furthermore, when $\|\bm{U} - \bm{U}\s \mathcal{O}_{\bm{U}\s,\bm{U}} \| \leq \frac{c}{\kappa(\bm{M})} \| \bm{U}\s \|$, we have that $\|\bm{U}\| \leq 2 \|\bm{U}\s\|$.  Therefore,
\begin{align}
    T_1 \leq \frac{16C}{\sqrt{ndr}} \|\bm{M} \| \|\bm{V} \|_F^2 = \frac{16 C \kappa}{\sqrt{ndr}} \lambda_{\min}(\bm{M}) \|\bm{V} \|_F^2.
\end{align}

\item \textbf{The term $T_2$:} We note that the term $T_2$ is a square of something times a nonnegative number, whence $T_2 \geq 0$.  
\item \textbf{The term $T_3$}: We will upper bound $T_3$.  Recall that if $f(\bm{U})_+ > 0$ then
\begin{align}
    f(\bm{U})_+ = \tau(\bm{U}) \frac{\|\bm{g}\|}{\sqrt{n}} - \frac{\langle \bm{H}, \bm{UU}\t - \bm{M}\rangle}{n}.
\end{align}
We have that
\begin{align}
     T_3 &:= \frac{4 \gamma_n}{n^2 \tau(\bm{U})} \frac{\|\bm{g}\|}{\sqrt{n}} \bigg( \frac{f(\bm{U})_+}{\tau^2(\bm{U})} - 2 \bigg) \bigg \langle \bm{HU}, \bm{V} \bigg \rangle \bigg \langle \big( \bm{UU}\t - \bm{M} \big) \bm{U}, \bm{V} \bigg \rangle \\
     &\leq \frac{4}{ndr \tau(\bm{U})} \frac{\|\bm{g}\|}{\sqrt{n}} \bigg( \frac{\tau(\bm{U}) \frac{\|\bm{g}\|}{\sqrt{n}} - \frac{\langle \bm{H}, \bm{UU}\t - \bm{M}\rangle}{n}}{\tau^2(\bm{U})} - 2 \bigg) \bigg \langle \bm{HU}, \bm{V} \bigg \rangle \bigg \langle \big( \bm{UU}\t - \bm{M} \big) \bm{U}, \bm{V} \bigg \rangle \\
     &\leq \frac{4}{ndr \tau(\bm{U})} \bigg( \frac{1}{\tau(\bm{U})} + \bigg| \frac{\langle \bm{H}, \bm{UU}\t - \bm{M} \rangle}{n \tau^2(\bm{U})}\bigg| + 2 \bigg) \bigg| \bigg\langle \bm{HU}, \bm{V} \bigg\rangle \bigg| \bigg\langle (\bm{UU}\t - \bm{M}) \bm{U} , \bm{V} \bigg \rangle \bigg| \\
     &\lesssim \frac{1}{ndr \tau(\bm{U})} \bigg( \frac{1}{\tau(\bm{U})} + \bigg| \frac{\langle \bm{H}, \bm{UU}\t - \bm{M} \rangle}{n \tau^2(\bm{U})}\bigg| + 2 \bigg) \bigg| \bigg \langle \bm{HU}, \bm{V} \bigg \rangle \bigg| \| (\bm{UU}\t - \bm{M} )\|_F \|\bm{V} \|_F \| \bm{U}\|.
\end{align}
On the event $\mathcal{E}_h$ it holds that 
\begin{align}
    \bigg| \bigg\langle \bm{H}, \bm{UU}\t - \bm{M} \bigg\rangle \bigg| &\lesssim \sqrt{dr} \| \bm{UU}\t - \bm{M} \|_F; \\
    \bigg| \bigg\langle \bm{HU}, \bm{V} \bigg\rangle \bigg| &\lesssim \sqrt{dr} \| \bm{VU}\t \|_F \lesssim \sqrt{dr} \| \bm{V} \|_F \| \bm{U}\|.
\end{align}
Plugging this in above yields
\begin{align}
    T_3 \lesssim \frac{\| \bm{UU}\t - \bm{M}\|_F}{n \sqrt{dr} \tau(\bm{U})} \bigg( \frac{1}{\tau(\bm{U})} + \frac{\sqrt{dr} \| \bm{UU}\t - \bm{M}\|_F}{n\tau^2(\bm{U})} + 2 \bigg) \| \bm{V} \|_F^2 \| \bm{U} \|^2.
\end{align}
We also recall that 
\begin{align}
    \frac{ \| \bm{UU}\t - \bm{M}\|_F}{\sqrt{n} \tau(\bm{U})} \leq 1.
\end{align}
Therefore,
\begin{align}
    T_3 &\lesssim \frac{\| \bm{V}\|_F^2 \| \bm{U}\|^2}{\sqrt{ndr}} \bigg( \frac{1}{\tau(\bm{U})} + \frac{\sqrt{dr}}{\sqrt{n} \tau(\bm{U})} + 2 \bigg) \\
    &\lesssim \frac{\|\bm{U}\|^2 \|\bm{V}\|_F^2}{\sqrt{ndr}} \\
    &\lesssim \|\bm{V}\|_F^2 \frac{\kappa \lambda_{\min}(\bm{M})}{\sqrt{ndr}}
\end{align}
since $\tau(\bm{U}) \geq \sigma \geq \sigma_{\min} > 0$ and $\|\bm{U}\|^2 \lesssim \| \bm{M}\| \lesssim \kappa \lambda_{\min}(\bm{M})$.  

\item \textbf{The term $T_4$:} We will lower bound $T_4$. We have that
\begin{align}
    T_4 &= \frac{2}{dr \tau(\bm{U})} f(\bm{U})_+ \frac{\|\bm{g}\|}{\sqrt{n}} \bigg \langle \bigg( \bm{VU}\t + \bm{UV}\t \bigg) \bm{U}, \bm{V} \bigg \rangle \\
    &\geq \frac{f(\bm{U})_+}{dr \tau(\bm{U})}  \bigg \langle \bigg( \bm{VU}\t + \bm{UV}\t \bigg) \bm{U}, \bm{V} \bigg \rangle.
\end{align}
Next, we note that on the event $\mathcal{E}_h$ it holds that
\begin{align}
    \bigg| \langle \bm{H}, \bm{UU}\t - \bm{M} \rangle \bigg| \leq C \sqrt{dr} \| \bm{UU}\t - \bm{M} \|_F.
\end{align}
In addition, $\tau(\bm{U}) \geq \frac{\|\bm{UU}\t - \bm{M}\|_F}{\sqrt{n}} $.  Together, these conditions imply that
\begin{align}
    f(\bm{U}) &= \tau(\bm{U}) \frac{\|\bm{g}\|}{\sqrt{n}} - - \frac{C \sqrt{dr}}{\sqrt{n}} \frac{\|\bm{UU}\t - \bm{M}\|_F}{\sqrt{n}}\\
    &\geq .5 \sigma + .5 \frac{\|\bm{UU}\t - \bm{M}\|_F}{\sqrt{n}} - \frac{C \sqrt{dr}}{\sqrt{n}} \frac{\|\bm{UU}\t - \bm{M}\|_F}{\sqrt{n}} \\
    &= .5\sigma + \bigg( \frac{1}{2} - C\sqrt{\frac{dr}{n}} \bigg) \frac{\|\bm{UU}\t - \bm{M}\|_F}{\sqrt{n}} \\
    &\geq \frac{1}{4} \sigma + \frac{1}{4} \frac{\|\bm{UU}\t - \bm{M}\|_F}{\sqrt{n}},
\end{align}
where the final inequality holds by \cref{ass2} as long as $C_3$ is sufficiently large.  Therefore, it holds that $\frac{f(\bm{U})_+}{\tau(\bm{U})} \geq \frac{1}{4}$, and hence that 
\begin{align}
    T_4 &\geq     
    \frac{1}{4dr} \bigg \langle \bigg( \bm{VU}\t + \bm{UV}\t \bigg) \bm{U}, \bm{V} \bigg \rangle. 
\end{align}
We now analyze the quantity $\bigg \langle \bigg( \bm{VU}\t + \bm{UV}\t \bigg) \bm{U}, \bm{V} \bigg \rangle$.  Observe that we can rewrite this quantity as
\begin{align}
    \| \bm{U}\t \bm{V} \|_F^2 + {\sf Tr}\big[ \bm{UV}\t \bm{UV}\t \big] \geq \lambda_{\min}(\bm{U})^2 \| \bm{V} \|_F^2 + {\sf Tr}\big[ \bm{UV}\t \bm{UV}\t \big].
\end{align}
Weyl's inequality implies that $\lambda_{\min}(\bm{U})^2 \geq \frac{1}{4} \lambda_{\min}(\bm{M})$ provided $c$ smaller than some absolute constant.  

To complete the lower bound, we note that $\bm{V} = \bm{U - U'}\mathcal{O}$, where $\mathcal{O}$ is the Frobenius-optimal orthogonal matrix aligning $\bm{U}$ and $\bm{U}'$. Consequently, we observe that
\begin{align}
    {\sf Tr}\big[ \bm{UV}\t \bm{UV}\t \big] &= {\sf Tr}\big[ \bm{U}(\bm{U - U'}\mathcal{O})\t \bm{U} ( \bm{U - U'}\mathcal{O})\t \big] \geq 0,
    \end{align}
    where the inequality follows from Lemma 35 of \citet{ma_implicit_2020}, which states that the matrix $\bm{U}(\bm{U} - \bm{U}' \mathcal{O})\t$ is a symmetric matrix.  Combining all of these bounds we obtain
    \begin{align}
        T_4 \geq \frac{\lambda_{\min}(\bm{M})}{16 dr} \|\bm{V}\|_F^2.
    \end{align}

\item \textbf{The term $T_5$}:  We will upper bound $T_5$.  First, we note that
\begin{align}
    T_5 &= \frac{2 \gamma_n f(\bm{U})_+}{\tau(\bm{U})} \frac{\|\bm{g}\|}{\sqrt{n}} \bigg \langle \bigg( \frac{\bm{UU\t - M - H}}{n} \bigg) \bm{V}, \bm{V} \bigg \rangle \\
    &\leq \frac{4 f(\bm{U})_+}{dr \tau(\bm{U})} \| \bm{V} \|_F^2 \| \bm{UU}\t - \bm{M} - \bm{H} \| \\
    &\leq \frac{4}{dr \tau(\bm{U})} \|\bm{V}\|_F^2 \bigg( \tau(\bm{U}) \frac{\|\bm{g}\|}{\sqrt{n}} - \frac{\langle \bm{H}, \bm{UU}\t - \bm{M}\rangle}{n} \bigg) \bigg( \| \bm{UU}\t - \bm{M} \| + \|\bm{H} \| \bigg) \\
    &\leq \frac{ 4\|\bm{V}\|_F^2}{dr} \bigg( 1 + \frac{C\sqrt{dr} \| \bm{UU}\t - \bm{M}\|_F}{n \tau(\bm{U})} \bigg) \bigg( \| \bm{UU}\t - \bm{M} \| + \|\bm{H} \| \bigg) \\
    &\leq  \frac{ 4\|\bm{V}\|_F^2}{dr} \bigg( 1 + \frac{C\sqrt{dr}}{\sqrt{n}} \bigg) \bigg( \| \bm{UU}\t - \bm{M} \| + \|\bm{H} \| \bigg) \\
    &\leq \frac{ 8\|\bm{V}\|_F^2}{dr} \bigg( \| \bm{UU}\t - \bm{M} \| + \| \bm{H} \| \bigg).
\end{align}
Note that
\begin{align}
    \| \bm{UU}\t - \bm{M} \| \leq  2\| \bm{U} - \bm{U}\s \mathcal{O}_{\bm{U}\s,\bm{U}} \| \| \bm{U}\s \| + \| \bm{U} - \bm{U}\s \mathcal{O}_{\bm{U}\s,\bm{U}} \|^2 \leq \frac{3c}{\kappa(\bm{M})} \| \bm{U}\s\|^2
\end{align}
which implies that (since $\|\bm{H} \| \leq 3 \sqrt{d}$ on the event $\mathcal{E}_h$)
\begin{align}
    T_5 \leq  \frac{ 8\|\bm{V}\|_F^2}{dr} \bigg( 3 c\lambda_{\min}(\bm{M}) + 3 \sqrt{d} \bigg) &= \frac{24(c \lambda_{\min}(\bm{M})+ \sqrt{d} )}{dr} \| \bm{V}\|_F^2. 
\end{align}
Under \cref{ass1} it holds that 
\begin{align}
    \lambda_{\min}(\bm{M}) \geq C_1 \sigma \sqrt{dr} \geq C_1 \sigma_{\min} \sqrt{dr} \geq c\sqrt{d}
\end{align}
as long as $C_1$ is sufficiently large.  Therefore,
\begin{align}
    T_5 &\leq \frac{48 c \lambda_{\min}(\bm{M})}{dr} \|\bm{V}\|_F^2.
\end{align}
\item \textbf{The term $T_6$:} We will keep $T_6$ as-is.  
\end{itemize} 
Combining all of our bounds, we obtain that
\begin{align}
    \nabla^2 h(\bm{U})[\bm{V},\bm{V}] &\geq \lambda_{\min}(\bm{M}) \|\bm{V}\|_F^2  \frac{1}{16 dr} - \lambda_{\min}(\bm{M}) \|\bm{V}\|_F^2  \frac{16 C \kappa }{ \sqrt{ndr }} \\
    &\quad - \lambda_{\min}(\bm{M}) \|\bm{V}\|_F^2 \frac{C\kappa}{\sqrt{ndr}} -  \lambda_{\min}(\bm{M}) \|\bm{V}\|_F^2  \frac{48 c}{dr} + \frac{2\lambda}{dr} \|\bm{V}\|_F^2 \\
    &\geq  \|\bm{V}\|_F^2 \frac{\lambda_{\min}(\bm{M})}{dr} \bigg( \frac{1}{16} - \frac{16 C \kappa \sqrt{dr}}{\sqrt{n}} - \frac{C \kappa \sqrt{dr}}{\sqrt{n}} - 48 c \bigg) + \frac{2\lambda}{dr} \|\bm{V}\|_F^2 \\
    &\geq \frac{\|\bm{V}\|_F^2}{dr} \bigg(\lambda_{\min}(\bm{M})\bigg[ \frac{1}{16} - 48 c - \frac{17 C \kappa \sqrt{dr}}{\sqrt{n}} \bigg] + 2 \lambda \bigg) \\
    &\geq \frac{\|\bm{V}\|_F^2}{dr} \bigg(\lambda_{\min}(\bm{M})\bigg[ \frac{1}{16} - 48 c - \frac{17 C }{\sqrt{C_2}} \bigg] + 2 \lambda \bigg) \\
    &\geq \frac{\|\bm{V}\|_F^2}{dr} \bigg( \frac{\lambda_{\min}(\bm{M})}{32} + 2 \lambda \bigg),
\end{align}
where the penultimate inequality holds since $\kappa = O(1)$ by \cref{ass1} and $n \geq C_3 dr^2$ by \cref{ass2}, and hence the final inequality  holds as long as $C_2$ larger than some absolute constant and $c$ is smaller than some absolute constant. 
\end{proof}

\subsubsection{Proof of \cref{lem:Eh_highprob}} \label{sec:Eh_highprob_proof}
\begin{proof}[Proof of \cref{lem:Eh_highprob}]
The proof follows via $\eps$-net.  Let $\bm{A}$ be any rank at most $2r$ matrix such that $\|\bm{A}\|_F\leq 1$.  Then $\langle \bm{H}, \bm{A} \rangle$ is Gaussian with variance at most $2 \|\bm{A}\|_F^2 \leq 2$.  Hoeffding's inequality implies that $\bigg| \langle \bm{H},\bm{A} \rangle \bigg| \leq C t$ with probability at least $1 - 2 \exp( - c t^2)$.  Next, let $\mathcal{N}_{\eps}$ be a covering of rank at most $2r$ matrices of Frobenius norm at most 1; by Lemma 3.1 of \citet{candes_tight_2011}, it holds that $|\mathcal{N}_{\eps}| \leq \bigg( \frac{9}{\eps} \bigg)^{3 dr}$.  The union bound implies that
\begin{align}
    \sup_{\bm{A} \in \mathcal{N}_{\eps}} \bigg| \langle \bm{H},\bm{A} \rangle \bigg| \leq C t
\end{align}
with probability at least $1 - 2\exp( 3dr \log (9/\eps) - t^2 )$.  Taking $t \geq \sqrt{dr}$ shows that this probability is at most $1 - 2 \exp( - c d r)$ as desired.

For the other event it holds that $\frac{\|\bm{g}\|}{\sqrt{n}}$ is a $\frac{1}{\sqrt{n}}$-Lipschitz function of $\bm{g}$ and consequently the event holds with probability at least $1 - 2\exp( - c n)$.  The bound on $\|\bm{H}\|$ follows directly from \citet{bandeira_sharp_2016}.  
\end{proof}

\subsubsection{Proof of \cref{lem:miolaneb1analogue}} \label{sec:miolaneb1analogueproof}
\begin{proof}[Proof of \cref{lem:miolaneb1analogue}]
First, fix any $\bm{U} \in \mathbb{B}(\bm{\hat U},K)$.    Define the function
\begin{align}
    J : t \mapsto h( \bm{U} + t ( \bm{\hat U} \mathcal{O}_{\bm{\hat U},\bm{U}} - \bm{U} ) ).
\end{align}
Then since
\begin{align}
    \frac{d}{dt^2} J(t) &= ( \bm{\hat U} \mathcal{O}_{\bm{\hat U},\bm{U}} - \bm{U} )\t \nabla^2 h(  \bm{U} + t ( \bm{\hat U} \mathcal{O}_{\bm{\hat U},\bm{U}} - \bm{U})( \bm{\hat U} \mathcal{O}_{\bm{\hat U},\bm{U}} - \bm{U} ) \geq \gamma \|  \bm{\hat U} \mathcal{O}_{\bm{\hat U},\bm{U}} - \bm{U}\|_F^2,
\end{align}
it holds that $J(t)$ is strongly convex with parameter $\gamma \| \bm{\hat U} \mathcal{O}_{\bm{\hat U},\bm{U}} - \bm{U} \|_F^2$.

Let $\bm{U}_0$ be any minimizer of $h(\cdot)$ within $\mathbb{B}(\bm{\hat U},K)$. Then the condition of the lemma implies that $h(\bm{\hat U}) \leq h(\bm{U}_0) + \eps$.   Applying the argument above with $\bm{U} = \bm{U}_0$ implies that for any $s \in [0,1]$,
\begin{align}
    J(s t_1 + (1- s) t_2 ) &\leq s J(t_1) + (1- s) J(t_2) - \frac{1}{2} \gamma s(1-s) | t_1 - t_2|^2   \|  \bm{\hat U} \mathcal{O}_{\bm{\hat U},\bm{U}_0} - \bm{U}\s \|_F^2 \\
    \Longrightarrow h( \bm{U}\s + s ( \bm{\hat U} \mathcal{O}_{\bm{\hat U},\bm{U}_0} - \bm{U}_0 )) &\leq s h( \bm{\hat U} \mathcal{O}) + (1-s) h(\bm{U}_0) - \frac{1}{2} \gamma s(1-s)   \|  \bm{\hat U} \mathcal{O} - \bm{U}_0 \|_F^2 \\
    \Longrightarrow \frac{\gamma s(1-s)}{2} \| \bm{\hat U}\mathcal{O}_{\bm{\hat U},\bm{U}_0} - \bm{U}_0 \|_F^2 &\leq s h(\bm{\hat U}\mathcal{O}_{\bm{\hat U},\bm{U}_0}) + (1 - s) h(\bm{U}_0) - h(\bm{U}_0 + s(\bm{\hat U} \mathcal{O}_{\bm{\hat U},\bm{U}_0} - \bm{U}_0 )) \\
    &\leq s \big( h (\bm{U}_0) + \eps \big) + (1-s) h(\bm{U}_0)  - h(\bm{U}_0 + s(\bm{\hat U}\mathcal{O}_{\bm{\hat U},\bm{U}_0} - \bm{U}\s )) \\
    &= s \eps + h(\bm{U}_0) - h(\bm{U}_0+ s (\bm{\hat U}\mathcal{O}_{\bm{\hat U},\bm{U}_0} - \bm{U}_0)) \\
    &\leq s \eps.
\end{align}
Rearranging this bound implies that 
\begin{align}
     \| \bm{\hat U} \mathcal{O}_{\bm{\hat U},\bm{U}_0} - \bm{U}_0 \|_F^2 &\leq \frac{2}{\gamma} \eps,
\end{align}
where we have implicitly used the fact that $h(\bm{U}\mathcal{O}) = h(\bm{U})$ for all orthogonal matrices $\mathcal{O}$.  If $2\eps/\gamma < K$, then $\bm{U}_0$ is in the interior of $\mathbb{B}(\bm{\hat U},K)$, and hence since $h$ is strongly convex within $\mathbb{B}(\bm{\hat U},K)$, $\bm{U}_0$ is unique up to orthogonal transformation. 

Now suppose $\bm{\tilde U} \in \mathbb{B}(\bm{\hat U},K)$ is such that $h(\bm{\tilde U}) > h(\bm{U}_0) + \eps$.  
Observe that if we redefine $J(t)$ via $t\mapsto h(\bm{U}_0 \mathcal{O}_{\bm{U}_0,\bm{\tilde U}} + t (\bm{\tilde U} - \bm{U}_0\mathcal{O}_{\bm{U}_0,\bm{\tilde U}} ) )$, then via a similar argument, $J$ is again strongly convex with parameter $\gamma \| \bm{\tilde U} - \bm{U}_0 \mathcal{O}_{\bm{U}_0,\bm{\tilde U}}\|_F^2$.  However, we then have for all $s \in (0,1)$,
\begin{align}
    \frac{1}{2} \gamma  s (1 - s) \| \bm{\tilde U} - \bm{U}_0 \mathcal{O}_{\bm{U}_0,\bm{\tilde U}} \|_F^2 &\leq s J(1) + (1 - s) J(0) - J(1) \\
    &= (s- 1) h(\bm{\tilde U}) + (1 - s) h(\bm{U}_0) \\
    &\leq (s-1) h(\bm{U}_0) + (s-1)\eps + (1-s) h(\bm{U}_0) \\
    &= (s-1)\eps < 0
\end{align}
which is a contradiction.  Therefore, we must have that $h(\bm{\tilde U}) \leq h(\bm{U}_0) + \eps$, and hence by the same argument above must imply that $\|\bm{\tilde U}  - \bm{U}_0 \mathcal{O}_{\bm{U}_0,\bm{\tilde U}} \|_F^2 \leq \frac{2}{\gamma}\eps$. Recalling that $\mathcal{O}_{\bm{\hat U},\bm{U}_0}$ is the Frobenius-norm minimizer shows that 
\begin{align}
    \| \bm{\tilde U}\mathcal{O}_{\bm{\tilde U},\bm{\hat U}} - \bm{\hat U}  \|_F &\leq \| \bm{\tilde U}\mathcal{O}_{\bm{U}_0 ,\bm{\tilde U}}\t - \bm{\hat U} \mathcal{O}_{\bm{\hat U},\bm{U}_0} \|_F \\
    &\leq \| \bm{\tilde U} \mathcal{O}_{\bm{U}_0,\bm{\tilde U}}\t - \bm{U}_0 \|_F + \| \bm{\hat U} \mathcal{O}_{\bm{\hat U},\bm{U}_0} - \bm{U}_0 \|_F \\
    &\leq 2 \sqrt{\frac{2\eps}{\gamma}}
\end{align}
which implies that 
\begin{align}
    \|\bm{\tilde U} \mathcal{O}_{\bm{\tilde U},\bm{\hat U}} - \bm{\hat U} \|_F^2 \leq \frac{8}{\gamma} \eps
\end{align}
as desired. 
\end{proof}

\subsubsection{Proof of \cref{lem:h_yhat_bounded}} \label{sec:h_yhat_bounded_proof}

\begin{proof}
We will argue similarly to the proof of \cref{lem:concentrationlemmapart1}. First, similar to the proof of that lemma, we have that 
\begin{align}
    L_{\lambda}^{(n)}&(\bm{Z}_{{\sf ST}}^{(\lambda/\zeta\s)}(\tau\s) - \bm{M} )  = \bigg\{ \bigg( \frac{\sigma^2}{2\tau\s} + \frac{\tau\s}{2} \bigg) \beta\s \gamma_n - \frac{\gamma_n (\beta\s)^2}{2} \\
    &\quad + \bigg\{ \frac{\gamma_n \beta\s}{2\tau\s n} \| \bm{Z}_{{\sf ST}}^{(\lambda/\zeta\s)}(\tau\s) - \bm{M}\|_F^2- \frac{\gamma_n \beta\s}{ \tau\s n} \langle \bm{H}, \bm{Z}^{(\lambda/\zeta\s)}(\tau\s) - \bm{M} \rangle + \frac{\lambda}{dr} \big\{ \| \bm{Z}^{\lambda/\zeta\s}(\tau\s) \|_* - \| \bm{M} \|_* \big\} \bigg\}.
\end{align}
By \cref{lem:Zstgoodproperties}, with probability at least $1 - \exp( - c dr ( \eps \vee \eps^2 ) )$ it holds that
\begin{align}
    \| \bm{Z}_{{\sf ST}}^{(\lambda/\zeta\s)}(\tau\s) - \bm{M} \|_F^2 \leq \mathbb{E} \|  \bm{Z}_{{\sf ST}}^{(\lambda/\zeta\s)}(\tau\s) - \bm{M} \|_F^2 + \eps dr.
\end{align}
Observe that
\begin{align}
    \frac{\mathbb{E} \|  \bm{Z}_{{\sf ST}}^{(\lambda/\zeta\s)}(\tau\s) - \bm{M} \|_F^2}{n} &= (\tau\s)^2 - \sigma^2 = O\bigg(  \sigma^2 \frac{dr}{n}\bigg),
\end{align}
where the final bound holds by \cref{lem:fxdpointprops}.  
Consequently,  with probability at least $1 - \exp( - c dr ( \eps \vee \eps^2 ) )$ it holds that
\begin{align}
      \| \bm{Z}_{{\sf ST}}^{(\lambda/\zeta\s)}(\tau\s) - \bm{M} \|_F^2 \leq C dr + \eps dr.
\end{align}
Take $\eps \asymp 1$ shows that with probability at least $1 - 2\exp( - c dr)$, one has $ \| \bm{Z}_{{\sf ST}}^{(\lambda/\zeta\s)}(\tau\s) - \bm{M} \|_F \leq C \sigma \sqrt{dr} = R'$.  Therefore, define the event
\begin{align}
    \mathcal{E}_1 := \bigg\{  \| \bm{Z}_{{\sf ST}}^{(\lambda/\zeta\s)}(\tau\s) - \bm{M} \|_F \leq R'\bigg\}.
\end{align}
On the event $\mathcal{E}_1$,
\begin{align}
    L_{\lambda}^{(n)}(\bm{Z}_{{\sf ST}}^{(\lambda/\zeta\s)}(\tau\s) - \bm{M} ) &= \bigg\{ \bigg( \frac{\sigma^2}{2\tau\s} + \frac{\tau\s}{2} \bigg) \beta\s \gamma_n - \frac{\gamma_n (\beta\s)^2}{2} \\
    &\qquad + \min_{\substack{\| \bm{W} \| \leq R'\\ \bm{W} \succcurlyeq - \bm{M}}} \bigg\{ \frac{\gamma_n \beta\s}{2\tau\s n} \| \bm{W} \|_F^2- \frac{\gamma_n \beta\s}{ \tau\s n} \langle \bm{H},\bm{W} \rangle + \frac{\lambda}{dr} \big\{ \| \bm{W} + \bm{M} \|_* - \| \bm{M} \|_* \big\} \bigg\} \\
    &=: \tilde L_{\lambda}^{(n)}(\bm{H}).
\end{align}
Notice that the function above is $\frac{\gamma_n \beta\s}{\tau\s n} R'$- Lipschitz in $\bm{H}$.  Define the event
\begin{align}
    \mathcal{E}_2: =  \bigg\{\bigg|    \tilde L_{\lambda}^{(n)} - \mathbb{E}_{\bm{H}} ( \tilde L_{\lambda}^* )\bigg| &\leq \eps + C (\exp( - c dr))  \bigg\},
\end{align}
where $C$ and $c$ are some universal constants.  
By Lipschitz concentration it holds that $$\p(\mathcal{E}_2) \geq 1 -2 \exp \big( - \frac{ (\eps + C \exp(- c dr) \big) ^2 (\tau\s)^2 n^2}{\gamma_n^2 (\beta\s)^2 (R')^2} ) \geq 1 - 2\exp( - C d r( C \exp(- c dr) + \eps)^2 ).$$  We further observe that $L_{\lambda}\s = \mathbb{E}  L_{\lambda}^{(n)}(\bm{Z}_{{\sf ST}}^{(\lambda/\zeta\s)}(\tau\s) - \bm{M} )$.  Therefore,
\begin{align}
\bigg|    L_{\lambda}\s - \mathbb{E} \tilde L_{\lambda}^{(n)} (\bm{H}) \bigg| &\leq \bigg(\frac{\gamma_n \beta\s}{\tau\s n} \| \bm{M} \|_F^2 + 2\frac{\lambda}{dr} \| \bm{M} \|_* \bigg) \p\big\{ \mathcal{E}_1^c \big\} \\
&\leq \bigg( \frac{\beta\s}{\tau\s dr}\|\bm{M} \|_F^2 + 2 \frac{C_4 \sigma \sqrt{d}}{dr} \| \bm{M}\|_* \bigg) \p\big\{ \mathcal{E}_1^c \big\} \\
&\leq \bigg( \frac{\beta\s}{\tau\s dr} r \kappa^2 \lambda_{r}(\bm{M})^2  + 2 \frac{C_4 \sigma \sqrt{d} r \lambda_r(\bm{M})}{dr} \bigg) \p\big\{ \mathcal{E}_1^c \big\} \\
&\leq \bigg( \frac{r \kappa^2 d \lambda_{\min}^2}{\tau\s dr}   + 2 \frac{C_4 \sigma d r \lambda_{\min}}{dr} \bigg) \p\big\{ \mathcal{E}_1^c \big\} \\
&\leq \bigg( C_2 \kappa^2 \sigma^2 r + 2 C_4 \sqrt{C_2} \sigma^2 \sqrt{r} \bigg) \p\big\{ \mathcal{E}_1^c \big\} \\
&= O\bigg( r \exp( - c dr ) \bigg) \\
&= O( \exp( - c dr)),
\end{align}
where we have applied \cref{ass1}.  Consequently it holds that
\begin{align}
\p\bigg\{ \bigg|    &L_{\lambda}^{(n)} \bigg( \bm{Z}_{{\sf ST}}^{(\lambda/\zeta\s)}(\tau\s) - \bm{M} \bigg) - L_{\lambda}^* \bigg| > \eps + C \exp( - c d r) \bigg\} \\
&\leq \p\bigg\{ \bigg|    L_{\lambda}^{(n)} \bigg( \bm{Z}_{{\sf ST}}^{(\lambda/\zeta\s)}(\tau\s) - \bm{M} \bigg) - L_{\lambda}^* \bigg| > \eps+ C \exp( - c d r) \cap \mathcal{E}_1 \bigg\} + O\big(\exp( - c dr ) \big)\\
&\leq \p\bigg\{ \bigg|    \tilde L_{\lambda}^{(n)} (\bm{H})  - L_{\lambda}^* \bigg| > \eps + C \exp( - c d r)   \bigg\} + O\big(\exp( - c dr ) \big) \\
&\leq O\bigg( \exp\big( - C dr \big( C \exp(- c dr) + \eps \big)^2 \big) + \exp( - c dr) \bigg).
\end{align}
This completes the proof.
\end{proof}

\subsubsection{Proof of \cref{lem:hessiancalculation}}  \label{sec:hessiancalculationproof}
\begin{proof}[Proof of \cref{lem:hessiancalculation}]
We observe that 
    \begin{align}
\nabla h(\bm{U}) &= \gamma_n f(\bm{U})_+ \nabla f(\bm{U}) + \frac{2\lambda}{dr} \bm{U}; \\
\nabla^2 h(\bm{U})[\bm{V},\bm{V}] &= \gamma_n \bigg\langle \nabla f(\bm{U}), \bm{V} \bigg \rangle^2 + \gamma_n f(\bm{U})_+ \bigg\langle \frac{d}{dt} \nabla f(\bm{U} + t \bm{V}), \bm{V} \bigg \rangle \bigg|_{t=0} + 2 \frac{\lambda}{dr} \| \bm{V} \|_F^2.
\end{align}
Therefore, in order to calculate $\nabla^2 h(\bm{U})[\bm{V},\bm{V}]$, we need the derivative of $f(\bm{U})$.  We have that 
\begin{align}
       \nabla \tau(\bm{U}) &= 2 \tau(\bm{U})\inv \frac{(\bm{UU}\t - \bm{M})}{n} \bm{U}; \\
        \nabla f(\bm{U}) &= \nabla \tau(\bm{U}) \frac{\|\bm{g}\|}{\sqrt{n}} - \frac{2 \bm{HU}}{n}; \equiv 2 \frac{\|\bm{g}\|}{\sqrt{n}}\tau(\bm{U})\inv \frac{(\bm{UU}\t - \bm{M})}{n} \bm{U}  - \frac{2 \bm{HU}}{n},
        \end{align}
    which implies that
    \begin{align}
    \langle \nabla f(\bm{U} + t \bm{V}), \bm{V} \rangle \\&= 2 \tau(\bm{U} + t \bm{V})\inv \frac{\|\bm{g}\|}{\sqrt{n}} \bigg \langle \frac{(\bm{U} + t \bm{V}) (\bm{U} + t \bm{V})\t - \bm{M} - \bm{H}}{n} \big( \bm{U} + t \bm{V} \big), \bm{V} \bigg\rangle.
    \end{align}
    Therefore,
    \begin{align}
     \frac{d}{dt}  &\langle \nabla f(\bm{U} + t \bm{V}), \bm{V} \rangle \bigg|_{t=0} \\
     &=  2  \frac{\|\bm{g}\|}{\sqrt{n}} \bigg\langle \bigg(\frac{ \bm{UU}\t  - \bm{M} - \bm{H}}{n}  \bigg)  \bm{U}  , \bm{V} \bigg \rangle \frac{d}{dt} \tau(\bm{U} + t \bm{V})\inv \bigg|_{t=0} \\
    &\quad + 2\tau(\bm{U})\inv \frac{\|\bm{g}\|}{\sqrt{n}} \frac{d}{dt} \bigg\langle \bigg(  \frac{\big(\bm{UU}\t + t \bm{VU}\t + t \bm{UV}\t + t^2 \bm{XX}\t\big) - \bm{M} - \bm{H}}{n}  \bigg)\big( \bm{U} + t \bm{V} \big), \bm{V} \bigg \rangle \bigg|_{t=0} \\
    &= 2\frac{\|\bm{g}\|}{\sqrt{n}} \bigg\langle \bigg(\frac{ \bm{UU}\t  - \bm{M}- \bm{H}}{n}  \bigg)  \bm{U}  , \bm{V} \bigg \rangle \frac{d}{dt} \tau(\bm{U} + t \bm{V})\inv \bigg|_{t=0} \\
    &\quad + 2\tau(\bm{U})\inv \frac{\|\bm{g}\|}{\sqrt{n}}  \bigg\langle \bigg(  \frac{\big(  \bm{VU}\t +  \bm{UV}\t \big) }{n}  \bigg)\bm{U}, \bm{V} \bigg \rangle \\
    &\quad + 2\tau(\bm{U})\inv \frac{\|\bm{g}\|}{\sqrt{n}}  \bigg\langle \bigg(  \frac{\bm{UU}\t - \bm{M}- \bm{H}}{n}  \bigg) \bm{V} , \bm{V} \bigg \rangle
\end{align}
We can compute
\begin{align}
    \frac{d}{dt} \tau(\bm{U} + t \bm{V} ) \inv \bigg|_{t=0} &= -1 \bigg[ \frac{\|\bm{UU}\t - \bm{M}\|_F^2}{n} + \sigma^2 \bigg]^{-3/2} \frac{\langle \bm{UU\t - M}, \bm{VU}\t + \bm{UV}\t \rangle}{n} \\
    &= - \tau(\bm{U})^{-3} \frac{\langle \bm{UU\t - M}, \bm{VU}\t + \bm{UV}\t \rangle}{n}.
\end{align}
This gives us the expression:
\begin{align}
      \frac{d}{dt}  \langle \nabla f(\bm{U} + t \bm{V}), \bm{V} \rangle \bigg|_{t=0} &= -2\frac{\|\bm{g}\|}{\sqrt{n}} \bigg\langle \bigg(\frac{ \bm{UU}\t  - \bm{M}- \bm{H}}{n}  \bigg)  \bm{U}  , \bm{V} \bigg \rangle \tau(\bm{U})^{-3} \frac{\langle \bm{UU\t - M}, \bm{VU}\t + \bm{UV}\t \rangle}{n} \\
    &\quad + 2\tau(\bm{U})\inv \frac{\|\bm{g}\|}{\sqrt{n}}  \bigg\langle \bigg(  \frac{\big(  \bm{VU}\t +  \bm{UV}\t \big) }{n}  \bigg)\bm{U}, \bm{V} \bigg \rangle \\
    &\quad + 2\tau(\bm{U})\inv \frac{\|\bm{g}\|}{\sqrt{n}}  \bigg\langle \bigg(  \frac{\bm{UU}\t - \bm{M}- \bm{H}}{n}  \bigg) \bm{V} , \bm{V} \bigg \rangle
\end{align}
Therefore:
\begin{align}
    \nabla^2 h(\bm{U})[\bm{V},\bm{V}] &= \gamma_n \bigg\langle \nabla f(\bm{U}), \bm{V} \bigg \rangle^2 + \gamma_n f(\bm{U})_+ \bigg\langle \frac{d}{dt} \nabla f(\bm{U} + t \bm{V}), \bm{V} \bigg \rangle \bigg|_{t=0} + 2 \frac{\lambda}{dr} \| \bm{V} \|_F^2 \\
    &= \gamma_n \bigg\langle 2 \frac{\|\bm{g}\|}{\sqrt{n}} \tau(\bm{U})\inv \frac{\bm{UU}\t - \bm{M}}{n} \bm{U} - \frac{2\bm{HU}}{n}, \bm{V} \bigg\rangle^2 \\
    &\quad - 2\gamma_n f(\bm{U})_+  \frac{\|\bm{g}\|}{\sqrt{n}} \bigg\langle \bigg(\frac{ \bm{UU}\t  - \bm{M}- \bm{H}}{n}  \bigg)  \bm{U}  , \bm{V} \bigg \rangle \tau(\bm{U})^{-3} \frac{\langle \bm{UU\t - M}, \bm{VU}\t + \bm{UV}\t \rangle}{n} \\
    &\quad +\gamma_n f(\bm{U})_+ 2\tau(\bm{U})\inv \frac{\|\bm{g}\|}{\sqrt{n}}  \bigg\langle \bigg(  \frac{\big(  \bm{VU}\t +  \bm{UV}\t \big) }{n}  \bigg)\bm{U}, \bm{V} \bigg \rangle \\
    &\quad +\gamma_n f(\bm{U})_+ 2\tau(\bm{U})\inv \frac{\|\bm{g}\|}{\sqrt{n}}  \bigg\langle \bigg(  \frac{\bm{UU}\t - \bm{M}- \bm{H}}{n}  \bigg) \bm{V} , \bm{V} \bigg \rangle \\
    &\quad + 2 \frac{\lambda}{dr} \| \bm{V} \|_F^2.
\end{align}
We  now expand out the first two terms.  We have that
\begin{align}
    \gamma_n \bigg\langle 2\frac{\|\bm{g}\|}{\sqrt{n}} \tau(\bm{U})\inv \frac{\bm{UU}\t - \bm{M}}{n} \bm{U} - \frac{2\bm{HU}}{n}, \bm{V} \bigg\rangle^2 &= 
\frac{4 \gamma_n }{n^2\tau^{2}(\bm{U})}\frac{\|\bm{g}\|^2}{n} \bigg \langle (\bm{UU}\t - \bm{M} ) \bm{U}, \bm{V} \bigg \rangle^2 \\
&\quad + \frac{4 \gamma_n }{n^2} \bigg \langle \bm{H} \bm{U}, \bm{V} \bigg \rangle^2 \\
&\quad - \frac{8 \gamma_n }{n^2\tau(\bm{U})} \frac{\|\bm{g}\|}{\sqrt{n}} \bigg\langle \big(\bm{UU}\t - \bm{M} \big) \bm{U} , \bm{V} \bigg \rangle \bigg \langle \bm{HU}, \bm{V} \bigg \rangle.
\end{align}
Similarly,
\begin{align}
     - 2\gamma_n f(\bm{U})_+  &\frac{\|\bm{g}\|}{\sqrt{n}} \bigg\langle \bigg(\frac{ \bm{UU}\t  - \bm{M}- \bm{H}}{n}  \bigg)  \bm{U}  , \bm{V} \bigg \rangle \tau(\bm{U})^{-3} \frac{\langle \bm{UU\t - M}, \bm{VU}\t + \bm{UV}\t \rangle}{n} \\ 
     &=  - 2\gamma_n f(\bm{U})_+  \frac{\|\bm{g}\|}{\sqrt{n}} \bigg\langle \bigg(\frac{ \bm{UU}\t  - \bm{M}}{n}  \bigg)  \bm{U}  , \bm{V} \bigg \rangle \tau(\bm{U})^{-3} \frac{\langle \bm{UU\t - M}, \bm{VU}\t + \bm{UV}\t \rangle}{n} \\ 
     &\quad + 2\gamma_n f(\bm{U})_+  \frac{\|\bm{g}\|}{\sqrt{n}} \bigg\langle \bigg(\frac{  \bm{H}}{n}  \bigg)  \bm{U}  , \bm{V} \bigg \rangle \tau(\bm{U})^{-3} \frac{\langle \bm{UU\t - M}, \bm{VU}\t + \bm{UV}\t \rangle}{n} \\ 
     &= \frac{-4 \gamma_n}{n^2 \tau^2(\bm{U})} \frac{f(\bm{U})_+}{\tau(\bm{U})} \frac{\|\bm{g}\|}{\sqrt{n}} \bigg \langle \big( \bm{UU}\t - \bm{M} \big) \bm{U} , \bm{V} \bigg \rangle^2 \\
     &\quad + \frac{4 \gamma_n}{n^2 \tau^2(\bm{U})} \frac{f(\bm{U})_+}{\tau(\bm{U})} \frac{\|\bm{g}\|}{\sqrt{n}} \bigg\langle \bm{HU}, \bm{V} \bigg \rangle \bigg\langle \big( \bm{UU}\t - \bm{M} \big) \bm{U}, \bm{V} \bigg \rangle.
\end{align}
Adding these two terms together we arrive at
\begin{align}
    \frac{4 \gamma_n}{n^2 \tau^2(\bm{U})} &\frac{\|\bm{g}\|}{\sqrt{n}} 5\bigg( \frac{\|\bm{g}\|}{\sqrt{n}} - \frac{f(\bm{U})_+}{\tau(\bm{U})} \bigg) \bigg \langle \big( \bm{UU}\t - \bm{M}\big) \bm{U}, \bm{V} \bigg \rangle^2 \\
    &\quad + \frac{4\gamma_n}{n^2} \bigg \langle \bm{HU}, \bm{V} \bigg\rangle^2 \\
    &\quad + \frac{4 \gamma_n}{n^2 \tau(\bm{U})} \frac{\|\bm{g}\|}{\sqrt{n}} \bigg( \frac{f(\bm{U})_+}{\tau^2(\bm{U})} - 2 \bigg) \bigg \langle \bm{HU}, \bm{V} \bigg \rangle \bigg \langle \big( \bm{UU}\t - \bm{M} \big) \bm{U}, \bm{V} \bigg \rangle 
\end{align}
Therefore, we have the final expression:
\begin{align}
    \nabla^2 h(\bm{U})[ \bm{V}, \bm{V}] &=  \frac{4 \gamma_n}{n^2 \tau^2(\bm{U})} \frac{\|\bm{g}\|}{\sqrt{n}} \bigg( \frac{\|\bm{g}\|}{\sqrt{n}} - \frac{f(\bm{U})_+}{\tau(\bm{U})} \bigg) \bigg \langle \big( \bm{UU}\t - \bm{M}\big) \bm{U}, \bm{V} \bigg \rangle^2 \\
    &\quad + \frac{4\gamma_n}{n^2} \bigg \langle \bm{HU}, \bm{V} \bigg\rangle^2 \\
    &\quad + \frac{4 \gamma_n}{n^2 \tau(\bm{U})} \frac{\|\bm{g}\|}{\sqrt{n}} \bigg( \frac{f(\bm{U})_+}{\tau^2(\bm{U})} - 2 \bigg) \bigg \langle \bm{HU}, \bm{V} \bigg \rangle \bigg \langle \big( \bm{UU}\t - \bm{M} \big) \bm{U}, \bm{V} \bigg \rangle \\
    &\quad +\gamma_n f(\bm{U})_+ 2\tau(\bm{U})\inv \frac{\|\bm{g}\|}{\sqrt{n}}  \bigg\langle \bigg(  \frac{\big(  \bm{VU}\t +  \bm{UV}\t \big) }{n}  \bigg)\bm{U}, \bm{V} \bigg \rangle \\
    &\quad +\gamma_n f(\bm{U})_+ 2\tau(\bm{U})\inv \frac{\|\bm{g}\|}{\sqrt{n}}  \bigg\langle \bigg(  \frac{\bm{UU}\t - \bm{M}- \bm{H}}{n}  \bigg) \bm{V} , \bm{V} \bigg \rangle \\
    &\quad + 2 \frac{\lambda}{dr} \| \bm{V} \|_F^2.
\end{align}
 
\end{proof}

\section{Proofs of Lemmas in \cref{sec:debnearequivalence}}

\label{sec:debunreg}

In this section we prove the two intermediate lemmas required for the proof of \cref{thm:debunreg}.

\subsection{Proof of \cref{lem:concentrationbound1}} \label{sec:concentrationbound1proof}
To prove \cref{lem:concentrationbound1} we first require the following lemma.

\begin{lemma} \label{lem:netargument}
Let $\{\bm{X}_i\}_{i=1}^{n} \sim {\sf GOE}(d)$, and let $\mathcal{X}(\cdot)$ be the operator that assigns entries $\mathcal{X}(\cdot)_i = \frac{1}{\sqrt{n}} \langle \bm{X}_i, \cdot\rangle $.  Then it holds that 
\begin{align}
  \sup_{\bm{V}  \in \mathbb{R}^{d \times r}:  \|\bm{V} \| \leq 1 }  \bigg\| \big( \mathcal{X}\s \mathcal{X} - \mathcal{I} \big) \big( \bm{VV}\t \big) \bigg\| \lesssim \frac{ r\sqrt{d}}{\sqrt{n}} 
\end{align}
with probability at least $1 - 2 \exp ( - C_1 dr) - 2 \exp( - C_2 n )$.    
\end{lemma}
\begin{proof}
    See \cref{sec:netargumentproof}.
\end{proof}

We now have the proof of \cref{lem:concentrationbound1}.  
\begin{proof}[Proof of \cref{lem:concentrationbound1}]
      Note that $\bm{U}^{(\lambda)}$ satisfies
\begin{align}
  \frac{1}{\sqrt{n}} \sum_{i}  \bigg( \bigg \langle \bm{X}_i , \bm{U}^{(\lambda)} (\bm{U}^{(\lambda)})\t - \bm{M} \bigg \rangle / \sqrt{n} - \eps_i \bigg) \bm{X}_i \bm{U}^{(\lambda)} + \lambda \bm{U}^{(\lambda)} = 0,
\end{align}
which shows that $\bm{U}^{(\lambda)}_{({\sf deb})}$ satisfies
\begin{align}
\nabla f^{(0)}_{{\sf ncvx}}&( \bm{U}^{(\lambda)}_{{\sf deb}} ) \\
& =   \frac{1}{\sqrt{n}}  \sum_{i}  \bigg\{ \bigg \langle \bm{X}_i , \bm{U}^{(\lambda)} \bigg( \bm{I} + \lambda \big( \bm{U}^{(\lambda)})\t \bm{U}^{(\lambda)} \big)\inv \bigg) (\bm{U}^{(\lambda)})\t - \bm{M} \bigg \rangle /\sqrt{n}- \eps_i \bigg\} \bm{X}_i \bm{U}^{(\lambda)} \bigg( \bm{I} + \lambda \big( \bm{U}^{(\lambda)})\t \bm{U}^{(\lambda)} \big)\inv \bigg)^{1/2} \\
      &=  \frac{1}{\sqrt{n}}\sum_{i}  \bigg\{ \bigg \langle \bm{X}_i , \bm{U}^{(\lambda)}  (\bm{U}^{(\lambda)})\t - \bm{M} \bigg \rangle/\sqrt{n} - \eps_i \bigg\} \bm{X}_i \bm{U}^{(\lambda)} \bigg( \bm{I} + \lambda \big( \bm{U}^{(\lambda)})\t \bm{U}^{(\lambda)} \big)\inv \bigg)^{1/2} \\
      &\quad+ \frac{\lambda}{n} \sum_{i} \bigg\{ \bigg \langle \bm{X}_i , \bm{U}^{(\lambda)} \big( \bm{U}^{(\lambda)})\t \bm{U}^{(\lambda)} \big)\inv  (\bm{U}^{(\lambda)})\t  \bigg \rangle  \bigg\} \bm{X}_i \bm{U}^{(\lambda)} \bigg( \bm{I} + \lambda \big( \bm{U}^{(\lambda)})\t \bm{U}^{(\lambda)} \big)\inv \bigg)^{1/2} \\
      &= \frac{1}{\sqrt{n}} \sum_{i}  \bigg\{ \bigg \langle \bm{X}_i , \bm{U}^{(\lambda)}  (\bm{U}^{(\lambda)})\t - \bm{M} \bigg \rangle /\sqrt{n}- \eps_i \bigg\} \bm{X}_i \bm{U}^{(\lambda)} \bigg( \bm{I} + \lambda \big( \bm{U}^{(\lambda)})\t \bm{U}^{(\lambda)} \big)\inv \bigg)^{1/2} \\
      &\quad+ \lambda \mathcal{X}\s \mathcal{X} \bigg( \bm{U}^{(\lambda)} \big( ( \bm{U}^{(\lambda)})\t \bm{U}^{(\lambda)} \big)\inv (\bm{U}^{(\lambda)} )\t  \bigg) \bm{U}^{(\lambda)} \bigg( \bm{I} + \lambda \big( \bm{U}^{(\lambda)})\t \bm{U}^{(\lambda)} \big)\inv \bigg)^{1/2} \\
      &= -\lambda \bm{U}^{(\lambda)}  \bigg( \bm{I} + \lambda \big( \bm{U}^{(\lambda)})\t \bm{U}^{(\lambda)} \big)\inv \bigg)^{1/2} \\
      &\quad + \lambda \bigg( \mathcal{X}\s \mathcal{X} - \mathcal{I} \bigg) \bigg[ \bm{U}^{(\lambda)} \big( (\bm{U}^{(\lambda)})\t  \bm{U}^{(\lambda)} \big)\inv (\bm{U}^{(\lambda)})\t \bigg] \bm{U}^{(\lambda)} \bigg( \bm{I} + \lambda \big( \bm{U}^{(\lambda)})\t \bm{U}^{(\lambda)} \big)\inv \bigg)^{1/2}  \\
      &\quad + \lambda \bm{U}^{(\lambda)} \bigg( \bm{I} + \lambda \big( \bm{U}^{(\lambda)})\t \bm{U}^{(\lambda)} \big)\inv \bigg)^{1/2} \\
      &= \lambda \bigg( \mathcal{X}\s \mathcal{X} - \mathcal{I} \bigg) \bigg[ \bm{U}^{(\lambda)} \big( (\bm{U}^{(\lambda)})\t  \bm{U}^{(\lambda)} \big)\inv (\bm{U}^{(\lambda)})\t \bigg] \bm{U}^{(\lambda)} \bigg( \bm{I} + \lambda \big( \bm{U}^{(\lambda)})\t \bm{U}^{(\lambda)} \big)\inv \bigg)^{1/2} .
\end{align}
Taking norms, we obtain that
\begin{align}
    \| \nabla f^{(0)}_{{\sf ncvx}} \big( \bm{U}^{(\lambda)}_{{\sf deb}} \big) \|_F &\leq \lambda \big\| \bm{U}^{(\lambda)} \big( \bm{I} + \lambda \big( \bm{U}^{(\lambda)\top} \bm{U}^{(\lambda)} \big)\inv \big)^{1/2}  \big\|_F \bigg\| \bigg( \mathcal{X}\s \mathcal{X} - \mathcal{I} \bigg) \big(  \bm{V}_U \bm{V}_U\t \big)\bigg \| \\
    &= \lambda \| \bm{U}^{(\lambda)}_{{\sf deb}} \|_F \bigg\| \bigg( \mathcal{X}\s \mathcal{X} - \mathcal{I} \bigg) \big(  \bm{V}_U \bm{V}_U\t \big)\bigg\|.
\end{align}
 By \cref{lem:netargument}, we have that with the advertised probability,
\begin{align}
    \bigg\| \bigg( \mathcal{X}\s \mathcal{X} - \mathcal{I} \bigg) \big(  \bm{V}_U \bm{V}_U\t \big)\bigg \| &\lesssim \frac{r \sqrt{d}}{\sqrt{n}},
\end{align}
whence the proof is completed.
\end{proof}

\subsubsection{Proof of \cref{lem:netargument}}
\label{sec:netargumentproof}

\begin{proof}[Proof of \cref{lem:netargument}]
We will proceed via $\eps$-net argument.  First, let $\bm{V}$ be any fixed $d \times r$  orthonormal matrix.  We have that
\begin{align}
\mathcal{X}\s \mathcal{X} \big( \bm{V} \bm{V}\t \big) &= 
    \frac{1}{n} \sum_{i=1}^{n} \bigg \langle \bm{X}_i, \bm{V} \bm{V}\t \bigg\rangle\bm{X}_i \bm{V} \bm{V}\t  \\
    &= \frac{1}{n} \sum_{i=1}^{n} \bigg \langle \bm{X}_i, \bm{V} \bm{V}\t \bigg \rangle  \bigg( \bm{I} - \bm{V} \bm{V}\t \bigg) \bm{X}_i \bm{V} \bm{V}\t \\
    &\quad + \frac{1}{n} \sum_{i=1}^{n} \bigg \langle \bm{X}_i, \bm{V} \bm{V}\t \bigg \rangle   \bm{V} \bm{V}\t  \bm{X}_i \bm{V} \bm{V}\t \\
    &= \frac{1}{n} \sum_{i=1}^{n} {\sf Tr} \bigg( \bm{V}\t \bm{X}_i \bm{V} \bigg) \bigg( \bm{I} - \bm{V}\bm{V}\t \bigg) \bm{X}_i \bm{V} \bm{V}\t \\
    &\quad + \frac{1}{n} \sum_{i=1}^{n} {\sf Tr}  \bigg( \bm{V}\t \bm{X}_i \bm{V} \bigg)  \bm{V} \bm{V}\t \bm{X}_i \bm{V} \bm{V}\t
\end{align}
As a result,
\begin{align}
\bigg\|    \bigg( \mathcal{X}\s \mathcal{X} - \mathcal{I} \bigg) \big( \bm{V} \bm{V}\t \big) \bigg\| &\leq \bigg\| \frac{1}{n} \sum_{i=1}^{n} {\sf Tr} \bigg( \bm{V}\t \bm{X}_i \bm{V} \bigg) \bigg( \bm{I} - \bm{V} \bm{V}\t \bigg) \bm{X}_i \bm{V} \bigg\| \\
&\qquad +  \bigg\|\bm{V} -  \frac{1}{n} \sum_{i=1}^{n} {\sf Tr} \bigg( \bm{V}\t \bm{X}_i \bm{V} \bigg)\bm{V} \bm{V}\t \bm{X}_i \bm{V}  \bigg\| \\
:&= T_1 + T_2,
\end{align}
where we have used the rotational invariance of $\|\cdot\|$.  We analyze each term in turn. 
\begin{itemize}
    \item \textbf{The quantity $T_1$:}  By rotational invariance the quantities $\bm{V}\t \bm{X}_i \bm{V}$ and $(\bm{I - VV}\t) \bm{X}_i \bm{V } $ are  independent from each other, and hence each quantity inside the summation is a product of two independent random variables. Let $\bm{V}_{\perp}$ denote the $d \times (d-r) $ matrix such that $\bm{V}_{\perp}\t \bm{V} = 0$, and let $c_i = {\sf Tr}\big( \bm{V}\t \bm{X}_i \bm{V} \big)$.  We may then write the summation as
    \begin{align}
        \frac{1}{n} \sum_{i=1}^{n} c_i \bm{V}_{\perp } \bm{V}_{\perp}\t \bm{X}_i \bm{V} &\overset{d}{=} \frac{1}{n} \sum_{i=1}^{n} c_i \bm{V}_{\perp } \bm{H}_i,
    \end{align}
    where $\bm{H}_i \in \mathbb{R}^{d-r \times r}$ is a matrix with $\mathcal{N}(0,1)$ entries. Furthermore, we have that
    \begin{align}
        \bigg\| \frac{1}{n} \sum_{i=1}^{n} c_i \bm{V}_{\perp} \bm{H}_i \bigg\| &\leq \bigg\| \frac{1}{n} \sum_{i=1}^{n} c_i \bm{H}_i \bigg\|.
    \end{align}
    So it suffices to bound the term above.  To do so, we proceed via $\eps$-net.  let $\bm{x}, \bm{y}$ be $d-$dimensional and $r$-dimensional unit vectors respectively.  Then it holds that conditional on $\bm{x}\t \bm{H}_i \bm{y}$, 
    \begin{align}
        \frac{1}{n} \sum_{i=1}^{n} c_i \bm{x}\t \bm{H}_i \bm{y} \sim \mathcal{N} \bigg( 0, \frac{r}{n^2} \sum_{i=1}^{n} \big(\bm{x}\t \bm{H}_i \bm{y} \big)^2 \bigg).
    \end{align}
    As a result, Hoeffding's inequality implies that with probability at least $1 - 2\exp(-c t^2)$,
    \begin{align}
        \bigg| \frac{1}{n}  \sum_{i=1}^{n} c_i \bm{x}\t \bm{H}_i \bm{y} \bigg| &\lesssim t \frac{\sqrt{r}}{n} \sqrt{\sum_{i=1}^{n} \big( \bm{x}\t \bm{H}_i \bm{y} \big)^2 }.
    \end{align}
    In addition, note that $\sqrt{\sum_{i=1}^{n} \big( \bm{x}\t \bm{H}_i \bm{y} \big)^2 }$ is simply the norm of a standard Gaussian random variable in $n$-dimensions. Let the event $\mathcal{E}$ be denoted
    \begin{align}
        \mathcal{E} := \bigg\{ \sqrt{\sum_{i=1}^{n} \big( \bm{x}\t \bm{H}_i \bm{y} \big)^2 } &\lesssim \sqrt{n} \bigg\}.
    \end{align}
    Then $\mathcal{E}$ holds with probability at least $1 -2 \exp( - c n)$.  Conditional on $\mathcal{E}$, with probability at least $1 - 2 \exp(- c t^2) $,
    \begin{align}
         \bigg| \frac{1}{n}  \sum_{i=1}^{n} c_i \bm{x}\t \bm{H}_i \bm{y} \bigg| &\lesssim t \frac{\sqrt{r}}{\sqrt{n}} 
    \end{align}
    By taking the supremum over a $1/4$-net, we have that by redefining $t = C( \sqrt{d} + t )$, 
    \begin{align}
        \bigg\| \frac{1}{n} \sum_{i=1}^{n} c_i \bm{H}_i \bigg\| &\lesssim  \sqrt{\frac{r}{n}} \bigg( \sqrt{d} + t \bigg)
    \end{align}
    with probability at least $1 - 2 \exp( - c d t^2) - 2 \exp( C_1 d - c n )$.  Under the assumption $n \gg dr$, this probability is at least $1 - 2 \exp( -c d t^2) - 2 \exp( -c n)$.    
    \item \textbf{The quantity $T_2$:} We note that $\mathbb{E} \frac{1}{n} \sum_{i=1}^{n} {\sf Tr} \big( \bm{V}\t \bm{X}_i \bm{V} \big) \bm{VV}\t \bm{X}_i \bm{V} = \bm{VV}\t$.  Let $\bm{\tilde H}_i$ denote the random variable $\bm{V}\t \bm{X}_i \bm{V}$.  Then $\bm{\tilde H}_i$ is equal in distributionn to an $r\times r$ GOE random matrix.  We have that
    \begin{align}
        \bm{V} - \frac{1}{n} \sum_{i=1}^{n} {\sf Tr} \big( \bm{V}\t \bm{X}_i \bm{V} \big) \bm{VV}\t \bm{X}_i \bm{V}  &\overset{d}{=} \bm{V} -  \frac{1}{n} \sum_{i=1}^{n} {\sf Tr} \big( \bm{\tilde H}_i \big) \bm{V} \bm{\tilde H}_i \\
        &= \bm{V} \bigg( \bm{I} -   \frac{1}{n} \sum_{i=1}^{n} {\sf Tr} \big( \bm{\tilde H}_i \big)  \bm{\tilde H}_i \bigg).
    \end{align}
    Consequently,
    \begin{align}
        \| \bm{V} \bigg( \bm{I} -   \frac{1}{n} \sum_{i=1}^{n} {\sf Tr} \big( \bm{\tilde H}_i \big)  \bm{\tilde H}_i \bigg) \big\| &\leq \| \bm{I} - \frac{1}{n} \sum_{i=1}^{n} {\sf Tr} \big( \bm{\tilde H}_i \big)  \bm{\tilde H}_i \|,
    \end{align}
    since $\|\bm{V} \| = 1$.  Let $\bm{x}$ by any fixed $r$-dimensional unit vector.  Then
    \begin{align}
       \bm{x}\t \bigg( \bm{I} - \frac{1}{n} \sum_{i=1}^{n} {\sf Tr} \big( \bm{\tilde H}_i \big)  \bm{\tilde H}_i  \bigg) \bm{x} &= \frac{1}{n} \sum_{i=1}^{n} \sum_{l=1}^{r} x_l^2 \bigg( (\bm{\tilde H}_i)_{ll}^2 - 1 \bigg) + \frac{1}{n} \sum_{i=1}^{n} \sum_{l=1}^r \sum_{j \neq l} x_j^2 \big( \bm{\tilde H}_i \big)_{ll} \big( \bm{\tilde H}_i \big)_{jj} \\
       &\quad + \frac{1}{n} \sum_{i=1}^{n} \sum_{l=1}^{r} \sum_{j < k,j\neq l} 2 \big( \bm{\tilde H}_i \big)_{ll} \big( \bm{\tilde H}_i \big)_{jk} \bm{x}_j \bm{x}_k \\
       &=: J_1 + J_2 + J_3.
    \end{align}
We bound each of these quantities separately.
\begin{itemize}
    \item \textbf{The term $J_1$:} We note that $J_1$ is simply a sum of $nr$ independent rescaled $\chi^2$ random variables.   By Lemma 1 of \citet{laurent_adaptive_2000} it holds that
    \begin{align}
        J_1 \leq 2\frac{\sqrt{\sum_{l=1}^{r} x_l^4}}{n} t + 4\frac{r}{n} t^2 \leq 2 \frac{t}{n} + 4 \frac{r}{n} t^2 
    \end{align}
    with probability at least $1 - \exp(- t^2)$. 
    \item \textbf{The term $J_2$:} Note that we can write
    \begin{align*}
        J_2 &= \sum_{\substack{j=1,l=1\\j\neq l}}^r \bm{A}_{lj} \langle \bm{\tilde H}_l, \bm{\tilde H}_j \rangle,
    \end{align*}
    where $\bm{\tilde H}_l$ is the $n$-dimensional vector with coordinates given by $(\bm{\tilde H}_i)_{ll}$, and $\bm{A}_{lj} = \frac{x_j^2}{n}$.
Note that $\bm{A}$ is the matrix whose rows are all given by $[\frac{x_1^2}{n}, \frac{x_2^2}{n} \cdots \frac{x_r^2}{n} ]$.  It is then evident that $\|\bm{A}\| = \frac{1}{n}$ and 
 \begin{align*}
     \| \bm{A} \|_F^2 &= r \sum_j \frac{x_j^4}{n^2} \leq \frac{r}{n^2}.
 \end{align*}
        By Proposition 2.5 of \citet{chen_hansonwright_2021}, it holds that
    \begin{align*}
        \p\big\{ | J_2 | > C t \big\} &\leq 2\exp\bigg( - C n\min\bigg[ \frac{ t^2}{r}, t \bigg] \bigg).
    \end{align*}
Redefining $t =  t \sqrt{\frac{r}{n}}$, we obtain that
    \begin{align}
        |J_2| \leq C t \sqrt{\frac{r}{n}}
    \end{align}
    with probability at least  $1 - 2 \exp( - c t^2)$ for all $t \leq \sqrt{nr}$.  
    \item \textbf{The term $J_3$:} We note that
\begin{align}
J_3 &= \frac{1}{n} \sum_{i=1}^n \sum_{l=1}^{r} 2 \big( \bm{\tilde H}_i \big)_{ll} \sum_{j < k,j\neq l} \big( \bm{\tilde H}_i \big)_{jk} \bm{x}_j \bm{x}_k \\
&= \frac{1}{n} \sum_{i=1}^{n} \sum_{j < k, j \neq l} 2 \big( \bm{\tilde H}_i \big)_{jk} x_j x_k y_i,
\end{align}
where $y_i =  \sum_{l=1}^{r} \big( \bm{\tilde H}_i \big)_{ll}.$  Conditional on $y$, Hoeffding's inequality implies that with probability at least $1 - 2 \exp( - c t^2)$, 
\begin{align*}
\bigg| \frac{1}{n} \sum_{i=1}^{n} \sum_{j < k, j \neq l} 2 \big( \bm{\tilde H}_i \big)_{jk} x_j x_k y_i \bigg| \leq t \frac{C}{n} \sqrt{\sum_{i=1}^{n} y_i^2}.
\end{align*}
Observe that each $y_i$ is a Gaussian random variable with variance $r$.  Consequently, with probability at least $1 - 2\exp( - c n) - 2 \exp( - C  t^2) $ it holds that 
\begin{align*}
    t \frac{C}{n} \sqrt{\sum_{i=1}^{n} y_i^2} \leq C \frac{t \sqrt{r}}{\sqrt{n}}.
\end{align*}
\end{itemize}
Combining all these bounds, we obtain that with probability at least $1 - O\bigg(   \exp( - c t^2) - 2 \exp( - c n) \bigg)$ for all $t \leq \sqrt{nr}$,
\begin{align*}
    |J_1 | + |J_2| + |J_3| \lesssim t \sqrt{\frac{r}{n}} + \frac{t}{n} +  \frac{r}{n} t^2 
\end{align*}
Define $t = C ( \sqrt{r} + s)$.  Then a net argument shows that for all $s \leq c \sqrt{nr}$,
\begin{align}
    |J_1| + |J_2| + |J_3| \lesssim ( \sqrt{r} + s ) \sqrt{\frac{r}{n}} + \frac{\sqrt{r}  + s }{n} + \frac{r}{n} ( r + s^2),
\end{align}
which holds with probability at least
\begin{align*}
    1 - \exp ( - c s ^2) - 2 \exp( - c n),
\end{align*}
where we have adjusted constants as necessary, using the fact that $r \leq n$ by assumption.    
\end{itemize}
Combining the bounds for $T_1$ and $T_2$, we see that with probability at least $1 - O\bigg( \exp( -c s^2) - \exp( - c n) - \exp( - d t^2) - \exp( - c n) \bigg)$ it holds that for any fixed $\bm{V}$,
\begin{align*}
    \bigg\| \bigg( \mathcal{X}\s \mathcal{X} - \mathcal{I} \bigg) ( \bm{VV}\t ) \bigg\| \lesssim  ( \sqrt{r} + s ) \sqrt{\frac{r}{n}} + \frac{\sqrt{r}  + s }{n} + \frac{r}{n} ( r + s^2) + \sqrt{\frac{r}{n}} \bigg( \sqrt{d} + t \bigg),
\end{align*}
provided $s \leq c \sqrt{nr}$. Let $\{\mathcal{V}_{\eps} \}$ denote an $\eps$-net of the set of matrices $\bm{V} \in \mathbb{R}^{d \times r}$ with respect to the metric $d(\bm{V}_1, \bm{V}_2) = \| \bm{V}_1 \bm{V}_1\t - \bm{V}_2 \bm{V}_2\t \|$, which has cardinality at most $\big(\frac{4+\eps}{\eps} \big)^{dr}$.  Let
\begin{align}
    \bm{\bar V} := \argmax_{\bm{V}} \bigg\| \bigg( \mathcal{X}\s \mathcal{X} - \mathcal{I} \bigg)  \big( \bm{VV}\t \big) \bigg\| .
\end{align}
and set
\begin{align}
    M := \max \bigg\| \bigg( \mathcal{X}\s \mathcal{X} - \mathcal{I} \bigg)  \big( \bm{VV}\t \big) \bigg\| .
\end{align}
Let $\bm{\bar V}_{\eps}$ denote the quantity in the $\eps$-net that is $\eps$-close to $\bm{\bar V}$.  Then it holds that 
\begin{align}
M &= \bigg\| \bigg( \mathcal{X}\s \mathcal{X} - \mathcal{I} \bigg)  \big( \bm{ \bar V \bar V}\t \big) \bigg\| \\
&\leq  \bigg\| \bigg( \mathcal{X}\s \mathcal{X} - \mathcal{I} \bigg)  \big( \bm{ \bar V \bar V}\t - \bm{\bar V}_{\eps} \bm{\bar V}_{\eps}\t \big) \bigg\| \\
    &\quad + \bigg\| \bigg( \mathcal{X}\s \mathcal{X} - \mathcal{I} \bigg)  \big(  \bm{\bar V}_{\eps} \bm{\bar V}_{\eps}\t \big) \bigg\| \\
    &\leq  M \eps + C ( \sqrt{r} + s ) \sqrt{\frac{r}{n}} + \frac{\sqrt{r}  + s }{n} + \frac{r}{n} ( r + s^2) + \sqrt{\frac{r}{n}} \bigg( \sqrt{d} + t \bigg)
\end{align}
which implies that
\begin{align}
    M \lesssim ( \sqrt{r} + s ) \sqrt{\frac{r}{n}} + \frac{\sqrt{r}  + s }{n} + \frac{r}{n} ( r + s^2) + \sqrt{\frac{r}{n}} \bigg( \sqrt{d} + t \bigg)
\end{align}
with probability at least $$1 - O\bigg( \exp( dr \log( (4+\eps)/\eps) -c s^2) - \exp( dr \log( (4+\eps)/\eps) - c n) - \exp( dr \log( (4+\eps)/\eps) - d t^2) \bigg) .$$  Taking $\eps = \frac{1}{4}$ and $t, s \asymp \sqrt{dr}$ (and further noting that $s \leq c \sqrt{nr}$) shows that with the advertised probability,
\begin{align*}
    M \lesssim \frac{r \sqrt{d}}{\sqrt{n}} + \frac{\sqrt{dr}}{n} + \frac{d r^2}{n}  + \sqrt{\frac{r}{n}} \sqrt{dr} \lesssim \frac{r \sqrt{d}}{\sqrt{n}},
\end{align*}
where the final bound holds by \cref{ass2}, since $n \geq C_3 d r^2$.  
\end{proof}

\subsection{Proof of \cref{lem:convexity}} \label{sec:landscapeproof}
The proof of this result relies on the following lemmas.  The first lemma provides the form of the second derivative.
\begin{lemma} \label{lem:secondderivcalc2}
    It holds that
    \begin{align}
      \nabla^2 f^{(0)}_{{\sf ncvx}}(\bm{U})[ \bm{V}, \bm{V}] &= \frac{1}{\sqrt{n}} \sum_{i=1}^{n}\frac{1}{\sqrt{n}} \bigg \langle \bm{X}_i, \bm{U V}\t + \bm{VU}\t  \bigg \rangle \bigg \langle \bm{X}_i \bm{U}, \bm{V} \bigg \rangle + \frac{1}{\sqrt{n}}\bigg \langle \bm{X}_i , \bm{UU}\t - \bm{M} \bigg \rangle \bigg \langle \bm{X}_i \bm{V}, \bm{V} \bigg \rangle  \\
    &\quad - \eps_i \bigg \langle \bm{X}_i \bm{V}, \bm{V} \bigg \rangle 
    \end{align}
\end{lemma}
\begin{proof}
    See \cref{sec:secondderivcalc2proof}.
\end{proof}

 Define the objective function
    \begin{align}
        f^{{\sf MF}}(\bm{U}) &= \frac{1}{4} \| \bm{UU}\t - \bm{M} \|_F^2.
    \end{align}
    A straightforward calculation shows that
\begin{align}
    \nabla^2 f^{{\sf MF}}(\bm{U})[\bm{V},\bm{V}] &=  .5 \| \bm{UV}\t + \bm{UV}\t \|_F^2 + \bigg\langle \bm{UU}\t - \bm{M}, \bm{VV}\t \bigg \rangle.
\end{align}
To prove \cref{lem:convexity} we first decompose $\nabla^2 f^{(0)}_{{\sf ncvx}}(\bm{U})[\bm{V}, \bm{V}]$ via
\begin{align}
    \nabla^2 f^{(0)}_{{\sf ncvx}} (\bm{U})[ \bm{V}, \bm{V}] &= \nabla^2 f^{{\sf MF}} (\bm{U})[ \bm{V}, \bm{V}] + \bigg( \nabla^2 f^{{\sf MF}} (\bm{U})[ \bm{V}, \bm{V}] - \nabla^2 f^{(0)}_{{\sf ncvx}} (\bm{U})[ \bm{V}, \bm{V}] \bigg).
\end{align}
The following lemma shows that the second term is sufficiently small with high probabiltiy.
\begin{lemma} \label{lem:secondderivcloseness}
      The following bound holds uniformly for all $\bm{V}$ and $\bm{U}$ of rank at most $2r$ with probability at least $1 - O\big( \exp( - c dr ) + \exp( - c d) \big) $:
    \begin{align}
        \bigg| \nabla^2 f^{(0)}_{{\sf ncvx}}(\bm{U})[ \bm{V}, \bm{V}] - \nabla^2 f^{{\sf MF}}(\bm{U})[ \bm{V}, \bm{V}] \bigg| \leq \| \bm{V}\|_F^2 \bigg( 2 \delta_{2r} \| \bm{U} \|^2 + \delta_{2r} \| \bm{UU}\t \|_F + \delta_{2r} \| \bm{M} \|_F + C \sigma \sqrt{dr} \bigg),
    \end{align}
    where $\delta_{2r}$ is as in \cref{lem:Egood}.  
\end{lemma}
\begin{proof}
    See \cref{sec:secondderivclosenessproof}.
\end{proof}

Finally, the next lemma provides a deterministic bound on $\nabla^2 f^{(0)}_{{\sf ncvx}}(\bm{U})[\bm{V},\bm{V}]$ for appropriate choice of $\bm{U}$ on the event from \cref{lem:secondderivcloseness}.

\begin{lemma} \label{lem:convexitydeterministic}
   Suppose  $\mathcal{E}_{\cref{lem:secondderivcloseness}}$ holds.  
  Let $\bm{U} \in \mathbb{R}^{d\times r}$, and let $\bm{V}= \bm{U} - \bm{U}' \mathcal{O}_{\bm{U}',\bm{U}}$.  Then 
    it holds that
\begin{align}
    \nabla^2 f^{(0)}_{{\sf ncvx}} (\bm{U})[ \bm{V},\bm{V}] \geq  \lambda_{\min}(\bm{U})^2\| \bm{V}\|_F^2 \bigg( 1  - \frac{\| \bm{UU}\t - \bm{M}\| + 2 \delta_{2r} \|\bm{U}\|^2 + \delta_{2r} \| \bm{UU}\t \|_F + \delta_{2r} \| \bm{M} \| + C \sigma \sqrt{dr}}{\lambda_{\min}(\bm{U})^2} \bigg)
\end{align}
and 
\begin{align}
    \bigg| \nabla^2 f^{(0)}_{{\sf ncvx}}(\bm{U})[\bm{V},\bm{V}] \bigg| \leq \|\bm{V}\|_F^2 \bigg( \| \bm{U}\|^2 + \| \bm{UU}\t - \bm{M} \| + 2 \delta_{2r} \| \bm{U}\|^2 + \delta_{2r} \| \bm{UU}\t \|_f + \delta_{2r} \| \bm{M} \|_F + C \sigma \sqrt{dr} \bigg).
\end{align}
\end{lemma}
\begin{proof}
    See \cref{sec:convexitydeterministicproof}.
\end{proof}

With these lemmas in hand we are prepared to prove \cref{lem:convexity}. 

\begin{proof}[Proof of \cref{lem:convexity}]
We observe that 
\begin{align}
    {\sf vec}(\bm{U}^{(0)} - \bm{U}_{{\sf deb}}^{(\lambda)} )\t \nabla^2 f^{(0)}_{{\sf ncvx}}\big( t \bm{U}^{(0)} + (1 - t) \bm{U}^{(\lambda)}_{{\sf deb}} \big){\sf vec}(\bm{U}^{(0)} - \bm{U}_{{\sf deb}}^{(\lambda)} ) &= \nabla^2 f^{(0)}_{{\sf ncvx}} \bigg( t \bm{U}^{(0)} + (1-t) \bm{U}_{{\sf deb}}^{(\lambda)} \bigg)[ \bm{V}, \bm{V}],
\end{align}
with $\bm{V} = \bm{U}^{(0)} - \bm{U}_{{\sf deb}}^{(\lambda)},$ where we have already assumed (without loss of generality) that $\mathcal{O}_{\bm{U}^{(0)}, \bm{U}_{{\sf deb}}^{(\lambda)}} = \bm{I}_r$.  Let $\bm{U} = t \bm{U}^{(0)} + (1 - t) \bm{U}^{(\lambda)}_{{\sf deb}}$. By \cref{lem:convexitydeterministic}, on the event $\mathcal{E}_{\cref{lem:secondderivcloseness}}$ it holds that
\begin{align}
     \nabla^2 f^{(0)}_{{\sf ncvx}} \big( \bm{U} \big)[ \bm{V}, \bm{V}] \geq \lambda_{\min}(\bm{U})^2\| \bm{V}\|_F^2 \bigg( 1  - \frac{\| \bm{UU}\t - \bm{M}\| + 2 \delta_{2r} \|\bm{U}\|^2 + \delta_{2r} \| \bm{UU}\t \|_F + \delta_{2r} \| \bm{M} \| + C \sigma \sqrt{dr}}{\lambda_{\min}(\bm{U})^2} \bigg).
\end{align}
To complete the proof of the lower bound we need only bound the numerator above.  
\begin{itemize}
    \item \textbf{Step 1: Studying $\bm{U}^{(0)}$ and $\bm{U}_{{\sf deb}}^{(\lambda)}$}: Through an argument similar to the proof of \cref{lem:uncvxgood} (take $\lambda = 0$ therein) shows that on the event $\mathcal{E}_{{\sf Good}}$
\begin{align}
    \| \bm{U}^{(0)} \bm{U}^{(0)\top} - \bm{M} \|_F \leq C \sigma \sqrt{dr}.
\end{align}
Therefore, Weyl's inequality shows that
\begin{align}
    \lambda_{\min}(\bm{U}^{(0)})^2 &= \lambda_{\min}(\bm{U}^{(0)} \bm{U}^{(0)\top}) \geq \lambda_{\min}(\bm{M}) - C \sigma \sqrt{dr} \geq \frac{\lambda_{\min}(\bm{M})}{2},
\end{align}
which holds by \cref{ass1}.  

We now consider $\bm{U}_{{\sf deb}}^{(\lambda)}$.  We have that by \cref{thm:udebzdeb} it holds that  on the event $\mathcal{E}_{{\sf Good}}$
\begin{align}
    \bm{U}_{{\sf deb}}^{(\lambda)} \bm{U}_{{\sf deb}}^{(\lambda)\top} = \bm{Z}_{{\sf deb}}^{(\lambda)} = \bm{V}_U\big( \bm{\Sigma} + \lambda \bm{I} \big) \bm{V}_U\t.
\end{align}
Furthermore, on this same event, by \cref{thm:relatecvxncvx} it holds that
\begin{align}
    \| \bm{Z}^{(\lambda)} - \bm{M} \|_F \leq C \sigma \sqrt{dr}.
\end{align}
Consequently, it holds that
\begin{align}
    \|\bm{U}_{{\sf deb}}^{(\lambda)} \bm{U}_{{\sf deb}}^{(\lambda)\top} - \bm{M}\|_F \leq C \sigma \sqrt{dr}.  
\end{align}
Therefore, the same argument shows that $\lambda_{\min}(\bm{U}_{{\sf deb}}^{(\lambda)})^2 \geq \frac{\lambda_{\min}(\bm{M})}{2}$.  The triangle inequality also shows that 
\begin{align}
    \|\bm{U}_{{\sf deb}}^{(\lambda)}\|^2 \leq \| \bm{M} \| +   \|\bm{U}_{{\sf deb}}^{(\lambda)} \bm{U}_{{\sf deb}}^{(\lambda)\top} - \bm{M}\|_F \leq 2 \| \bm{M} \|.
\end{align}
The same bound holds for $\bm{U}^{(0)}$ as well.  
    \item \textbf{Step 2: Bounding $\bm{U}$.}  As a consequence of the previous step, we have that 
\begin{align}
    \lambda_{\min}(\bm{U})^2 &= \lambda_{\min}\bigg( t \bm{U}^{(0)} + (1 - t) \bm{U}^{(\lambda)}_{{\sf deb}} \bigg)^2 \\
    &= \lambda_{\min} \bigg( t^2 \bm{U}^{(0)\top} \bm{U}^{(0)} + (1 - t)^2 \bm{U}_{{\sf deb}}^{(\lambda)\top} \bm{U}_{{\sf deb}}^{(\lambda)} + t(1-t) \bm{U}^{(0)\top}\bm{U}_{{\sf deb}}^{(\lambda)}+ t(1-t)\bm{U}_{{\sf deb}}^{(\lambda)\top}\bm{U}^{(0)} \bigg) \\
    &\geq  \frac{t^2}{2} \lambda_{\min}(\bm{M}) + \frac{(1 - t)^2}{2} \lambda_{\min}(\bm{M}) + 2 t(1-t) \lambda_{\min}(\bm{U}_{{\sf deb}}^{(\lambda)\top}\bm{U}^{(0)}) \\
    &\geq \frac{t^2 + (1-t)^2}{2} \lambda_{\min}(\bm{M}) \\
    &\geq \frac{\lambda_{\min}(\bm{M})}{8},
\end{align}
where we have recognized that $\bm{U}_{{\sf deb}}^{(\lambda)\top}\bm{U}^{(0)}$ is a positive  diagonal matrix by the fact that they are already aligned.  A similar argument gives
\begin{align}
    \| \bm{U} \|^2 \leq 4 \| \bm{M} \|.
\end{align}
    By Lemma 6 of \citet{ge_no_2017} it holds that
\begin{align}
    \| \bm{U}^{(0)} \mathcal{O}_{\bm{U}^{(0)},\bm{U}\s} - \bm{U}\s \|_F \leq \frac{1}{2 \lambda_{\min}^{1/2}(\bm{M})} \| \bm{U}^{(0)} \bm{U}^{(0)\top} - \bm{M} \|_F \leq \frac{C \sigma \sqrt{dr}}{2\lambda_{\min}^{1/2}(\bm{M})}.
    \end{align}
    \item \textbf{Step 3: Bounding $\bm{UU}\t - \bm{M}$.}
    First, define the set
    \begin{align}
        \mathbb{B}_x(\bm{U}\s) := \{ \bm{U}: \| \bm{U} \mathcal{O}_{\bm{U},\bm{U}\s} - \bm{U}\s \|_F \leq x \}.
    \end{align}
    We note that  $\mathbb{B}_x(\bm{U}\s)$ is the geodesic ball on the Riemannian quotient space $[\bm{U}] := \{ \bm{U} \mathcal{O}: \mathcal{OO}\t = \bm{I}_r \}$. By Theorem 2 of \citet{luo_nonconvex_2022}, the set $\mathbb{B}_x(\bm{U}\s)$ is geodesically convex for any $x \leq \lambda_{\min}(\bm{U}\s)/3$. Moreover, by Lemma 2 of \citet{luo_nonconvex_2022} the function
    \begin{align}
        \bm{U}(t) := \bm{U}_{{\sf deb}}^{(\lambda)} + t \big( \bm{U}^{(0)} - \bm{U}_{{\sf deb}}^{(\lambda)} \big)
    \end{align}
    is the unique minimizing geodesic from $[\bm{U}_{{\sf deb}}^{(\lambda)}]$ to $[\bm{U}^{(0)}]$ (where we have used the fact that without loss of generality both matrices are already properly aligned).  Since both $\bm{U}^{(0)}$ and $\bm{U}_{{\sf deb}}^{(\lambda)}$ are full rank from the previous arguments, it holds that $[\bm{U}] \in \mathbb{B}_x(\bm{U}\s)$.  In particular, it holds that
    \begin{align}
        \| \bm{U}(t) \mathcal{O}_{\bm{U},\bm{U}\s} - \bm{U}\s \|_F \leq \frac{C \sigma \sqrt{dr}}{2\lambda_{\min}^{1/2}(\bm{M})}.
    \end{align}
     for all $t$.  Lemma 11 of \citet{luo_nonconvex_2022} shows that
     \begin{align}
         \| \bm{UU}\t - \bm{M} \| \leq \sqrt{\kappa} C \sigma \sqrt{dr} \leq \frac{\lambda_{\min}(\bm{M})}{10},
     \end{align}
     where the final bound follows from \cref{ass1}.  
     \end{itemize}
Combining these bounds we have that on the event $\mathcal{E}_{{\sf Good}}$ 
\begin{align}
\| \bm{UU}\t - \bm{M} \| + 2 \delta_{2r} \| \bm{U} \|^2 + \delta_{2r} \| \bm{UU}\t\|_F + \delta_{2r} \| \bm{M} \| + C \sigma \sqrt{dr}  &\leq \sqrt{\kappa} C \sigma \sqrt{dr} + 8 \frac{c_0}{\sqrt{r}} \| \bm{M} \| + 4 c_0 \| \bm{M}\| + C \sigma \sqrt{dr} \\
&\leq C' \sigma \sqrt{dr} + C' c_0 \| \bm{M} \| \\
&\leq \frac{\lambda_{\min}(\bm{M})}{20}
\end{align}
provided $c_0$ is smaller than some constant, where $c_0$ is such that $\delta_{2r} \leq \frac{c_0}{\sqrt{r}}$.  Therefore,
\begin{align}
    \nabla^2 f(\bm{U})[\bm{V},\bm{V}] \geq\frac{\lambda_{\min}(\bm{M})}{8} \|\bm{V} \|_F^2 \big( 1 - \frac{2}{5} \big) \geq \frac{\lambda_{\min}(\bm{M})}{20} \|\bm{V} \|_F^2 
\end{align}
as desired.  The upper bound follows from a similar argument.  
 
\end{proof}

\subsubsection{Proof of \cref{lem:secondderivcalc2}} \label{sec:secondderivcalc2proof}
\begin{proof}[Proof of \cref{lem:secondderivcalc2}]
    The hessian satisfies
\begin{align}
    \nabla^2 f^{(0)}_{{\sf ncvx}}(\bm{U})[\bm{V},\bm{V}] &= \frac{d}{dt} \frac{1}{\sqrt{n}} \sum_{i=1}^{n} \bigg[ \frac{1}{\sqrt{n}}\bigg\langle \bm{X}_i, \big( \bm{U} + t \bm{V} \big) \big( \bm{U} + t \bm{V} \big)\t - \bm{M} \bigg \rangle - \eps_i \bigg] \bigg \langle \bm{X}_i \big( \bm{U} + t \bm{V} \big), \bm{V} \bigg\rangle .
\end{align}
We have that
\begin{align}
     \bigg[ \frac{1}{\sqrt{n}}\bigg\langle &\bm{X}_i, \big( \bm{U} + t \bm{V} \big) \big( \bm{U} + t \bm{V} \big)\t - \bm{M} \bigg \rangle - \eps_i \bigg] \bigg \langle \bm{X}_i \big( \bm{U} + t \bm{V} \big), \bm{V} \bigg\rangle \\
     &=  \bigg[ \frac{1}{\sqrt{n}}\bigg\langle \bm{X}_i, \big( \bm{U} + t \bm{V} \big) \big( \bm{U} + t \bm{V} \big)\t - \bm{M} \bigg \rangle \bigg] \bigg \langle \bm{X}_i \big( \bm{U} + t \bm{V} \big), \bm{V} \bigg\rangle  - \eps_i\bigg \langle \bm{X}_i \big( \bm{U} + t \bm{V} \big), \bm{V} \bigg\rangle \\
     &=  \bigg[ \frac{1}{\sqrt{n}}\bigg\langle \bm{X}_i, \big( \bm{U} + t \bm{V} \big) \big( \bm{U} + t \bm{V} \big)\t - \bm{M} \bigg \rangle \bigg] \bigg \langle \bm{X}_i  \bm{U} , \bm{V} \bigg\rangle \\
     &\quad+ t \bigg[ \frac{1}{\sqrt{n}}\bigg\langle \bm{X}_i, \big( \bm{U} + t \bm{V} \big) \big( \bm{U} + t \bm{V} \big)\t - \bm{M} \bigg \rangle \bigg] \bigg \langle \bm{X}_i  \bm{V} , \bm{V} \bigg\rangle \\
     &\quad  - \eps_i\bigg \langle \bm{X}_i \big( \bm{U} + t \bm{V} \big), \bm{V} \bigg\rangle
     \\
     &=  \bigg[ \frac{1}{\sqrt{n}}\bigg\langle \bm{X}_i,  \bm{U}  \big( \bm{U} + t \bm{V} \big)\t - \bm{M} \bigg \rangle \bigg] \bigg \langle \bm{X}_i  \bm{U} , \bm{V} \bigg\rangle \\
&\quad +     t \bigg[\frac{1}{\sqrt{n}} \bigg\langle \bm{X}_i,   \bm{V}  \big( \bm{U} + t \bm{V} \big)\t \bigg \rangle \bigg] \bigg \langle \bm{X}_i  \bm{U} , \bm{V} \bigg\rangle \\
     &\quad+ t \bigg[ \frac{1}{\sqrt{n}}\bigg\langle \bm{X}_i, \big( \bm{U} + t \bm{V} \big) \big( \bm{U} + t \bm{V} \big)\t - \bm{M} \bigg \rangle \bigg] \bigg \langle \bm{X}_i  \bm{V} , \bm{V} \bigg\rangle \\
     &\quad  - \eps_i\bigg \langle \bm{X}_i \big( \bm{U} + t \bm{V} \big), \bm{V} \bigg\rangle
\end{align}
Taking only terms linear in $t$ yields that
\begin{align}
    \nabla^2 f^{(0)}_{{\sf ncvx}}(\bm{U})[ \bm{V}, \bm{V}] &= \frac{1}{\sqrt{n}} \sum_{i=1}^{n}\frac{1}{\sqrt{n}} \bigg \langle \bm{X}_i, \bm{U W}\t + \bm{VU}\t  \bigg \rangle \bigg \langle \bm{X}_i \bm{U}, \bm{V} \bigg \rangle + \frac{1}{\sqrt{n}}\bigg \langle \bm{X}_i , \bm{UU}\t - \bm{M} \bigg \rangle \bigg \langle \bm{X}_i \bm{V}, \bm{V} \bigg \rangle  \\
    &\quad - \eps_i \bigg \langle \bm{X}_i \bm{V}, \bm{V} \bigg \rangle 
\end{align}
as desired. 
\end{proof}

\subsubsection{Proof of \cref{lem:secondderivcloseness}}
\label{sec:secondderivclosenessproof}
\begin{proof}[Proof of \cref{lem:secondderivcloseness}]
    First, we write 
    \begin{align}
        \nabla^2 f^{(0)}_{{\sf ncvx}}(\bm{U})[\bm{V}, \bm{V}] :=   \nabla^2 f^{({\sf clean})}(\bm{U})[\bm{V}, \bm{V}]
     - \frac{1}{\sqrt{n}} \sum_{i=1}^{n}\eps_i \bigg \langle \bm{X}_i \bm{V}, \bm{V} \bigg \rangle, 
    \end{align}
    where
    \begin{align}
     \nabla^2 f^{({\sf clean})}(\bm{U})[\bm{V}, \bm{V}]   &= \frac{1}{n} \sum_{i=1}^{n}  \bigg \langle \bm{X}_i, \bm{U V}\t + \bm{VU}\t  \bigg \rangle \bigg \langle \bm{X}_i \bm{U}, \bm{V} \bigg \rangle +  \bigg \langle \bm{X}_i , \bm{UU}\t - \bm{M} \bigg \rangle \bigg \langle \bm{X}_i \bm{V}, \bm{V} \bigg \rangle.
    \end{align}
First we bound the difference between $\nabla^2 f^{{\sf clean}}$ and $\nabla^2 f^{{\sf MF}}$, uniformly over matrices of rank at most $2r$.  Lemma 3 of \citet{chi_nonconvex_2019} shows that if $\mathcal{X}$ satisfies the restricted isometry property with constant $\delta_{2r}$, then 
\begin{align}
\bigg|    \frac{1}{n} \sum_{i=1}^{n} \langle \bm{X}_i, \bm{Q}_1 \rangle \langle \bm{X}_i, \bm{Q}_2 \rangle - \langle \bm{Q}_1, \bm{Q}_2 \rangle \bigg| &\leq \delta_{2r} \| \bm{Q}_1 \|_F \| \bm{Q}_2\|_F,
\end{align}
uniformly over matrices $\bm{Q}_1, \bm{Q}_2$ of rank at most $r$.  Consequently, when the restricted isometry property holds, by breaking up the sum it holds that
\begin{align}
    \bigg| \nabla^2 f^{{(\sf clean)}}(\bm{U})[\bm{V}, \bm{V}] - \nabla^2 f^{{\sf MF}}(\bm{U})[\bm{V}, \bm{V}] \bigg| &\leq \delta_{2r} \bigg( 2\| \bm{UV}\t \|_F^2 + \| \bm{UU}\t \|_F \| \bm{VV}\t \|_F + \| \bm{M} \|_F \| \bm{VV}\t \|_F \bigg) \\
    &\leq \delta_{2r} \| \bm{V}\|_F^2 \bigg( 2 \| \bm{U} \|^2 + \| \bm{UU}\t \|_F + \| \bm{M} \|_F \bigg).
\end{align}
We now bound the noise term.  Observe that for fixed matrix $\bm{VV}\t$ of rank at most $r$,   it holds that 
\begin{align}
    \frac{1}{\sqrt{n}} \sum_{i=1}^{n} \eps_i \bigg \langle \bm{X}_i ,\bm{V} \bm{V}\t \bigg \rangle &\overset{d}{=} \frac{ \| \bm{V} \bm{V}\t \|_F}{\sqrt{n}} \langle \bm{\eps}, \bm{Z} \rangle,
\end{align}
where $\bm{Z}$ and $\bm{\eps}/\sigma $ are standard Gaussian random vectors.  Standard concentration arguments show that for fixed matrix $\bm{VV}\t$, conditional on $\bm{\eps}$, 
\begin{align}
  \bigg|  \frac{1}{\sqrt{n}} \sum_{i=1}^{n}  \eps_i \bigg \langle \bm{X}_i ,\bm{V} \bm{V}\t \bigg \rangle \bigg| &\leq C \sigma \| \bm{VV}\t \|_F \frac{\| \bm{\eps} \|}{\sqrt{n}} s
\end{align}
with probability at least $1 - 2\exp(-c s^2)$.  Recognizing that $\| \bm{\eps} \|/\sqrt{n} \leq C$ with probability at least $1 - 2 \exp( - c n)$ shows that for any fixed $\bm{V}$ it holds that 
\begin{align}
     \bigg|  \frac{1}{\sqrt{n}} \sum_{i=1}^{n}  \eps_i \bigg \langle \bm{X}_i ,\bm{V} \bm{V}\t \bigg \rangle \bigg| &\leq C \sigma  \| \bm{VV}\t \|_F  s
\end{align}
with probability at least $1 - 2 \exp(- c s^2) - 2 \exp( - c n)$.  Taking $s \asymp \sqrt{dr}$ and a standard net argument shows that this bound holds uniformly over matrices $\bm{VV}\t$ of rank at most $r$ with probability  at least $1 - O( \exp( - c dr ) )$.  
\end{proof}

\subsubsection{Proof of \cref{lem:convexitydeterministic}} \label{sec:convexitydeterministicproof}

\begin{proof}[Proof of \cref{lem:convexitydeterministic}]
  We have that 
   \begin{align}
       \nabla^2 f^{(0)}_{{\sf ncvx}}(\bm{U})[\bm{V}, \bm{V}] &= \underbrace{ .5\| \bm{UV}\t + \bm{VU}\t \|_F^2 }_{=:T_1} + \underbrace{ \langle \bm{UU}\t - \bm{M}, \bm{VV}\t\rangle}_{=:T_2} + 
       \underbrace{\nabla^2 f^{(0)}_{{\sf ncvx}}(\bm{U})[ \bm{V},\bm{V}] - \nabla^2 f^{{(\sf MF)}}(\bm{U})[\bm{V},\bm{V}]}_{=:T_3}.
   \end{align}
We will upper and lower bound $T_1$ and upper bound $T_2$ and $T_3$.  
\begin{itemize}
    \item \textbf{The Term $T_1$}: The upper bound 
    \begin{align}
        |T_1| \leq \| \bm{U} \|^2 \| \bm{V} \|_F^2
    \end{align}
    is immediate.  For the lower bound, we have
    \begin{align}
         T_1 &= .5 \| \bm{U} \bm{V}\t + \bm{V} \bm{U}^{\top} \|_F^2 \\
        &= \frac{1}{2} \| \bm{U} \bm{V}\t \|_F^2 + \frac{1}{2} \| \bm{V} \bm{U}^{\top} \|_F^2 + {\sf Tr}\bigg( \bm{U}^{\top} \bm{V} \bm{U}^{\top} \bm{V} \bigg) \\
        &\geq \lambda_{\min}(\bm{U})^2 \| \bm{V} \|_F^2 + {\sf Tr}\bigg( \bm{U}^{\top} \bm{V} \bm{U}^{\top} \bm{V} \bigg) \\
        &\geq \lambda_{\min}(\bm{U})^2 \| \bm{V} \|_F^2,
    \end{align}
    where the final inequality follows by Lemma 35 of \citet{ma_implicit_2020} as $\bm{V} = \bm{U} - \bm{U}' \mathcal{O}_{\bm{U}',\bm{U}}$.
    \item \textbf{The term $T_2$: } We have that
    \begin{align}
        |T_2| &= \langle \big(\bm{UU}\t - \bm{M}\big) \bm{V}, \bm{V} \rangle \leq \| \bm{UU}\t - \bm{M} \| \| \bm{V} \|_F^2.
    \end{align}
     \item \textbf{The term $T_3$:} Let $\mathcal{E}_{\cref{lem:secondderivcloseness}}$ denote the event from \cref{lem:secondderivcloseness}.  On this event we have that 
\begin{align}
    | T_3 |  \leq \| \bm{V}\|_F^2 \bigg( 2 \delta_{2r} \| \bm{U} \|^2 + \delta_{2r} \| \bm{UU}\t \|_F + \delta_{2r} \| \bm{M} \|_F + C \sigma \sqrt{dr} \bigg). \label{eq:t3bound}
\end{align}
\end{itemize}
Combining bounds we have that
\begin{align}
    \nabla^2 f^{(0)}_{{\sf ncvx}} (\bm{U})[ \bm{V},\bm{V}] \geq  \lambda_{\min}(\bm{U})^2\| \bm{V}\|_F^2 \bigg( 1  - \frac{\| \bm{UU}\t - \bm{M}\| + 2 \delta_{2r} \|\bm{U}\|^2 + \delta_{2r} \| \bm{UU}\t \|_F + \delta_{2r} \| \bm{M} \| + C \sigma \sqrt{dr}}{\lambda_{\min}(\bm{U})^2} \bigg).
\end{align}
For the upper bound we have that
\begin{align}
    \bigg| \nabla^2 f^{(0)}_{{\sf ncvx}}(\bm{U})[\bm{V},\bm{V}] \bigg| \leq \|\bm{V}\|_F^2 \bigg( \| \bm{U}\|^2 + \| \bm{UU}\t - \bm{M} \| + 2 \delta_{2r} \| \bm{U}\|^2 + \delta_{2r} \| \bm{UU}\t \|_f + \delta_{2r} \| \bm{M} \|_F + C \sigma \sqrt{dr} \bigg).
\end{align}
\end{proof}

\section{Proof of Corollary \ref{cor:maincor}}

\label{sec:corproof}

\begin{proof}
By \cref{thm:mainthm}, when $\frac{n}{dr^2} \to \infty$ with $\gamma_n/(dr) \to 0$ it holds that with probability tending to one
\begin{align*}
   \frac{1}{dr} \bigg| \| \bm{Z}^{(\lambda)} - \bm{M} \|_F^2 - \mathbb{E} \big\| \bm{Z}_{{\sf ST}}^{(\lambda)} - \bm{M} \|_F^2 \bigg| &= o(1), 
\end{align*}
with the same result holding for $\bm{U}^{(0)} \bm{U}^{(0)\top}$ and $\bm{Z}_{{\sf HT}}$ respectively (to make this function Lipschitz, strictly speaking we first restrict to a bounded set and consider its Lipschitz extension, but this restriction is minor).  Consequently, since the right hand side above is $o(1)$, we have in probability
\begin{align*}
\liminf \frac{\|\bm{Z}^{(\lambda)} - \bm{M}\|_F^2}{\| \bm{U}^{(0)} \bm{U}^{(0)\top} - \bm{M} \|_F^2} &\geq \frac{  \mathbb{E}_{\bm{H}} \| \bm{Z}_{{\sf ST}}^{(\lambda)} - \bm{M} \|_F^2 }{  \mathbb{E}_{\bm{H}} \| \bm{Z}_{{\sf HT}} - \bm{M} \|_F^2}.
\end{align*}
Therefore, we need only consider the mean squared error for $\bm{Z}_{{\sf HT}}$ and $\bm{Z}_{{\sf ST}}^{(\lambda)}$.  When $r$ is fixed, this is given in Lemma 4 of \citet{gavish_optimal_2014}.  Matching notation and rescaling appropriately to their asymptotic framework, it holds that
\begin{align*}
    \frac{1}{d \sigma^2 } \| \bm{Z}_{{\sf HT}} - \bm{M} \|_F^2 &\to 2r + \sum_{i=1}^{r} 3 \frac{\sigma^2}{\lambda_i^2}; \\
    \frac{1}{d \sigma^2} \| \bm{Z}_{{\sf ST}}^{(\lambda)} - \bm{M} \|_F^2 &\to 2r + \sum_{i=1}^{r} 3 \frac{\sigma^2}{\big( \lambda_i - \frac{\lambda}{\sqrt{d}} \big)^2},
\end{align*}
where the convergence is almost surely.  Since the right hand side is a constant, we conclude by the Dominated Convergence Theorem that the same limits hold for the expectations.   Since $\lambda_i > \frac{\lambda}{\sqrt{d}}$ by \cref{ass1,ass3}, it holds that $\frac{\sigma^2}{\lambda_i^2} \leq \frac{\sigma^2}{(\lambda_i - \frac{\lambda}{\sqrt{d}}})^2$, whence the result is immediate.  Strict inequality for any choices of $\lambda_i$ satisfying \cref{ass1} follows. 
\end{proof}

\bibliography{mat_sensing_bib,mat_sensing_bib2}

\begin{thebibliography}{76}
\providecommand{\natexlab}[1]{#1}
\providecommand{\url}[1]{\texttt{#1}}
\expandafter\ifx\csname urlstyle\endcsname\relax
  \providecommand{\doi}[1]{doi: #1}\else
  \providecommand{\doi}{doi: \begingroup \urlstyle{rm}\Url}\fi

\bibitem[Bandeira and Handel(2016)]{bandeira_sharp_2016}
Afonso~S. Bandeira and Ramon~van Handel.
\newblock Sharp nonasymptotic bounds on the norm of random matrices with independent entries.
\newblock \emph{The Annals of Probability}, 44\penalty0 (4):\penalty0 2479--2506, July 2016.
\newblock ISSN 0091-1798, 2168-894X.
\newblock \doi{10.1214/15-AOP1025}.

\bibitem[Bayati and Montanari(2012)]{bayati_lasso_2012}
Mohsen Bayati and Andrea Montanari.
\newblock The {LASSO} {Risk} for {Gaussian} {Matrices}.
\newblock \emph{IEEE Transactions on Information Theory}, 58\penalty0 (4):\penalty0 1997--2017, April 2012.
\newblock ISSN 1557-9654.
\newblock \doi{10.1109/TIT.2011.2174612}.

\bibitem[Bellec and Koriyama(2024)]{bellec_existence_2024}
Pierre~C. Bellec and Takuya Koriyama.
\newblock Existence of solutions to the nonlinear equations characterizing the precise error of {M}-estimators, October 2024.
\newblock arXiv:2312.13254 [math].

\bibitem[Bellec and Zhang(2023)]{bellec_debiasing_2023}
Pierre~C. Bellec and Cun-Hui Zhang.
\newblock Debiasing convex regularized estimators and interval estimation in linear models.
\newblock \emph{The Annals of Statistics}, 51\penalty0 (2):\penalty0 391--436, April 2023.
\newblock ISSN 0090-5364, 2168-8966.
\newblock \doi{10.1214/22-AOS2243}.

\bibitem[Benaych-Georges and Nadakuditi(2011)]{benaych-georges_eigenvalues_2011}
Florent Benaych-Georges and Raj~Rao Nadakuditi.
\newblock The eigenvalues and eigenvectors of finite, low rank perturbations of large random matrices.
\newblock \emph{Advances in Mathematics}, 227\penalty0 (1):\penalty0 494--521, May 2011.
\newblock ISSN 0001-8708.
\newblock \doi{10.1016/j.aim.2011.02.007}.

\bibitem[Berthier et~al.(2020)Berthier, Montanari, and Nguyen]{berthier_state_2020}
Raphaël Berthier, Andrea Montanari, and Phan-Minh Nguyen.
\newblock State evolution for approximate message passing with non-separable functions.
\newblock \emph{Information and Inference: A Journal of the IMA}, 9\penalty0 (1):\penalty0 33--79, March 2020.
\newblock ISSN 2049-8772.
\newblock \doi{10.1093/imaiai/iay021}.

\bibitem[Bhojanapalli et~al.(2016)Bhojanapalli, Neyshabur, and Srebro]{bhojanapalli_global_2016}
Srinadh Bhojanapalli, Behnam Neyshabur, and Nati Srebro.
\newblock Global {Optimality} of {Local} {Search} for {Low} {Rank} {Matrix} {Recovery}.
\newblock In \emph{Advances in {Neural} {Information} {Processing} {Systems}}, volume~29. Curran Associates, Inc., 2016.

\bibitem[Burer and Monteiro(2003)]{burer_nonlinear_2003}
Samuel Burer and Renato~D.C. Monteiro.
\newblock A nonlinear programming algorithm for solving semidefinite programs via low-rank factorization.
\newblock \emph{Mathematical Programming}, 95\penalty0 (2):\penalty0 329--357, February 2003.
\newblock ISSN 1436-4646.
\newblock \doi{10.1007/s10107-002-0352-8}.

\bibitem[Cai and Zhang(2018)]{cai_rate-optimal_2018}
T.~Tony Cai and Anru Zhang.
\newblock Rate-optimal perturbation bounds for singular subspaces with applications to high-dimensional statistics.
\newblock \emph{The Annals of Statistics}, 46\penalty0 (1):\penalty0 60--89, February 2018.
\newblock ISSN 0090-5364, 2168-8966.
\newblock \doi{10.1214/17-AOS1541}.

\bibitem[Cai et~al.(2016{\natexlab{a}})Cai, Li, and Ma]{cai_optimal_2016}
T.~Tony Cai, Xiaodong Li, and Zongming Ma.
\newblock Optimal rates of convergence for noisy sparse phase retrieval via thresholded {Wirtinger} flow.
\newblock \emph{The Annals of Statistics}, 44\penalty0 (5):\penalty0 2221--2251, October 2016{\natexlab{a}}.
\newblock ISSN 0090-5364, 2168-8966.
\newblock \doi{10.1214/16-AOS1443}.

\bibitem[Cai et~al.(2016{\natexlab{b}})Cai, Liang, and Rakhlin]{cai_geometric_2016}
T.~Tony Cai, Tengyuan Liang, and Alexander Rakhlin.
\newblock Geometric inference for general high-dimensional linear inverse problems.
\newblock \emph{The Annals of Statistics}, 44\penalty0 (4):\penalty0 1536--1563, August 2016{\natexlab{b}}.
\newblock ISSN 0090-5364, 2168-8966.
\newblock \doi{10.1214/15-AOS1426}.

\bibitem[Candès and Plan(2011)]{candes_tight_2011}
Emmanuel~J. Candès and Yaniv Plan.
\newblock Tight {Oracle} {Inequalities} for {Low}-{Rank} {Matrix} {Recovery} {From} a {Minimal} {Number} of {Noisy} {Random} {Measurements}.
\newblock \emph{IEEE Transactions on Information Theory}, 57\penalty0 (4):\penalty0 2342--2359, April 2011.
\newblock ISSN 1557-9654.
\newblock \doi{10.1109/TIT.2011.2111771}.

\bibitem[Candès et~al.(2015)Candès, Li, and Soltanolkotabi]{candes_phase_2015}
Emmanuel~J. Candès, Xiaodong Li, and Mahdi Soltanolkotabi.
\newblock Phase {Retrieval} via {Wirtinger} {Flow}: {Theory} and {Algorithms}.
\newblock \emph{IEEE Transactions on Information Theory}, 61\penalty0 (4):\penalty0 1985--2007, April 2015.
\newblock ISSN 1557-9654.
\newblock \doi{10.1109/TIT.2015.2399924}.

\bibitem[Celentano et~al.(2023)Celentano, Montanari, and Wei]{celentano_lasso_2023}
Michael Celentano, Andrea Montanari, and Yuting Wei.
\newblock The {Lasso} with general {Gaussian} designs with applications to hypothesis testing.
\newblock \emph{The Annals of Statistics}, 51\penalty0 (5), October 2023.
\newblock ISSN 0090-5364.
\newblock \doi{10.1214/23-AOS2327}.

\bibitem[Chandrasekher et~al.(2023)Chandrasekher, Pananjady, and Thrampoulidis]{chandrasekher_sharp_2023}
Kabir~Aladin Chandrasekher, Ashwin Pananjady, and Christos Thrampoulidis.
\newblock Sharp global convergence guarantees for iterative nonconvex optimization with random data.
\newblock \emph{The Annals of Statistics}, 51\penalty0 (1):\penalty0 179--210, February 2023.
\newblock ISSN 0090-5364, 2168-8966.
\newblock \doi{10.1214/22-AOS2246}.

\bibitem[Chandrasekher et~al.(2024)Chandrasekher, Lou, and Pananjady]{chandrasekher_alternating_2024}
Kabir~Aladin Chandrasekher, Mengqi Lou, and Ashwin Pananjady.
\newblock Alternating minimization for generalized rank-1 matrix sensing: sharp predictions from a random initialization.
\newblock \emph{Information and Inference: A Journal of the IMA}, 13\penalty0 (3):\penalty0 iaae025, September 2024.
\newblock ISSN 2049-8772.
\newblock \doi{10.1093/imaiai/iaae025}.

\bibitem[Charisopoulos et~al.(2021)Charisopoulos, Chen, Davis, Díaz, Ding, and Drusvyatskiy]{charisopoulos_low-rank_2021}
Vasileios Charisopoulos, Yudong Chen, Damek Davis, Mateo Díaz, Lijun Ding, and Dmitriy Drusvyatskiy.
\newblock Low-{Rank} {Matrix} {Recovery} with {Composite} {Optimization}: {Good} {Conditioning} and {Rapid} {Convergence}.
\newblock \emph{Foundations of Computational Mathematics}, 21\penalty0 (6):\penalty0 1505--1593, December 2021.
\newblock ISSN 1615-3383.
\newblock \doi{10.1007/s10208-020-09490-9}.

\bibitem[Chen and Yang(2021)]{chen_hansonwright_2021}
Xiaohui Chen and Yun Yang.
\newblock Hanson–{Wright} inequality in {Hilbert} spaces with application to \${K}\$-means clustering for non-{Euclidean} data.
\newblock \emph{Bernoulli}, 27\penalty0 (1):\penalty0 586--614, February 2021.
\newblock ISSN 1350-7265.
\newblock \doi{10.3150/20-BEJ1251}.

\bibitem[Chen et~al.(2019{\natexlab{a}})Chen, Chi, Fan, and Ma]{chen_gradient_2019}
Yuxin Chen, Yuejie Chi, Jianqing Fan, and Cong Ma.
\newblock Gradient {Descent} with {Random} {Initialization}: {Fast} {Global} {Convergence} for {Nonconvex} {Phase} {Retrieval}.
\newblock \emph{Mathematical Programming}, 176\penalty0 (1-2):\penalty0 5--37, July 2019{\natexlab{a}}.
\newblock ISSN 0025-5610, 1436-4646.
\newblock \doi{10.1007/s10107-019-01363-6}.

\bibitem[Chen et~al.(2019{\natexlab{b}})Chen, Fan, Ma, and Yan]{chen_inference_2019}
Yuxin Chen, Jianqing Fan, Cong Ma, and Yuling Yan.
\newblock Inference and {Uncertainty} {Quantification} for {Noisy} {Matrix} {Completion}.
\newblock \emph{Proceedings of the National Academy of Sciences}, 116\penalty0 (46):\penalty0 22931--22937, November 2019{\natexlab{b}}.
\newblock ISSN 0027-8424, 1091-6490.
\newblock \doi{10.1073/pnas.1910053116}.

\bibitem[Chen et~al.(2020)Chen, Chi, Fan, Ma, and Yan]{chen_noisy_2020}
Yuxin Chen, Yuejie Chi, Jianqing Fan, Cong Ma, and Yuling Yan.
\newblock Noisy {Matrix} {Completion}: {Understanding} {Statistical} {Guarantees} for {Convex} {Relaxation} via {Nonconvex} {Optimization}.
\newblock \emph{SIAM Journal on Optimization}, 30\penalty0 (4):\penalty0 3098--3121, January 2020.
\newblock ISSN 1052-6234.
\newblock \doi{10.1137/19M1290000}.

\bibitem[Chen et~al.(2021{\natexlab{a}})Chen, Chi, Fan, and Ma]{chen_spectral_2021}
Yuxin Chen, Yuejie Chi, Jianqing Fan, and Cong Ma.
\newblock Spectral {Methods} for {Data} {Science}: {A} {Statistical} {Perspective}.
\newblock \emph{Foundations and Trends® in Machine Learning}, 14\penalty0 (5):\penalty0 566--806, 2021{\natexlab{a}}.
\newblock ISSN 1935-8237, 1935-8245.
\newblock \doi{10.1561/2200000079}.

\bibitem[Chen et~al.(2021{\natexlab{b}})Chen, Fan, Ma, and Yan]{chen_bridging_2021}
Yuxin Chen, Jianqing Fan, Cong Ma, and Yuling Yan.
\newblock Bridging convex and nonconvex optimization in robust {PCA}: {Noise}, outliers and missing data.
\newblock \emph{The Annals of Statistics}, 49\penalty0 (5):\penalty0 2948--2971, October 2021{\natexlab{b}}.
\newblock ISSN 0090-5364, 2168-8966.
\newblock \doi{10.1214/21-AOS2066}.

\bibitem[Chen et~al.(2021{\natexlab{c}})Chen, Fan, Wang, and Yan]{chen_convex_2021}
Yuxin Chen, Jianqing Fan, Bingyan Wang, and Yuling Yan.
\newblock Convex and {Nonconvex} {Optimization} {Are} {Both} {Minimax}-{Optimal} for {Noisy} {Blind} {Deconvolution} {Under} {Random} {Designs}.
\newblock \emph{Journal of the American Statistical Association}, 0\penalty0 (0):\penalty0 1--11, July 2021{\natexlab{c}}.
\newblock ISSN 0162-1459.
\newblock \doi{10.1080/01621459.2021.1956501}.

\bibitem[Chi et~al.(2019)Chi, Lu, and Chen]{chi_nonconvex_2019}
Yuejie Chi, Yue~M. Lu, and Yuxin Chen.
\newblock Nonconvex {Optimization} {Meets} {Low}-{Rank} {Matrix} {Factorization}: {An} {Overview}.
\newblock \emph{IEEE Transactions on Signal Processing}, 67\penalty0 (20):\penalty0 5239--5269, October 2019.
\newblock ISSN 1053-587X, 1941-0476.
\newblock \doi{10.1109/TSP.2019.2937282}.

\bibitem[Donhauser et~al.(2022)Donhauser, Ruggeri, Stojanovic, and Yang]{donhauser_fast_2022}
Konstantin Donhauser, Nicolò Ruggeri, Stefan Stojanovic, and Fanny Yang.
\newblock Fast rates for noisy interpolation require rethinking the effect of inductive bias.
\newblock In \emph{Proceedings of the 39th {International} {Conference} on {Machine} {Learning}}, pages 5397--5428. PMLR, June 2022.
\newblock ISSN: 2640-3498.

\bibitem[Donoho and Montanari(2016)]{donoho_high_2016}
David Donoho and Andrea Montanari.
\newblock High dimensional robust {M}-estimation: asymptotic variance via approximate message passing.
\newblock \emph{Probability Theory and Related Fields}, 166\penalty0 (3):\penalty0 935--969, December 2016.
\newblock ISSN 1432-2064.
\newblock \doi{10.1007/s00440-015-0675-z}.

\bibitem[Dudeja et~al.(2023)Dudeja, Lu, and Sen]{dudeja_universality_2023}
Rishabh Dudeja, Yue~M. Lu, and Subhabrata Sen.
\newblock Universality of approximate message passing with semirandom matrices.
\newblock \emph{The Annals of Probability}, 51\penalty0 (5):\penalty0 1616--1683, September 2023.
\newblock ISSN 0091-1798, 2168-894X.
\newblock \doi{10.1214/23-AOP1628}.

\bibitem[Dudeja et~al.(2024)Dudeja, Sen, and Lu]{dudeja_spectral_2024}
Rishabh Dudeja, Subhabrata Sen, and Yue~M. Lu.
\newblock Spectral {Universality} in {Regularized} {Linear} {Regression} {With} {Nearly} {Deterministic} {Sensing} {Matrices}.
\newblock \emph{IEEE Transactions on Information Theory}, 70\penalty0 (11):\penalty0 7923--7951, November 2024.
\newblock ISSN 1557-9654.
\newblock \doi{10.1109/TIT.2024.3458953}.

\bibitem[El~Karoui et~al.(2013)El~Karoui, Bean, Bickel, Lim, and Yu]{el_karoui_robust_2013}
Noureddine El~Karoui, Derek Bean, Peter~J. Bickel, Chinghway Lim, and Bin Yu.
\newblock On robust regression with high-dimensional predictors.
\newblock \emph{Proceedings of the National Academy of Sciences}, 110\penalty0 (36):\penalty0 14557--14562, September 2013.
\newblock \doi{10.1073/pnas.1307842110}.

\bibitem[Gavish and Donoho(2014)]{gavish_optimal_2014}
Matan Gavish and David~L. Donoho.
\newblock The {Optimal} {Hard} {Threshold} for {Singular} {Values} is 4/{\textbackslash}sqrt 3.
\newblock \emph{IEEE Transactions on Information Theory}, 60\penalty0 (8):\penalty0 5040--5053, August 2014.
\newblock ISSN 1557-9654.
\newblock \doi{10.1109/TIT.2014.2323359}.

\bibitem[Ge et~al.(2017)Ge, Jin, and Zheng]{ge_no_2017}
Rong Ge, Chi Jin, and Yi~Zheng.
\newblock No {Spurious} {Local} {Minima} in {Nonconvex} {Low} {Rank} {Problems}: {A} {Unified} {Geometric} {Analysis}.
\newblock In \emph{Proceedings of the 34th {International} {Conference} on {Machine} {Learning}}, pages 1233--1242. PMLR, July 2017.
\newblock ISSN: 2640-3498.

\bibitem[Haeffele et~al.(2014)Haeffele, Young, and Vidal]{haeffele2014structured}
Benjamin Haeffele, Eric Young, and Rene Vidal.
\newblock Structured low-rank matrix factorization: Optimality, algorithm, and applications to image processing.
\newblock In \emph{International conference on machine learning}, pages 2007--2015. PMLR, 2014.

\bibitem[Han(2023)]{han_noisy_2023}
Qiyang Han.
\newblock Noisy linear inverse problems under convex constraints: {Exact} risk asymptotics in high dimensions.
\newblock \emph{The Annals of Statistics}, 51\penalty0 (4):\penalty0 1611--1638, August 2023.
\newblock ISSN 0090-5364, 2168-8966.
\newblock \doi{10.1214/23-AOS2301}.

\bibitem[Han(2024)]{han_entrywise_2024}
Qiyang Han.
\newblock Entrywise dynamics and universality of general first order methods, June 2024.
\newblock arXiv:2406.19061 [cs, math, stat].

\bibitem[Han and Shen(2023)]{han_universality_2023}
Qiyang Han and Yandi Shen.
\newblock Universality of regularized regression estimators in high dimensions.
\newblock \emph{The Annals of Statistics}, 51\penalty0 (4):\penalty0 1799--1823, August 2023.
\newblock ISSN 0090-5364, 2168-8966.
\newblock \doi{10.1214/23-AOS2309}.

\bibitem[Han and Xu(2023)]{han_distribution_2023}
Qiyang Han and Xiaocong Xu.
\newblock The distribution of {Ridgeless} least squares interpolators, July 2023.
\newblock arXiv:2307.02044 [cs, math, stat].

\bibitem[Javanmard and Montanari(2014{\natexlab{a}})]{javanmard_confidence_2014}
Adel Javanmard and Andrea Montanari.
\newblock Confidence {Intervals} and {Hypothesis} {Testing} for {High}-{Dimensional} {Regression}.
\newblock \emph{Journal of Machine Learning Research}, 15\penalty0 (82):\penalty0 2869--2909, 2014{\natexlab{a}}.
\newblock ISSN 1533-7928.

\bibitem[Javanmard and Montanari(2014{\natexlab{b}})]{javanmard_hypothesis_2014}
Adel Javanmard and Andrea Montanari.
\newblock Hypothesis {Testing} in {High}-{Dimensional} {Regression} {Under} the {Gaussian} {Random} {Design} {Model}: {Asymptotic} {Theory}.
\newblock \emph{IEEE Transactions on Information Theory}, 60\penalty0 (10):\penalty0 6522--6554, October 2014{\natexlab{b}}.
\newblock ISSN 1557-9654.
\newblock \doi{10.1109/TIT.2014.2343629}.

\bibitem[Javanmard and Montanari(2018)]{javanmard_debiasing_2018}
Adel Javanmard and Andrea Montanari.
\newblock Debiasing the lasso: {Optimal} sample size for {Gaussian} designs.
\newblock \emph{The Annals of Statistics}, 46\penalty0 (6A):\penalty0 2593--2622, December 2018.
\newblock ISSN 0090-5364, 2168-8966.
\newblock \doi{10.1214/17-AOS1630}.

\bibitem[Javanmard and Soltanolkotabi(2022)]{javanmard_precise_2022}
Adel Javanmard and Mahdi Soltanolkotabi.
\newblock Precise statistical analysis of classification accuracies for adversarial training.
\newblock \emph{The Annals of Statistics}, 50\penalty0 (4):\penalty0 2127--2156, August 2022.
\newblock ISSN 0090-5364, 2168-8966.
\newblock \doi{10.1214/22-AOS2180}.

\bibitem[Javanmard et~al.(2020)Javanmard, Soltanolkotabi, and Hassani]{javanmard_precise_2020}
Adel Javanmard, Mahdi Soltanolkotabi, and Hamed Hassani.
\newblock Precise {Tradeoffs} in {Adversarial} {Training} for {Linear} {Regression}.
\newblock In \emph{Proceedings of {Thirty} {Third} {Conference} on {Learning} {Theory}}, pages 2034--2078. PMLR, July 2020.
\newblock ISSN: 2640-3498.

\bibitem[Koriyama and Bellec(2025)]{koriyama_phase_2025}
Takuya Koriyama and Pierre~C. Bellec.
\newblock Phase transitions for the existence of unregularized {M}-estimators in single index models, May 2025.
\newblock arXiv:2501.03163 [math].

\bibitem[Laurent and Massart(2000)]{laurent_adaptive_2000}
B.~Laurent and P.~Massart.
\newblock Adaptive estimation of a quadratic functional by model selection.
\newblock \emph{The Annals of Statistics}, 28\penalty0 (5):\penalty0 1302--1338, October 2000.
\newblock ISSN 0090-5364, 2168-8966.
\newblock \doi{10.1214/aos/1015957395}.

\bibitem[Li et~al.(2019)Li, Lu, Arora, Haupt, Liu, Wang, and Zhao]{li_symmetry_2019}
Xingguo Li, Junwei Lu, Raman Arora, Jarvis Haupt, Han Liu, Zhaoran Wang, and Tuo Zhao.
\newblock Symmetry, {Saddle} {Points}, and {Global} {Optimization} {Landscape} of {Nonconvex} {Matrix} {Factorization}.
\newblock \emph{IEEE Transactions on Information Theory}, 65\penalty0 (6):\penalty0 3489--3514, June 2019.
\newblock ISSN 1557-9654.
\newblock \doi{10.1109/TIT.2019.2898663}.

\bibitem[Liang and Sur(2022)]{liang_precise_2022}
Tengyuan Liang and Pragya Sur.
\newblock A precise high-dimensional asymptotic theory for boosting and minimum-ell\_1-norm interpolated classifiers.
\newblock \emph{The Annals of Statistics}, 50\penalty0 (3):\penalty0 1669--1695, June 2022.
\newblock ISSN 0090-5364, 2168-8966.
\newblock \doi{10.1214/22-AOS2170}.

\bibitem[Loureiro et~al.(2021)Loureiro, Gerbelot, Cui, Goldt, Krzakala, Mezard, and Zdeborová]{loureiro_learning_2021}
Bruno Loureiro, Cedric Gerbelot, Hugo Cui, Sebastian Goldt, Florent Krzakala, Marc Mezard, and Lenka Zdeborová.
\newblock Learning curves of generic features maps for realistic datasets with a teacher-student model.
\newblock In \emph{Advances in {Neural} {Information} {Processing} {Systems}}, volume~34, pages 18137--18151. Curran Associates, Inc., 2021.

\bibitem[Luo and Trillos(2022)]{luo_nonconvex_2022}
Yuetian Luo and Nicolas~Garcia Trillos.
\newblock Nonconvex {Matrix} {Factorization} is {Geodesically} {Convex}: {Global} {Landscape} {Analysis} for {Fixed}-rank {Matrix} {Optimization} {From} a {Riemannian} {Perspective}, November 2022.
\newblock arXiv:2209.15130 [math].

\bibitem[Ma et~al.(2020)Ma, Wang, Chi, and Chen]{ma_implicit_2020}
Cong Ma, Kaizheng Wang, Yuejie Chi, and Yuxin Chen.
\newblock Implicit {Regularization} in {Nonconvex} {Statistical} {Estimation}: {Gradient} {Descent} {Converges} {Linearly} for {Phase} {Retrieval}, {Matrix} {Completion}, and {Blind} {Deconvolution}.
\newblock \emph{Foundations of Computational Mathematics}, 20\penalty0 (3):\penalty0 451--632, June 2020.
\newblock ISSN 1615-3375, 1615-3383.
\newblock \doi{10.1007/s10208-019-09429-9}.

\bibitem[Ma et~al.(2021)Ma, Li, and Chi]{ma_beyond_2021}
Cong Ma, Yuanxin Li, and Yuejie Chi.
\newblock Beyond {Procrustes}: {Balancing}-{Free} {Gradient} {Descent} for {Asymmetric} {Low}-{Rank} {Matrix} {Sensing}.
\newblock \emph{IEEE Transactions on Signal Processing}, 69:\penalty0 867--877, 2021.
\newblock ISSN 1941-0476.
\newblock \doi{10.1109/TSP.2021.3051425}.

\bibitem[Ma and Fattahi(2023)]{ma_optimization_2023}
Jianhao Ma and Salar Fattahi.
\newblock On the {Optimization} {Landscape} of {Burer}-{Monteiro} {Factorization}: {When} do {Global} {Solutions} {Correspond} to {Ground} {Truth}?, February 2023.
\newblock arXiv:2302.10963 [cs, math].

\bibitem[Ma et~al.(2023)Ma, Bi, Lavaei, and Sojoudi]{ma_geometric_2023}
Ziye Ma, Yingjie Bi, Javad Lavaei, and Somayeh Sojoudi.
\newblock Geometric {Analysis} of {Noisy} {Low}-{Rank} {Matrix} {Recovery} in the {Exact} {Parametrized} and the {Overparametrized} {Regimes}.
\newblock \emph{INFORMS Journal on Optimization}, April 2023.
\newblock ISSN 2575-1484.
\newblock \doi{10.1287/ijoo.2023.0090}.

\bibitem[Miolane and Montanari(2021)]{miolane_distribution_2021}
Léo Miolane and Andrea Montanari.
\newblock The distribution of the {Lasso}: {Uniform} control over sparse balls and adaptive parameter tuning.
\newblock \emph{The Annals of Statistics}, 49\penalty0 (4):\penalty0 2313--2335, August 2021.
\newblock ISSN 0090-5364, 2168-8966.
\newblock \doi{10.1214/20-AOS2038}.

\bibitem[Montanari et~al.(2023)Montanari, Ruan, Saeed, and Sohn]{montanari_universality_2023}
Andrea Montanari, Feng Ruan, Basil Saeed, and Youngtak Sohn.
\newblock Universality of max-margin classifiers, September 2023.
\newblock arXiv:2310.00176 [math, stat].

\bibitem[Montanari et~al.(2025)Montanari, Ruan, Sohn, and Yan]{montanari_generalization_2025}
Andrea Montanari, Feng Ruan, Youngtak Sohn, and Jun Yan.
\newblock The generalization error of max-margin linear classifiers: {Benign} overfitting and high dimensional asymptotics in the overparametrized regime.
\newblock \emph{The Annals of Statistics}, 53\penalty0 (2):\penalty0 822--853, April 2025.
\newblock ISSN 0090-5364, 2168-8966.
\newblock \doi{10.1214/25-AOS2489}.

\bibitem[Negahban and Wainwright(2011)]{negahban_estimation_2011}
Sahand Negahban and Martin~J. Wainwright.
\newblock Estimation of (near) low-rank matrices with noise and high-dimensional scaling.
\newblock \emph{Annals of Statistics}, 39\penalty0 (2):\penalty0 1069--1097, April 2011.
\newblock ISSN 0090-5364, 2168-8966.
\newblock \doi{10.1214/10-AOS850}.

\bibitem[Salehi et~al.(2019)Salehi, Abbasi, and Hassibi]{salehi_impact_2019}
Fariborz Salehi, Ehsan Abbasi, and Babak Hassibi.
\newblock The {Impact} of {Regularization} on {High}-dimensional {Logistic} {Regression}.
\newblock In \emph{Advances in {Neural} {Information} {Processing} {Systems}}, volume~32. Curran Associates, Inc., 2019.

\bibitem[Soltanolkotabi et~al.(2025)Soltanolkotabi, Stöger, and Xie]{soltanolkotabi_implicit_2025}
Mahdi Soltanolkotabi, Dominik Stöger, and Changzhi Xie.
\newblock Implicit {Balancing} and {Regularization}: {Generalization} and {Convergence} {Guarantees} for {Overparameterized} {Asymmetric} {Matrix} {Sensing}.
\newblock \emph{IEEE Transactions on Information Theory}, 71\penalty0 (4):\penalty0 2991--3037, April 2025.
\newblock ISSN 1557-9654.
\newblock \doi{10.1109/TIT.2025.3530335}.

\bibitem[Stojanovic et~al.(2024)Stojanovic, Donhauser, and Yang]{stojanovic_tight_2024}
Stefan Stojanovic, Konstantin Donhauser, and Fanny Yang.
\newblock Tight bounds for maximum \${\textbackslash}ell\_1\$-margin classifiers.
\newblock In \emph{Proceedings of {The} 35th {International} {Conference} on {Algorithmic} {Learning} {Theory}}, pages 1055--1112. PMLR, March 2024.
\newblock ISSN: 2640-3498.

\bibitem[Stöger and Soltanolkotabi(2021)]{stoger_small_2021}
Dominik Stöger and Mahdi Soltanolkotabi.
\newblock Small random initialization is akin to spectral learning: {Optimization} and generalization guarantees for overparameterized low-rank matrix reconstruction.
\newblock In \emph{Advances in {Neural} {Information} {Processing} {Systems}}, volume~34, pages 23831--23843. Curran Associates, Inc., 2021.

\bibitem[Taheri et~al.(2020)Taheri, Pedarsani, and Thrampoulidis]{taheri_sharp_2020}
Hossein Taheri, Ramtin Pedarsani, and Christos Thrampoulidis.
\newblock Sharp {Asymptotics} and {Optimal} {Performance} for {Inference} in {Binary} {Models}.
\newblock In \emph{Proceedings of the {Twenty} {Third} {International} {Conference} on {Artificial} {Intelligence} and {Statistics}}, pages 3739--3749. PMLR, June 2020.
\newblock ISSN: 2640-3498.

\bibitem[Taheri et~al.(2021)Taheri, Pedarsani, and Thrampoulidis]{taheri_fundamental_2021}
Hossein Taheri, Ramtin Pedarsani, and Christos Thrampoulidis.
\newblock Fundamental {Limits} of {Ridge}-{Regularized} {Empirical} {Risk} {Minimization} in {High} {Dimensions}.
\newblock In \emph{Proceedings of {The} 24th {International} {Conference} on {Artificial} {Intelligence} and {Statistics}}, pages 2773--2781. PMLR, March 2021.
\newblock ISSN: 2640-3498.

\bibitem[Thrampoulidis et~al.(2015{\natexlab{a}})Thrampoulidis, Oymak, and Hassibi]{thrampoulidis_gaussian_2015}
Christos Thrampoulidis, Samet Oymak, and Babak Hassibi.
\newblock The {Gaussian} min-max theorem in the {Presence} of {Convexity}, March 2015{\natexlab{a}}.
\newblock arXiv:1408.4837 [cs, math].

\bibitem[Thrampoulidis et~al.(2015{\natexlab{b}})Thrampoulidis, Oymak, and Hassibi]{thrampoulidis_regularized_2015}
Christos Thrampoulidis, Samet Oymak, and Babak Hassibi.
\newblock Regularized {Linear} {Regression}: {A} {Precise} {Analysis} of the {Estimation} {Error}.
\newblock In \emph{Proceedings of {The} 28th {Conference} on {Learning} {Theory}}, pages 1683--1709. PMLR, June 2015{\natexlab{b}}.
\newblock ISSN: 1938-7228.

\bibitem[Thrampoulidis et~al.(2018)Thrampoulidis, Abbasi, and Hassibi]{thrampoulidis_precise_2018}
Christos Thrampoulidis, Ehsan Abbasi, and Babak Hassibi.
\newblock Precise {Error} {Analysis} of {Regularized} {M} -{Estimators} in {High} {Dimensions}.
\newblock \emph{IEEE Transactions on Information Theory}, 64\penalty0 (8):\penalty0 5592--5628, August 2018.
\newblock ISSN 1557-9654.
\newblock \doi{10.1109/TIT.2018.2840720}.

\bibitem[Tong et~al.(2021)Tong, Ma, and Chi]{tong_low-rank_2021}
Tian Tong, Cong Ma, and Yuejie Chi.
\newblock Low-{Rank} {Matrix} {Recovery} {With} {Scaled} {Subgradient} {Methods}: {Fast} and {Robust} {Convergence} {Without} the {Condition} {Number}.
\newblock \emph{IEEE Transactions on Signal Processing}, 69:\penalty0 2396--2409, 2021.
\newblock ISSN 1941-0476.
\newblock \doi{10.1109/TSP.2021.3071560}.

\bibitem[Tu et~al.(2016)Tu, Boczar, Simchowitz, Soltanolkotabi, and Recht]{tu_low-rank_2016}
Stephen Tu, Ross Boczar, Max Simchowitz, Mahdi Soltanolkotabi, and Ben Recht.
\newblock Low-rank {Solutions} of {Linear} {Matrix} {Equations} via {Procrustes} {Flow}.
\newblock In \emph{Proceedings of {The} 33rd {International} {Conference} on {Machine} {Learning}}, pages 964--973. PMLR, June 2016.
\newblock ISSN: 1938-7228.

\bibitem[Wang et~al.(2025)Wang, , and Fan]{wang_robust_2025}
Bingyan Wang, , and Jianqing Fan.
\newblock Robust {Matrix} {Completion} with {Heavy}-{Tailed} {Noise}.
\newblock \emph{Journal of the American Statistical Association}, 120\penalty0 (550):\penalty0 922--934, April 2025.
\newblock ISSN 0162-1459.
\newblock \doi{10.1080/01621459.2024.2375037}.

\bibitem[Wang et~al.(2022)Wang, Donhauser, and Yang]{wang_tight_2022}
Guillaume Wang, Konstantin Donhauser, and Fanny Yang.
\newblock Tight bounds for minimum \${\textbackslash}ell\_1\$-norm interpolation of noisy data.
\newblock In \emph{Proceedings of {The} 25th {International} {Conference} on {Artificial} {Intelligence} and {Statistics}}, pages 10572--10602. PMLR, May 2022.
\newblock ISSN: 2640-3498.

\bibitem[Xia(2019)]{xia_confidence_2019}
Dong Xia.
\newblock Confidence {Region} of {Singular} {Subspaces} for {Low}-{Rank} {Matrix} {Regression}.
\newblock \emph{IEEE Transactions on Information Theory}, 65\penalty0 (11):\penalty0 7437--7459, November 2019.
\newblock ISSN 1557-9654.
\newblock \doi{10.1109/TIT.2019.2924900}.

\bibitem[Xu et~al.(2023)Xu, Shen, Chi, and Ma]{xu_power_2023}
Xingyu Xu, Yandi Shen, Yuejie Chi, and Cong Ma.
\newblock The {Power} of {Preconditioning} in {Overparameterized} {Low}-{Rank} {Matrix} {Sensing}.
\newblock In \emph{Proceedings of the 40th {International} {Conference} on {Machine} {Learning}}, pages 38611--38654. PMLR, July 2023.
\newblock ISSN: 2640-3498.

\bibitem[Yang and Ma(2023)]{yang_optimal_2023}
Yuepeng Yang and Cong Ma.
\newblock Optimal {Tuning}-{Free} {Convex} {Relaxation} for {Noisy} {Matrix} {Completion}.
\newblock \emph{IEEE Transactions on Information Theory}, 69\penalty0 (10):\penalty0 6571--6585, October 2023.
\newblock ISSN 1557-9654.
\newblock \doi{10.1109/TIT.2023.3284341}.

\bibitem[Zhang and Zhang(2014)]{zhang_confidence_2014}
Cun-Hui Zhang and Stephanie~S. Zhang.
\newblock Confidence {Intervals} for {Low} {Dimensional} {Parameters} in {High} {Dimensional} {Linear} {Models}.
\newblock \emph{Journal of the Royal Statistical Society Series B: Statistical Methodology}, 76\penalty0 (1):\penalty0 217--242, January 2014.
\newblock ISSN 1369-7412.
\newblock \doi{10.1111/rssb.12026}.

\bibitem[Zhang et~al.(2023)Zhang, Chiu, and Zhang]{zhang_fast_2023}
Gavin Zhang, Hong-Ming Chiu, and Richard~Y. Zhang.
\newblock Fast and {Minimax} {Optimal} {Estimation} of {Low}-{Rank} {Matrices} via {Non}-{Convex} {Gradient} {Descent}, May 2023.
\newblock arXiv:2305.17224 [cs, math, stat].

\bibitem[Zhang(2021)]{zhang_sharp_2021}
Richard~Y. Zhang.
\newblock Sharp {Global} {Guarantees} for {Nonconvex} {Low}-{Rank} {Matrix} {Recovery} in the {Overparameterized} {Regime}, April 2021.
\newblock arXiv:2104.10790 [math].

\bibitem[Zheng and Lafferty(2015)]{zheng_convergent_2015}
Qinqing Zheng and John Lafferty.
\newblock A {Convergent} {Gradient} {Descent} {Algorithm} for {Rank} {Minimization} and {Semidefinite} {Programming} from {Random} {Linear} {Measurements}.
\newblock In \emph{Advances in {Neural} {Information} {Processing} {Systems}}, volume~28. Curran Associates, Inc., 2015.

\end{thebibliography}

\end{document}